\documentclass[11pt,a4paper]{amsart} 

\usepackage[T1]{fontenc}
 \usepackage[english]{babel}
 \usepackage[latin1]{inputenc}
\usepackage[dvips]{graphics}
\usepackage{xcolor}
\usepackage{url}
\usepackage{stmaryrd}
\usepackage{rotating}
\usepackage{longtable}
\usepackage{multicol}
\usepackage{etoolbox}

\usepackage{amsmath}
\usepackage{amsthm}
\usepackage{amssymb}
\usepackage{bbm}
\usepackage{enumerate}
\usepackage[shortlabels]{enumitem} \setlist{leftmargin=5.5mm}
\usepackage[all,poly,necula]{xy} 
\theoremstyle{plain}
\newtheorem{theorem}{Theorem}[section]
\newtheorem{theoremi}{Theorem}

\newtheorem{proposition}[theorem]{Proposition}
\newtheorem{lemma}[theorem]{Lemma}
\newtheorem{corollary}[theorem]{Corollary}

\theoremstyle{definition}
\newtheorem{definition}[theorem]{Definition}
\newtheorem{example}[theorem]{Example}

\theoremstyle{remark}
\newtheorem{remark}[theorem]{Remark}

\DeclareMathOperator{\id}{Id}

\DeclareMathOperator{\ima}{im}

\DeclareMathOperator{\Hom}{Hom}

\DeclareMathOperator{\End}{End}

\DeclareMathOperator{\add}{add}

\DeclareMathOperator{\op}{op}

\DeclareMathOperator{\soc}{soc}

\DeclareMathOperator{\Kb}{K^b}
\DeclareMathOperator{\proj}{proj}

\newenvironment{sbmatrix}{\left[\begin{smallmatrix}}{\end{smallmatrix}\right]}

\renewcommand{\leq}{\leqslant}
\renewcommand{\geq}{\geqslant}

\newcommand{\Br}{\mathcal{B}}

\newcommand{\Er}{\mathcal{E}}

\renewcommand{\Pr}{\mathcal{P}}
\newcommand{\M}{\mathbb{M}}
\newcommand{\N}{\mathbb{N}}
\newcommand{\Z}{\mathbb{Z}}
\newcommand{\Q}{\mathbb{Q}}
\newcommand{\R}{\mathbb{R}}

\newcommand{\e}{\mathbf{e}}

\renewcommand{\P}{\mathbb{P}}

\renewcommand{\epsilon}{\varepsilon}

\newcommand{\RR}{\mathcal{R}}
\newcommand{\tens}{\otimes}

\renewcommand{\phi}{\varphi}

\renewcommand{\bar}[1]{\overline{#1}}

\newcommand{\comment}[1]{}

\newcommand{\Gs}{\Gamma_\sigma}
\newcommand{\Gso}{\Gamma_\sigma^\circ}

\newcommand{\uQ}{\underline{\smash{Q}}}

\newcommand{\surj}{\twoheadrightarrow}

\renewcommand{\tilde}[1]{\smash{\widetilde{#1}}}

\definecolor{orange}{rgb}{1,0.90,0.5}
\let\oldmarginpar\marginpar
\renewcommand\marginpar[1]{(**)\oldmarginpar[\raggedleft\footnotesize \fcolorbox{blue}{orange}{\parbox{\marginparwidth}{\color{blue}{(**) #1}}}]%
{\raggedright\footnotesize \fcolorbox{blue}{orange}{\parbox{\marginparwidth}{\color{blue}{(**) #1}}}}}

\newcommand{\forloop}[5][1] { \setcounter{#2}{#3} \ifthenelse{#4} { #5 \addtocounter{#2}{#1} \forloop[#1]{#2}{\value{#2}}{#4}{#5} }{ } } 
\newcounter{i}
\newcounter{j}
\newcommand{\dets}[4]{
    \displaystyle \left|\;
    \begin{matrix}
      \forloop{i}{1}{\value{i} < 7}{
       \forloop{j}{1}{\value{j} < 7}{
        \ifthenelse{\value{i} < #1 \or \value{i} > #2 \or \value{j} < #3 \or \value{j} > #4}{\ifthenelse{\value{i}<\value{j}}{\cdot}{\ifthenelse{\value{i} = \value{j}}{1}{}}}{\bullet} \ifthenelse{\value{j}<6}{&}{}
       } \ifthenelse{\value{i}<6}{\\}{}
      }
    \end{matrix}
    \;\right|^{\vphantom{M^{M^M}}}_{\vphantom{M_{M_M}}}}

\makeatletter
\def\overarrowb@#1#2#3{\vbox{\vspace*{-.3ex}\ialign{##\crcr\vspace*{-.3ex}#1#2\crcr
 \noalign{\nointerlineskip}$\m@th\hfil#2#3\hfil$\crcr}}}
\newcommand{\amsvect}{%
  \mathpalette{\overarrowb@\rightarrowfill@}}
\makeatother

\newcommand{\io}[1]{\mathopen] #1 \mathclose[}
\newcommand{\ic}[1]{\mathopen[ #1 \mathclose]}
\newcommand{\vct}[1]{\amsvect{#1}}

\newcommand{\rs}[1]{\setcounter{enumi}{#1}}
\renewcommand{\mod}{\operatorname{mod}}

\newsavebox\locboxinminipage
\newlength\locboxinminipagel
\newcommand{\boxinminipage}[1]
{%
 \sbox\locboxinminipage{#1}%
 \settowidth\locboxinminipagel{\usebox{\locboxinminipage}}%
 \begin{minipage}{\locboxinminipagel}\usebox{\locboxinminipage}\end{minipage}%
}

\newenvironment{itemizedec}{
  \begin{itemize}[labelindent=1cm,leftmargin=*]}{\end{itemize}}

\AtBeginEnvironment{figure}{\stepcounter{theorem}}
\AtBeginEnvironment{equation}{\stepcounter{theorem}}

\begin{document}


\setlength{\parindent}{5mm}
\title[Algebras of partial triangulations]{Algebras of partial triangulations}
\author{Laurent Demonet}
\address{Graduate School of Mathematics, Nagoya University, Furo-cho, Chikusa-ku, 464-8602 Nagoya, Japan}
\email{Laurent.Demonet@normalesup.org}

\date{}

\begin{abstract}
 We introduce two classes of algebras coming from partial triangulations of marked surfaces. The first one, called frozen algebra of a partial triangulation, is generally of infinite rank and contains frozen Jacobian algebras of triangulations of marked surfaces. The second one, called algebra of a partial triangulation, is always of (explicit) finite rank and contains classical Jacobian algebras of triangulations of marked surfaces and Brauer graph algebras. We classify the partial triangulations, depending on the complexity of their frozen algebras (some are free of finite rank, some are lattices over a formal power series ring and most of them are not finitely generated over their centre). For algebras of partial triangulations, we prove that they are symmetric when the surface has no boundary. From a more representation theoretical point of view, we prove that these algebras of partial triangulations are of tame representation type and we define a combinatorial operation on partial triangulation, generalizing Kauer moves of Brauer graphs and flips of triangulations, which give derived equivalences of the corresponding algebras.
\end{abstract}

\maketitle

\setcounter{tocdepth}{1}
\tableofcontents

\section{Introduction}

The aim of this paper is to introduce two new classes of algebras, called \emph{frozen} and \emph{algebras of a partial triangulation}, generalizing Brauer graph algebras on the one hand and Jacobian algebras coming from triangulations of surfaces on the other hand. Then we give some properties of these algebras which justify the interest of their study.

In the forties, Brauer introduced \emph{Brauer tree algebras} which are of finite representation type. These algebras have been then generalized by various authors to \emph{Brauer graph algebras}. Brauer graph algebras are defined from the combinatorial datum of a ribbon graph and have many nice properties: they are finite dimensional and symmetric, they are of tame representation type with completely classified modules, and their tilting theory is well understood. For more details about Brauer graph algebras, see for example \cite{AdAiCh, Ka98, Ro98, WaWa85}. Some generalizations, going in different direction than this paper, have already been proposed (see for example \cite{GrSc}).

On the other hand, \emph{Jacobian algebras} have been introduced more recently \cite{DeWeZe08}, and in particular \emph{Jacobian algebras of triangulations of surfaces} \cite{La09, CeLa12}. These algebras are defined by using triangulations of oriented surfaces with marked points (with or without boundary). They share many nice properties with Brauer graph algebras. They are finite dimensional, they are symmetric when the surface have no boundary \cite{La}, they are tame \cite{GeLaSc16}. Moreover, in certain cases, their derived equivalence classes are understood \cite{La2}.

We will now give an overview of this paper. All along, $k$ is a commutative ring with unit. We fix a compact connected oriented surface $\Sigma$ with or without boundary and a non-empty finite set $\M$ of marked points. For each 
$M \in \M$, we fix $m_M \in \N_{>0}$ and $\lambda_M \in k$ invertible. We have to exclude few degenerated cases for simplicity (see beginning of Section \ref{algparttri}). A \emph{partial triangulation} $\sigma$ of $(\Sigma, \M)$ is, roughly speaking, a subset of a triangulation of $(\Sigma, \M)$. 

To any partial triangulation, in Section \ref{algparttri}, we associate a quiver $Q_\sigma$, the vertices of which are indexed by the edges of $\sigma$ and the arrows of which winds counter-clockwisely around marked points. Then we get the \emph{frozen algebra $\Gamma_\sigma = \Gamma_\sigma^\lambda$ associated with the partial triangulation $\sigma$} by factoring out some relations in the (complete) path algebra of $Q_\sigma$. 

In Section \ref{nonfroz}, we introduce the \emph{algebra $\Delta_\sigma = \Delta_\sigma^\lambda$ associated with $\sigma$}. It is the quotient of $\Gamma_\sigma$ by the ideal generated by the idempotent corresponding to the boundary of $\Sigma$. An important structural result about $\Gamma_\sigma$ and $\Delta_\sigma$, which permits to do inductive arguments, is the following one:
\begin{theoremi}[Theorem \ref{subtri} and Corollary \ref{subtri2}]
 Let $\tau \subset \sigma$. Then we have 
 $$\Gamma_\tau \cong e_\tau \Gamma_\sigma e_\tau \quad \text{and} \quad \Delta_\tau \cong e_\tau \Delta_\sigma e_\tau$$
 where $e_\tau$ is the idempotent of $\Gamma_\sigma$ or $\Delta_\sigma$ corresponding to arcs in $\tau$. 
\end{theoremi}

In Sections \ref{casetri} and \ref{brauer}, we give the following results, which can be seen as the first motivation to introduce $\Gamma_\sigma$ and $\Delta_\sigma$:
\begin{theoremi}[Theorems \ref{thmmin2} and \ref{thmbrauer}] \label{recov}
 \begin{enumerate}[\rm (1)]
  \item If all $m_M$ are invertible in $k$ and $\sigma$ is a triangulation, then $\Gamma_\sigma$ (respectively $\Delta_\sigma$) is the frozen (respectively classical) Jacobian algebra of a quiver with potential.
  \item If $m_M = 1$ for every marked point $M$ and $\sigma$ is a triangulation, then $\Gamma_\sigma$ and $\Delta_\sigma$ correspond to the quiver with potential introduced in \cite{La09}.
  \item If $\sigma$ is sparse (that is ``far'' from a triangulation, see Definition \ref{sparse}), then $\Delta_\sigma$ is the Brauer graph algebra of the ribbon graph underlying to $\sigma$.
  \item Every Brauer graph algebra is obtained from a sparse partial triangulation of a surface without boundary.
 \end{enumerate}
\end{theoremi}
In addition to Theorem \ref{recov}, notice that \emph{tiling algebras}, introduced in \cite{GaMC}, are also special cases of algebras of partial triangulations.
For the definition of a frozen Jacobian algebra of a quiver with potential, see for example \cite{BuIyReSm11, DeLu16, DeLu16-2} (see also Section \ref{casetri}). We use the expression ``classical Jacobian algebra'' to refer to usual Jacobian algebras of a quiver with potential as defined in \cite{DeWeZe08}.

In Section \ref{order}, we classify partial triangulations $\sigma$ in function of the complexity of $\Gamma_\sigma$. We prove that:
\begin{theoremi}[Theorem \ref{classiford2}]
 If $\sigma$ is connected, we have
 \begin{enumerate}[\rm (1)]
  \item $\Gamma_\sigma$ is free of finite rank as a $k$-module if $\sigma$ is not connected to the boundary of $\Sigma$.
  \item $\Gamma_\sigma$ is a lattice over $k \llbracket x \rrbracket$ (\emph{i.e.} free of finite rank) over $k \llbracket x \rrbracket$) if
   \begin{itemize}
    \item[] $\Sigma$ is a polygon with no puncture and at most one $m_M$ greater than $1$;
    \item[or] $\Sigma$ is a polygon with one puncture and $m_M = 1$ for all $M$ on the boundary;
    \item[or] all $M$ on the boundary of $\Sigma$ satisfy $m_M = 1$ and all arcs in $\sigma$ are homotopic to a part of the boundary.
   \end{itemize} 
  \item $\Gamma_\sigma$ is not finitely generated over its centre in any other case. 
 \end{enumerate}
\end{theoremi}
Notice that the second case generalizes slightly \cite{DeLu16, DeLu16-2}. Moreover, if $\sigma$ is not connected, the result can be applied independently to each connected component.

The rest of the paper is dedicated to prove a certain number of properties of $\Delta_\sigma$, already known for Brauer graph algebras and for some Jacobian algebras coming from triangulations. In Theorem \ref{basisDsig}, we give a basis of $\Delta_\sigma$ and we deduce:
\begin{theoremi}[Corollary \ref{dimd}]
 The $k$-algebra $\Delta_\sigma$ is a free $k$-module of rank
 $$\sum_{M \in \M \setminus \P} \frac{d_M(d_M-1)}{2} + \sum_{M \in \P} m_M d_M^2 + f$$
 where $\P \subset \M$ is the set of punctures (\emph{i.e.} non-boundary marked points), for $M \in \M$, $d_M$ is the degree of $M$ in the graph $\sigma$ (without counting boundary components), and $f$ is the number of arcs in $\sigma$ with both endpoints on boundaries. 
\end{theoremi}

We also get the following property:
\begin{theoremi}[Theorem \ref{symm}]
 If $\sigma$ has no arc incident to the boundary, then $\Delta_\sigma$ is a symmetric $k$-algebra (\emph{i.e.} $\Hom_k(\Delta_\sigma, k) \cong \Delta_\sigma$ as $\Delta_\sigma$-bimodules). 
\end{theoremi}


We then deal with two representation theoretical questions. In Section \ref{reptype}, we prove:
\begin{theoremi}[Theorem \ref{tame}]
 If $k$ is an algebraically closed field, then $\Delta_\sigma$ is tame (or representation-finite).
\end{theoremi}

We can even prove (Proposition \ref{qq}) that, if $\Sigma$ has no boundary and $\sigma$ is a triangulation, then $\Delta_\sigma$ is of \emph{quasi-quaternion type} in the sense of \cite{La3} (see also \cite{Er16, Sk16}).

In Section \ref{secflip}, we define flips $\mu_u(\sigma)$ of a partial triangulation $\sigma$ with respect to most arcs $u$ (see Definition \ref{defflip}) which generalize flips for triangulations and Kauer moves for Brauer graph algebras. We also define coefficients $\mu_u(\lambda)_M$ for $M \in \M$. Then we obtain:
\begin{theoremi}[Theorem \ref{thmtilt}]
 If $u$ is not close to the boundary (in the sense of Definition \ref{defflip}), then the algebras $\Delta_\sigma^\lambda$ and $\Delta^{\mu_u(\lambda)}_{\mu_u(\sigma)}$ are derived equivalent.
\end{theoremi}

Notice that most of arcs are not close to the boundary. Notice also that this theorem recovers the result for Brauer graph algebras, which is considered to be known (\cite{Ai15, Ka98}), but not proven completely. The term Kauer move was introduced in \cite{MaSc14}, where authors analyse the case of Brauer graph algebras coming from $m$-angulations of an oriented surface.


\section*{Acknowledgments}

The author would like to thank Osamu Iyama, Takahide Adachi and Aaron Chan for interesting discussions in relation to this subject. The author was partially supported by JSPS Grant-in-Aid for Young Scientist (B) 26800008.

\section{Frozen algebras of partial triangulations} \label{algparttri}

We consider a connected compact oriented bordered surface $\Sigma$ with non-empty finite set of marked points $\M \subset \Sigma$. For each marked point $M \in \M$, we fix an invertible scalar $\lambda_M \in k$ and a positive integer $m_M$. If a marked point $M \in \M$ is not on the boundary $\partial \Sigma$ of $\Sigma$, it is called a \emph{puncture}. We define $\lambda_\M := \prod_{M \in \M} \lambda_M$. If $\Sigma$ is a sphere (without boundary), $\# \M = 4$ and $m_M = 1$ for all $M \in \M$, we define $\nu_\M = 1 - \lambda_\M$. Otherwise, we set $\nu_\M = 1$. We assume that
\begin{itemize}
 \item if $\Sigma$ is a sphere (without boundary), then $\# \M \geq 3$;
 \item if $\Sigma$ is a sphere (without boundary) and $\# \M = 3$, then $m_M > 1$ for all $M \in \M$;
 \item $\nu_\M$ is invertible;
 \item if $\Sigma$ is a disc and $\# \M < 3$ then $\M$ contains at least one puncture $M$ satisfying $m_M > 1$ or at least two punctures.
\end{itemize}

%

An \emph{oriented edge} of $(\Sigma, \M)$ is a continuous map $\vec u: [0, 1] \to \Sigma$ such that
\begin{itemize}
 \item $\vec u(\{0,1\}) \subset \M$;
 \item $\vec u|_{(0,1)}$ is injective from $(0,1)$ to $\Sigma \setminus \M$;
 \item $\vec u$ is not homotopic to a constant path relatively to its endpoints.
\end{itemize}
The \emph{opposite} of the oriented edge $\vec u$ is the oriented edge $-\vec u$ defined by $(-\vec u)(t) = \vec u (1-t)$. For an oriented edge $\vec u$ of $(\Sigma, \M)$, we denote $s(\vec u) := \vec u(0)$, $t(\vec u) := \vec u(1)$ and we call $u := \{\vec u, -\vec u\}$ the corresponding \emph{(non-oriented) edge}.  We say that two oriented edges $\vec u$ and $\vec v$ are \emph{homotopic} if they are homotopic relative to $\{s(\vec v) = s(\vec u), t(\vec v) = t(\vec u)\}$ in $\Sigma \setminus \M$. Two (non-oriented) edges $u = \{\vec u, -\vec u\}$ and $v = \{\vec v, -\vec v\}$ are homotopic if $\vec v$ is homotopic to $\vec u$ or $-\vec u$.
 We call \emph{boundary edge} an edge homotopic to an edge 
entirely included in a boundary component and we call \emph{arc} an edge which is not a boundary edge. 

Two arcs of $(\Sigma, \M)$ are compatible if they are not homotopic and they do not cross except maybe at their endpoints. A \emph{partial triangulation} $\sigma$ of $(\Sigma, \M)$ is a set of compatible edges containing boundary edges. An oriented edge (respectively arc) of a partial triangulation $\sigma$ is a orientation of an edge (respectively arc) of $\sigma$. 

In a partial triangulation, the set of oriented edges starting at a given $M \in \M$ can be ordered counter-clockwisely up to cyclic permutation. Notice that if an edge $u$ has its two endpoints at $M$, its two orientations $\vec u$ and $-\vec u$ appear at different positions in this order. We call this order \emph{the cyclic order of oriented edges around $M$}.  
Finally, we call as usual \emph{triangulation} a maximal partial triangulation.

\begin{remark}
 In some contexts, in particular concerning Brauer graph algebras, an oriented edge is called a \emph{half edge}.
\end{remark}

\begin{example} \label{expartri} In the following diagrams, we draw two partial triangulations. A partial triangulation of a disc and a partial triangulation of a torus.
 $$
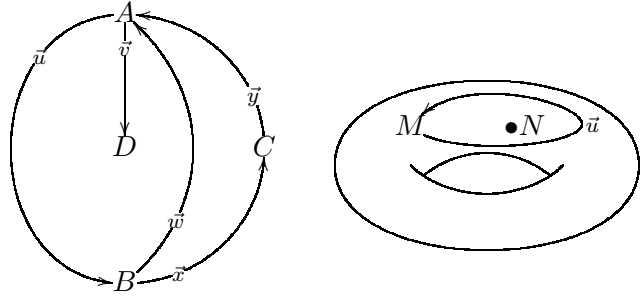
\xymatrix@L=.05cm@R=1.5cm@C=1.5cm@M=0.01cm{
   & A \ar@`{c+(-20,0),p+(-20,0)}[dd]|(.25){\vec u}^{{}}="AB" \ar[d]|(.25){\vec v}^{{}}="AD" \\
   & D & C \ar@`{c+(0,10),p+(10,0)}[ul]|(.25){\vec y}^{{}}="CA" \\
   & B \ar@`{c+(10,0),p+(0,-10)}[ur]|(.25){\vec x}^{{}}="BC" \ar@`{c+(12,10), p+(12,-10)}[uu]|(.25){\vec w}^{{}}="BA"
 } \quad \quad \xymatrix@L=.05cm@!=0cm@R=.5cm@C=.5cm@M=0.00cm{
   & \\
   & \\
   & \\
   & & M \ar@`{c+(5,-4),c+(35,3),p+(5,8)}[]_{\vec u}|(.12)*+<.1cm>{{}}="su"|(.88){{}}="tu" & & & \bullet N\\
   \ar@{-}@`{c+(0,15),p+(0,15)}[rrrrrrrr] \ar@{-}@`{c+(0,-15),p+(0,-15)}[rrrrrrrr] & & \ar@{-}@`{c+(5,-5),p+(-5,-5)}[rrrr]|(.1){{}}="P"|(.9){{}}="Q" &  & & & & & \\
   & \\
   & 
   \ar@{-}@`{c+(5,4),p+(-5,4)}"P";"Q"
 }$$
 In the first case, $\vec u$, $\vec x$ and $\vec y$ are boundary edges, while $\vec v$ and $\vec w$ are arcs. In the second case, $\vec u$ is an arc. In the first example, $A = s(\vec u) = s(\vec v) = t(\vec w) = t(\vec y)$ and the cyclic order around $A$ is $\vec u$, $\vec v$, $-\vec w$, $-\vec y$.
\end{example}

We associate to the partial triangulation $\sigma$ a quiver:

\begin{definition}
 We define $Q_\sigma$ to be a quiver with set of vertices the set of edges of $\sigma$. Then, each vertex $u$ has exactly two outgoing arrows, constructed in the following way: for both possible orientations $\vec u$ of the edge $u$, an arrow $\ic{\vec u, \vec v}$ points to the next (oriented) edge $\vec v$ around $s(\vec u)$.

 If $M \in \M$ has at least one incident edge in $\sigma$ and the oriented edges starting at $M$ are ordered $\vec u_1$, $\vec u_2$, \dots, $\vec u_n$, we fix the following paths of $Q_\sigma$:
 \begin{itemize}
  \item $\ic{\vec u_i, \vec u_j} := \ic{\vec u_i, \vec u_{i+1}}\ic{\vec u_{i+1}, \vec u_{i+2}} \cdots \ic{\vec u_{j-1},\vec u_j}$ composed of at least one arrow and at most $n$ arrows;
  \item $\io{\vec u_i, \vec u_j} := \lambda_M \ic{\vec u_i,\vec u_i}^{m_M - 1}\ic{\vec u_i, \vec u_{i+1}}\ic{\vec u_{i+1}, \vec u_{i+2}} \cdots \ic{\vec u_{j-1},\vec u_j}$ composed of at least $n(m_M-1)$ arrows and at most $n m_M - 1$ arrows. It is $\lambda_M$ times the idempotent at $u_i$ if $\vec u_i = \vec u_j$ and $m_M = 1$.
 \end{itemize}
 Thus we get in particular $\ic{\vec u_i, \vec u_j}\cdot \io{\vec u_j, \vec u_i} = \io{\vec u_i, \vec u_j}\cdot \ic{\vec u_j, \vec u_i} = \lambda_M \ic{\vec u_i, \vec u_i}^{m_M}$. 
\end{definition}

\begin{example} \label{exquiv}
 The quivers corresponding to partial triangulations of Example \ref{expartri} are represented by dashed arrows in the following diagrams:
$$
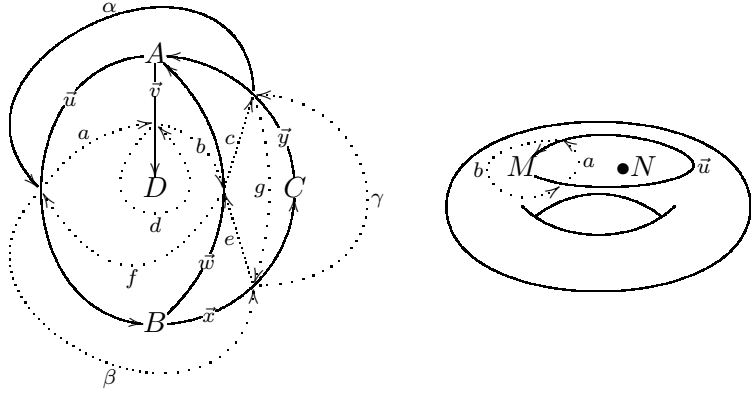
\xymatrix@L=.05cm@R=1.5cm@C=1.5cm@M=0.01cm{
   & A \ar@`{c+(-20,0),p+(-20,0)}[dd]|(.25){\vec u}^{{}}="AB" \ar[d]|(.25){\vec v}^{{}}="AD" \\
   & D & C \ar@`{c+(0,10),p+(10,0)}[ul]|(.25){\vec y}^{{}}="CA" \\
   & B \ar@`{c+(10,0),p+(0,-10)}[ur]|(.25){\vec x}^{{}}="BC" \ar@`{c+(12,10), p+(12,-10)}[uu]|(.25){\vec w}^{{}}="BA"
   \ar@{..>}@/^/^a"AB";"AD"
   \ar@{..>}@/^/|b"AD";"BA"
   \ar@{..>}^c"BA";"CA"
   \ar@{..>}@`{c+(-8,-8),p+(0,-16),p+(8,-8)}_d"AD";"AD"
   \ar@{..>}^e"BC";"BA"
   \ar@{..>}@/^1cm/^f"BA";"AB"
   \ar@{..>}@/^/_g"CA";"BC"
   \ar@{..>}@`{c+(-10,-10),c+(20,-20), p+(0,-10)}_\beta"AB";"BC"
   \ar@{..>}@`{c+(10,0),c+(10,10),p+(10,0)}_\gamma"BC";"CA"
   \ar@{..>}@{<-}@`{c+(-10,10),c+(20,20), p+(0,10)}^\alpha"AB";"CA"
 } \quad \quad \xymatrix@L=.05cm@!=0cm@R=.5cm@C=.5cm@M=0.00cm{
   & \\
   & \\
   & \\
   & & M \ar@`{c+(5,-4),c+(35,3),p+(5,8)}[]_{\vec u}|(.12)*+<.1cm>{{}}="su"|(.88){{}}="tu" & & & \bullet N\\
   \ar@{-}@`{c+(0,15),p+(0,15)}[rrrrrrrr] \ar@{-}@`{c+(0,-15),p+(0,-15)}[rrrrrrrr] & & \ar@{-}@`{c+(5,-5),p+(-5,-5)}[rrrr]|(.1){{}}="P"|(.9){{}}="Q" &  & & & & & \\
   & \\
   & 
   \ar@{-}@`{c+(5,4),p+(-5,4)}"P";"Q"
   \ar@/_/@{..>}"su";"tu"_a
   \ar@`{c+(-5,0),c+(-10,-3),p+(-5,-4)}@{..>}"tu";"su"_b
 }$$
 In the first diagram, the arrows of the quiver are
 $$\begin{array}{rrrrr} \ic{\vec u, \vec v} = a,& \ic{\vec v, -\vec w} = b,& \ic{-\vec w, -\vec y} = c,& \ic{-\vec v, -\vec v} = d,& \ic{\vec x, \vec w} = e, \\ \ic{\vec w, -\vec u} = f,& \ic{\vec y, -\vec x} = g,& \ic{-\vec y, \vec u} = \alpha,& \ic{-\vec u, \vec x} = \beta,& \ic{-\vec x, \vec y} = \gamma,\end{array}$$
 and we also consider symbolic compositions like $\ic{\vec u, -\vec y} = abc$. Moreover, if we suppose that $m_A = 1$ and $m_D = 3$, we have for example
 $\io{-\vec y, \vec v} = \lambda_A \alpha a$ and $\io{-\vec v, -\vec v} = \lambda_D d^2$.

 In the second diagram, we have $\ic{\vec u, -\vec u} = a$ and $\ic{-\vec u, \vec u} = b$ but also $\ic{\vec u, \vec u} = ab$, $\ic{-\vec u, -\vec u} = ba$, $\io{\vec u, -\vec u} = \lambda_M (ab)^{\lambda_M-1} a$.
\end{example}

%
%
%

We construct a first algebra associated to $\sigma$:

\begin{definition}
 For each oriented edge $\vec u$ of $\sigma$, we denote $$C_{\vec u} := \lambda_{s(\vec u)} \ic{\vec u, \vec u}^{m_{s(\vec u)}} - \lambda_{t(\vec u)} \ic{-\vec u, -\vec u}^{m_{t(\vec u)}}$$ and we denote by $I_\sigma^\circ$ the ideal of $k Q_\sigma$ generated by all possible $C_{\vec u}$. We denote $\Gso := k Q_\sigma / I_\sigma^\circ$.

 We say that a path $\omega$ of $Q_\sigma$ is \emph{$C$-irreducible} if it can not be written as $\omega = \omega_1 \ic{\vec u, \vec u}^{m_{s(\vec u)}} \omega_2$ for some paths $\omega_1$ and $\omega_2$ and some oriented edge $\vec u$.
\end{definition}

\begin{example} \label{exgso} 
 If we continue Example \ref{exquiv}, we have, for the first example, if $m_A = m_B = 1$, $m_C = 2$ and $m_D = 3$,
 $$\begin{array}{ll} C_{\vec u} = \lambda_A abc\alpha - \lambda_B \beta e f, & C_{\vec v} = \lambda_A bc\alpha a - \lambda_D d^3, \\ C_{\vec w} = \lambda_B f\beta e - \lambda_A c \alpha a b, & C_{\vec x} = \lambda_B ef\beta - \lambda_C (\gamma g)^2, \\ C_{\vec y} = \lambda_C (g \gamma)^2 - \lambda_A \alpha a b c.\end{array}$$
 And we have, in the second example, $$C_{\vec u} = \lambda_M \left((ab)^{m_M} - (ba)^{m_M}\right).$$
\end{example}

We give a convenient structural description of $\Gso$.

\begin{proposition} \label{Csig}
 For each edge $u$ of $\sigma$, choose an orientation $\vec u$ and denote $C_\sigma := \sum_{u \in \sigma} \lambda_{s(\vec u)} \ic{\vec u, \vec u}^{m_{s(\vec u)}} \in \Gamma^\circ_\sigma$. Then 
 \begin{enumerate}[\rm (1)]
  \item $C_\sigma$ does not depend on the chosen orientations;
  \item there is an injection from $k[x]$ to the centre of $\Gamma^\circ_\sigma$ mapping $x$ to $C_\sigma$; 
  \item $\Gamma^\circ_\sigma$ is free over $k[x]$ with basis consisting of the $C$-irreducible paths.
 \end{enumerate}
\end{proposition}

Proof of Proposition \ref{Csig} is given in Subsection \ref{proofCsig}

\begin{definition} \label{definpoly}
 For $n \in \Z_{\geq 1}$ we fix a \emph{model $n$-gon without hole} $\Pr(n)$ which is an (oriented) closed disc with $n$ marked points on its boundary, called its \emph{vertices}. We number the vertices of $\Pr(n)$ from $1$ to $n$ in the counter-clockwise order. We also fix a \emph{model $n$-gon with hole} $\Pr^\circ(n)$ which is $\Pr(n) \setminus D$ where $D$ is a closed disc contained in the interior of $\Pr(n)$. 

 Let us consider a partial triangulation $\sigma$ of $(\Sigma, \M)$. An \emph{$n$-gon without hole} (respectively \emph{$n$-gon with hole}) of $\sigma$ is an oriented continuous map $P: \Pr(n) \rightarrow \Sigma$ (respectively $P: \Pr^\circ(n) \rightarrow \Sigma$) satisfying:
 \begin{itemize}
  \item $P$ is injective on the interior of $\Pr(n)$ (respectively $\Pr^\circ(n)$);
  \item each side of $\Pr(n)$ (respectively $\Pr^\circ(n)$) is mapped injectively to an edge of $\sigma$;
  \item in the case of a polygon with hole, we cannot fill the hole: $P$ cannot be extended to a polygon without hole $P' : \Pr(n) \rightarrow \Sigma$ along the canonical inclusion $\Pr^\circ(n) \subset \Pr(n)$.
 \end{itemize}
 We call \emph{interior of $P$} the image by $P$ of the interior of $\Pr(n)$ (respectively $\Pr^\circ(n)$). If $P$ has no hole, we call \emph{set of punctures of $P$} the intersection $\M_P$ of $\M$ with the interior of $P$. If $P$ has a hole, we will use the symbolic notation $\# \M_P = \infty$.

 We denote by $P_i$ the image by $P$ of the vertex number $i$ of $\Pr(n)$ (respectively $\Pr^\circ(n)$). We denote by $\vct{P_i P_{i+1}}$ the oriented edge of $\sigma$ on which is mapped the corresponding side of $\Pr(n)$ and $\vct{P_{i+1} P_i} = - \vct{P_i P_{i+1}}$. We call $\vct{P_i P_{i+1}}$ an \emph{(oriented) side} of $P$. In the same way we define $\vct{P_n P_1}$ and $\vct{P_1 P_n}$. 

 We call \emph{special monogon} a polygon without hole $P$ with one side and one puncture $M$ satisfying $m_M = 1$, called \emph{special puncture}.
\end{definition}

\begin{example} \label{expoly} We continue with the figures of Example \ref{exquiv}. 
 The first one has five polygons:
 \begin{itemize}
  \item $P$ with sides $\vec u$, $\vec w$: we have $\M_{P} = \{D\}$;
  \item $Q$ with sides $-\vec w$, $\vec x$, $\vec y$: we have $\M_{Q} = \emptyset$;
  \item $R$ with sides $\vec u$, $\vec x$, $\vec y$: we have $\M_{R} = \{D\}$;
  \item $S$ with sides $\vec u$, $\vec w$, $\vec v$, $-\vec v$: we have $\M_{S} = \emptyset$;
  \item $T$ with sides $\vec u$, $\vec x$, $\vec y$, $\vec v$, $-\vec v$: we have $\M_{T} = \emptyset$.
 \end{itemize}
 The second has two polygons: a monogon $P$ with side $\vec u$ ($\M_P = \{N\}$) and a monogon $Q$ with side $-\vec u$ ($\# \M_Q = \infty$ as $Q$ has a hole). 
\end{example}


The following proposition, proven in Subsection \ref{proofcaracpoly}, gives a combinatorial description of polygons of $\sigma$.

\begin{proposition} \label{caracpoly}
 We consider a sequence of oriented edges $\vec u_1$, $\vec u_2$, \dots, $\vec u_n$, with indices considered modulo $n$, such that $t(\vec u_i) = s(\vec u_{i+1})$ for $i = 1 \dots n$. The following conditions are equivalent:
 \begin{enumerate}[\rm (i)]
  \item \label{caracpoly1} There is an $n$-gon having oriented sides $\vec u_1$, $\vec u_2$, \dots, $\vec u_n$ in this order.
  \item \label{caracpoly3} The following conditions are satisfied:
    \begin{itemize}
     \item if $i \neq j$ then $\vec u_i \neq \vec u_j$;
     \item for each $i$, if $\vec u_i$ is an oriented boundary edge, then it is oriented clockwisely around the boundary;
     \item for any $i$ and $j$ such that $M := s(\vec u_i) = s(\vec u_j)$, we have that $-\vec u_{i-1}$, $\vec u_i$, $-\vec u_{j-1}$ and $\vec u_j$ are ordered clockwisely around $M$;
     \item for any oriented boundary edge $\vec v$ and $i$ such that $M := s(\vec u_i) = s(\vec v)$, we have that $-\vec u_{i-1}$, $\vec u_i$ and $\vec v$ are ordered clockwisely around $M$.
    \end{itemize}
 \end{enumerate}
\end{proposition}

\begin{definition}
 Two $n$-gons without hole $P, P': \Pr(n) \rightarrow \Sigma$ of $\sigma$ are said to be \emph{equivalent} if there exist an oriented automorphism $\psi$ of $\Pr(n)$ permuting the sides such that $P = P' \circ \psi$. Two $n$-gons with hole $P, P': \Pr^\circ(n) \rightarrow \Sigma$ of $\sigma$ are said to be \emph{equivalent} if there exist two injection $\iota, \iota': \Pr^\circ(n) \hookrightarrow \Pr^\circ(n)$ which map the sides of the $n$-gon to the sides of the $n$-gon and satisfy $P \circ \iota = P' \circ \iota'$.
\end{definition}

Notice that this equivalence relation relates polygons which either have both hole, either have the same set of punctures. Moreover, an equivalence class of polygons is entirely determined by the sequence of its sides. From now on, we will consider polygons up to this equivalence relation. Using Proposition \ref{caracpoly} permits to consider polygons as combinatorial objects.

We associate a second algebra to $\sigma$:

\begin{definition} \label{defgamma}
 Consider a $n$-gon $P$. For $i \in \{1, 2, \dots, n\}$, we call \emph{internal path winding around $P_i$} the path $$\omega_i^P := \ic{\vct{P_i P_{i+1}}, \vct{P_i P_{i-1}}}$$ and we call \emph{coefficiented external path winding around $P_i$} the (multiple of a) path $$\xi_i^P := \io{\vct{P_i P_{i-1}}, \vct{P_i P_{i+1}}}.$$ Notice that we have $\omega^P_i \xi^P_i = \xi^P_{i+1} \omega^P_{i+1}$ in $\Gso$. For $i \in \{1, 2, \dots, n\}$, we denote
 \begin{align*} & \Gso \ni R_{P, i} := \omega_{i+1}^P \omega_i^P -  \\
 &\left\{\begin{array}{ll}
                                               \xi_{i+2}^P \xi_{i+3}^P \cdots \xi_{i-2}^P \xi_{i-1}^P, & \text{if $\M_P = \emptyset$;} \\
                                               \lambda_M \omega_{i+1}^P (\xi_{i+1}^P \xi_{i+2}^P \cdots \xi_{i}^P)^{m_M-1} \xi_{i+1}^P \cdots \xi_{i-1}^P, & \text{if  $\M_P = \{M\}$;} \\
					       0, & \text{if $\# \M_P \geq 2$.} 
                                            \end{array}\right.
 \end{align*}
  (Notice that, if $n = 1$ or $n = 2$, we have $\M_P \neq \emptyset$).

 Finally, we denote by $I_\sigma$ the ideal of $\Gso$ generated by all $R_{P,i}$ for all polygons of $\sigma$. We call $$\Gs = \Gs^\lambda := \widehat{\Gamma^\circ_\sigma / I_\sigma}^{J}$$ the \emph{frozen algebra associated to $\sigma$} where the completion is taken with respect to the ideal 
 $J$ of $\Gamma^\circ_\sigma / I_\sigma$ generated by the arrows $q \in Q_1$ such that $q$ does not divide an idempotent in $\Gamma^\circ_\sigma / I_\sigma$ or equivalently $q$ is not of the form $\ic{\vec u, -\vec u}$, $\ic{\vec u, \vec v}$ or $\ic{\vec v, -\vec u}$ where $\vec u$ is the oriented side of a special monogon and $\vec v$ is the oriented edge pointing to the special puncture if it is in $\sigma$.
\end{definition}

\begin{example} \label{exalgb} We continue Examples \ref{exquiv}, \ref{exgso} and \ref{expoly}. For the first partial triangulation, in addition to relations $C_{\vec u}$, $C_{\vec v}$, $C_{\vec w}$, $C_{\vec x}$ and $C_{\vec y}$, we have the following relations:
  \begin{itemize}
   \item Coming from $P$: $$fab = \lambda_D f (\lambda_B \lambda_A \beta e c\alpha )^2  \lambda_B \beta e \quad \text{and} \quad abf = \lambda_D ab (\lambda_A \lambda_B c \alpha \beta e)^2 \lambda_A c \alpha;$$
   \item Coming from $Q$: $$ge = \lambda_A \alpha a b \quad \text{and} \quad ec = \lambda_C \gamma g \gamma\quad \text{and} \quad cg = \lambda_B f \beta;$$
   \item Coming from $R$: \begin{align*} gef &= \lambda_D g (\lambda_C \lambda_A \lambda_B \gamma g \gamma \alpha \beta)^2 \lambda_C \lambda_A \gamma g \gamma \alpha \\ \text{and} \quad efabc &= \lambda_D ef (\lambda_B \lambda_C \lambda_A \beta \gamma g \gamma \alpha)^2 \lambda_B \lambda_C \beta \gamma g \gamma \\ \text{and} \quad abcg &= \lambda_D abc (\lambda_A \lambda_B \lambda_C \alpha \beta \gamma g \gamma)^2 \lambda_A \lambda_B \alpha \beta;\end{align*}
   \item Coming from $S$: \begin{align*} ad &= \lambda_B \beta e \lambda_A c \alpha a \quad \text{and} \quad db = \lambda_A bc \alpha \lambda_B \beta e \\ \text{and} \quad bf &= \lambda_D d^2  \lambda_A bc\alpha \quad \text{and} \quad fa = \lambda_A c\alpha a \lambda_D d^2; \end{align*}
   \item Coming from $T$: \begin{align*} ad &= \lambda_B \beta \lambda_C \gamma g \gamma \lambda_A \alpha a \quad \text{and} \quad dbc = \lambda_A bc\alpha \lambda_B \beta \lambda_C \gamma g \gamma \\ \text{and} \quad bcg &= \lambda_D d^2 \lambda_A bc \alpha \lambda_B \beta\quad \text{and} \quad gef = \lambda_A \alpha a \lambda_D d^2 \lambda_A bc \alpha \\ \text{and} \quad efa &= \lambda_C \gamma g \gamma \lambda_A \alpha a \lambda_D d^2.\end{align*}
  \end{itemize}
  The ideal $J$ is generated by all arrows of $Q_\sigma$. Notice that these relations are redundant. For instance, relations for $T$ can be recovered from the relations of $\Gamma^\circ_\sigma$, the relations for $Q$ and the relations from $S$. For example, using relations of $Q$, $S$ and relations of $\Gamma^\circ_\sigma$, we find
  $$dbc = \lambda_A bc\alpha \lambda_B \beta ec = \lambda_A bc\alpha \lambda_B \beta \lambda_C \gamma g \gamma$$
  which is a relation coming from $T$. This remark will be generalized in Theorem \ref{thmmin}.

  For the second partial triangulation, in addition to $C_{\vec u}$, we get the relation $a^2 = \lambda_N a (\lambda_M (ba)^{m_M - 1}b)^{m_N - 1}$ from $P$ and $b^2 = 0$ from $Q$. Using the second relation, the first relation can be simplified to 
  $$a^2 = \left\{\begin{array}{ll}
                  \lambda_N a & \text{if $m_N = 1$;} \\
                  \lambda_N \lambda_M (ab)^{m_M} & \text{if $m_N = 2$;} \\
                  0 & \text{else.}
                 \end{array}\right.$$
  The ideal $J$ is generated by $a$ and $b$ if $m_N \geq 2$ and by only $b$ if $m_N = 1$. We find easily
   $$\Gamma_\sigma = \left\{\begin{array}{ll}
                             k \left. \left(\boxinminipage{\xymatrix@L=.05cm{e_1 \ar@`{p+(-5,-5),p+(-10,0),p+(-5,5)}[]^(.8){b_{11}} \ar@/^/[r]^{b_{12}} & e_2 \ar@`{p+(5,5),p+(10,0),p+(5,-5)}[]^(.2){b_{22}} \ar@/^/[l]^{b_{21}}     }}\right)\right/ \left(
				\begin{array}{l}  b_{11}^2 + b_{12}b_{21}, \\ b_{11} b_{12} + b_{12} b_{22}, \\ b_{21} b_{11} + b_{22} b_{21}, \\ b_{21} b_{12} + b_{22}^2, \\ b_{11}^{m_M-1} b_{12}, b_{21} b_{11}^{m_M-1} \end{array}   \right) & \text{if $m_N = 1$;} \\ 
                             k \langle a, b \rangle / \left( a^2 - \lambda_N \lambda_M (ab)^{m_M}, a^2 - \lambda_N \lambda_M (ba)^{m_M}, b^2\right) & \text{if $m_N = 2$;} \\ 
                             k \langle a, b \rangle / \left( (ab)^{m_M} - (ba)^{m_M}, a^2, b^2\right) & \text{else}
                            \end{array} \right.$$
  (in the first case, we separated the idempotents $a / \lambda_N$ and $1 - a/\lambda_N$). Notice that in all cases, the ideal contains all paths of length $m_M+1$ so there is no need of completion here.
\end{example}

\begin{definition}
 An arc $u$ of a partial triangulation $\sigma$ is said to be a \emph{reduction arc} for a polygon $P$ if it is connected to at least one vertex of $P$ and if the interior of $P$ intersects $u$. 

 A $n$-gon $P$ is said to be \emph{minimal} if it does not have any reduction arc. 
\end{definition}

\begin{example} 
 In Example \ref{expoly}, minimal polygons are $Q$ and $S$.
\end{example}

We finish this Section by giving two important theorems to understand the algebra $\Gamma_\sigma$. These theorems are proven in Subsections \ref{proofmainthmsa} and \ref{proofmainthmsb}.

\begin{theorem} \label{thmmin}
 The ideal $I_\sigma$ is the ideal of $\Gamma^\circ_\sigma$ generated by all $R_{P,i}$ for $P$ minimal.
\end{theorem}

\begin{theorem} \label{subtri}
 Let $\tau \subset \sigma$ be an inclusion of partial triangulations. Let $e_\tau$ be the idempotent of $\Gs$ corresponding to the set of edges of $\tau$. There is an isomorphism of (non-unital) algebras
 $$\phi: \Gamma_\tau \xrightarrow{\sim} e_\tau \Gs e_\tau$$
 mapping each idempotent $e_u$ to $e_u$ and each arrow $\ic{\vec u, \vec v}$ to the path $\ic{\vec u, \vec v}$.
\end{theorem}

\section{When is $\Gs$ module-finite over its centre?} \label{order} 

In this section, we classify partial triangulations in three families depending on the complexity of $\Gs$. Recall the lattice over a commutative ring $R$ is an $R$-algebra which is free of finite rank as an $R$-module. The main theorem of this section is the following one:
\begin{theorem} \label{classiford2}
 Let $E \subset \sigma$ be a connected component of $\sigma$. Let $e$ be the idempotent of $\Gs$ corresponding to $E$. Then:
 \begin{enumerate}[\rm (1)]
  \item If $E$ is disconnected from the boundary, then $e\Gs e$ is free of finite rank as a $k$-module.
  \item If we are in one of the following three cases, then $e\Gs e$ is a lattice over $R := k \llbracket x \rrbracket$:
   \begin{enumerate}[\rm (a)]
    \item $(\Sigma, \M)$ is a polygon without puncture and $m_M = 1$ for all $M \in \M$ except at most one;
    \item $(\Sigma, \M)$ is a polygon with one puncture and $m_M = 1$ for all $M \in \M$ on the boundary;
    \item all marked points $M$ incident to $E$ satisfy $m_M = 1$ and any arc of $E$ is homotopic to a part of a boundary component.
   \end{enumerate}
  \item In any other case, $e \Gs e$ is not finitely generated as a module over its centre. 
 \end{enumerate}
\end{theorem}

Proof of Theorem \ref{classiford2} is given in Section \ref{proofclassiford2}. Notice that $\Gs$ is the direct product of the $e \Gs e$ corresponding to the connected components of $\sigma$ so Theorem \ref{classiford2} exhausts all cases. Notice also that Theorem \ref{classiford2} (2) generalizes slightly the observations of \cite{DeLu16} and \cite{DeLu16-2}. We give now a precise description of the lattices of Theorem \ref{classiford2} (b) when $\sigma$ contains only boundary edges. We fix $E$ and $e$ as in Theorem \ref{classiford2}.
%
\begin{proposition} \label{boundord}
 Suppose that $E$ is a boundary component of $\Sigma$. Let $P$ be the $n$-gon corresponding to this boundary component. We have:
 \begin{enumerate}[\rm (1)]
  \item If $\M_P = \emptyset$, $m_{P_1} = m$ and $m_{P_i} = 1$ for $2 \leq i \leq n$, then
   $$e \Gamma_\sigma e \cong \begin{bmatrix}
                                  R' & R' & R' & \cdots & R' & t^{-m} R' \\
				  t^m R' & \Delta & \Delta & \cdots & \Delta & R' \\
				  t^{m+1} R' & t^m R' & \Delta & \cdots & \Delta & R' \\
				  \vdots & \vdots & \vdots & \ddots & \vdots & \vdots \\
				  t^{m+1} R' & t^{m+1} R' & t^{m+1} R' & \cdots & \Delta & R' \\
				  t^{m+1} R' & t^{m+1} R' & t^{m+1} R' & \cdots & t^m R' & R'
                                 \end{bmatrix}_{n \times n}$$
  where $R' := k \llbracket t \rrbracket$, $\Delta := k \oplus t^m R'$ and $R$ is identified to a subalgebra of $R'$ and $\Delta$ by mapping $x$ to $t^m$.
 \item If $\M_P = \{M\}$, $m_{P_i} = 1$ for $1 \leq i \leq n$ and $m_M = m$, then
  $$e \Gamma_\sigma e \cong \begin{bmatrix}
                                 R_m & R_m & R_m & \cdots & R_m & R_{m-1} \\
				 xR_{m-1} & R_m & R_m & \cdots & R_m & R_m \\
				 xR_{m} & xR_{m-1} & R_m & \cdots & R_m & R_m \\
				 \vdots & \ddots & \ddots & \ddots & \vdots & \vdots \\
				 xR_{m} & xR_{m} & xR_{m} & \ddots & R_m & R_m \\
                                 xR_{m} & xR_{m} & xR_{m} & \cdots & xR_{m-1} & R_m 
                                \end{bmatrix}_{n \times n}$$
  where $R_j := \{(P, Q) \in R^2 \mid P-Q \in x^j R\}$ for $j \geq 0$. Notice that if $n = 1$, it degenerates to $e_u \Gamma_\sigma e_u \cong R_{m-1}$.
 \item If $\# \M_P > 1$ and $m_{P_i} = 1$ for $1 \leq i \leq n$, then
  $$e \Gamma_\sigma e \cong \begin{bmatrix}
                                 R' & R' & R' & \cdots & R' & R' & x^{-1} I \\
				 I & R' & R' & \cdots & R' & R' & R' \\
				 x R' & I & R' & \cdots & R' & R' & R' \\
				 \vdots & \vdots & \vdots & \ddots & \vdots & \vdots & \vdots \\
                                 x R' & x R' & x R' & \cdots & R' & R' & R' \\
				 x R' & x R' & x R' & \cdots & I & R' & R' \\
				 x R' & x R' & x R' & \cdots &xR' & I & R' \\
                                \end{bmatrix}_{n \times n}$$
  where $R' := R[\varepsilon]/(\varepsilon^2)$ and $I := (\varepsilon, x) \subset R'$.
 \end{enumerate}
\end{proposition}

Proposition \ref{boundord} is proven in Subsection \ref{proofboundord}.

\section{Algebras of a partial triangulations} \label{nonfroz}

In this section, we define algebras of partial triangulations and we give first results about these algebras. As usual, $\sigma$ is a partial triangulation of $(\Sigma, \M)$. All results are proven in Subsections \ref{proofaltpres0}, \ref{proofaltpres2}, \ref{proofbasisDsig} and \ref{proofsymm}.

\begin{definition}
 We call \emph{algebra associated to $\sigma$} the algebra $$\Delta_\sigma = \Delta_\sigma^\lambda := \Gs^\lambda / (e_0)$$ where $e_0$ is the sum of the idempotent corresponding to boundary components.
\end{definition}

First, notice that we can use the same techniques as for frozen algebras:
\begin{corollary}[of Theorem \ref{subtri}] \label{subtri2}
 If $\tau \subset \sigma$ then $\Delta_\tau = e_\tau \Delta_\sigma e_\tau$.
\end{corollary}

By abuse of notation, we denote by $J$ the ideal of $\Delta_\sigma$ obtained by projection of $J$.
We start by giving an alternative presentation of $\Delta_\sigma$, which does not involve completion. We consider the full subquiver $\uQ_\sigma$ of $Q_\sigma$ with vertices corresponding to arcs (\emph{i.e.} non-boundary edges). For each oriented arc $\vec v$ of $\sigma$, we define a relation $\RR_{\vec v}$ in the following way. Let $P$ be the minimal polygon of $\sigma$ having $\vec v$ as an oriented side, and $\vec u$ (respectively $\vec w$) be the side of $P$ following (respectively preceding) $\vec v$. If one at least of $\vec u$ or $\vec w$ is a boundary edge, we put $\RR_{\vec v} = 0$. Otherwise, we put $\RR_{\vec v} = \ic{\vec u, -\vec v} \ic{\vec v, -\vec w} - f_{\vec v}$, where, denoting by $n$ the number of sides of $P$:
 \begin{itemize}
  \item if $n = 3$ and $\M_P = \emptyset$, $f_{\vec v} = \io{-\vec u, \vec w}$;
  \item if $n = 1$ and $\M_P = \{M\}$ with $m_M = 1$, $f_{\vec v} = \lambda_M \ic{\vec v, -\vec v}$;
  \item if $2 \notin k^\times$, $n = 1$ and $\M_P = \{M\}$ with $m_M = 2$, $f_{\vec v} = \lambda_M \lambda_{s(\vec v)} \ic{\vec v, \vec v}^{m_{s(\vec v)}}$;
  \item in any other case, $f_{\vec v} = 0,$
 \end{itemize}
 where any path of $Q_\sigma$ passing through a boundary component is $0$ in $k \uQ_\sigma$.

 Then we have the following result:
 \begin{theorem} \label{altpres2}
  We have an isomorphism
  $$\Delta_\sigma \cong \Delta^s_\sigma := \frac{k \uQ_\sigma}{JC_\sigma + (C_{\vec v}, \RR_{\vec v})_{\vec v \in \sigma}}$$
  which maps $e_u$ to $e_u$ for any $u \in \sigma$, such that $C_\sigma$ comes from the same element of $\uQ_\sigma$ on both side.
 \end{theorem}

 \begin{remark}
  We will see that the isomorphism of Theorem \ref{altpres2} does not come from the identity of $k \uQ_\sigma$. In particular, the isomorphism of Corollary \ref{subtri2} cannot any more be realized by the naive identifications $\ic{\vec u, \vec v}$ in $\Delta^s_\tau$ and $\Delta^s_\sigma$. We will give in Proposition \ref{altpres0} an other presentation of $\Delta_\sigma$ for which this naive identification still works.
 \end{remark}

We now give a convenient basis of $\Delta_\sigma$. Let $\Br$ the subset of $k \uQ_\sigma/(C_{\vec v})_{\vec v \in \sigma}$ consisting of $e_u$ for $u \in \sigma$ (a non-boundary edge), $c_u := e_u C_\sigma$ for $u$ an arc of $\sigma$ not connected to a boundary and $\ic{\vec u, \vec u}^\ell \ic{\vec u, \vec v}$ where
 \begin{itemize}
  \item $\vec u$ and $\vec v$ are non-boundary oriented edges satisfying $s(\vec u) = s(\vec v)$;
  \item $0 \leq \ell < m_{s(\vec u)}$;
  \item If $\ell = m_{s(\vec u)} - 1$ then $\vec u \neq \vec v$;  
  \item If there is a boundary edge $\vec b$ with $s(\vec b) = s(\vec u)$ then $\vec u$, $\vec v$ and $\vec b$ are strictly ordered counter-clockwisely around $s(\vec u)$ and $\ell = 0$.
 \end{itemize}

\begin{theorem} \label{basisDsig}
 \begin{enumerate}[\rm (1)]
  \item The set $\Br$ is mapped to $k$-bases of $\Delta_\sigma$ and $\Delta^s_\sigma$.
  \item For two elements $x$, $y$ of $\Br$, the product $xy$ is a scalar multiple of an element of $\Br$ in $\Delta_\sigma$ and in $\Delta_\sigma^s$.
  \item If $\sigma$ has no arc incident to the boundary, then for any $x \in \Delta_\sigma$, $J x = x J = 0$ if and only if $x \in (C_\sigma)$.
 \end{enumerate}
\end{theorem}

We deduce the following easy corollary.

\begin{corollary} \label{dimd}
 The $k$-algebra $\Delta_\sigma$ is free of rank
 $$\sum_{M \in \M \setminus \P} \frac{d_M(d_M-1)}{2} + \sum_{M \in \P} m_M d_M^2 + f$$
 where $\P \subset \M$ is the set of punctures, for $M \in \M$, $d_M$ is the degree of $M$ in the graph $\sigma$ (without counting boundary components), and $f$ is the number of arcs in $\sigma$ with both endpoints on boundaries. 
\end{corollary}

\begin{example}
 In Examples \ref{exquiv}, \ref{exgso}, \ref{exalgb}, the rank of the algebra of the left partial triangulation is $1 + 3 + 1 = 5$ and the one of the right partial triangulation is $4 m_M$.
\end{example}

\begin{remark}
 It is immediate that $\Delta_\sigma$ does not depend of the values of $m_M$ for $M \in \M$ on the boundary of $\Sigma$. This is reflected in the rank formula.
\end{remark}

We also get the following theorem:

\begin{theorem} \label{symm}
 If $\sigma$ has no arc incident to the boundary, then $\Delta_\sigma$ is a symmetric $k$-algebra (\emph{i.e.} $\Hom_k(\Delta_\sigma, k) \cong \Delta_\sigma$ as $\Delta_\sigma$-bimodules). 
\end{theorem}

We construct a triangulation $\sigma'$ of a surface without boundary from the triangulation $\sigma$ having the property that $\Delta_{\sigma} = \Delta_{\sigma'}/(e_0)$ for an idempotent $e_0 \in \Delta_{\sigma'}$:

\begin{definition} \label{defsp}
 For each boundary of $\Sigma$ with $n$ marked points, we patch a $n$-gon with two punctures to form a marked surface $(\Sigma', \M')$ without boundary that we call \emph{augmented surface of $(\Sigma, \M)$}. We call $\sigma'$ the partial triangulation of $(\Sigma', \M')$ having the same edges and we call it its \emph{augmented partial triangulation}. 
\end{definition}

Notice that it is immediate by definition that $\Delta_{\sigma'}^\lambda = \Gamma_{\sigma'}^\lambda$ does not depend on the choice of the $\lambda_M$'s for the added punctures $M$.

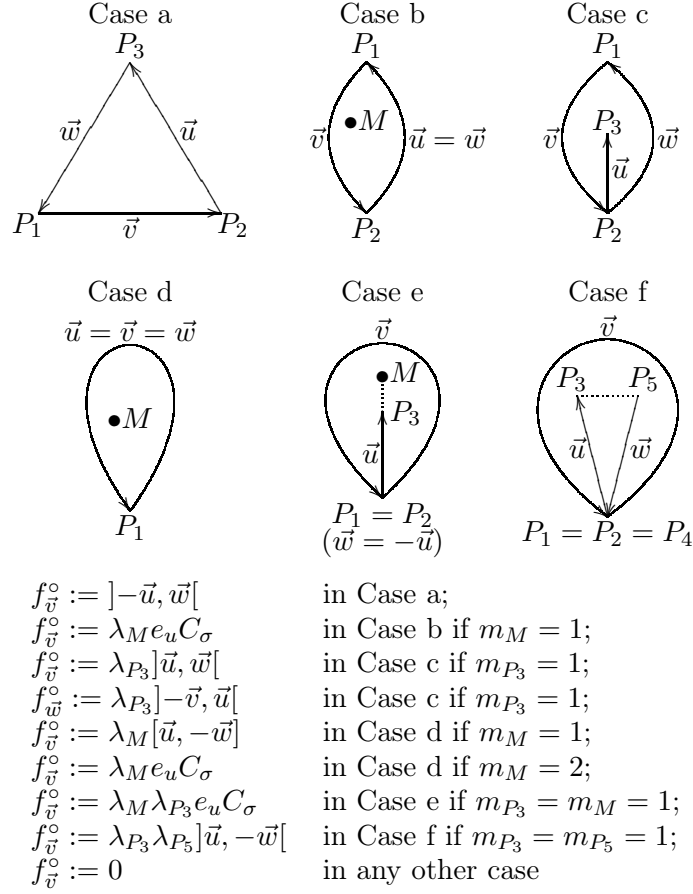
\begin{figure}
 \begin{tabular}{ccc}
  Case a & Case b & Case c \\
\boxinminipage{$\renewcommand{\labelstyle}{\textstyle} \xymatrix@L=.05cm@!=0cm@R=.4cm@C=.4cm@M=0.00cm{ 
    & & & \ar[dddddlll]_{\vec w} \ar@{{}{ }{}}@`{p+(0,3), p+(0,3)}[]|{P_3}\\ \\ \\ \\ \\ 
   \ar[rrrrrr]_{\vec v} \ar@{{}{ }{}}@`{p+(-2,-2), p+(-2,-2)}[]|{P_1} & & & & & & \ar[uuuuulll]_{\vec u} \ar@{{}{ }{}}@`{p+(2,-2), p+(2,-2)}[]|{P_2}
}$}
&   
\boxinminipage{$\renewcommand{\labelstyle}{\textstyle} \xymatrix@L=.05cm@!=0cm@R=.4cm@C=.4cm@M=0.00cm{ 
    & & & \ar@/_.5cm/[ddddd]_{\vec v} \ar@{{}{ }{}}@`{p+(0,3), p+(0,3)}[]|{P_1} & & & \\ \\ & & & \bullet M \\ \\ \\ 
    & & & \ar@/_.5cm/[uuuuu]_{\vec u = \vec w} \ar@{{}{ }{}}@`{p+(0,-3), p+(0,-3)}[]|{P_2}
}$}
& 
\boxinminipage{$\renewcommand{\labelstyle}{\textstyle} \xymatrix@L=.05cm@!=0cm@R=.4cm@C=.4cm@M=0.00cm{ 
    & & & \ar@/_.6cm/[ddddd]_{\vec v} \ar@{{}{ }{}}@`{p+(0,3), p+(0,3)}[]|{P_1} & & & \\ \\ & & & P_3 \\ \\ \\ 
    & & & \ar[uuu]_{\vec u} \ar@/_.6cm/[uuuuu]_{\vec w} \ar@{{}{ }{}}@`{p+(0,-3), p+(0,-3)}[]|{P_2}
}$}
  \\
  & & \\
  Case d & Case e & Case f \\
\boxinminipage{$\renewcommand{\labelstyle}{\textstyle} \xymatrix@L=.05cm@!=0cm@R=.4cm@C=.4cm@M=0.00cm{ 
    & & & & & & \\ \\ & & & \bullet M \\ \\ \\ 
    & & & \ar@`{p+(10,14),p+(0,30),p+(-10,14)}[]_{\vec u = \vec v = \vec w} \ar@{{}{ }{}}@`{p+(0,-3), p+(0,-3)}[]|{P_1} \ar@{{}{ }{}}@`{p+(0,-8), p+(0,-8)}[]|{\quad}
}$}
   &  
\boxinminipage{$\renewcommand{\labelstyle}{\textstyle} \xymatrix@L=.05cm@!=0cm@R=.4cm@C=.4cm@M=0.00cm{ 
    & & & & & & \\ & & & \bullet \ar@{{}{ }{}}@`{p+(4,1), p+(4,1)}[]|{M} \\ & & & \ar@{..}[u] \ar@{{}{ }{}}@`{p+(4,-1), p+(4,-1)}[]|{P_3} \\ \\ \\ 
    & & & \ar@`{p+(13,12),p+(0,29),p+(-13,12)}[]_{\vec v} \ar[uuu]^{\vec u} \ar@{{}{ }{}}@`{p+(0,-3), p+(0,-3)}[]|{P_1 = P_2} \ar@{{}{ }{}}@`{p+(0,-8), p+(0,-8)}[]|{(\vec w = -\vec u)}
}$}
 & 
\boxinminipage{$\renewcommand{\labelstyle}{\textstyle} \xymatrix@L=.05cm@!=0cm@R=.4cm@C=.4cm@M=0.00cm{ 
    & & & & & & \\ & & \ar@{{}{ }{}}@`{p+(-1,3), p+(-1,3)}[]|{P_3} \ar@{..}[rr] & & \ar@{{}{ }{}}@`{p+(1,3), p+(1,3)}[]|{P_5} \ar[ddddl]^(.4){\vec w} \\ \\ \\ \\ 
    & & & \ar@`{p+(16,12),p+(0,35),p+(-16,12)}[]_{\vec v} \ar[uuuul]^(.6){\vec u} \ar@{{}{ }{}}@`{p+(0,-3), p+(0,-3)}[]|{P_1 = P_2 = P_4} \ar@{{}{ }{}}@`{p+(0,-8), p+(0,-8)}[] \ar@{{}{ }{}}@`{p+(0,-8), p+(0,-8)}[]|{\quad}
}$}
 \end{tabular}
 
 $$\begin{array}{ll}
  f^\circ_{\vec v} := \io{-\vec u, \vec w} & \text{ in Case a;} \\
  f^\circ_{\vec v} := \lambda_M e_u C_\sigma & \text{ in Case b if $m_M = 1$;} \\
  f^\circ_{\vec v} := \lambda_{P_3} \io{\vec u, \vec w} & \text{  in Case c if $m_{P_3} = 1$;} \\
  f^\circ_{\vec w} := \lambda_{P_3} \io{-\vec v, \vec u} & \text{  in Case c if $m_{P_3} = 1$;} \\
  f^\circ_{\vec v} := \lambda_{M} \ic{\vec u, -\vec w} & \text{  in Case d if $m_M = 1$;} \\
  f^\circ_{\vec v} := \lambda_{M} e_u C_\sigma & \text{  in Case d if $m_M = 2$;} \\
  f^\circ_{\vec v} := \lambda_{M} \lambda_{P_3} e_u C_\sigma & \text{  in Case e if $m_{P_3} = m_M = 1$;} \\
  f^\circ_{\vec v} := \lambda_{P_3} \lambda_{P_5} \io{\vec u, -\vec w} & \text{  in Case f if $m_{P_3} = m_{P_5} = 1$;} \\
  f^\circ_{\vec v} := 0 & \text{  in any other case} \end{array}
 $$
 \caption{Polygons inducing non-zero relations in $\Delta_\sigma$ and corresponding $f^\circ_{\vec v}$}\label{polnonz} \label{fcircv}
\end{figure}

Finally, we give a variant of Theorem \ref{altpres2}, which has the advantage to be compatible with Corollary \ref{subtri2}, but gives a more complicated presentation of $\Delta_\sigma$. We keep notations introduced at the beginning of this section.

For each $\vec v \in \sigma$, we define $\RR^\circ_{\vec v} \in k \uQ_\sigma/(C_{\vec x})_{\vec x \in \sigma}$. If one at least of $\vec u$, $\vec w$ is a boundary edge, we put $\RR^\circ_{\vec v} = 0$. Otherwise, we define $\RR^\circ_{\vec v} := \ic{\vec u, -\vec v} \ic{\vec v, -\vec w} - f^\circ_{\vec v}$ where $f^\circ_{\vec v}$ depends on the situation as in Figure \ref{polnonz}. On this figure, there are no hole or puncture other than the one depicted inside the drawn polygon (but there can be an arc in Case e from $P_3$ to $M$, and in Case $f$ from $P_3$ to $P_5$). Moreover, for Case a, we can have that $P_1$, $P_2$, $P_3$ or $u$, $v$, $w$ are not distinct. In Cases b and c, $P_1$ and $P_2$ are not necessarily distinct. Let us define $f^\circ_{\vec v}$ as in Figure \ref{fcircv} (any path of $Q_\sigma$ passing through a boundary component is $0$ in $k \uQ_\sigma$).

 \begin{proposition} \label{altpres0}
  We have the equality
  $$\Delta_\sigma = \frac{k \uQ_\sigma}{JC_\sigma + (C_{\vec v}, \RR^\circ_{\vec v})_{\vec v \in \sigma}}.$$
 \end{proposition}

 Notice that in Proposition \ref{altpres0}, in contrast to Theorem \ref{altpres2} where we need a non-trivial endomorphism of $k \uQ_\sigma$, we just give an alternative set of generators of the ideal of relations defining $\Delta_\sigma$.

\section{Case of triangulations} \label{casetri}

Let us suppose in this section that $(\Sigma, \M)$ admits at least one marked point on each boundary component and that $\sigma$ is a triangulation. We suppose also that for any $M \in \M$, $m_M$ is invertible in $k$. We will prove that in this case the relations defining $\Gs$ and $\Delta_\sigma$ come from a potential. Up to Morita equivalence, $\Delta_\sigma$ is the algebra introduced by Labardini-Fragoso in \cite{La09} and Cerulli Irelli, Labardini-Fragoso in \cite{CeLa12} when $m_M = 1$ for all $M \in \M$. For general definitions and results about Jacobian algebras and frozen Jacobian algebras, we refer to \cite{DeWeZe08} and \cite{BuIyReSm11}. 

We call \emph{potential attached to $\sigma$} the following linear combination of cycles in $k Q_\sigma / [k Q_\sigma, k Q_\sigma]$:
$$W_\sigma = \sum_{P \text{ triangle of $\sigma$}} \omega_3^P \omega_2^P \omega_1^P - \sum_{M \in \M} \frac{\lambda_M}{m_M} \alpha_M^{m_M}$$
where the first sum runs over all (minimal) triangles of $\sigma$ and for each $M \in \M$, $\alpha_M$ is the cycle running around $M$ with arbitrary starting vertex.

For any arrow $\alpha \in Q_{\sigma, 1}$, and any cycle $u = u_1 u_2 \cdots u_n$ of $Q_\sigma$, we define the \emph{cyclic derivative}
 $$\partial_\alpha (u) := e_{t(\alpha)} \sum_{u_i = \alpha} u_{i+1} u_{i+2} \cdots u_{i-1}$$
and we extend this definition to $k Q_\sigma / [k Q_\sigma, k Q_\sigma]$.

Let us call an arrow $\ic{\vec u, \vec v}$ \emph{frozen} if $-\vec v$ and $\vec u$ are two consecutive oriented boundary edges winding counter-clockwisely around a hole and let us call $F$ the set of frozen arrows. We call \emph{frozen Jacobian ideal} of $W_\sigma$ the ideal $J(W_\sigma)$ of $k Q_\sigma$ generated by the elements $\partial_\alpha W_\sigma$ where $\alpha$ runs over non-frozen arrows of $Q_\sigma$. Adapting \cite{BuIyReSm11}, we call \emph{frozen Jacobian algebra} of the frozen quiver with potential $(Q_\sigma, W_\sigma, F)$ the algebra
$$\Pr(Q_\sigma, W_\sigma, F) := \widehat{\left(\frac{k Q_\sigma}{J(W_\sigma)}\right)}^J.$$

The main theorem of this Section is:

\begin{theorem} \label{thmmin2}
 The identity map of $k Q_\sigma$ induces an isomorphism $$\Gamma_\sigma \cong \Pr(Q_\sigma, W_\sigma, F).$$
\end{theorem}

\begin{proof}
 Let $\ic{\vec u, \vec v}$ be a non-frozen arrow of $Q_\sigma$. Then $\vec u$ is not a boundary edge winding counter-clockwisely around a hole. Thus, thanks to Proposition \ref{findminpol}, there is a unique minimal polygon $P$ having $\vec u$ as a oriented side. As $\sigma$ is a triangulation, it is immediate that $P$ is a triangle and $\M_P = \emptyset$. Moreover, $P$ has also $-\vec v$ as an oriented side and, up to numbering correctly the vertices of $P$, $\ic{\vec u, \vec v} = \omega_3^P$. Notice also that $\omega_3^P$ is part of $\alpha_M$ only for $M = s(\vec u)$. Thus, an elementary computation gives $\partial_{\omega_3^P} W_\sigma = R_{P,1}$. Moreover, all $R_{P, i}$ for minimal triangles $P$ are obtained in this way.

 Let $u \in Q_{\sigma, 0}$. There is an orientation $\vec u$ of $u$ such that $\vec u$ is not a boundary edge winding counter-clockwisely around a hole. Then taking the same notation as before for this $\vec u$, we get $$C_{\vec u} = \lambda_{s(\vec u)} \ic{\vec u, \vec u}^{m_{s(\vec u)}} - \lambda_{t(\vec u)} \ic{-\vec u, -\vec u}^{m_{t(\vec u)}} = (\partial_{\omega_1^P} W_\sigma) \omega_1^P - \omega_3^P (\partial_{\omega_3^P} W_\sigma) $$
 so, using Theorem \ref{thmmin}, we get $\Gamma_\sigma \cong \Pr(Q_\sigma, W_\sigma, F)$.
\end{proof}

We deduce easily the following result:
\begin{corollary}
 Using notation of Section \ref{nonfroz} for $\uQ_\sigma$, and denoting by $\underline{W}_\sigma$ the potential on $\uQ_\sigma$ induced by $W_\sigma$, we get
  $$\Delta_\sigma \cong \Pr(\uQ_\sigma, \underline{W}_\sigma)$$
 where $\Pr(\uQ_\sigma, \underline{W}_\sigma) := \Pr(\uQ_\sigma, \underline{W}_\sigma, \emptyset)$ is the usual Jacobian algebra.
\end{corollary}

Notice that in the case where there exist at least one special monogon in a triangulation, then $\Gamma_\sigma$ is not a basic algebra. We give a sketch of the method which permits to get a Morita equivalent basic Jacobian algebra. Suppose that there is in our triangulation a special monogon inducing a self-folded triangle. Let $\vec u$ be the oriented arc enclosing the special monogon and $\vec v$ the oriented arc pointing toward the special puncture. We can write the potential in the following way:
 $$W_\sigma = \ic{\vec v, -\vec u} \ic{\vec u, \vec v} \ic{-\vec v, -\vec v} - \lambda_{t(\vec v)} \ic{-\vec v, -\vec v} - \ic{\vec v, -\vec u} \cdot \io{-\vec u, \vec u} \cdot \ic{\vec u, \vec v} + \tilde W_\sigma$$
 where $\tilde W_\sigma$ does not contain any occurrence of $\ic{\vec v, -\vec u}$, $\ic{\vec u, \vec v}$ or $\ic{-\vec v, -\vec v}$. In particular
  $\partial_{\ic{-\vec v, -\vec v}} W_\sigma = \ic{\vec v, -\vec u} \ic{\vec u, \vec v} - \lambda_{t(\vec v)} e_v$
 and $\Pr(Q_\sigma, W_\sigma, F)$ is Morita equivalent to $(1-e_v)\Pr(Q_\sigma, W_\sigma, F) (1-e_v)$. Let us denote
 $$e_u^{\bullet} := \ic{\vec u, \vec v} \ic{\vec v, -\vec u} / \lambda_{t(\vec v)} \quad \text{and} \quad e_u^{\bowtie} := e_u - e_u^{\bullet}$$
 which are orthogonal idempotents of $(1-e_v)\Pr(Q_\sigma, W_\sigma, F)(1-e_v)$. We construct a quiver $Q'_\sigma$ in the following way: 
 $Q'_{\sigma, 0} := Q_{\sigma, 0} \setminus \{u, v\} \cup \{u^\bullet, u^{{\bowtie}}\}$ and $Q'_{\sigma, 1}$ consists of 
 \begin{itemize}
  \item all arrows of $Q_\sigma$ which are not incident to $u$ or $v$;
  \item for each arrow $\alpha$ pointing toward $u$ except $\ic{\vec v, -\vec u}$, two arrows $\alpha^\bullet$ and $\alpha^{\bowtie}$ with $s(\alpha^\bullet) = s(\alpha^{\bowtie}) = s(\alpha)$ and $t(\alpha^\bullet) = u^\bullet$ and $t(\alpha^{\bowtie}) = u^{\bowtie}$;
  \item for each arrow $\beta$ pointing from $u$ except $\ic{\vec u, \vec v}$, two arrows ${}^\bullet\beta$ and ${}^{\bowtie} \beta$ with $t({}^\bullet\beta) = t({}^{\bowtie}\beta) = t(\beta)$ and $s({}^\bullet \beta) = u^\bullet$ and $s({}^{\bowtie} \beta) = u^{\bowtie}$.
 \end{itemize}
 Finally, we define a potential on $Q'_\sigma$ by $W'_\sigma := - \lambda_{t(\vec v)} \underline{ \io{-\vec u, \vec u} } + \underline{ \tilde W_\sigma }$
 where
 \begin{itemize}
  \item $\underline{ \io{-\vec u, \vec u} }$ is obtained from $\io{-\vec u, \vec u}$ by substituting any arrow $\alpha$ pointing toward (respectively from) $u$ by $\alpha^\bullet$ (respectively ${}^\bullet \alpha$) and replacing the product $\ic{\vec u, \vec v} \ic{\vec v, -\vec u}$ as many times as it appears by a factor $\lambda_{t(\vec v)}$;
  \item $\underline{ \tilde W_\sigma }$ is obtained from $W_\sigma$ by substituting any arrow $\alpha$ pointing toward (respectively from) $u$ by $\alpha^\bullet + \alpha^{\bowtie} $ (respectively ${}^\bullet \alpha + {}^{\bowtie} \alpha$).
 \end{itemize}
 Then, it is easy to see that there is an isomorphism of algebras $$(1-e_v)\Pr(Q_\sigma, W_\sigma, F)(1-e_v) \cong \Pr(Q'_\sigma, W'_\sigma, F),$$ mapping 
 \begin{itemize}
  \item common elements of $Q_\sigma$ and $Q'_\sigma$ to themselves;
  \item $e_u^\bullet$ to $e_{u^\bullet}$ and $e_u^{\bowtie}$ to $e_{u^{\bowtie}}$;
  \item any arrow $\alpha$ pointing toward (respectively from) $u$ to $\alpha^\bullet + \alpha^{\bowtie}$ (respectively ${}^\bullet \alpha + {}^{\bowtie} \alpha$);
  \item $\ic{\vec u, \vec v}  \cdot \io{-\vec v, -\vec v}^\ell \cdot \ic{\vec v, -\vec u}$ to $\lambda_{s(\vec v)}\underline{ \io{-\vec u, \vec u} }^\ell e_{u^\bullet}$ for any $\ell \geq 0$.
 \end{itemize}
 Therefore, it gives a method to iteratively find a basic Jacobian algebra Morita equivalent to $\Pr(Q_\sigma, W_\sigma, F)$ as special monogons never share any arc under the hypotheses of this paper. We refer to \cite{CeLa12} for a more direct construction of the basic quiver with potential in the case where $m_M = 1$ for any $M \in \M$.

 As a corollary, we get:
 \begin{corollary}
  If $\sigma$ is a triangulation and $n_M = 1$ for any $M \in \M$, then the algebra $\Delta_\sigma$ is Morita equivalent to the classical Jacobian algebra corresponding to this triangulation in \cite{CeLa12}.
 \end{corollary}

\section{Case of Brauer graph algebras} \label{brauer}

\begin{definition} \label{sparse}
 We say that a partial triangulation $\sigma$ is \emph{sparse} if 
 \begin{itemize}
  \item no arc of $\sigma$ is incident to a marked point on the boundary of $\Sigma$;
  \item $\sigma$ does not contain any triangle without puncture;
  \item $\sigma$ does not contain any special monogon;
  \item if $2 \notin k^\times$, $\sigma$ does not contain any monogon enclosing a unique puncture $M$ with $m_M = 2$.
 \end{itemize}
\end{definition}

\begin{theorem} \label{thmbrauer}
 The following hold:
 \begin{enumerate}[\rm (a)]
  \item If $\sigma$ is sparse then $\Delta_\sigma \cong k \uQ_\sigma / I$ where $I$ is the ideal generated by the relations $C_{\vec u}$ and the relations $\ic{\vec u, \vec v} \ic{-\vec v, -\vec w}$ for any triple of oriented arcs $\vec u, \vec v, \vec w$ satisfying $s(\vec u) = s(\vec v)$ and $t(\vec v) = t(\vec w)$. In other terms, $\Delta_\sigma$ is the \emph{Brauer graph algebra} corresponding to the Brauer graph underlying to $\sigma$.
  \item Any Brauer graph algebra is $\Delta_\sigma$ for a partial triangulation of a marked surface.
 \end{enumerate}  
\end{theorem}

\begin{proof}
 (a) is an immediate consequence of Theorem \ref{altpres2}.

 (b) can be proven in the following classical way. Attach to each vertex $M$ of a Brauer graph a closed $2d_M$-gon $A_{\vec v_1} B_{\vec v_1} A_{\vec v_2} B_{\vec v_2} \cdots A_{\vec v_{d_M}} B_{\vec v_{d_M}}$ with $M$ at its centre, where $\vec v_1$, \dots, $\vec v_{d_M}$ are the oriented edges of the Brauer graph starting at $M$, cyclically ordered. Then, define $$\Sigma := \frac{\bigsqcup_{M} A_{\vec v_1} B_{\vec v_1} A_{\vec v_2} B_{\vec v_2} \cdots A_{\vec v_{d_M}} B_{\vec v_{d_M}}}{\sim}$$
 where the equivalence relation $\sim$ identifies $A_{\vec v} B_{\vec v}$ with $B_{-\vec v} A_{-\vec v}$ for any oriented edge $\vec v$ of the Brauer graph. Then the Brauer graph can be embedded in $\Sigma$ by drawing each oriented edge $\vec x$ orthogonally to $A_{\vec x} B_{\vec x} = B_{-\vec x} A_{-\vec x}$. In this case, no marked point of $\Sigma$ is on the boundary and every polygon has a hole so $\sigma$ is sparse.
\end{proof}

Thanks to Theorem \ref{thmbrauer}, we generalize (mildly) a result of Schroll \cite{Sc15}. For that recall that the trivial extension of a $k$-algebra $A$ is the $k$-algebra $A \oplus \Hom_k(A, k)$ with multiplication defined by $(a, f)(a',f') = (aa', af' + f'a)$. 

\begin{proposition} \label{trivext}
 If $\sigma$ is a partial triangulation of $(\Sigma, \M)$ such that every arc of $\sigma$ links two marked points on the boundary, without special monogon, then the trivial extension of $\Delta_\sigma$ is (canonically) isomorphic to the Brauer graph algebra with Brauer graph consisting of the arcs of $\sigma$ and multiplicity $1$ at each vertex.
\end{proposition}

In Proposition \ref{trivext}, \emph{canonically} means that $\uQ_\sigma$ is canonically embedded in the quiver defining the Brauer graph algebra and that this embedding induces the isomorphism.

\begin{proof}
 First of all, notice that all relations $\RR^\circ_{\vec v}$ of Proposition \ref{altpres0} are of the form $\RR^\circ_{\vec v} = \ic{\vec u, -\vec v} \ic{\vec v, -\vec w}$ (indeed, $f^\circ_{\vec v}$ always goes through a boundary). Thus, up to adding some punctures (without incident arcs), we can suppose that every polygon $P$ has at least three punctures and we can suppose that $m_M = 1$ for any $M \in \M$. Thus, it is enough to check that $\Delta_{\sigma'}$ (see Definition \ref{defsp}) is the trivial extension of $\Delta_\sigma$ ($\sigma'$ is sparse so $\Delta_{\sigma'}$ is the Brauer graph algebra by Theorem \ref{thmbrauer}). Using Definition \ref{trace}, Lemma \ref{commC2}, there is an isomorphism $\Hom_k(\Delta_{\sigma'}, k)$ to $\Delta_{\sigma'}$ given by
 $$e_u^* \mapsto c_u, \quad \ic{\vec u, \vec v}^* \mapsto \lambda_{s(\vec u)} \ic{\vec v, \vec u}, \quad c_u^* \mapsto e_u$$
 (notice that in this case, there are no $2$-special arcs). Moreover, $e_u \in \Delta_\sigma$ and $c_u \notin \Delta_\sigma$ for any arc $u$ of $\sigma$ and exactly one of $\ic{\vec u, \vec v}$ and $\ic{\vec v, \vec u}$ is in $\Delta_\sigma$ for any choice of two oriented arcs $\vec u, \vec v \in \sigma$ such that $s(\vec u) = s(\vec v)$. It permits to conclude.
\end{proof}

\section{Representation type of $\Delta_\sigma$} \label{reptype}

In this section, we prove the following result, generalizing \cite{GeLaSc16} in the case of triangulations of surfaces and \cite{WaWa85} in the case of Brauer graph algebras:

\begin{theorem} \label{tame}
 If $k$ is an algebraically closed field, the algebra $\Delta_\sigma$ is representation tame for any partial triangulation $\sigma$ of $(\Sigma, \M)$.
\end{theorem}

When $\Sigma$ has no boundary and $\sigma$ is a triangulation, we can be more precise. Recall the following definition from Ladkani \cite{La3}:

\begin{definition}[{\cite{La3}}]
 An algebra $A$ is of \emph{quasi-quaternion type} if it is of tame representation type, symmetric, indecomposable and for any $X \in \underline{\mod} A$, $\Omega^4 X \cong X$ where $\Omega$ is the syzygy functor in the stable category of $\mod A$.
\end{definition}

\begin{proposition} \label{qq}
 If $k$ is an algebraically closed field, $\Sigma$ has no boundary and $\sigma$ is a triangulation of $(\Sigma, \M)$, then $\Delta_\sigma$ is of quasi-quaternion type.
\end{proposition}

\begin{remark}
 In \cite{La3}, Ladkani states a similar result. The final version of this paper, containing proofs, is not available yet.
\end{remark}

Proposition \ref{qq} is a consequence of Theorems \ref{symm}, \ref{tame} together with the following lemma, which will be proven at the end of Subsection \ref{proofsymm}:

\begin{proposition} \label{bimodres}
 If $\Sigma$ has no boundary and $\sigma$ is a triangulation of $(\Sigma, \M)$, then there is an exact sequence of $\Delta_\sigma$-bimodules of the form:
 $$\xymatrix{ 
   0 \ar[r] & \Delta_\sigma \ar[r]^-{\alpha^*} & \bigoplus_{u \in \sigma} \Delta_\sigma e_u \otimes e_u \Delta_\sigma \ar[r]^-{\beta^*} & \bigoplus_{\vec u \in \sigma} \Delta_\sigma e_{\vec u^+} \otimes e_{\vec u} \Delta_\sigma \ar[d]^-\gamma \\
   0 & \Delta_\sigma \ar[l] & \bigoplus_{u \in \sigma} \Delta_\sigma e_u \otimes e_u \Delta_\sigma \ar[l]^-\alpha & \bigoplus_{\vec u \in \sigma} \Delta_\sigma e_{\vec u} \otimes e_{\vec u^+} \Delta_\sigma \ar[l]^-\beta
 }$$
 where $\vec u^+$ is the oriented arc following immediately $\vec u$ around $s(\vec u)$ and $-^*$ is the duality of Theorem \ref{symm}.
\end{proposition}

Notice that Proposition \ref{bimodres} does not need $k$ to be an algebraically closed field. From now on, we suppose that $k$ is an algebraically closed field.
Following the same strategy than in \cite{GeLaSc16}, we use the known result for Brauer graph algebras and use the following result of Crawley-Boevey:

\begin{theorem}[{\cite[Theorem B]{CB95}}] \label{thcb}
 Let $A$ be a finite dimensional $k$-algebra, let $X$ be an irreducible algebraic variety over $k$ and let $g_1, \dots, g_r: X \to A$ be morphisms of varieties. For $x \in X$, denote $A_x := A/(g_i(x))_{i = 1 \dots r}$. If there exists a non-empty open subset $U$ of $X$ such that $A_x \cong A_{x'}$ for all $x, x' \in U$ and there exists $x_0 \in X$ such that $A_{x_0}$ is of tame representation type then $A_x$ is of tame representation type for $x \in U$.
\end{theorem}


To prove Theorem \ref{tame}, let us take notations of Definition \ref{defsp}. The quotient $\Delta_{\sigma'} \surj \Delta_\sigma$ induces a full and faithful functor $\mod \Delta_\sigma \to \mod \Delta_{\sigma'}$ so we can suppose that $\Sigma$ has no boundary. In the same way, for $\tau \subset \sigma$, the isomorphism $e_\tau \Delta_\sigma e_\tau \cong \Delta_\tau$ of Corollary \ref{subtri2} gives a full and faithful functor $\Delta_\sigma \otimes_{\Delta_\tau} -: \mod \Delta_\tau \to \mod \Delta_\sigma$ so we can suppose that $\sigma$ is a triangulation. In this case, minimal polygons of $\sigma$ are all triangles (Case a of Figure \ref{polnonz}).

Define
 $$A := \frac{k \uQ_\sigma}{JC_\sigma + (C_{\vec v})_{\vec v}}$$
which is obviously finite dimensional. We will define a family of functions $\RR_{\vec v}: k \to A$ indexed by oriented arcs $\vec v$ of $\sigma$ such that $A_x := A/(\RR_{\vec v}(x))_{\vec v}$ satisfies that $A_x \cong \Delta_\sigma$ if $x \neq 0$ and $A_0$ is a quotient of the Brauer graph algebra with Brauer graph coinciding with the partial triangulation. It will permit to conclude thanks to Theorem \ref{thcb} because $A_0$ is representation tame.

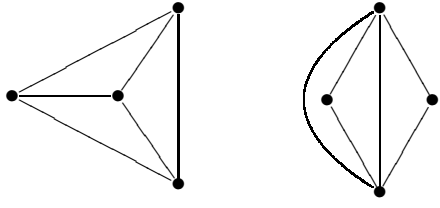
\begin{figure}
 $$\xymatrix@L=.05cm@R=0.3cm@C=0.3cm@M=0.0cm{
   & & & & & &  \ar@{-}[llddd] \ar@{-}[dddddd] \bullet \\ \\ \\ 
    \ar@{-}[uuurrrrrr] \ar@{-}[dddrrrrrr] \ar@{-}[rrrr] \bullet & & & & \ar@{-}[rrddd] \bullet \\ \\ \\
   & & & & & & \bullet
 } \quad \quad
 \xymatrix@L=.05cm@R=0.5cm@C=0.5cm@M=0.0cm{
  & & & \bullet \ar@{-}[ddl] \ar@/^-1cm/@{-}[dddd] \ar@{-}[dddd]  \ar@{-}[ddr] \\ \\
  & & \bullet \ar@{-}[ddr] & & \bullet \ar@{-}[ddl] \\ \\ 
  & & & \bullet
 }
$$
\caption{Two special triangulations of the sphere with four marked point and multiplicities $1$} \label{spectria}
\end{figure}

We first state a key lemma:

\begin{lemma} \label{pl}
 If $\sigma$ is not one of the two triangulations of Figure \ref{spectria}, there exists $p: Q_{\sigma, 1} \to \N_{> 0}$ satisfying that
 \begin{enumerate}[\rm (1)]
  \item for $M \in \M$, $\kappa := m_M \sum_{i = 1}^{d_M} p_{\ic{\vec u_i, \vec u_{i+1}}}$ is a constant independent of $M$, where $\vec u_1$, \dots, $\vec u_{d_M}$ are the oriented arcs starting at $M$ cyclically ordered;
  \item For any (minimal) triangle $P$ of $\sigma$, 
   \begin{itemize}
    \item $\kappa_P := p_{\omega^P_1} + p_{\omega^P_2} + p_{\omega^P_3} > \kappa$ if $P$ is a special self-folded triangle;
    \item $\kappa_P := p_{\omega^P_1} + p_{\omega^P_2} + p_{\omega^P_3} < \kappa$ else.
   \end{itemize}
 \end{enumerate}
\end{lemma}

We finish the construction before proving Lemma \ref{pl}. We fix the notations of Lemma \ref{pl}. For an oriented arc $\vec v$ of $\sigma$ and $x \in k$, taking $\vec u$ and $\vec w$ such that $\vec w$, $\vec v$ and $\vec u$ form a minimal triangle $P$ in this order, we define:
 $$\RR_{\vec v}(x) := \left\{\begin{array}{ll}
			      x^{\kappa_P - \kappa} \ic{\vec u, -\vec v} \ic{\vec v, -\vec w} -  \io{-\vec u, \vec w} & \text{if $P$ is special self-folded;} \\
                              \ic{\vec u, -\vec v} \ic{\vec v, -\vec w} - x^{\kappa - \kappa_P} \io{-\vec u, \vec w} & \text{else.}                              
                             \end{array} \right.$$

Then, we prove that $A_x \cong \Delta_\sigma$ for any non-zero $x$.
 We consider the automorphism $\psi$ of $k Q_\sigma$ defined by $\psi(q) := x^{p_q} q$ for any $q \in Q_{\sigma, 1}$. Then, for any oriented arc $\vec v$ and the minimal triangle $P$ with sides $\vec w$, $\vec v$, $\vec u$ containing it, we have $\psi(C_{\vec v}) = x^\kappa C_{\vec v}$, $\psi(J C_\sigma) = J C_\sigma$ and 
 $$\psi(\RR_{\vec v}) = \left\{\begin{array}{ll} x^{\kappa - \kappa_{\ic{-\vec u, \vec w}}}\RR_{\vec v}(x) & \text{if $P$ is special self-folded;} \\ x^{\kappa_P - \kappa_{\ic{-\vec u, \vec w}}}\RR_{\vec v}(x) & \text{else.} \end{array}\right.$$
 so $\psi$ induces an isomorphism from $\Delta_\sigma$ to $A_x$.

Moreover, it is immediate that $A_0$ is the quotient of the Brauer graph algebra with Brauer graph induced by the partial triangulation modulo the idempotents corresponding to self-folded edges of special self-folded triangle. It concludes the proof of Theorem \ref{tame} except in the two cases of Figure \ref{spectria}. In this latter case, we use Theorem \ref{thmtilt}. Indeed, by \cite[Corollary 2.2]{Ri89b}, for self-injective algebras, derived equivalence implies stable equivalence, which obviously implies invariance of the representation type. As there are other triangulations of the sphere with four punctures, the result for cases of Figure \ref{spectria} follows.

\begin{proof}[Proof of Lemma \ref{pl}]
 Up to rescaling, we can choose the $p_q$'s in $\Q_{> 0}$ and $\kappa = 1$ as these properties are invariant by scalar multiplication by positive integers. For $\vec u, \vec v$ such that $\ic{\vec u, \vec v}$ is an arrow, denote $p^\circ_{\ic{\vec u, \vec v}} := 1/(d_{s(\vec u)} m_{s(\vec u)})$. We start by proving that the $p^\circ_{\ic{\vec u, \vec v}}$'s almost satisfy (1) and (2). Indeed (1) is satisfied. We know that $1/a + 1/b + 1/c < 1$ except if $(a,b,c)$ is in
 $$E = \mathfrak{S}_3 \{(1, b, c), (2,2,c), (2,3,3), (2,3,4), (2,3,5), (2,3,6), (2,4,4), (3,3,3)\}.$$
 So we should identify (minimal) triangles $P$ of $\sigma$ such that $$(d_{P_1}m_{P_1}, d_{P_2}m_{P_2}, d_{P_3}m_{P_3}) \in E.$$
 Let us take such a $P$ and choose $P_1$ such that $d_{P_1} m_{P_1} \leq d_{P_2}m_{P_2}, d_{P_3}m_{P_3}$. 

 If $d_{P_1} m_{P_1} = 1$, then it is immediate that $P$ satisfies (2) (this is the case where $P$ is a special self-folded triangle). 

 If $d_{P_1} \in \{1,2\}$, an easy case by case analysis using hypotheses on $(\Sigma, \M)$ shows that each vertex $N$ sharing an edge with $P_1$ satisfies $d_N m_N \geq 4$. Thus the only possibility which makes (2) fail is $(d_{P_1} m_{P_1}, d_{P_2} m_{P_2}, d_{P_3} m_{P_3}) = (2,4,4)$. In this case, we have $p^\circ_{\omega^P_1} + p^\circ_{\omega^P_2} + p^\circ_{\omega^P_3} = \kappa$.

 If $m_{P_1} = 1$ and $d_{P_1} = 3$, a quick analysis shows that at least one $N$ connected to $P_1$ satisfies $d_N m_N \geq 4$. Thus, $P$ satisfy (2) in this case. Then for any (minimal) triangle $P$ of $\sigma$, 
 \begin{itemize}
  \item $p^\circ_{\omega^P_1} + p^\circ_{\omega^P_2} + p^\circ_{\omega^P_3} > \kappa$ if $P$ is a special self-folded triangle;
  \item $p^\circ_{\omega^P_1} + p^\circ_{\omega^P_2} + p^\circ_{\omega^P_3} \leq \kappa$ else.
 \end{itemize}
 We get the $p_q$'s by adding some small perturbations to the $p^\circ_q$'s. It is immediate that it is possible if every vertex $M$ with $d_M m_M = 4$ is in at least one triangle where the strict inequality is already satisfied (in this case, it is enough to add small numbers to the $p^\circ_q$'s appearing in a triangle satisfying the equality and to compensate this addition by subtracting a small number to the $p^\circ_q$ appearing in a triangle satisfying the strict inequality).

 Thus, we suppose that $M$ with $d_M m_M = 4$ appears only in triangles $P$ satisfying $p^\circ_{\omega^P_1} + p^\circ_{\omega^P_2} + p^\circ_{\omega^P_3} = \kappa$. We have $d_M = 4$ and $m_M = 1$ (other cases have already been excluded). Denote by $\vec u_1$, $\vec u_2$, $\vec u_3$ and $\vec u_4$ the cyclically ordered oriented arcs starting at $M$. According to our hypotheses, we can suppose that $d_{t(\vec u_1)} m_{t(\vec u_1)} = d_{t(\vec u_3)} m_{t(\vec u_3)} = 2$ and $d_{t(\vec u_2)} m_{t(\vec u_2)} = d_{t(\vec u_4)} m_{t(\vec u_4)} = 4$. We consider two cases:
 \begin{itemize}
  \item If $t(\vec u_2) = M$. In this case, $\vec u_4 = - \vec u_2$. As $\sigma$ is a triangulation, using Proof of Lemma \ref{findminpol}, there is a (minimal) triangle with sides $-\vec u_3$, $\vec u_2$, $\vec u_3$ in this order and therefore $d_{t(\vec u_3)} = 1$. In the same way, $d_{t(\vec u_1)} = 1$ and $(\Sigma, \M)$ is a sphere with three punctures. It contradicts our hypothesis that, in this case, $m_M \geq 2$.
  \item If $t(\vec u_2) \neq M$. As before, there is a (minimal) triangle with sides $-\vec u_3$, $\vec u_2$, $\vec v$ in this order for a certain $\vec v$. Notice that $s(\vec v) = t(\vec u_2)$ and $t(\vec v) = t(\vec u_3)$. So $v$ is not a $u_i$. So, using the hypotheses, $d_{t(\vec u_3)} = 2$ and $m_{t(\vec u_3)} = 1$. We get in the same way triangles with sides $-\vec u_4$, $\vec u_3$, $-\vec v$ and $-\vec u_1$, $\vec u_4$, $\vec v'$ and $-\vec u_2$, $\vec u_1$, $-\vec v'$ for a certain $\vec v'$. An easy analysis proves that these four triangles cover $\Sigma$ and therefore we are in the second case of Figure \ref{spectria}. It contradicts our hypothesis. \qedhere
 \end{itemize}   
\end{proof}

\section{Flipping partial triangulations} \label{secflip}

We start by defining a combinatorial operation on partial triangulations of $\sigma$.

\begin{definition} \label{defflip}
 Let $\vec u$ be an oriented arc of $\sigma$. If $u$ is the only arc incident to $s(\vec u)$, we define $\vec u^+ := \emptyset$. If $\vec u$ is followed by $\vec v$ around $s(\vec u)$ with $\vec v \neq -\vec u$, we define $\vec u^+ := \vec v$. If $\vec u$ is followed by $-\vec u$ and $u$ is not the only arc incident to $s(\vec u)$, we define $\vec u^+ := (-\vec u)^+$. We says that $u$ is \emph{close to the boundary} if $\vec u^+$ or $(-\vec u)^+$ is a boundary edge.

 Suppose now that $\vec u$ is not close to the boundary. Let $\tau$ be the partial triangulation obtained from $\sigma$ by removing $u$. The \emph{mutation} or \emph{flip of $\sigma$ with respect to $u$} is the partial triangulation $\mu_u(\sigma)$ obtained from $\tau$ by adding the oriented arc $\vec u^*$ constructed in the following way: 
 \begin{enumerate}[\rm (a)]
  \item if $\vec u^+ \neq \emptyset$, it starts at $t(\vec u^+)$ just after $-\vec u^+$ in the cyclic ordering around $t(\vec u^+)$, it follows $-\vec u^+$, winds around $s(\vec u)$ clockwisely; \\ if $\vec u^+ =\emptyset$, it starts at $s(\vec u)$ at the same position as $\vec u$;
  \item if follows $\vec u$;
  \item if $(-\vec u)^+ \neq \emptyset$ it winds around $t(\vec u)$ counter-clockwisely, follows $(-\vec u)^+$ and ends at $t((-\vec u)^+)$ just after $(-\vec u)^+$; \\  if $(-\vec u)^+ = \emptyset$, it ends at $t(\vec u)$ at the same position as $\vec u$.  
 \end{enumerate}
\end{definition}

To summarize Definition \ref{defflip}, we depict in Figure \ref{figflip} the three main possibilities (up to reorientation of $\vec u$). 
\begin{figure}
 \begin{center} \begin{tabular}{ccccc}
                 \boxinminipage{$$\xymatrix@L=.05cm@!=0cm@R=.5cm@C=.7cm@M=0.01cm{
                  & & & \\
		    & & &  \\
			  & & & & \\
		    \ar@{..}[d] & & \ar@{-->}[uur]^{\vec u^*} \ar[ll] \ar@{..}[d]  & & \ar[uul]_{\vec u} \ar[ll]^{\vec u^+} \ar@{..}[d]  \\ & & & & 
                 }$$} &  &
		\boxinminipage{$$\xymatrix@L=.05cm@!=0cm@R=.5cm@C=.7cm@M=0.01cm{
                     & & &  \cdots \\
			   & & & \\
		   \ar@{..}[d] & & \ar@`{p+(14,10),p+(8,14),p+(0,10)}@{-->}[]_(.7){\vec u^*} \ar@{..}[d] \ar[ll] & & \ar@`{p+(0,10),p+(-8,14),p+(-14,10)}_(.3){\vec u} \ar[ll]_{\vec u^+}^{ (-\vec u)^+} \ar@{..}[d]  \\ & & & & 
                 }$$} &  & 
                \boxinminipage{$$\xymatrix@L=.05cm@!=0cm@R=.5cm@C=.7cm@M=0.01cm{
                  & & & & & \\
		    & \ar@{..}[u] \ar[rr]^{(-\vec u)^{+}}  & & \ar@{..}[u] \ar[rr] & & \ar@{..}[u]  \\ & & & & 
			  & \\
		    \ar@{..}[d] & & \ar@{-->}[uur]^(.2){\vec u^*} \ar[ll] \ar@{..}[d]  & & \ar[uulll]_(.2){\vec u} \ar[ll]^{\vec u^+} \ar@{..}[d]  \\ & & & & 
                 }$$} \\
	    (F1) & &  (F2) & & (F3)
                \end{tabular}  
 \end{center}
 \caption{Flip of an oriented arc in a partial triangulation} \label{figflip}
\end{figure}
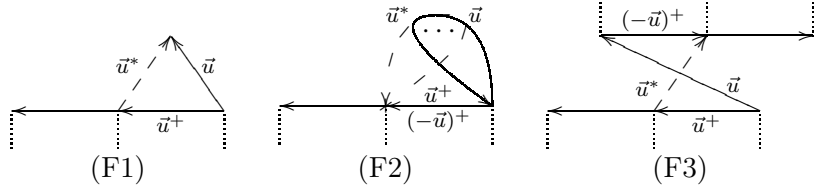

We need also to define a mutation for coefficients:
\begin{definition} \label{defmutcoeff} We define coefficients $\mu_u(\lambda)_M = \lambda^*_M$ in the following way:
 \begin{itemize}
  \item If there exists a monogon $\vec x$ enclosing a unique puncture $M$ and  
  \begin{itemizedec}
   \item[either] $\vec u = \pm \vec x$ and $\sigma$ does not contain an arc incident to $M$,
   \item[or] $\vec u$ is incident to $M$,
  \end{itemizedec}
   $$\lambda^*_M := -\lambda_M ;\quad \lambda^*_{s(\vec x)} := \nu_\M^{-1} \lambda_{s(\vec x)};  \quad \lambda_N^* := \lambda_N \text{ for $N \in \M \setminus \{s(\vec x), M\}$}.$$
  \item For $\vec x = \pm \vec u$, if $(-\vec x)^+ = \emptyset$ and $\vec x$ is not a side of a self-folded triangle, $$\lambda^*_{t(\vec x)} = -\lambda_{t(\vec x)} ;\quad \lambda^*_{s(\vec x)} := (-1)^{m_{s(\vec x)}} \lambda_{s(\vec x)};  \quad \lambda_N^* := \lambda_N \text{ for other $N \in \M$}.$$
  \item In any other case, $\lambda^*_N := \lambda_N$ for any $N \in \M$.
 \end{itemize}
\end{definition}

The operation which maps $\sigma$ to $\mu_u(\sigma)$ is sometimes called \emph{Kauer move} as it was first introduced in \cite{Ka98} for Brauer graphs. The main theorem of this section generalizes \cite{Ka98} (see also \cite{Ai15}):

\begin{theorem} \label{thmtilt}
 If $u$ is an arc of a partial triangulation $\sigma$ of $(\Sigma, \M)$ which is not close to the boundary, then $\Delta^\lambda_\sigma$ and $\Delta^{\mu_u(\lambda)}_{\mu_u(\sigma)}$ are derived equivalent.
\end{theorem}

\begin{example}
 We consider the two following partial triangulations of a disc with three punctures and no marked point on the boundary:
 \begin{center}
  \boxinminipage
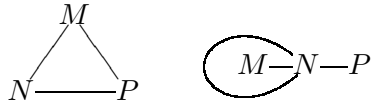
{$$\xymatrix@L=.05cm@!=0cm@R=1cm@C=.7cm@M=0.01cm{
                  & M \ar@{-}[dl] \ar@{-}[dr] & \\
                  N \ar@{-}[rr] & & P
                 }$$} \quad \quad \quad 
  \boxinminipage{$$\xymatrix@L=.05cm@!=0cm@R=.5cm@C=.7cm@M=0.01cm{
                   \\
                  M \ar@{-}[r] & N \ar@{-}[r] \ar@{-}@`{p+(-7,7), p+(-20,0), p+(-7,-7)}[] & P  \\
                 }$$}  
 \end{center}
 They are related by a flip so the following algebras, obtained for $\lambda_M = \lambda_N = \lambda_P = 1$ and $m_M = m_N = m_P = m$ are derived equivalent:
 $$\frac{k \left(\boxinminipage{\xymatrix@L=.05cm@!=0cm@R=2cm@C=1.4cm{
        & \bullet \ar@/_/[dr]_x \ar@/_/[dl]_y \\
        \bullet \ar@/_/[ur]_x \ar@/_/[rr]_y & & \bullet \ar@/_/[ul]_y \ar@/_/[ll]_x}}\right)}{(x^2-(yx)^{m-1}y, y^2)}
  \quad \text{and} \quad
   \frac{k \left(\boxinminipage{\xymatrix@L=.05cm@!=0cm@R=1cm@C=1.2cm{
            \\
         & \bullet \ar@/_/[r]_{\beta_1} \ar@`{p+(-5,5), p+(-15,0), p+(-5,-5)}[]_\alpha & \bullet \ar@/_/[r]_{\gamma_1} \ar@/_/[l]_{\beta_2} & \bullet \ar@/_/[l]_{\gamma_2} \ar@`{p+(5,-5), p+(15,0), p+(5,5)}[]_\delta & \\  & }}
        \right)}{\left(\begin{array}{c} \beta_2 \alpha - (\gamma_1 \gamma_2 \beta_2 \beta_1)^{m-1} \gamma_1 \gamma_2 \beta_2, \\ \alpha \beta_1 - (\beta_1 \gamma_1 \gamma_2 \beta_2)^{m-1} \beta_1 \gamma_1 \gamma_2, \\ \beta_1 \beta_2 - \alpha^{m-1}, \gamma_1 \delta, \delta \gamma_2, \gamma_2 \gamma_1,\\ \delta^m - (\gamma_2 \beta_2 \beta_1 \gamma_1)^m \end{array}\right)}.
$$

\end{example}

\begin{remark}
 We can of course reverse this construction to obtain an inverse of $\mu_u$ which also gives a derived equivalence.
\end{remark}

Before giving the proof, we deal with two particular cases coming from special monogons: 
\begin{remark} \label{avoidcases}
\begin{itemize}
 \item  In case (F2), if $\vec u$ encloses a special monogon and the arc $v$ pointing to the special puncture is not in $\sigma$. By Proposition \ref{revidem}, and the discussion preceding it, $\Delta^\lambda_\sigma$ is Morita equivalent to $\Delta^{\mu_u(\lambda)}_{\sigma \cup \{v\}}$. Then, we get easily that $\Delta^{\mu_u(\lambda)}_{\mu_u(\sigma)}$ is Morita equivalent to $\Delta^{\mu_u(\lambda)}_{\mu_v(\mu_u(\sigma \cup \{v\}))}$.
 \item If $\vec u$ is the double side of a special self-folded triangle, pointing toward the special puncture, $\mu_u(\sigma) \cong \sigma$ and $\Delta^{\mu_u(\lambda)}_{\mu_u(\sigma)}$ is Morita equivalent to $\Delta^\lambda_\sigma$ according to Proposition \ref{revidem} so the result is trivial in this case. 

  Notice that, in this case, $\Delta^{\mu_u(\lambda)}_{\mu_u(\sigma)}$ is not obtained by a non-trivial tilting. Let $\vec v$ be the oriented side enclosing the special monogon with special puncture $t(\vec u)$ or $s(\vec u)$. If $\vec v$ is not close to the boundary, a non-trivial tilting can be obtained by putting $\mu'_u(\sigma) = \mu_v(\sigma)$ (notice that $\Delta^{\mu_u(\lambda)}_{\mu'_u(\sigma)} \neq \Delta^{\mu_v(\lambda)}_{\mu_v(\sigma)}$).
 Indeed, using Proposition \ref{revidem}, $\Delta^\lambda_\sigma$ is Morita equivalent to $\Delta^{\mu_u(\lambda)}_\sigma$, ``exchanging'' the idempotents $\ic{\vec v, -\vec v}/\lambda_{t(\vec u)}$ and $e_v - \ic{\vec v, -\vec v}/\lambda_{t(\vec u)}$. Moreover, we will see that $\mu_v(\sigma)$ corresponds to tilting at the idempotent $e_v - \ic{\vec v, -\vec v}/\lambda_{t(\vec u)}$ so computing the algebra of $\mu_v(\sigma)$ with respect to the coefficients ${\mu_u(\lambda)}$ consists to a tilting at the idempotent $\ic{\vec v, -\vec v}/\lambda_{t(\vec u)}$ which is equivalent to $e_u$.
\end{itemize}
\end{remark}

In view of Remark \ref{avoidcases}, we suppose that, in Case (F2), $\vec u$ does not enclose a special monogon, and that, in Case (F1), $\vec u$ is not the double side of a special self-folded triangle.

To prove Theorem \ref{thmtilt}, we use the Okuyama-Rickard complex \cite{Ri89b}. We prove that it is tilting under our assumptions. From now on, $\sigma$, $\vec u$, $\tau$, $\lambda^* := \mu_u(\lambda)$ and $\sigma^* := \mu_u(\sigma)$ are fixed as in Definitions \ref{defflip} and \ref{defmutcoeff}. We denote 
 $$\e := \left\{\begin{array}{ll}
            e_u - \lambda_M \ic{\vec u, -\vec u} & \text{in case (F3), if $\vec u$ encloses a special monogon;} \\
            e_u - \lambda_M \ic{-\vec u, \vec u} & \text{in case (F3), if $-\vec u$ encloses a special monogon;} \\
            e_u & \text{else,}
           \end{array}\right.$$
 and $X_\e$ is the maximal indecomposable module supported by $\e$.

We consider the following object of the homotopy category $\Kb(\proj \Delta_\sigma)$ of right $\Delta_\sigma$-modules: $T := e_\tau \Delta_\sigma \oplus P_u^*$ where $P_u^*$ is the following complex concentrated in degree $0$ and $1$:
$$e_{\vec u^+} \Delta_\sigma \oplus e_{(-\vec u)^+} \Delta_\sigma \xrightarrow{\begin{sbmatrix} \alpha_1 & \alpha_2 \end{sbmatrix}} e_u \Delta_\sigma$$
where $\alpha_1 := \ic{\vec u, \vec u^+}$ and $\alpha_2 := \ic{-\vec u, (-\vec u)^+}$ with the convention that $e_\emptyset = 0$, $\ic{\vec u, \emptyset} = \ic{-\vec u, \emptyset} = 0$. This complex can be understood as a projective presentation of $X_\e$. In other terms, $\begin{sbmatrix} \alpha_1 & \alpha_2 \end{sbmatrix}$ is a minimal right $\add (e_\tau \Delta_\sigma)$-approximation of $e_u \Delta_\sigma$. Notice that $P_u^*$ contains a split direct summand if $u$ encloses a special monogon.

This is known that $T$ is a tilting complex when $\Delta_\sigma$ is symmetric. We prove that it is still the case here:

\begin{lemma}
 The complex $T$ is tilting. In other terms, 
 \begin{enumerate}[\rm (a)]
  \item $\Hom_{\Kb(\proj \Delta_\sigma)}(T, T[i]) = 0$ for any $i \neq 0$ where $[i]$ is the $i$-shift functor;
  \item $\Delta_\sigma$ is in the triangulated subcategory of $\Kb(\Delta_\sigma)$ generated by $T$ (and therefore $\Kb(\proj \Delta_\sigma)$ is generated by $T$).
 \end{enumerate}
\end{lemma}

\begin{proof}
 (a) The only possibly non-trivial terms of $\Hom_{\Kb(\proj \Delta_\sigma)}(T, T[i])$ are 
  \begin{itemize}
   \item $\Hom_{\Kb(\proj \Delta_\sigma)}(e_\tau \Delta_\sigma, P_u^*[1])$: as there is no morphism from $e_\tau \Delta_\sigma$ to $X_\e$, it is immediate that any element of $\Hom_{\Kb(\proj \Delta_\sigma)}(e_\tau \Delta_\sigma, P_u^*[1])$ is homotopic to $0$.
   \item $\Hom_{\Kb(\proj \Delta_\sigma)}(P_u^*, e_\tau \Delta_\sigma[-1])$: such a non-zero morphism gives a non-zero map from $X_\e$ to $e_\tau \Delta_\sigma$. But, using Theorem \ref{basisDsig}, the socle of $e_\tau \Delta_\sigma$ has to be concentrated on $\tau$ and arcs which are close to the boundary. As $u$ is not close to the boundary, we get a contradiction.
   \item $\Hom_{\Kb(\proj \Delta_\sigma)}(P_u^*, P_u^*[-1])$: it is similar to the previous case.
   \item $\Hom_{\Kb(\proj \Delta_\sigma)}(P_u^*, P_u^*[1])$: a morphism in this space is induced by a morphism from $e_{\vec u^+} \Delta_\sigma \oplus e_{(-\vec u)^+} \Delta_\sigma$ to $e_u  \Delta_\sigma$. As $\begin{sbmatrix}\alpha_1 & \alpha_2\end{sbmatrix}$ is a right $\add(e_\tau \Delta_\sigma)$-approximation, such a morphism is homotopic to $0$.
  \end{itemize}
 (b) It is enough to prove that $e_u \Delta_\sigma$ is in the triangulated category generated by $T$. This is immediate: $e_u \Delta_\sigma$ is isomorphic to the cone of the canonical map $P_u^* \to e_{\vec u^+} \Delta_\sigma \oplus e_{(-\vec u)^+} \Delta_\sigma$.
\end{proof}

Then, to prove Theorem \ref{thmtilt}, by the famous Theorem of Rickard \cite{Ri89}, it is enough to prove the following proposition:

\begin{proposition} \label{endo}
 There is an isomorphism $\End_{\Kb(\proj \Delta_\sigma)} (T) \cong \Delta^{\lambda^*}_{\sigma^*}$. 
\end{proposition}

Proposition \ref{endo} is proven in Subsection \ref{proofendo}. We give a useful corollary which permits to get easily some isomorphisms:

\begin{corollary} \label{isoa}
 We suppose that $\Sigma$ has no boundary. Fix three partial triangulations $\sigma_1$, $\sigma_2$ and $\sigma^\circ$ of $(\Sigma, \M)$ such that:
 \begin{enumerate}[\rm (a)]
  \item $\sigma_1 \cup \sigma^\circ$ is a triangulation;
  \item $\sigma_1 \cap \sigma^\circ = \sigma_2 \cap \sigma^\circ = \emptyset$;
  \item any triangle of $\sigma_1 \cup \sigma^\circ$ that has a side in $\sigma^\circ$ has all its sides in $(\sigma_1 \cap \sigma_2) \cup \sigma^\circ$.
 \end{enumerate}
 We also consider a second family of coefficient $(\mu_M)_{M \in \M}$. 

 Then, for any isomorphism $\psi_1: \Delta^{\mu}_{\sigma_1} \to \Delta^{\lambda}_{\sigma_1}$ such that $\psi_1(e_u) = e_u$ for all $u \in \sigma_1$, there exists an isomorphism $\psi_2: \Delta^{\mu}_{\sigma_2} \to \Delta^\lambda_{\sigma_2}$ such that $\psi_2(e_u) = e_u$ for all $u \in \sigma_2$ and $\psi_2|_{\Delta^{\mu}_{\sigma_1 \cap \sigma_2}} = \psi_1|_{\Delta^{\mu}_{\sigma_1 \cap \sigma_2}}$, where $\Delta^{\mu}_{\sigma_1 \cap \sigma_2}$ is identified to subalgebras of $\Delta^{\mu}_{\sigma_1}$ and $\Delta^{\mu}_{\sigma_2}$ via Corollary \ref{subtri2}.
\end{corollary}

\begin{remark} \label{isoarem}
 Proof of Corollary \ref{isoa} only relies on case (F3). This is important as we will use it for the proof of case (F1).
\end{remark}

\begin{proof}
 Let us first suppose that $\sigma_2 \cup \sigma^\circ$ is a triangulation. Then, it is a classical fact that there exists a sequence of flips from $\sigma_1 \cup \sigma^\circ$ to $\sigma_2 \cup \sigma^\circ$ which do not involve arcs in their intersection $(\sigma_1 \cap \sigma_2) \cup \sigma^\circ$. Notice also that all flips involved are of type (F3) as other flips stabilize triangulations.

 Notice that the first flip $u$ of the sequence can be applied to $\sigma_1$. Indeed, $u \notin (\sigma_1 \cap \sigma_2) \cup \sigma^\circ$, so by (c) the quadrilateral $u$ is a diagonal of has all its sides in $\sigma_1$. Moreover, the partial triangulation $\mu_u(\sigma_1)$ satisfies the same hypotheses (a), (b) and (c) as $\sigma$. Thus, we suppose that $\sigma_2 = \mu_u(\sigma_1)$, and the result is obtained by induction.

 We consider the tilting object $T$ of $\Kb(\proj \Delta^\lambda_{\sigma_1})$ defined for $u$. As $T$ is a projective presentation of $X_\e$, we get that $\psi_1^*(T)$ is the tilting object defined for $u$ in $\Kb(\proj \Delta^\mu_{\sigma_1})$. Then, using Proposition \ref{endo}, we get an isomorphism
  $$\psi_2 : \Delta^\mu_{\sigma_2} \xrightarrow{\phi_a} \End_{\Kb(\proj \Delta^\mu_{\sigma_1})} (\psi_1^*(T)) \xrightarrow{(\psi_1^*)^{-1}} \End_{\Kb(\proj \Delta^\lambda_{\sigma_1})} (T) \xrightarrow{\phi_b} \Delta^\lambda_{\sigma_2}$$
 (notice that we have $\mu_u(\lambda) = \lambda$ and $\mu_u(\mu) = \mu$ as we are in case (F3)). The fact that $\psi_2$ satisfies $\psi_2(e_u) = e_u$ for $u \in \sigma_2$ and $\psi_2|_{\Delta^\mu_{\sigma_1 \cap \sigma_2}} = \psi_1|_{\Delta^\mu_{\sigma_1 \cap \sigma_2}}$ is an immediate consequence of the construction of $\phi_a$, $\phi_b$ (see Lemma \ref{defmorphtilt}).

 Finally, if $\sigma_2 \cup \sigma^\circ$ is not a triangulation, the result is obtained by completing $\sigma_2 \cup \sigma^\circ$ to a triangulation and using Corollary \ref{subtri2}.
\end{proof}

\section{Dealing with special monogons} \label{spemon}

Let us fix $\vec u \in \sigma$ enclosing a special monogon with special puncture $M$. We have that $\ic{\vec u, -\vec u} / \lambda_M$ is an idempotent and therefore also $e_{s(\vec u)} - \ic{\vec u, -\vec u} / \lambda_M$. It leads to several observations. First of all, if $\sigma$ contains the arc $v$ joining $s(\vec u)$ to $M$, $\Gamma_\sigma$ and $\Delta_\sigma$ are not basic. In fact $\Gamma_\sigma$ and $\Gamma_{\sigma \setminus \{v\}}$ are Morita equivalent in this case. This is also the case for $\Delta_\sigma$ and $\Delta_{\sigma \setminus \{v\}}$. 

Another aspect of this remark is that, in view of Corollary \ref{subtri2}, $e \Delta_\sigma e$ can be obtained as an algebra of partial triangulation for any idempotent $e$ except if $e \cdot e_{s(\vec u)} = e_{s(\vec u)} - \ic{\vec u, -\vec u} / \lambda_M$ for a special monogon. The next proposition gives a change of basis which permit to get rid of this issue:

\begin{proposition} \label{revidem}
 We suppose that $\sigma$ does not contain the arc connecting $M$ to $s(\vec u)$. For $N \in \M$, denote 
 $$\mu_N = \left\{\begin{array}{ll} -\lambda_N & \text{if $N = M$;} \\ \nu_\M^{-1} \lambda_N & \text{if $N = s(\vec u)$;} \\ \lambda_M & \text{else.} \end{array} \right.$$
 Then there is an isomorphism $\psi: \Delta^{\mu}_{\sigma} \to \Delta^{\lambda}_\sigma$  satisfying:
 \begin{align*}
   \psi(\ic{\vec u, -\vec u}) &= \ic{\vec u, -\vec u} - \lambda_M e_u; \\
   \psi(e_v) &= e_v & \text{for any $v \in \sigma$}; \\
   \psi(\ic{\vec x, -\vec x}) & = \ic{\vec x, -\vec x} & \text{for any $\vec x \neq \vec u$ special.}
 \end{align*}
\end{proposition}

Before proving Proposition \ref{revidem}, we state the following corollary:

\begin{corollary} \label{redidemp}
 For any idempotent $e$ of $\Delta_\sigma$, $e \Delta_\sigma e$ is Morita equivalent to an algebra of partial triangulation of $(\Sigma, \M)$.
\end{corollary}

\begin{proof}
 Using the preliminary remark, we can suppose that $\sigma$ does not contain any arc incident to a special puncture up to Morita equivalence. Then, Proposition \ref{revidem} permits to exchange the role of idempotents $\ic{\vec u, -\vec u} / \lambda_M$ and $e_{s(\vec u)} - \ic{\vec u, -\vec u} / \lambda_M$ for each special monogon to make sure that, in any case, $e \cdot e_{s(\vec u)} \neq e_{s(\vec u)} - \ic{\vec u, -\vec u} / \lambda_M$. Thus, Corollary \ref{subtri2} permits to conclude.
\end{proof}

\begin{remark}
 Proposition \ref{revidem} corresponds, using the vocabulary of tagged triangulation \cite{CeLa12}, to changing the tags at $M$.
\end{remark}

The proof of Proposition \ref{revidem} relies on this key lemma:

\begin{lemma} \label{initcase}
 Suppose that $\sigma$ contains a triangle without puncture or hole, with sides $-\vec u$, $\vec v$ and $-\vec w$ such that $s(\vec v) = t(\vec v)$:
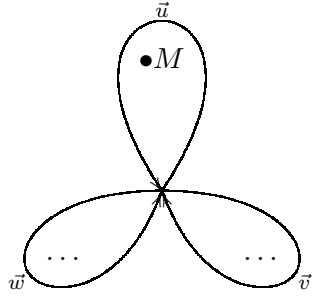
 $$\xymatrix@L=.05cm@R=.75cm@C=.75cm@M=0.01cm{
    & \bullet M & \\ \\
    & \ar@`{p+(10,15),p+(0,30),p+(-10,15)}[]_{\vec u} \ar@`{p+(15,0),p+(25,-15),p+(5,-14)}[]^{\vec v} \ar@`{p+(-15,0),p+(-25,-15),p+(-5,-14)}[]_{\vec w} & \\
    \dots & & \dots
 }$$

 Then Proposition \ref{revidem} holds for $\sigma$.
\end{lemma}

Lemma \ref{initcase} is proven in Subsection \ref{proofinitcase}. We deduce Proposition \ref{revidem}:

\begin{proof}[Proof of Proposition \ref{revidem}]
 Thanks to Definition \ref{defsp} and easy observations, we can suppose without loss of generality that $\Sigma$ has no boundary. The strategy is to use Corollary \ref{isoa}. Let $\sigma_2 = \sigma$ and take $\sigma^\circ = \{\vec t\}$ where $\vec t$ is the special arc pointing at $M$. Then we take a partial triangulation $\sigma_1$ which contains all special monogons of $\sigma$, which does not contain $\vec t$, which satisfies the hypothesis of Lemma \ref{initcase} and which is maximal for these properties (it is an easy observation that taking all special monogons of $\sigma$ does not prevent to complete the partial triangulation as in Lemma \ref{initcase}). According to Lemma \ref{initcase}, there is an isomorphism $\psi_1: \Delta_{\sigma_1}^\mu \to \Delta_{\sigma_1}^\lambda$ satisfying
 \begin{align*}
   \psi_1(\ic{\vec u, -\vec u}) &= \ic{\vec u, -\vec u} - \lambda_M e_u; \\
   \psi_1(e_v) &= e_v & \text{for any $v \in \sigma$}; \\
   \psi_1(\ic{\vec x, -\vec x}) & = \ic{\vec x, -\vec x} & \text{for any $\vec x \neq \vec u$ special.}
 \end{align*}
 Then we can apply Corollary \ref{isoa} to get an isomorphism $\psi_2: \Delta_{\sigma_2}^\mu \to \Delta_{\sigma_2}^\lambda$ satisfying the same conclusions.
\end{proof}

\section{Proofs}

\subsection{Proof of Proposition \ref{Csig}} \label{proofCsig}
 (1) The element $C_\sigma \in \Gamma^\circ_\sigma$ clearly does not depend on the choice of the orientations by definition of $I_\sigma^\circ$.

 (2) Let two oriented edges $\vec u$ and $\vec v$ of $\sigma$ such that $s(\vec u) = s(\vec v)$. We get
 $$\ic{\vec u, \vec v} C_\sigma = \ic{\vec u, \vec v} \cdot \io{\vec v, \vec u} \cdot \ic{\vec u, \vec v} = C_\sigma \ic{\vec u, \vec v} $$
 and, as $C_\sigma$ clearly commutes with idempotents of $Q_\sigma$, $C_\sigma$ is in the centre of $\Gamma^\circ_\sigma$. Moreover, we can grade $k Q_\sigma$ by letting the degree of an arrow $\ic{\vec u, \vec v}$ be $1/ d_{s(\vec u)} m_{s(\vec u)}$ where $d_{s(\vec u)}$ is the number of oriented edges of $\sigma$ starting at $s(\vec u)$. Thus, it is clear that all the $C_{\vec u}$ have degree $1$ and $\Gamma^\circ_\sigma$ is graded. As the degree of $C_\sigma^\ell$ is $\ell$, the $C_\sigma^\ell$'s are linearly independent in $\Gamma^\circ_\sigma$ so $k[C_\sigma]$ is included in the centre of $\Gamma^\circ_\sigma$. 

 (3) Let us prove that for a path $\alpha$ and two scalar multiple of paths $\omega$ and $\omega'$, if $\alpha \omega = \alpha \omega'$ in $\Gamma^\circ_\sigma$ then $\omega = \omega'$ in $\Gamma^\circ_\sigma$. We denote by $\ell$ the minimal number of relations $C_{\vec w}$ to apply to relate $\alpha \omega$ and $\alpha \omega'$ and by $n$ the length of $\alpha$. We will do an induction on $(\ell, n)$ and we will also prove that we need at most $\ell$ relations to go from $\omega$ to $\omega'$. If $\ell = 1$, the result is immediate. Suppose that $\ell > 1$.

 Suppose that $n = 1$, \emph{i.e.} $\alpha = \ic{\vec u, \vec v}$ is an arrow of $Q_\sigma$. Fix a sequence of arrows 
 $\alpha_0 = \alpha$, $\alpha_1$, \dots, $\alpha_{\ell-1}$, $\alpha_\ell = \alpha$ and a sequence of scalar multiples of paths $\omega_0 = \omega$, \dots, $\omega_{\ell-1}$, $\omega_\ell = \omega'$ such that $\alpha_i \omega_i$ is related with $\alpha_{i-1} \omega_{i-1}$ in one step (\emph{i.e.} by one $C_{\vec w}$). If $\alpha_i = \alpha$ for some $i \neq 0, \ell$, applying the induction hypothesis is immediate. So we can suppose that $\alpha_1 = \alpha_2 = \dots = \alpha_{\ell-1} = \ic{-\vec u, \vec v'}$ where $\ic{-\vec u, \vec v'}$ is the only other arrow starting at $u$. The first and last relations applied have to be $C_{\vec u}$, so one must have $$\omega_1 = \io{\vec v', -\vec u} \omega_1' \quad \text{and} \quad \omega_{\ell - 1} = \io{\vec v', -\vec u} \omega_{\ell-1}'.$$ These two scalar multiples of paths are equal in $\Gamma^\circ_\sigma$ in $\ell-2$ steps so, by induction hypothesis, $\omega_1' = \omega_{\ell-1}'$ in $\Gamma^\circ_\sigma$. We have $\omega = \io{\vec v, \vec u} \omega_1'$ and $\omega' =  \io{\vec v, \vec u} \omega_{\ell-1}'$ so the result is true in this case. 

 Suppose now that $\alpha = \alpha_0 \alpha'$ where $\alpha_0$ has length $1$ and $\alpha'$ has length $n-1$. By induction hypothesis, as $\alpha_0 \alpha' \omega$ and $\alpha_0 \alpha' \omega'$ are equal in $\Gso$, then $\alpha' \omega$ and $\alpha' \omega'$ are equal. Applying once again the induction hypothesis, we get $\omega' = \omega$ in $\Gso$.

 According to (2) and by definition of $C$-irreducible paths, it is immediate that elements $C_\sigma^\ell \omega$ for $\ell \geq 0$ and $\omega$ a $C$-irreducible path generates $\Gamma_\sigma^\circ$ over $k$. Thus, it is enough to prove that they are linearly independent over $k$. 

 Let $E$ be the set of paths of $Q_\sigma$. Let $\sim$ be the smallest equivalence relation on $E$ such that $\omega_1 \ic{\vec u, \vec u}^{m_{s(\vec u)}} \omega_2 \sim \omega_1 \ic{-\vec u, -\vec u}^{m_{t(\vec u)}} \omega_2$ for any $\omega_1, \omega_2 \in E$ and any $\vec u \in \sigma$ such that $\omega_1 \ic{\vec u, \vec u}^{m_{s(\vec u)}} \omega_2 \neq 0$. Suppose that $C_\sigma^\ell \omega = \lambda C_\sigma^{\ell'} \omega'$ for $\ell \leq \ell'$ and $\omega$, $\omega'$ two $C$-irreducible paths and $\lambda \in k$. Then, according to the previous discussion, we have $\omega = \lambda C_\sigma^{\ell' - \ell} \omega'$. As $\omega$ does not appear in any relation of $\Gamma^\circ_\sigma$, we get that $\ell' = \ell$, $\lambda = 1$ and $\omega = \omega'$. Thus, there is at most one multiple of a $C_\sigma^\ell \omega$ in each equivalence class of $E$. As relations relate only multiple of paths in the same equivalence class, it implies that the $C_\sigma^\ell \omega$'s are linearly independent over $k$. \qed

\subsection{Proof of Proposition \ref{caracpoly}} \label{proofcaracpoly}

We need the following technical preliminaries to be able to construct easily polygons.

\begin{definition}
 Let us define
 \begin{align*} & \Delta^\circ := \{(x,y) \in \R^2 \,|\, x, y > 0 \text{ and } x+y \leq 1\} \\ \subset \,\,& \Delta := \{(x,y) \in \R^2 \,|\, x, y \geq 0 \text{ and } x+y \leq 1\}.\end{align*}
 We call \emph{angle} of $\sigma$ a continuous map $\hat \phi: \Delta \rightarrow \Sigma$ such that  
 \begin{itemize}
  \item $\hat \phi$ is injective and oriented on $\Delta^\circ$;
  \item $\hat \phi|_{[0,1] \times \{0\}}$ is an oriented edge called \emph{first side} of $\hat \phi$;
  \item $\hat \phi|_{\{0\} \times [0,1]}$ is an oriented edge called \emph{second side} of $\hat \phi$.
 \end{itemize} 
 The \emph{interior of $\hat \phi$} is the image of $\Delta^\circ$.

 We define now 
 $$\Delta^{(1)} := \{(x,y) \in \R^2 \,|\, y \geq 0 \text{ and } x \geq 2 y \text{ and } x+y \leq 1\} \subset \Delta$$
 $$\Delta^{(2)} := \{(x,y) \in \R^2 \,|\, x \geq 0 \text{ and } y \geq 2 x \text{ and } x+y \leq 1\} \subset \Delta.$$
 We say that two angles $\hat \alpha$ and $\hat \beta$ such that the second side of $\hat \alpha$ is opposite to the first side of $\hat \beta$ are \emph{compatible} if  
 \begin{itemize}
  \item $\hat \beta^{-1}(\hat \alpha(\Delta)) \setminus (\{0\} \times (0,1] ) = \Delta^{(1)}$;
  \item $\hat \alpha^{-1}(\hat \beta(\Delta)) \setminus ((0,1] \times \{0\}) = \Delta^{(2)}$;
  \item $\hat \beta|_{\Delta^{(1)}} = \hat \alpha|_{\Delta^{(2)}} \circ \psi$ where $\psi: \Delta^{(1)} \rightarrow \Delta^{(2)}$ is the oriented homeomorphism defined by $\phi(x, y) = (y, 1-x+y)$.
 \end{itemize}
\end{definition}

We will prove the following more precise version of Proposition \ref{caracpoly}:

\begin{lemma} 
 We consider a sequence of oriented edges $\vec u_1$, $\vec u_2$, \dots, $\vec u_n$, with indices considered modulo $n$, such that $t(\vec u_i) = s(\vec u_{i+1})$ for $i = 1 \dots n$. The following conditions are equivalent:
 \begin{enumerate}[\rm (i)]
  \item \label{lcaracpoly1} There is a $n$-gon having oriented sides $\vec u_1$, $\vec u_2$, \dots, $\vec u_n$ in this order.
  \item \label{lcaracpoly2} There exist a sequence $\hat \alpha_1$, $\hat \alpha_2$, \dots, $\hat\alpha_n$ of angles of $\sigma$ such that
   \begin{itemize}
    \item the first side of $\hat \alpha_i$ is $\vec u_{i+1}$ and its second side is $-\vec u_i$;
    \item if $j \neq i-1, i, i+1$ then the interiors of $\hat \alpha_i$ and $\hat \alpha_j$ do not intersect;
    \item $\hat \alpha_{i+1}$ and $\hat \alpha_i$ are compatible for any $i$.
   \end{itemize}
  \item \label{lcaracpoly3} The following conditions are satisfied:
    \begin{itemize}
     \item if $i \neq j$ then $\vec u_i \neq \vec u_j$;
     \item for each $i$, if $\vec u_i$ is a boundary component, then it is oriented clockwisely around the boundary;
     \item for any $i$ and $j$ such that $M := s(\vec u_i) = s(\vec u_j)$, we have that $-\vec u_{i-1}$, $\vec u_i$, $-\vec u_{j-1}$ and $\vec u_j$ are ordered clockwisely around $M$;
     \item for any oriented boundary component $\vec v$ and $i$ such that $M := s(\vec u_i) = s(\vec v)$, we have that $-\vec u_{i-1}$, $\vec u_i$ and $\vec v$ are ordered clockwisely around $M$.
    \end{itemize}
 \end{enumerate}
\end{lemma}

\begin{proof}
 ~\ref{lcaracpoly1} $\Rightarrow$ \ref{lcaracpoly2}. Suppose that $\vec u_1$, \dots, $\vec u_n$ forms a $n$-gon $P: \Pr^{?}(n) \rightarrow \Sigma$. It is immediate that we can take angles $\hat \alpha^\circ_i$ in $\Pr^?(n)$ satisfying the conditions expected for the $\hat \alpha_i$'s. Then the angles $\hat \alpha_i := P \circ \hat \alpha^\circ_i$ satisfy the expected conditions.

 \ref{lcaracpoly2} $\Rightarrow$ \ref{lcaracpoly1}. We denote by $\overline{\Pr^\circ(n)}$ the closure of $\Pr^\circ(n)$ in $\Pr(n)$. We fix $\hat \alpha_i$ as in \ref{lcaracpoly2}. We also fix angles $\hat \alpha_i^\circ$ of $\overline{\Pr^\circ(n)}$ satisfying the same hypotheses and such that the maps $\hat \alpha_i^\circ: \Delta \rightarrow \overline{\Pr^\circ(n)}$ are injective and their images cover entirely $\overline{\Pr^\circ(n)}$ (this is easy). Let us prove that there is a unique map $P: \overline{\Pr^\circ(n)} \rightarrow \Sigma$ such that $\hat \alpha_i = P \circ \hat \alpha_i^\circ$ for every $i$. First of all, as the images of the $\hat \alpha_i^\circ$'s cover $\overline{\Pr^\circ(n)}$, such a map has to be unique if it exists. Suppose that $\hat \alpha_i^\circ(x) = \hat \alpha_j^\circ(y)$ for some $i$, $j$ and $x, y \in \Delta$. If $i = j$ then $x = y$ by hypothesis and therefore $\hat \alpha_i (x) = \hat \alpha_j(y)$. If $i \neq j$, the hypotheses imply that $j = i \pm 1$, say $j = i+1$ without loss of generality, and $y = \psi(x)$ by definition of compatibility of angles (using also the fact that $\hat \alpha_i^\circ$ and $\hat \alpha_j^\circ$ are injective), and then $\hat \alpha_i(x) = \hat \alpha_j(y)$ also by definition of compatibility. We proved that $P$ is well defined. By 
 continuity of the $\hat \alpha_i$'s and the $\hat \alpha_i^\circ$'s, $P$ is also continuous. The map $P$ is also injective on the interior of $\overline{\Pr^\circ(n)}$ and maps injectively each side of $\overline{\Pr^\circ(n)}$ to an edge of $\Sigma$ by definition of angles and compatibility. Finally, up to filling the hole of $P$ if possible, $P$ is an $n$-gon.

 \ref{lcaracpoly2} $\Rightarrow$ \ref{lcaracpoly3}. We fix $\hat \alpha_i$ as in \ref{lcaracpoly2}. Notice that $\vec u_i = \hat \alpha_{i-1} |_{[0,1] \times \{0\}}$ so we immediately get that $\vec u_i \neq \vec u_j$ by hypotheses on the angles and definition of compatible angles. In the same way, if $\vec u_i$ is an oriented boundary edge, as it has a left neighbourhood which is in the image of $\hat \alpha_{i-1}$, it has to winds around the boundary in the clockwise direction. Finally, suppose that $M := s(\vec u_i) = s(\vec u_j)$. If $i = j$, clearly $- \vec u_{i-1}$, $\vec u_i$, $-\vec u_{j-1}$ and $\vec u_j$ are ordered clockwisely around $M$. If $i \neq j$, we use the fact that $-\vec u_{i-1}$ and $\vec u_i$ are the second and first sides of $\hat \alpha_{i-1}$ and the fact that the interior of $-\vec u_{j-1}$ does not intersect the image of $\hat \alpha_{i-1}$ to see that $-\vec u_{i-1}$, $\vec u_i$ and $-\vec u_{j-1}$ are ordered clockwisely. The rest of the orderings comes by analogous arguments. The last point is proved in the same way.

 \ref{lcaracpoly3} $\Rightarrow$ \ref{lcaracpoly2}. For each $i$, let us fix an injective oriented map $f_i: \Delta^{(1)} \rightarrow \Sigma$ such that
 \begin{itemize}
  \item $f_i |_{[0,1] \times \{0\}} = \vec u_i$;
  \item the sets $f_i(\Delta^{(1)} \setminus ([0,1] \times \{0\}))$ do not intersect any of the $\vec u_i$'s, the boundary of $\Sigma$ and do not intersect each other.
 \end{itemize}
 It is clearly possible by the first two hypotheses of \ref{lcaracpoly3}. For each $i$, we can choose an arc $\gamma_i$ of $\Sigma$ satisfying
 \begin{itemize}
  \item $\gamma_i$ links $f_i(2/3,1/3)$ to $f_{i+1}(2/3,1/3)$;
  \item $\gamma_i$ is homotopic to $f_i([(2/3,1/3),(1,0)]) \cup f_{i+1}([(0,0), (2/3,1/3)])$ relatively to its endpoints;
  \item the interiors of the $\gamma_j$'s do not intersect the $\vec u_j$'s and do not intersect each other.  
 \end{itemize}
 This is because, by construction, the arcs $$-\vec u_i, \quad f_i([(2/3,1/3),(1,0)]), \quad f_{i+1}([(0,0), (2/3,1/3)]), \quad \vec u_{i+1}$$ are ordered clockwisely around $t(\vec u_i) = f_i(1,0) = f_{i+1}(0,0) = s(\vec u_{i+1})$, and, according to the third and fourth conditions, there is neither another $\vec u_j$, nor a boundary component which can be inserted between $-\vec u_i$ and $\vec u_{i+1}$.

 Then, we can construct angles $\hat \alpha_i$ satisfying
 \begin{itemize}
  \item $\hat \alpha_i|_{\Delta^{(1)}} = f_{i+1}$;
  \item $\hat \alpha_i|_{\Delta^{(2)}} = f_i \circ \psi^{-1}$;
  \item the boundary of $\hat \alpha_i(\Delta \setminus (\Delta^{(1)} \cup \Delta^{(2)}))$ is $$f_i([(2/3,1/3),(1,0)]) \cup f_{i+1}([(0,0), (2/3,1/3)]) \cup \gamma_i.$$
 \end{itemize}
 where the last point is obtained using the fact that an homeomorphism of the circle can be extended to an homeomorphism of the disc. Then the $\hat \alpha_i$'s satisfy all the conditions.
\end{proof}

\subsection{Proof of Theorem \ref{thmmin}} \label{proofmainthmsa}

Before proving Theorem \ref{thmmin}, we need to develop some tools:

The following lemma permits to understand better relations between $J$ and external paths winding around polygons:

\begin{lemma} \label{xiJ}
 \begin{enumerate}[\rm (1)] \item For two oriented edges $\vec u$ and $\vec v$ of $\sigma$ starting at the same marked point, the following are equivalent:
  \begin{enumerate}[\rm (i)]
   \item $\ic{\vec u, \vec v} \notin J$;
   \item $\ic{\vec u, \vec v}$ is of the form $\ic{\vec u', -\vec u'}$, $\ic{\vec u', \vec v'}$ or $\ic{\vec v', -\vec u'}$ where $\vec u'$ is the oriented side enclosing a special monogon and $\vec v'$ is the oriented edge pointing to the special puncture if it is in $\sigma$.
  \end{enumerate}
 \end{enumerate}
 Let $P$ be a polygon. We have
 \begin{enumerate}[\rm (1)] \rs{1}
  \item $\xi_i^P \notin J$ if and only if $m_{P_i} = 1$ and either $\vct{P_i P_{i-1}} = \vct{P_i P_{i+1}}$ or $P$ is complementary to a special monogon.
  \item We have $\xi_i^P \xi_{i+1}^P \notin J$ if and only if $m_{P_i} = m_{P_{i+1}} = 1$ and either $P$ is a \emph{flat} digon (\emph{i.e.} a digon with oriented sides $\vec u$ and $-\vec u$) or $P$ is complementary to a special monogon.
  \item We have $C_\sigma \in J$.
 \end{enumerate}
\end{lemma}

\begin{proof}
 (1) (i) $\Rightarrow$ (ii). Suppose that $\ic{\vec u, \vec v} \notin J$. The decomposition as product of arrows is
   $$\ic{\vec u, \vec v} = \ic{\vec u_0, \vec u_1} \ic{\vec u_1, \vec u_2} \cdots \ic{\vec u_{\ell-1}, \vec u_\ell}$$
  where $\vec u_0 = \vec u$, $\vec u_1$, \dots, $\vec u_\ell = \vec v$ are successive oriented edges of $\sigma$ around the common starting point. None of the $\ic{\vec u_{i-1}, \vec u_i}$ is in $J$ so for each of them, there exists a special monogon with oriented side $\vec v_i$ and possibly $\vec v_i'$ pointing to the special puncture such that we are in one of the following three cases
  \begin{itemize}
   \item $\vec u_{i-1} = \vec v_i$ and $\vec u_i = \vec v_i'$;
   \item $\vec u_{i-1} = \vec v_i$ and $\vec u_i = - \vec v_i$;
   \item $\vec u_{i-1} = \vec v_i'$ and $\vec u_i = - \vec v_i$.
  \end{itemize}
  As we excluded the case where $\Sigma$ is a sphere without boundary, $\# \M = 3$ and $m_M = 1$ for some $M \in \M$, it is not possible that $\vec v_i$ and $-\vec v_i$ both enclose a special monogon. Thus, an easy analysis gives that $\ell = 1$ or $\ell = 2$ and the result follows. 

 (ii) $\Rightarrow$ (i). Take $\vec u'$ and $\vec v'$ as in (ii) ($\vec v'$ is not necessarily present). Let us prove that $\ic{\vec u', \vec v'}$, $\ic{\vec u', -\vec u'}$ and $\ic{\vec v', -\vec u'}$ are not in $J$. Let $E$ be the set of paths of $k Q_\sigma$ generated by $e_{u'}$, $e_{v'}$, $\ic{\vec u', \vec v'}$, $\ic{\vec u', -\vec u'}$ and $\ic{\vec v', -\vec u'}$. Consider the linear map
 $$ \psi: k Q_\sigma \rightarrow k, \quad \omega \notin E \mapsto 0, \quad \omega \in E \mapsto \lambda_{t(\vec v')}^{a(\omega)} $$ where $a(\omega)$ is the number of times $\omega$ goes through $u'$. By definition $J \subset \ker \psi$ and $\ic{\vec u', \vec v'}, \ic{\vec u', -\vec u'}, \ic{\vec v', -\vec u'} \notin \ker \psi$ so it is enough to prove that $I_\sigma + I_\sigma^\circ \subset \ker \psi$. We clearly have $C_{\vec u} \in \ker \psi$ for all $\vec u$. Consider a relation $R_{P, i} \in \Gamma^\circ_\sigma$ coming from a $n$-gon $P$.

 If $\omega_{i+1}^P \omega_i^P \in E$, by a similar argument as in the converse part of the proof, we get that $\vct{P_i P_{i+1}} = \vec u'$. The two only possibilities for $P$ is then the special monogon itself or an induced self-folded triangle. In the first case, the relation is
 \begin{equation} R_{P,i} = \ic{\vec u', -\vec u'}^2 - \lambda_{t(\vec v')} \ic{\vec u', -\vec u'} \label{rp1} \end{equation}
 and in the second case
 \begin{equation} R_{P,i} = \ic{\vec v', -\vec u'} \ic{\vec u', \vec v'} - \lambda_{t(\vec v')} e_{v'}\label{rp2} \end{equation}
 so in both case $R_{P,i} \in \ker \psi$.

 Suppose now that $\alpha := R_{P,i} - \omega_{i+1}^P \omega_i^P$ is a non-zero multiple of an element of $E$. We know that $\alpha$ is not a multiple of $C_\sigma$ so, using the definition of $R_{P,i}$, there are two cases:
 \begin{enumerate}[\rm (a)]
  \item If $\M_P = \{M\}$, we need to have $m_M = 1$ and $P$ is a monogon. So $P$ is a special monogon and it is immediate that $R_{P,i}$ is the one defined in \eqref{rp1}.
  \item If $\M_P = \emptyset$, we have $n \geq 3$ and $\xi_{i+2}^P, \xi_{i+3}^P, \dots, \xi_{i-1}^P \in E$. Let us discuss the possible values of $\xi_{i+2}^P$:
   \begin{itemize}
    \item If $\xi_{i+2}^P$ is a scalar multiple of $e_{u'}$, it means that $-\vec u'$, $\vec u'$ are consecutive sides of $P$. It is impossible thanks to Proposition \ref{caracpoly}.
    \item If $\xi_{i+2}^P$ is a scalar multiple of $\ic{\vec u', \vec v'}$, it means that $-\vec u'$, $\vec v'$ are consecutive sides of $P$. Thus, $-\vec v'$ should be the next side. It is impossible thanks to Proposition \ref{caracpoly}.
    \item In the same way it is impossible that $\xi_{i+2}^P$ is a scalar multiple of $\ic{\vec v', -\vec u'}$.
    \item If $\xi_{i+2}^P$ is a scalar multiple of $\ic{\vec u', -\vec u'}$, it means that $-\vec u'$ is consecutive to $-\vec u'$ which is only possible if $P$ is a monogon. It contradicts $n \geq 3$.
    \item If $\xi_{i+2}^P$ is a scalar multiple of $e_{v'}$, it means that $\vct{P_{i-1} P_i} = \vec v'$ and $\vct{P_{i} P_{i+1}} = -\vec v'$. If $n > 3$, $\xi_{i+3}^P$ should satisfy the same condition. It is impossible. So $n = 3$. As a consequence, $\vct{P_{i+1} P_{i+2}}$ is a loop starting at $s(\vec v')$. As $\M_P = \emptyset$, we get $\vct{P_{i+1} P_{i+2}} = \vec u'$. Finally, $R_{P, i}$ is the relation defined in \eqref{rp2}.    
   \end{itemize}
 \end{enumerate}
 
 (2) First of all, if $\vct{P_i P_{i-1}} = \vct{P_i P_{i+1}}$ and $m_{P_i} = 1$, then $$\xi_i^P = \io{\vct{P_i P_{i-1}}, \vct{P_i P_{i+1}}} = \lambda_{P_1} e_{P_i P_{i+1}}$$ is multiple of an idempotent so is not in $J$. If $P$ is a monogon with special complementary and $m_{P_i} = 1$ then $\xi_i^P \notin J$ using (1). Suppose now that $\xi_i^P \notin J$. If $m_{P_i} > 1$, we get that $\ic{\vct {P_i P_{i-1}},\vct {P_i P_{i-1}}}$ divides $\xi_i^P$ which is a contradiction thanks to (1). Suppose that $m_{P_i} = 1$ and $\vct{P_i P_{i-1}} \neq \vct{P_i P_{i+1}}$. Then $\lambda_{P_i}^{-1} \xi_i^P = \ic{\vct{P_i P_{i-1}}, \vct{P_i P_{i+1}}}$ is of the form $\ic{\vec u, \vec v}$ or $\ic{\vec u, -\vec u}$ or $\ic{\vec v, -\vec u}$ where $\vec u$ is enclosing a special monogon and $\vec v$ is possibly pointing to the special puncture thanks to (1). By a similar reasoning as in (1) (ii) $\Rightarrow$ (i), the only possibility is that $\vct{P_i P_{i-1}} = \vec u$ and $\vct{P_i P_{i+1}} = -\vec u$ and we are in the case of a monogon with special complementary.

 (3) If $\xi_i^P \xi_{i+1}^P \notin J$ then $\xi_i^P \notin J$ and $\xi_{i+1}^P \notin J$ so $m_{P_i} = m_{P_{i+1}} = 1$ and $P$ is a flat digon or a monogon with special complementary thanks to (2). The converse is immediate using the map $\psi$ defined in (1) (ii) $\Rightarrow$ (i).

 (4) Each term of $C_\sigma$ is multiple of a $\ic{\vec u, \vec u}$ so this is immediate as, thanks to (1), $\ic{\vec u, \vec u} \in J$.
\end{proof}

The following lemma tells us that there are always minimal polygons:
\begin{lemma} \label{findminpol}
 For any oriented edge $\vec u$ of $\sigma$ such that $\vec u$ is not a counter-clockwise oriented boundary edge, there exist a unique minimal polygon $P$ having $\vec u$ as an oriented edge, up to equivalence. Moreover, each internal path winding around a vertex of this polygon is an arrow of $Q_\sigma$.
\end{lemma}

\begin{proof}
 We prove by induction on $n$ that there is a unique sequence $\vec u_0$, $\vec u_1$, \dots, $\vec u_n$ such that
 \begin{itemize}
  \item $\vec u_0 = \vec u$;
  \item for each $i \leq n$, if $\vec u_i$ is a boundary edge then $\vec u_i$ is oriented clockwisely around the boundary;
  \item for each $i < n$, $s(\vec u_{i+1}) = t(\vec u_i)$ and $\ic{\vec u_{i+1}, -\vec u_i}$ is an arrow of $Q_\sigma$.
 \end{itemize}
 For $n = 0$, it is obvious. Let us suppose the result proven for $n$ and let us prove it for $n+1$. By construction of $Q_\sigma$, there is a unique $\vec u_{n+1}$ satisfying that $\ic{\vec u_{n+1}, -\vec u_n}$ is an arrow of $Q_\sigma$. Then, it is easy to see that $\vec u_{n+1}$ can not be a counter-clockwise boundary component, as otherwise $\vec u_n$ would be also one. It finishes the induction.

 By definition, it is immediate that a polygon is minimal if and only if every internal path winding around a vertex in an arrow of $Q_\sigma$. Thus, if a minimal polygon containing $\vec u$ exists, using Proposition \ref{caracpoly}, it has to have sides $\vec u_0$, \dots $\vec u_{n-1}$ where $n$ is the smallest integer satisfying $\vec u_{n} = \vec u_0$. Conversely, let us fix the smallest possible $n$ such that $\vec u_{n} = \vec u_0$. We have that the conditions of Proposition \ref{caracpoly} \ref{caracpoly3} are satisfied for $(\vec u_i)_{0 \leq i \leq n-1}$. Notice in particular that for any $i$ and for any oriented edge $\vec v$ of $\sigma$ such that $s(\vec v) = s(\vec u_i)$, $-\vec u_{i-1}$, $\vec u_i$ and $\vec v$ are ordered clockwisely around $s(\vec u_i)$ as by definitions of arrows of $Q_\sigma$, there is no oriented edge between $-\vec u_{i-1}$ and $\vec u_i$ in the clockwise order.
\end{proof}

\begin{definition}
 Suppose that $(\Sigma, \M)$ is a sphere without boundary and with four punctures and $m_{M} = 1$ for all $M \in \M$. This is the case where $\nu_\M \neq 1$. Let $P$ be a $n$-gon. For $1 \leq \ell \leq n$, if $\vct{P_{\ell+1} P_\ell}$ encloses a special monogon, we call $S_i^P$ the relation coming from this monogon. Otherwise we denote $S_i^P = 0$. We call \emph{special ideal} of $P$ the ideal $S^P := (S_1^P, S_2^P, \dots, S_n^P)$.

 In any other case, we put $S^P = 0$.
\end{definition}

\begin{lemma} \label{RRp}
 Suppose that an $n$-gon $P$ satisfies $\vct{P_{i} P_{i+1}} = - \vct{P_{j} P_{j+1}}$ for two sides $\vct{P_{i} P_{i+1}}$ and $\vct{P_{j} P_{j+1}}$ and one of the following conditions:
 \begin{enumerate}[\rm (a)]
  \item $\vct{P_{i} P_{i+1}}$ and $\vct{P_{j} P_{j+1}}$ are not consecutive;
  \item $\# \M_P > 0$.
 \end{enumerate}
 Then the following statements hold in $\widehat{\Gamma^\circ_\sigma}^J$:
 \begin{enumerate}[\rm (1)]
  \item $(R_{P,i}) + S^P = (\omega_{i+1}^P \omega_i^P) + S^P$;
  \item $(R_{P,j}) + S^P = (\omega_{j+1}^P \omega_j^P) + S^P$;
  \item for $1 \leq \ell \leq n$, 
   \begin{itemize}
    \item $\omega^P_\ell (R_{P,\ell} - \omega^P_{\ell+1} \omega^P_\ell), (R_{P,\ell} - \omega^P_{\ell+1} \omega^P_\ell)\omega^P_{\ell+1} \in ( \omega^P_{i+1} \omega^P_i, \omega^P_{j+1} \omega^P_j)$, if $i + 1 = \ell = j - 1$ or $j + 1 = \ell = i - 1$;
    \item $R_{P,\ell} - \omega^P_{\ell+1} \omega^P_\ell \in ( \omega^P_{i+1} \omega^P_i, \omega^P_{j+1} \omega^P_j)$, else.
   \end{itemize}
 \end{enumerate}
\end{lemma}

\begin{proof}
 We can suppose that $\# \M_P \leq 1$ as the result is trivial for $\# \M_P \geq 2$. We also suppose that $i = 1$.

 (1) (a) We suppose first that $n = 2$: this is the case of a flat digon and we have $\M_P = \{M\}$ for some $M \in \M$. Hence $(\Sigma, \M)$ is a sphere with three punctures. So $m_M, m_{P_1}, m_{P_2} > 1$. Thus:
 \begin{align*}R_{P, 1} &= \omega^P_2 \omega^P_1 - \lambda_M \omega^P_2 (\xi^P_2 \xi^P_1)^{m_M-1} \xi^P_2 \\ &= \omega^P_2 \omega^P_1 - \lambda_\M (\omega_2^P)^{m_{P_2} - 1} (\omega_2^P \omega_1^P) (\omega_1^P)^{m_{P_1} - 2}  (\xi^P_2 \xi^P_1)^{m_M-2} \xi_2^P\end{align*}
 which implies the result as $\xi_2^P \in J$.

 (b) Suppose now that $n \geq 3$. We suppose that $j < n$. Indeed, we cannot have $n = j = 2$. Thus $j = n \neq 2$ is analogous to $j = 2 \neq n$. We get easily
  $$R_{P,1} =  \omega_{2}^P \omega_1^P - \alpha \io{\vct{P_jP_{j-1}}, \vct{P_{2} P_{3}}}  \omega_{2}^P \omega_1^P  \io{\vct{P_1 P_{n}}, \vct{P_{j+1} P_{j+2}}} \beta$$
    where  
   \begin{align*}  \alpha &:= \left\{\begin{array}{ll}\xi_{3}^P \xi_{4}^P \cdots \xi_{j-1}^P &\text{if } \M_P = \emptyset\text{ (and $j>2$)}, \\
                                       \lambda_M (\xi_{3}^P \cdots \xi_{n}^P \xi_{1}^P \xi_{2}^P)^{m_M-1} &\text{if } \M_P = \{M\} \text{ and } j = 2, \\
                                       \lambda_M \omega_{2}^P (\xi_{2}^P  \cdots \xi_{n}^P \xi_{1}^P)^{m_M-1} \xi_{2}^P \xi_{3}^P \cdots \xi_{j-1}^P  &\text{if } \M_P = \{M\} \text{ and } j > 2,
                                     \end{array}\right.
\\  \beta &:= \xi_{j+2}^P \cdots \xi_{n}^P.\end{align*}
  If at least one of $\alpha$, $\beta$, $\io{\vct{P_jP_{j-1}}, \vct{P_{2} P_{3}}}$ and $\io{\vct{P_1 P_{n}}, \vct{P_{j+1} P_{j+2}}}$ is in $J$, it is immediate that $(R_{P,1}) = (\omega^P_{2} \omega^P_1)$. Suppose that they are not in $J$.

  As $\io{\vct{P_jP_{j-1}}, \vct{P_2 P_3}} \notin J$, thanks to Lemma \ref{xiJ}, we get that $m_{P_2} = 1$ and either $\vct{P_{j-1} P_j} = -\vct{P_2 P_3}$, or $-\vct{P_{j-1} P_{j}} = -\vct{P_2 P_3}$ encloses a special monogon. More precisely, we are in one of these cases:
  \begin{itemize}
   \item $j = 2$ (and $\M_P = \{M\}$ as $\vct{P_1 P_2}$ and $\vct{P_j P_{j+1}}$ are consecutive).
   \item $-\vct{P_{j-1} P_j} = -\vct{P_2 P_3}$ encloses a special monogon and $j = 3$ as $P$ cannot have twice the same side.
   \item $\vct{P_{j-1} P_j} = -\vct{P_2 P_3}$ and $j > 3$. In this case, as $\alpha \notin J$, we have $\xi_3^P \notin J$ so, as $P$ is not a monogon, $m_{P_3} = 1$ and $\vct{P_3 P_4} = - \vct{P_2 P_3}$. As two different (oriented) sides of $P$ cannot be equal, we deduce that $j = 4$.
  \end{itemize}
  In the same way, as $\io{\vct{P_1 P_{n}}, \vct{P_{j+1} P_{j+2}}} \notin J$ and $j \neq n$, we have $m_{P_1} = 1$ and we are in one of these cases:
  \begin{itemize}
   \item $n = j+1$ and $-\vct{P_{n} P_{1}}$ encloses a special monogon.
   \item $\vct{P_n P_{1}} = \vct{P_{j+1} P_{j+2}}$, $n = j+2$ and $m_{P_n} = 1$.
  \end{itemize}
  As $P$ is not a monogon or a self-folded triangle without puncture, we get $\omega_2^P \in J$ so, as $\alpha \notin J$, we get either $\M_P = \emptyset$ or $j = 2$ and $\M_P = \{M\}$. Moreover, in this last case, we get also $m_M = 1$ as $\xi_1^P \in J$. We summarize all possible cases in Figure \ref{csfour} (all on a sphere with four punctures and without boundary).
  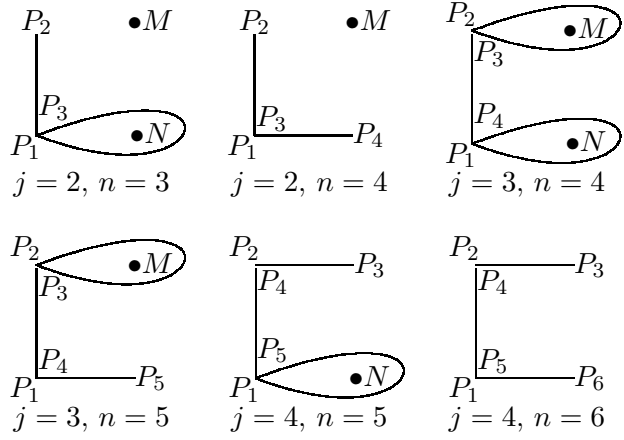
\begin{figure}
\begin{center}
 \begin{tabular}{ccccc}
 \boxinminipage{$\renewcommand{\labelstyle}{\textstyle} \xymatrix@L=.05cm@!=0cm@R=1.5cm@C=1.5cm@M=0.00cm{ 
  P_2 \ar@{-}[d] & \bullet M \\
  \ar@{{}{ }{}}@`{p+(-2,-2),p+(-2,-2)}[]|{P_1} \ar@{{}{ }{}}@`{p+(3,5),p+(3,5)}[]|{P_3} \ar@{-}@`{p+(13,5),p+(26,2),p+(13,-5)}[] & \bullet N
 }$}
& & 
 \boxinminipage{$\renewcommand{\labelstyle}{\textstyle} \xymatrix@L=.05cm@!=0cm@R=1.5cm@C=1.5cm@M=0.00cm{ 
  P_2 \ar@{-}[d] & \bullet M \\
  \ar@{{}{ }{}}@`{p+(-2,-2),p+(-2,-2)}[]|{P_1} \ar@{{}{ }{}}@`{p+(3,3),p+(3,3)}[]|{P_3} \ar@{-}[r] & P_4
 }$}
 & & 
 \boxinminipage{$\renewcommand{\labelstyle}{\textstyle} \xymatrix@L=.05cm@!=0cm@R=1.5cm@C=1.5cm@M=0.00cm{ 
  \ar@{-}[d] \ar@{{}{ }{}}@`{p+(-2,3),p+(-2,3)}[]|{P_2} \ar@{{}{ }{}}@`{p+(3,-4),p+(3,-4)}[]|{P_3} \ar@{-}@`{p+(13,5),p+(26,2),p+(13,-5)}[] & \bullet M \\
  \ar@{{}{ }{}}@`{p+(-2,-2),p+(-2,-2)}[]|{P_1} \ar@{{}{ }{}}@`{p+(3,5),p+(3,5)}[]|{P_4} \ar@{-}@`{p+(13,5),p+(26,2),p+(13,-5)}[] & \bullet N
 }$} \\
   $j = 2$, $n = 3$ & & $j = 2$, $n = 4$ & & $j = 3$, $n = 4$ \\ \\
 \boxinminipage{$\renewcommand{\labelstyle}{\textstyle} \xymatrix@L=.05cm@!=0cm@R=1.5cm@C=1.5cm@M=0.00cm{ 
  \ar@{-}[d] \ar@{{}{ }{}}@`{p+(-2,3),p+(-2,3)}[]|{P_2} \ar@{{}{ }{}}@`{p+(3,-4),p+(3,-4)}[]|{P_3} \ar@{-}@`{p+(13,5),p+(26,2),p+(13,-5)}[] & \bullet M \\
  \ar@{{}{ }{}}@`{p+(-2,-2),p+(-2,-2)}[]|{P_1} \ar@{{}{ }{}}@`{p+(3,3),p+(3,3)}[]|{P_4} \ar@{-}[r] & P_5
 }$} 
  & &
\boxinminipage{$\renewcommand{\labelstyle}{\textstyle} \xymatrix@L=.05cm@!=0cm@R=1.5cm@C=1.5cm@M=0.00cm{ 
  \ar@{-}[d] \ar@{{}{ }{}}@`{p+(-2,3),p+(-2,3)}[]|{P_2} \ar@{{}{ }{}}@`{p+(3,-3),p+(3,-3)}[]|{P_4} \ar@{-}[r] & P_3 \\
  \ar@{{}{ }{}}@`{p+(-2,-2),p+(-2,-2)}[]|{P_1} \ar@{{}{ }{}}@`{p+(3,5),p+(3,5)}[]|{P_5} \ar@{-}@`{p+(13,5),p+(26,2),p+(13,-5)}[] & \bullet N
 }$} & &
\boxinminipage{$\renewcommand{\labelstyle}{\textstyle} \xymatrix@L=.05cm@!=0cm@R=1.5cm@C=1.5cm@M=0.00cm{ 
  \ar@{-}[d] \ar@{{}{ }{}}@`{p+(-2,3),p+(-2,3)}[]|{P_2} \ar@{{}{ }{}}@`{p+(3,-3),p+(3,-3)}[]|{P_4} \ar@{-}[r] & P_3 \\
  \ar@{{}{ }{}}@`{p+(-2,-2),p+(-2,-2)}[]|{P_1} \ar@{{}{ }{}}@`{p+(3,3),p+(3,3)}[]|{P_5} \ar@{-}[r] & P_6
 }$} \\
   $j = 3$, $n = 5$ & & $j = 4$, $n = 5$ & & $j = 4$, $n = 6$ 
 \end{tabular}
\end{center}
  \caption{Special polygons in a sphere with four punctures} \label{csfour}   
  \end{figure}
  In all cases, for all $N \in \M$, we have $m_N = 1$. We can simplify $\alpha \io{\vct{P_jP_{j-1}}, \vct{P_{2} P_{3}}}$ and $\io{\vct{P_1 P_{n}}, \vct{P_{j+1} P_{j+2}}} \beta$ in the following way:
  \begin{itemize}
   \item If $j = 2$, $\alpha  \io{\vct{P_jP_{j-1}}, \vct{P_{2} P_{3}}} = \lambda_M \lambda_{P_2} e_\alpha$ where $e_\alpha := e_{P_2 P_3}$.
   \item If $j = 3$, $\alpha  \io{\vct{P_jP_{j-1}}, \vct{P_{2} P_{3}}} = \lambda_{P_2} \lambda_M e_\alpha$ where $e_\alpha := \ic{\vct{P_3 P_2}, \vct{P_{2} P_{3}}} / \lambda_{M}$ is an idempotent modulo $S^P$.
   \item If $j = 4$, $\alpha  \io{\vct{P_jP_{j-1}}, \vct{P_{2} P_{3}}} = \lambda_{P_3} \lambda_{P_4} e_\alpha$ where $e_\alpha := e_{P_2 P_3}$.
   \item If $n = j+1$, $\io{\vct{P_1 P_{n}}, \vct{P_{j+1} P_{j+2}}} \beta = \lambda_N \lambda_{P_1} e_\beta$ where $e_\beta := \ic{\vct{P_1 P_n}, \vct{P_{n} P_{1}}} / \lambda_{N}$ is idempotent modulo $S^P$.
   \item If $n = j+2$, $\io{\vct{P_1 P_{n}}, \vct{P_{j+1} P_{j+2}}} \beta = \lambda_{P_n} \lambda_{P_1} e_\beta$ where $e_\beta := e_{P_n P_1}$.
  \end{itemize}
  so, in any case, $R_{P,1} = \omega_{2}^P \omega_1^P - \lambda_\M e_\alpha \omega_{2}^P \omega_1^P  e_\beta$. Thus we get \begin{align*} ((1-e_\alpha) R_{P,1}) + S^P &= ((1-e_\alpha) \omega_{2}^P \omega_1^P) + S^P \\ (R_{P,1}(1-e_\beta)) + S^P &= (\omega_{2}^P \omega_1^P (1-e_\beta)) + S^P \end{align*} and $e_\alpha R_{P,1} e_\beta = \nu_\M e_\alpha \omega_2^P \omega_1^P e_\beta$ so $$(e_\alpha R_{P,1} e_\beta) + S^P= (e_\alpha \omega_2^P \omega_1^P e_\beta) + S^P$$ as $\nu_\M$ is invertible. We conclude that $(R_{P,1}) + S^P = (\omega_2^P \omega_1^P) + S^P$.

 (2) This is the analogous to (1).

 (3) If $n = 2$, this is an easy consequence of (1) and (2). So, up to swapping $i = 1$ and $j$, we can suppose that $j \neq n$. 

 (a) If $j \neq 2$, for any $\ell \neq 1, 2, n$, $R_{P,\ell} - \omega^P_{\ell+1} \omega^P_\ell$ is right divisible by
  \begin{align*}
   &  \xi^P_{\ell+2} \cdots \xi^P_n \xi^P_1 \xi^P_2 \xi^P_3 \cdots \xi^P_{\ell-1} \\
	    =\,\, & \xi^P_{\ell+2} \cdots \xi^P_n \cdot \io{\vct{P_1 P_n}, \vct{P_{j+1} P_{j+2}}} \cdot \omega_{j+1}^P \omega_j^P \cdot \io{\vct{P_jP_{j-1}},\vct{P_2 P_3}} \cdot \xi^P_3  \cdots \xi^P_{\ell-1}
  \end{align*}
  so $R_{P,\ell} - \omega_{\ell+1}^P \omega_\ell^P \in (\omega^P_{j+1} \omega^P_j)$. In the same way, $R_{P,\ell} - \omega_{\ell+1}^P \omega_\ell^P \in (\omega^P_2 \omega^P_1)$ if $\ell \neq j-1, j, j+1$. As $j \neq 1, 2, n$, the only remaining possibilities are $i + 1 = 2 = \ell = j - 1$ and $j+1 = \ell = n = i - 1$. 

  By symmetry of the situation, we suppose that $\ell = 2$ and $j = 3$ so that $(R_{P,\ell} - \omega^P_{\ell+1} \omega^P_\ell) \omega^P_{\ell+1}$ is right divisible by 
   $$ \xi^P_{4} \cdots \xi^P_n \xi^P_1 \omega^P_3  = \xi^P_{4} \cdots \xi^P_n \cdot \io{\vct{P_1 P_n}, \vct{P_{4} P_{5}}} \cdot \omega_{4}^P \omega_3^P $$
  and $\omega^P_{\ell} (R_{P,\ell} - \omega^P_{\ell+1} \omega^P_\ell) $ is left divisible by 
   $$ \omega^P_2 \xi^P_{4} \cdots \xi^P_n \xi^P_1 = \omega^P_2 \omega^P_1 \cdot \io{\vct{P_1 P_n}, \vct{P_{4} P_{5}}} \cdot \xi^P_5 \cdots \xi^P_n \xi^P_1$$
  so $\omega^P_\ell (R_{P,\ell} - \omega^P_{\ell+1} \omega^P_\ell), (R_{P,\ell} - \omega^P_{\ell+1} \omega^P_\ell)\omega^P_{\ell+1} \in ( \omega^P_{i+1} \omega^P_i, \omega^P_{j+1} \omega^P_j)$.

  (b) If $j = 2$, for $\ell \neq 1, n$, $\omega^P_{\ell+1} \omega^P_\ell - R_{P,\ell}$ is right divisible by
  \begin{align*}
   & C_\sigma \xi^P_{\ell+2} \cdots \xi^P_n \xi^P_1 \xi^P_2 \xi^P_3 \cdots \xi^P_{\ell-1} \\
	    =\,\, & \xi^P_{\ell+2} \cdots \xi^P_n \cdot \io{\vct{P_1 P_n}, \vct{P_{3} P_{4}}} \cdot \omega_{3}^P C_\sigma \xi^P_2 \xi^P_3  \cdots \xi^P_{\ell-1} \\
            =\,\, & \xi^P_{\ell+2} \cdots \xi^P_n \cdot \io{\vct{P_1 P_n}, \vct{P_{3} P_{4}}} \cdot \omega_{3}^P \omega_2^P \xi_2^P \cdot \xi^P_2 \xi^P_3  \cdots \xi^P_{\ell-1}
  \end{align*}
  so $\omega^P_{\ell+1} \omega^P_\ell - R_{P,\ell} \in (\omega^P_{j+1} \omega^P_j)$. For $\ell \neq 2, 3$, a similar computation gives the result. The last possibility is $\ell = n = 3$. In this case, $j+1 = \ell = i-1$ and $(\omega^P_{\ell+1} \omega^P_\ell - R_{P,\ell}) \omega^P_{\ell+1}$ is right divisible by
  $C_\sigma \omega^P_{\ell+1} = \xi^P_2 \omega^P_2 \omega^P_1$
  so $(\omega^P_{\ell+1} \omega^P_\ell - R_{P,\ell}) \omega^P_{\ell+1} \in (\omega^P_2 \omega^P_1)$. In the same way, $\omega^P_{\ell} (\omega^P_{\ell+1} \omega^P_\ell - R_{P,\ell}) \in (\omega^P_{\ell} \omega^P_{\ell-1}) = (\omega^P_{j+1} \omega^P_j)$.
\end{proof}

We will need the following easy observation:
\begin{lemma} \label{RRp2}
 If $P$ is a flat $4$-gon (\emph{i.e.} $\vct{P_1 P_2} = - \vct{P_4 P_1}$ and $\vct{P_2 P_3} = - \vct{P_3 P_4}$) then we have the equalities
 $$(R_{P,1}, R_{P,2}) = (\omega^P_2 \omega^P_1, \omega^P_3 \omega^P_2) \quad \text{and} \quad (R_{P,3}, R_{P,4}) = (\omega^P_4 \omega^P_3, \omega^P_1 \omega^P_4).$$
\end{lemma}

\begin{proof}
 If we are not in the case of a sphere without boundary and with $\# \M = 3$, then $\# \M_P \geq 1$ and this is an easy consequence of Lemma \ref{RRp}. So we suppose that $\Sigma$ is a sphere without boundary and $\# \M = 3$. We get easily
 \begin{align*} &R_{P,1} = \omega^P_2 \omega^P_1 - \lambda_{P_2} \lambda_{P_3} (\omega^P_3)^{m_{P_3} - 1} (\omega^P_2 \omega^P_4)^{m_{P_2} - 1} \omega^P_2 \\ \text{and} \quad & R_{P,2} = \omega^P_3 \omega^P_2 - \lambda_{P_1} \lambda_{P_2}  (\omega^P_2 \omega^P_4)^{m_{P_2} - 1} \omega^P_2   (\omega^P_1)^{m_{P_1} - 1} \end{align*}
 and, as we did the hypothesis that $m_{P_1}, m_{P_2}, m_{P_3} > 1$, we get easily that $\omega^P_3 \omega^P_2$ is a strict factor of $ (\omega^P_3)^{m_{P_3} - 1} (\omega^P_2 \omega^P_4)^{m_{P_2} - 1} \omega^P_2$ and $\omega^P_2 \omega^P_1$ is a strict factor of $(\omega^P_2 \omega^P_4)^{m_{P_2} - 1} \omega^P_2   (\omega^P_1)^{m_{P_1} - 1}$. This gives the first equality by completion with respect to $J$ (which contains all arrows in this case). The second equality is similar.
\end{proof}

Let $P$ be a polygon. We call \emph{reduction situation} the datum of an oriented edge $\vec u$ such that $s(\vec u) = P_1$ and two polygons $P'$ and $P''$ such that one of the following holds:
\begin{enumerate}[({R}a)]
 \item $P' = P''$ has $n+2$ sides and $\vct{P'_\ell P'_{\ell+1}} = \vct{P_\ell P_{\ell+1}}$ for $\ell = 1 \dots n-1$, $\vct{P'_n P'_{n+1}} = \vct{P_n P_1}$, $\vct{P'_{n+1} P'_{n+2}} = \vec u$ and $\vct{P'_{n+2} P'_1} = -\vec u$. Notice that either $\# \M_P = \# \M_{P'} = \infty$, or $\M_P = \M_{P'} \cup \{t(\vec u)\}$.
 \item $P' = P''$ has $n' > n+2$ sides and $\vct{P'_\ell P'_{\ell+1}} = \vct{P_\ell P_{\ell+1}}$ for $\ell = 1 \dots n-1$, $\vct{P'_n P'_{n+1}} = \vct{P_n P_1}$, $\vct{P'_{n+1} P'_{n+2}} = \vec u$ and $\vct{P'_{n'} P'_1} = -\vec u$.
 \item $P'$ has $n'$ sides and $P''$ has $n''$ sides with $n'+n'' = n+2$ and $\vct{P'_\ell P'_{\ell+1}} = \vct{P_\ell P_{\ell+1}}$ for $\ell = 1 \dots n'-1$, $\vct{P'_{n'} P'_1} = - \vec u$, $\vct{P''_\ell P''_{\ell+1}} = \vct{P_{n'-1+\ell} P_{n'+\ell}}$ for $\ell = 1 \dots n''-1$, $\vct{P''_{n''} P''_1} = \vec u$.
\end{enumerate}

Reductions situations are illustrated on Figure \ref{reds}. Notice that, for $i = 1, \dots, n'$ in case (Rc) and $i = 1, \dots, n+1$ in cases (Ra) and (Rb), we have $\omega_i^P = \alpha_i \omega_i^{P'} \beta_i$ and $\xi_i^{P'} = \beta_i \xi_i^P \alpha_i$ where 
  $$\alpha_i = \left\{\begin{array}{ll}
                       \omega_1^{P'} & \text{if } i = n+1, \\
		       1 & \text{else,}
                      \end{array}\right. \quad 
    \beta_i = \left\{\begin{array}{ll}
                       \omega_{n+1}^{P'} & \text{if } i = 1, \\
			1 & \text{else,}
                      \end{array}\right.$$
in cases (Ra) and (Rb) and
  $$\alpha_i = \left\{\begin{array}{ll}
                       \omega_1^{P''} & \text{if } i = n' , n'' \neq 1, \\
		       \omega_1^{P'} \omega_1^{P''} & \text{if } i = n' , n'' = 1, \\
			1 & \text{else,}
                      \end{array}\right. \quad 
    \beta_i = \left\{\begin{array}{ll}
                       \omega_{n''}^{P''} & \text{if } i = 1 , n'' \neq 1, \\
		       \omega_{n''}^{P''} \omega_{n'}^{P'} & \text{if } i = 1 , n'' = 1, \\
			1 & \text{else.}
                      \end{array}\right.$$
in case (Rc). 

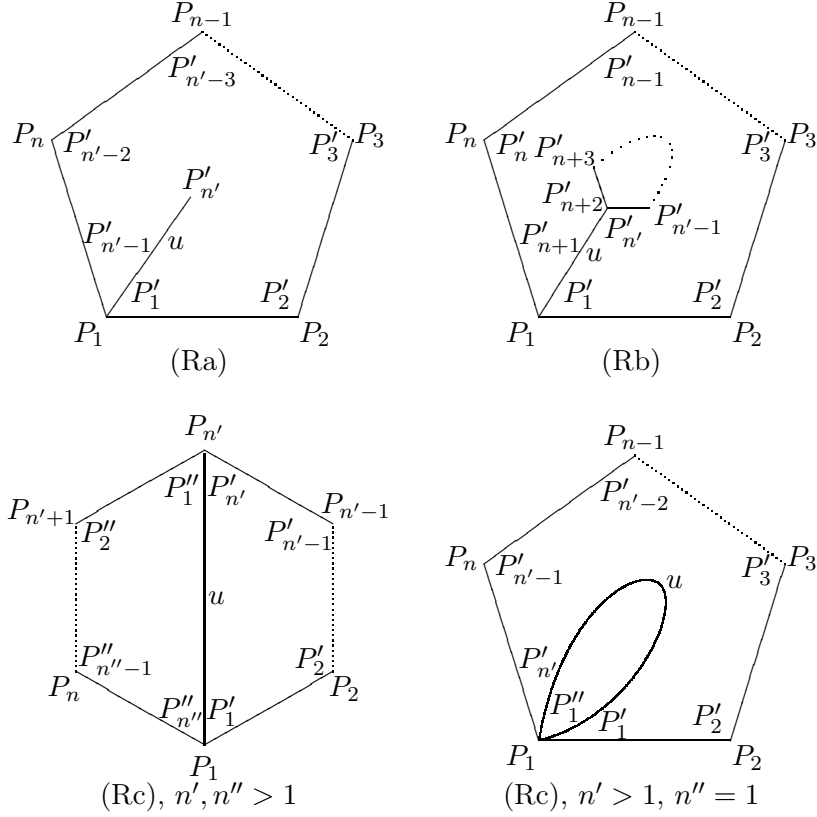
\begin{figure}
\begin{center}
 \begin{tabular}{ccc}
   \boxinminipage{$\renewcommand{\labelstyle}{\textstyle} \xymatrix@L=.05cm@!=0cm@R=.18cm@C=.18cm@M=0.00cm{ 
    & & & & & & & & & & & \ar@{..}[rrrrrrrrrrrdddddddd] \ar@{{}{ }{}}@`{p+(0,3),p+(0,3)}[]|{P_{n-1}} \ar@{{}{ }{}}@`{p+(0,-7),p+(0,-7)}[]|{P'_{n'-3}} \\ \\ \\ \\ \\ \\ \\ \\
   \ar@{-}[rrrrrrrrrrruuuuuuuu] \ar@{{}{ }{}}@`{p+(-4,1),p+(-4,1)}[]|{P_n} \ar@{{}{ }{}}@`{p+(8,-1),p+(8,-1)}[]|{P'_{n'-2}} & & & & & & & & & & & & & & & & & & & & & & \ar@{-}[llllddddddddddddd] \ar@{{}{ }{}}@`{p+(3,1),p+(3,1)}[]|{P_3} \ar@{{}{ }{}}@`{p+(-5,-1),p+(-5,-1)}[]|{P_3'} \\ \\ \\ 
    & & & & & & & & & & & P'_{n'} \ar@{-}[llllllldddddddddd]^(.4)u \\ \\ \\ \\ \\ \\ \\ \\ \\ \\
    & & & & \ar@{{}{ }{}}@`{p+(-3,-3),p+(-3,-3)}[]|{P_1} \ar@{{}{ }{}}@`{p+(7,4),p+(7,4)}[]|{P_1'} \ar@{{}{ }{}}@`{p+(2,14),p+(2,14)}[]|{P'_{n'-1}} \ar@{-}[lllluuuuuuuuuuuuu] & & & & & & & & & & & & & & \ar@{-}[llllllllllllll] \ar@{{}{ }{}}@`{p+(3,-3),p+(3,-3)}[]|{P_2} \ar@{{}{ }{}}@`{p+(-4,4),p+(-4,4)}[]|{P_2'}
 }$}
 & & 
\boxinminipage{$\renewcommand{\labelstyle}{\textstyle} \xymatrix@L=.05cm@!=0cm@R=.18cm@C=.18cm@M=0.00cm{ 
    & & & & & & & & & & & \ar@{..}[rrrrrrrrrrrdddddddd] \ar@{{}{ }{}}@`{p+(0,3),p+(0,3)}[]|{P_{n-1}} \ar@{{}{ }{}}@`{p+(0,-7),p+(0,-7)}[]|{P'_{n-1}} \\ \\ \\ \\ \\ \\ \\ \\
   \ar@{-}[rrrrrrrrrrruuuuuuuu] \ar@{{}{ }{}}@`{p+(-4,1),p+(-4,1)}[]|{P_n} \ar@{{}{ }{}}@`{p+(5,-1),p+(5,-1)}[]|{P'_{n}} & & & & & & & & & & & & & & & & & & & & & & \ar@{-}[llllddddddddddddd] \ar@{{}{ }{}}@`{p+(3,1),p+(3,1)}[]|{P_3} \ar@{{}{ }{}}@`{p+(-5,-1),p+(-5,-1)}[]|{P_3'} \\ \\ 
    & & & & & & & & \ar@{{}{ }{}}@`{p+(-5,3),p+(-5,3)}[]|{P'_{n+3}} \ar@{..}@`{c+(6,6),p+(8,12)}[dddrrrr] \\ \\ \\
    & & & & & & & & & \ar@{{}{ }{}}@`{p+(3,-4),p+(3,-4)}[]|{P'_{n'}} \ar@{{}{ }{}}@`{p+(-6,2),p+(-6,2)}[]|{P'_{n+2}} \ar@{-}[llllldddddddd]^(.37)u \ar@{-}[rrr] \ar@{-}[uuul] & & &  \ar@{{}{ }{}}@`{p+(7,-2),p+(7,-2)}[]|{P'_{n'-1}}  \\ \\ \\ \\ \\ \\ \\ \\
    & & & & \ar@{{}{ }{}}@`{p+(-3,-3),p+(-3,-3)}[]|{P_1} \ar@{{}{ }{}}@`{p+(7,4),p+(7,4)}[]|{P_1'} \ar@{{}{ }{}}@`{p+(2,14),p+(2,14)}[]|{P'_{n+1}} \ar@{-}[lllluuuuuuuuuuuuu] & & & & & & & & & & & & & & \ar@{-}[llllllllllllll] \ar@{{}{ }{}}@`{p+(3,-3),p+(3,-3)}[]|{P_2} \ar@{{}{ }{}}@`{p+(-4,4),p+(-4,4)}[]|{P_2'}
 }$}
\\
   (Ra) & & (Rb) \\ \\
\boxinminipage{$\renewcommand{\labelstyle}{\textstyle} \xymatrix@L=.05cm@!=0cm@R=.089cm@C=.089cm@M=0.00cm{ 
    & & & & & & & & & & & & & & & & & & & \ar@{-}[rrrrrrrrrrrrrrrrrrrddddddddddd] \ar@{-}[dddddddddddddddddddddddddddddddddddddddddddd]^u \ar@{{}{ }{}}@`{p+(0,4),p+(0,4)}[]|{P_{n'}} \ar@{{}{ }{}}@`{p+(4,-7),p+(4,-7)}[]|{P'_{n'}} \ar@{{}{ }{}}@`{p+(-4,-7),p+(-4,-7)}[]|{P''_1} \\ \\ \\ \\ \\ \\ \\ \\ \\ \\ \\
    \ar@{-}[rrrrrrrrrrrrrrrrrrruuuuuuuuuuu] \ar@{{}{ }{}}@`{p+(4,-2),p+(4,-2)}[]|{P''_2} \ar@{{}{ }{}}@`{p+(-6,2),p+(-6,2)}[]|{P_{n'+1}} & & & & & & & & & & & & & & & & & & & & & & & & & & & & & & & & & & & & & & \ar@{..}[dddddddddddddddddddddd] \ar@{{}{ }{}}@`{p+(-6,-2),p+(-6,-2)}[]|{P'_{n'-1}} \ar@{{}{ }{}}@`{p+(4,3),p+(4,3)}[]|{P_{n'-1}} \\ \\ \\ \\ \\ \\ \\ \\ \\ \\ \\ \\ \\ \\ \\ \\ \\ \\ \\ \\ \\ \\
    \ar@{..}[uuuuuuuuuuuuuuuuuuuuuu] \ar@{{}{ }{}}@`{p+(-2,-3),p+(-2,-3)}[]|{P_n} \ar@{{}{ }{}}@`{p+(7,2),p+(7,2)}[]|{P''_{n''-1}} & & & & & & & & & & & & & & & & & & & & & & & & & & & & & & & & & & & & & & \ar@{-}[dddddddddddlllllllllllllllllll] \ar@{{}{ }{}}@`{p+(2,-3),p+(2,-3)}[]|{P_2} \ar@{{}{ }{}}@`{p+(-4,2),p+(-4,2)}[]|{P'_2} \\ \\ \\ \\ \\ \\ \\ \\ \\ \\ \\
    & & & & & & & & & & & & & & & & & & & \ar@{-}[llllllllllllllllllluuuuuuuuuuu] \ar@{{}{ }{}}@`{p+(0,-4),p+(0,-4)}[]|{P_1} \ar@{{}{ }{}}@`{p+(3,6),p+(3,6)}[]|{P'_1} \ar@{{}{ }{}}@`{p+(-4,6),p+(-4,6)}[]|{P''_{n''}}
 }$}
& & 
\boxinminipage{$\renewcommand{\labelstyle}{\textstyle} \xymatrix@L=.05cm@!=0cm@R=.18cm@C=.18cm@M=0.00cm{ 
    & & & & & & & & & & & \ar@{..}[rrrrrrrrrrrdddddddd] \ar@{{}{ }{}}@`{p+(0,3),p+(0,3)}[]|{P_{n-1}} \ar@{{}{ }{}}@`{p+(0,-7),p+(0,-7)}[]|{P'_{n'-2}} \\ \\ \\ \\ \\ \\ \\ \\
   \ar@{-}[rrrrrrrrrrruuuuuuuu] \ar@{{}{ }{}}@`{p+(-4,1),p+(-4,1)}[]|{P_n} \ar@{{}{ }{}}@`{p+(8,-1),p+(8,-1)}[]|{P'_{n'-1}} & & & & & & & & & & & & & & & & & & & & & & \ar@{-}[llllddddddddddddd] \ar@{{}{ }{}}@`{p+(3,1),p+(3,1)}[]|{P_3} \ar@{{}{ }{}}@`{p+(-5,-1),p+(-5,-1)}[]|{P_3'} \\ \\ \\ \\ \\ \\ \\ \\ \\ \\ \\ \\ \\
    & & & & \ar@{{}{ }{}}@`{p+(-3,-3),p+(-3,-3)}[]|{P_1} \ar@{{}{ }{}}@`{p+(13,3),p+(13,3)}[]|{P_1'} \ar@{{}{ }{}}@`{p+(0,14),p+(0,14)}[]|{P'_{n'}} \ar@{-}[lllluuuuuuuuuuuuu] \ar@{-}@`{p+(3,18),p+(25,30),p+(12,3)}[]^u \ar@{{}{ }{}}@`{p+(5,6),p+(5,6)}[]|{P_1''} & & & & & & & & & & & & & & \ar@{-}[llllllllllllll] \ar@{{}{ }{}}@`{p+(3,-3),p+(3,-3)}[]|{P_2} \ar@{{}{ }{}}@`{p+(-4,4),p+(-4,4)}[]|{P_2'}
 }$}
\\
   (Rc), $n', n'' > 1$ & & (Rc), $n' > 1$, $n'' = 1$ 
 \end{tabular} 
\end{center}
  \caption{Reduction situations} \label{reds}   
  \end{figure}

\begin{lemma} \label{redideal}
 In a reduction situation, for any $i = 1, \dots, \min(n'-1,n)$, there exists $\kappa$ invertible in $\widehat{\Gamma^\circ_\sigma}^J / S^P$ such that
 $R_{P,i} \kappa - \alpha_{i+1} R_{P',i} \beta_i \in I_{P,u} + S^{P}$
 where
 \begin{itemize}
  \item $I_{P,u} = (R_{P', n+1}, R_{P', n'})$ in cases (Ra) or (Rb);
  \item $I_{P,u} = (R_{P',n'}, R_{P'', n''})$ in case (Rc).
 \end{itemize}
\end{lemma}

\begin{proof}
 Let us consider the three cases separately:
 \begin{enumerate}[({R}a)]
  \item If $\# \M_{P'} \geq 1$, we have $R_{P,i} = \alpha_{i+1} \omega^{P'}_{i+1} \omega^{P'}_i \beta_i$ and thanks to Lemma \ref{RRp}, we get:
  $$ \alpha_{i+1}\left(R_{P',i} - \omega^{P'}_{i+1} \omega^{P'}_i\right) \beta_i \in (\omega_{1}^{P'} \omega_{n'}^{P'}, \omega_{n'}^{P'} \omega_{n'-1}^{P'}) + S^{P'} = I_{P,u} + S^{P}$$
  so $R_{P,i} - \alpha_{i+1} R_{P',i} \beta_i \in I_{P,u} + S^{P}$. 

So we suppose that $\M_{P'} = \emptyset$. Then, modulo $R_{P',n+2} = R_{P',n'}$ we get
  \begin{align*}
   \alpha_{n+1} \xi_{n+2}^{P'} &= \lambda_{P'_{n+2}} \omega_1^{P'} (\omega_{n+2}^{P'})^{m_{P'_{n+2} } -1} = \lambda_{P'_{n+2}} \xi_2^{P'} \cdots \xi_n^{P'} \xi_{n+1}^{P'} (\omega_{n+2}^{P'})^{m_{P'_{n+2} } - 2} \\
   &= \lambda_{P'_{n+2}} \xi_2^{P} \cdots \xi_n^{P} \xi_{1}^{P} \omega_1^{P'} (\omega_{n+2}^{P'})^{m_{P'_{n+2} } - 2} \\ &= \cdots = \lambda_{P'_{n+2}} (\xi_2^{P} \cdots \xi_n^{P} \xi_{1}^{P})^{m_{P'_{n+2} } - 1} \omega_1^{P'}.
  \end{align*}

  Then, modulo $R_{P', n'}$, for $i = 1 \dots n$, we have
  \begin{align*} \alpha_{i+1} R_{P',i} \beta_i &= \alpha_{i+1} \omega_{i+1}^{P'} \omega_i^{P'} \beta_i - \alpha_{i+1} \xi_{i+2}^{P'} \cdots \xi_n^{P'} \xi_{n+1}^{P'} \xi_{n+2}^{P'} \xi_1^{P'} \xi_2^{P'} \cdots \xi_{i-1}^{P'} \beta_i \\ 
    &= \omega_{i+1}^P \omega_i^P - \xi_{i+2}^P \cdots \xi_n^P \xi_1^P \alpha_{n+1} \xi_{n+2}^{P'} \beta_1 \xi_1^P \xi_2^P \cdots \xi_{i-1}^P \\
    &= \omega_{i+1}^P \omega_i^P - \xi_{i+2}^P \cdots \xi_1^P \lambda_{P'_{n+2}} (\xi_2^{P} \cdots \xi_{1}^{P})^{m_{P'_{n+2} - 1}} \omega_1^{P'}\beta_1 \xi_1^P \cdots \xi_{i-1}^P \\
   &= \omega_{i+1}^P \omega_i^P - \lambda_{P'_{n+2}} (\xi_{i+2}^{P} \cdots \xi_{i+1}^{P})^{m_{P'_{n+2} - 1}} \xi_{i+2}^P \cdots \xi_1^P \omega_1^{P} \xi_1^P \cdots \xi_{i-1}^P \\
    &= \omega_{i+1}^P \omega_i^P - \lambda_{P'_{n+2}} C_\sigma (\xi_{i+2}^{P} \cdots \xi_{i+1}^{P})^{m_{P'_{n+2} - 1}} \xi_{i+2}^P \cdots \xi_{i-1}^P = R_{P,i}.
  \end{align*}
  \item In this case, we get that $\# \M_P \geq 2$ so we can use Lemma \ref{RRp} as in the first case of (Ra).
  \item As $1 \leq i \leq n'-1$, if $\# \M_{P'} > 1$, we get that $R_{P,i} = \alpha_{i+1} R_{P',i} \beta_i$ so we can suppose that $\# \M_{P'} \leq 1$. If $\M_{P'} = \{M\}$, denote
  $$\theta = \lambda_M C_\sigma (\xi_{1}^{P} \xi_{2}^{P} \cdots \xi_{n'}^{P} \alpha_{n'} \beta_1)^{m_M - 1} \quad \text{and} \quad \theta' = \lambda_M C_\sigma (\xi_{1}^{P'} \xi_{2}^{P'} \cdots \xi_{n'}^{P'})^{m_M - 1} $$
  and if $\M_{P'} = \emptyset$, denote $\theta = \theta'= 1$, so that, in both cases, $\beta_1 \theta = \theta' \beta_1$. By an easy analysis, we have
  \begin{align*} \alpha_{i+1} R_{P',i} \beta_i &= \alpha_{i+1} \omega_{i+1}^{P'} \omega_i^{P'} \beta_i - (\alpha_{i+1} \xi_{i+2}^{P'} \cdots \xi_{n'}^{P'}) \theta' (\xi_1^{P'} \cdots \xi_{i-1}^{P'} \beta_i)\\ 
    &= \omega_{i+1}^P \omega_i^P - (\xi_{i+2}^P \cdots \xi_{n'-1}^P \xi_{n'}^P \alpha_{n'}) \theta' (\beta_1 \xi_1^P \xi_2^P \cdots \xi_{i-1}^P) \\
    &= \omega_{i+1}^P \omega_i^P - \xi_{i+2}^P \cdots \xi_{n'-1}^P \xi_{n'}^P \alpha_{n'} \beta_1 \theta \xi_1^P \xi_2^P \cdots \xi_{i-1}^P.
  \end{align*}
  
 Write $\gamma = \omega_1^{P'}$ and $\delta = \omega_{n'}^{P'}$ if $n'' = 1$ and $\gamma = \delta = 1$ else. If $\# \M_{P''} > 1$ then $\gamma R_{P'',n''} \delta = \gamma \omega_1^{P''} \omega_{n''}^{P''} \delta = \alpha_{n'} \beta_1$ and therefore $\alpha_{i+1} R_{P',i} \beta_i = R_{P,i}$ modulo $R_{P'',n''}$ and we conclude in this case.
  
  If $\# \M_{P''} \leq 1$, let 
    $$\eta' = \lambda_N C_\sigma (\xi_{n''}^{P''} \xi_1^{P''} \cdots \xi_{n''-1}^{P''})^{m_N-1} \  \text{and} \  \eta = \lambda_N C_\sigma (\xi_1^{P} \omega_1^{P'} \omega_{n'}^{P'} \xi_{n'}^{P} \cdots \xi_{n}^{P})^{m_N-1}$$
    if $\M_{P''} = \{N\}$ and $\eta = \eta' = 1$ if $\M_{P''} = \emptyset$. In any case, $\eta' \delta = \delta \eta$ so that
  \begin{align*}\gamma R_{P'', n''} \delta &= \gamma \omega_1^{P''} \omega_{n''}^{P''} \delta - \gamma \xi_2^{P''} \cdots \xi_{n''-1}^{P''} \eta' \delta \\ &= \alpha_{n'} \beta_1 - \gamma \xi_2^{P''} \cdots \xi_{n''-1}^{P''} \delta \eta =  \alpha_{n'} \beta_1 - \xi_{n'+1}^{P} \cdots \xi_n^{P} \eta  
 \end{align*}
 with the convention that $\xi_{n+2}^P \cdots \xi_n^{P} C_\sigma = \omega_1^P$ (if $n' = n+1$).
  So, modulo $R_{P'',n''}$, we get,
  $$\alpha_{i+1} R_{P',i} \beta_i = \omega_{i+1}^P \omega_i^P - \xi_{i+2}^P \cdots \xi_{n}^P \eta \theta \xi_1^P \xi_2^P \cdots \xi_{i-1}^P$$
  If $\M_{P'} = \M_{P''} = \emptyset$, we have $\eta = \theta = 1$ so $\alpha_{i+1} R_{P', i} \beta_i = R_{P,i}$ modulo $R_{P'', n''}$. 
  If $\M_{P'} = \{M\}$ and $\M_{P''} = \emptyset$, modulo $R_{P'', n''}$, we get
  $\theta =  \lambda_M C_\sigma (\xi_{1}^{P} \xi_{2}^{P} \cdots \xi_{n}^{P})^{m_M - 1}$
  so we get
  \begin{align*}\alpha_{i+1} R_{P',i} \beta_i &= \omega_{i+1}^P \omega_i^P - \lambda_M C_\sigma \xi_{i+2}^P \cdots \xi_{n}^P (\xi_{1}^{P} \xi_{2}^{P} \cdots \xi_{n}^{P})^{m_M - 1} \xi_1^P \xi_2^P \cdots \xi_{i-1}^P \\ &= \omega_{i+1}^P \omega_i^P - \lambda_M C_\sigma (\xi_{i+2}^P \cdots  \xi_{i+1}^{P})^{m_M - 1} \xi_{i+2}^P \cdots \xi_{i-1}^P = R_{P,i}
  \end{align*}

  Symmetrically, we get that if $\M_{P'} = \emptyset$ and $\M_{P''} = \{N\}$, modulo $R_{P'',n''}$ and $R_{P', n'}$, $\alpha_{i+1} R_{P',i} \beta_i = R_{P,i}$.

  If $\M_{P'} = \{M\}$ and $\M_{P''} = \{N\}$, the equality modulo $R_{P'', n''}$ becomes $$\alpha_{i+1} R_{P',i} \beta_i = \omega_{i+1}^P \omega_i^P(1 - \lambda_N \lambda_M \nu) = R_{P,i}(1 - \lambda_N \lambda_M \nu)$$
 where
 $$\nu = \xi_i^P \cdots \xi_n^P (\xi_1^P \omega_1^{P'} \omega_{n'}^{P'} \xi_{n'}^P \cdots \xi_n^P)^{m_N-1} (\xi_1^P \cdots \xi_{n'}^P \alpha_{n'} \beta_1)^{m_M-1} \xi_1^P \xi_2^P \cdots \xi_{i-1}^P.$$
 If $\nu \in J$, we get the result immediately taking $\kappa := 1-\lambda_M \lambda_N \nu$. 

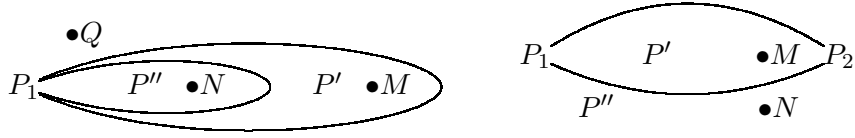
\begin{figure}
\begin{center}
\boxinminipage{$\renewcommand{\labelstyle}{\textstyle} \xymatrix@L=.05cm@!=0cm@R=.7cm@C=.8cm@M=0.00cm{ 
    & \bullet Q \\
    P_1 \ar@{-}@`{p+(15,6),p+(50,0),p+(15,-6)}[] \ar@{-}@`{p+(20,10),p+(90,0),p+(20,-10)}[] & & P'' & \bullet N & & P' &  \bullet M & & \\
 }$}
\boxinminipage{$\renewcommand{\labelstyle}{\textstyle} \xymatrix@L=.05cm@!=0cm@R=.7cm@C=.8cm@M=0.00cm{ 
    \\
    P_1 \ar@/_.5cm/@{-}[rrrrr] & & P' & & \bullet M & P_2 \ar@/_.7cm/@{-}[lllll] \\
    & P'' & & & \bullet N
 }$}
\end{center}
  \caption{Two special cases} \label{excdec}   
  \end{figure}
 Suppose that $\nu \notin J$. 
 A quick analysis, using Lemma \ref{xiJ} (2) proves that the only possibility is that $(\Sigma, \M)$ is a sphere with four punctures with $P'$ and $P''$ as in Figure \ref{excdec} and $m_{M'} = 1$ for $M' \in \M$. For the left diagram, we get $\nu = \xi_1^P = \lambda_{P_1} \ic{\vct{P_2 P_1}, \vct{P_1 P_2}}$
 and $e := \ic{\vct{P_2 P_1}, \vct{P_1 P_2}} / \lambda_{Q}$ is an idempotent modulo $S^{P}$. For the right diagram, we get $\nu = \lambda_{P_1} \lambda_{P_2} e$ where $e := e_{\vct{P_1 P_2}}$. Finally, we find that, in both cases, modulo $R_{P'', n''}$, 
 $$\alpha_{i+1} R_{P',i} \beta_i = R_{P,i}(1 - \lambda_\M e)$$ 
 which induces the result, as $\kappa := 1 - \lambda_\M e$ satisfies $\kappa^{-1} = 1 + \lambda_\M \nu_\M^{-1} e$. \qedhere
 \end{enumerate}
\end{proof}

The following lemma permits to do inductions:
\begin{lemma} \label{defpoly}
 \begin{enumerate}[\rm (1)]
  \item In a reduction situation, $P'$ and $P''$ have less reduction arcs than $P$.
  \item If $P$ is a polygon of $\sigma$ which is not minimal with reduction arc $u$ then up to replacing $P$ by an equivalent polygon $\tilde P$, there exist a reduction situation involving $P$ and an orientation $\vec u$ of $u$.
 \end{enumerate}
\end{lemma}

\begin{proof}
 First of all, (1) is immediate. Then (2) is an immediate consequence of Proposition \ref{caracpoly}.
\end{proof}

Now, we can conclude the proof of Theorem \ref{thmmin}:
\begin{proof}[Proof of Theorem \ref{thmmin}]
 It is easy to prove, using Lemmas \ref{defpoly} and \ref{redideal}, by descending induction on $N \geq 0$ that relations coming from polygons having at most $N$ reduction arcs and special monogons generate $I_\sigma$. So $I_\sigma$ is generated by relations coming from minimal polygons and special monogons. Let $P$ be a non-minimal special monogon. By hypothesis, a loop of $\sigma$ can not cut $(\Sigma, \M)$ into two special monogons, so $S^P = 0$ and thanks to Lemma \ref{redideal}, we conclude that the relation coming from $P$ is in the ideal generated by relations coming from minimal polygons.
\end{proof}

\subsection{Proof of Theorem \ref{subtri}} \label{proofmainthmsb}

 Denote by $\phi^{\circ\circ}: k Q_\tau \hookrightarrow e_\tau k Q_\sigma e_\tau$ the only injective morphism of algebras such that $e_u$ for $u \in \tau$ is mapped to $e_u$ and $\ic{\vec u, \vec v}$ is mapped to $\ic{\vec u, \vec v}$ for all arrows $\ic{\vec u, \vec v}$ of $Q_\tau$. 

 First, $\phi^{\circ\circ}$ induces an injective morphism $\phi^\circ: \Gamma^\circ_\tau \hookrightarrow e_\tau \Gamma^\circ_\sigma e_\tau$. Indeed, the relation $C_{\vec u}$ is mapped to $C_{\vec u}$ for any oriented edge $\vec u$. The injectivity is a clear consequence of Proposition \ref{Csig}. 

 Then, $\phi^\circ$ induces a morphism $\phi: \Gamma_\tau \rightarrow e_\tau \Gs e_\tau$. Indeed, every relation defining $\Gamma_\tau$ is also a relation defining $\Gs$ (since every polygon of $\tau$ is a polygon of $\sigma$). To prove that it is an isomorphism, it is enough to look at the case where the difference between $\tau$ and $\sigma$ is only one edge $u$ (then the result comes by an immediate induction). Notice that $C_\sigma = \phi(C_\tau) + \lambda_{s(\vec u)} \ic{\vec u, \vec u}^{m_{s(\vec u)}}$. 

 Let us prove that $\phi$ is an isomorphism. By Lemma \ref{findminpol}, there are two minimal polygons $P'$ and $P''$, unique up to equivalence, the first one having $\vec u$ as a side and the second one having $-\vec u$ as a side. We suppose that the number $n'$ of sides of $P'$ is greater or equal that the number $n''$ of sides of $P''$. If $P'$ and $P''$ both involve only the edge $u$, then we get easily that $u$ forms a connected component of $\sigma$ and the result is immediate. Otherwise, using Lemma \ref{caracpoly}, we construct a reduction situation involving a polygon $P$ having the sides of $P'$ and $P''$ except $u$.

 (a) We suppose first that we are in case (Ra) or in case (Rb) or in case (Rc) with $\# \M_{P'}, \# \M_{P''} \in \{0, \infty\}$. We consider the following, clearly well defined, morphism of algebras:
 \begin{align*}
  \psi^{\circ\circ} : e_\tau k Q_\sigma e_\tau & \rightarrow \Gamma_\tau \\
   e_v & \mapsto e_v & v \in \tau \\
   \ic{\vec v, \vec w}& \mapsto \ic{\vec v, \vec w} & \vec v, \vec w \in \tau, s(\vec v) = s(\vec w) \\
   \omega^{P'}_1 (\omega^{P'}_{n'})^\ell \omega^{P'}_{n'-1} & \mapsto \left(\delta_{\M_{P'}, \emptyset} \xi_2^P \xi_3^P \cdots \xi_1^P \right)^\ell \omega_1^P & \text{(Ra), $\ell \geq 0$}  \\
   \omega^{P'}_1 \omega^{P'}_{n'}, \omega^{P'}_{n+2} \omega^{P'}_{n+1} & \mapsto 0 & \text{(Rb)} \\
   \omega^{P'}_1 \omega^{P'}_{n'} & \mapsto \delta_{\M_{P'}, \emptyset}\xi^{P}_2 \xi^P_3 \cdots \xi^P_{n'-1}  & \text{(Rc)} \\
   \omega^{P''}_1 \omega^{P''}_{n''} & \mapsto \delta_{\M_{P''}, \emptyset} \xi^{P}_{n'+1} \xi^P_{n'+2} \cdots \xi^P_{n}  & \text{(Rc), $n'' > 1$} \\
   \omega^{P'}_1 (\omega^{P''}_1)^\ell \omega^{P'}_{n'} & \mapsto \delta_{\ell, 1} \omega_1^P & \text{(Rc), $\ell \geq 1$, $n'' = 1$}.
 \end{align*}

 We will prove that $\psi^{\circ \circ}$ induces an inverse of $\phi$. Let us prove that it induces a morphism $\psi^\circ: e_\tau \Gamma^\circ_\sigma e_\tau \rightarrow \Gamma_\tau$. If $\vec v$ is not an orientation of $\vec u$ then $\psi^{\circ\circ} (C_{\vec v}) = C_{\vec v} = 0$. We need to look at generators of $e_\tau (C_{\vec u}) e_\tau$ in each case:
 \begin{enumerate}[\rm ({R}a)]
  \item The generator of $e_\tau (C_{\vec u}) e_\tau$ are $\omega^{P'}_1 (\omega_{n'}^{P'})^{\ell_1} C_{\vec u} (\omega_{n'}^{P'})^{\ell_2} \omega^{P'}_{n'-1}$ for $\ell_1, \ell_2 \geq 0$ and we have
 \begin{align*}
   & \psi^{\circ\circ}\left(\omega^{P'}_1 (\omega_{n'}^{P'})^{\ell_1} C_{\vec u} (\omega_{n'}^{P'})^{\ell_2} \omega^{P'}_{n'-1}\right) \\ =\, & \psi^{\circ\circ}\left(\omega^{P'}_1 (\omega_{n'}^{P'})^{\ell_1} (\omega^{P'}_{n'-1} \xi^P_1 \omega^{P'}_1- \xi^{P'}_{n'} \omega^{P'}_{n'}) (\omega_{n'}^{P'})^{\ell_2} \omega^{P'}_{n'-1}\right) \\
    =\, & \left(\delta_{\M_{P'}, \emptyset} \xi_2^P \xi_3^P \cdots \xi_1^P \right)^{\ell_1} \omega_1^P \xi_1^P \left(\delta_{\M_{P'}, \emptyset} \xi_2^P \xi_3^P \cdots \xi_1^P \right)^{\ell_2} \omega_1^P \\ & - \psi^{\circ\circ}\left(\lambda_{P'_{n'}} \omega^{P'}_1 (\omega^{P'}_{n'})^{\ell_1 + m_{P'_{n'}} +  \ell_2} \omega^{P'}_{n'-1}\right) \\
    =\, & \left(\delta_{\M_{P'}, \emptyset} \xi_2^P \xi_3^P \cdots \xi_1^P \right)^{\ell_1 + \ell_2} \left(\xi_2^P \omega_2^P \omega_1^P -  \lambda_{P'_{n'}} (\delta_{\M_{P'}, \emptyset} \xi_2^P \xi_3^P \cdots \xi_1^P)^{m_{P'_{n'}}}  \omega_1^P \right) \\ =\, & \left(\delta_{\M_{P'}, \emptyset} \xi_2^P \xi_3^P \cdots \xi_1^P \right)^{\ell_1 + \ell_2} \xi^P_2 R_{P,1} = 0.
  \end{align*}
  \item The generators are $\omega^{P'}_1 C_{\vec u} \omega^{P'}_{n'}$, $\omega^{P'}_1 C_{\vec u} \omega^{P'}_{n+1}$, $\omega^{P'}_{n+2} C_{\vec u} \omega^{P'}_{n'}$, $\omega^{P'}_{n+2} C_{\vec u} \omega^{P'}_{n+1}$. Denoting $\rho := \io{\vct{P_{n'} P_{n'-1}}, \vct{P_{n+2} P_{n+3}}}$, we have
  \begin{align*}
   \psi^{\circ\circ}(\omega^{P'}_1 C_{\vec u} \omega^{P'}_{n'}) &= \psi^{\circ\circ}(\omega^{P'}_1 (\omega_{n+1}^{P'} \xi_1^{P} \omega_1^{P'} - \omega_{n'}^{P'} \rho \omega_{n+2}^{P'} ) \omega^{P'}_{n'}) \\
           &= C_\tau \psi^{\circ\circ}(\omega_1^{P'} \omega_{n'}^{P'}) - \psi^{\circ\circ}(\omega_1^{P'} \omega_{n'}^{P'}) C_\tau = 0;
  \end{align*}
  \begin{align*}
   \psi^{\circ\circ}(\omega^{P'}_1 C_{\vec u} \omega^{P'}_{n+1}) &= \psi^{\circ\circ}(\omega^{P'}_1 (\omega_{n+1}^{P'} \xi_1^{P} \omega_1^{P'} - \omega_{n'}^{P'} \rho \omega_{n+2}^{P'} ) \omega^{P'}_{n+1}) \\
           &= C_\tau \omega_1^P - \psi^{\circ\circ}(\omega_1^{P'} \omega_{n'}^{P'}) \rho \psi^{\circ\circ}(\omega_{n+2}^{P'} \omega_{n+1}^{P'}) \\
	   &= \xi_2^P \omega_2^P \omega_1^P = 0
  \end{align*}
  as $\# \M_P \geq 2$. The two other cases are similar.
\item with $n', n'' > 1$. Recall that we supposed that $\# \M_{P'}, \# \M_{P''} \in \{0, \infty\}$. The generators of $e_\tau (C_{\vec u}) e_\tau$ are $\omega^{P'}_1 C_{\vec u} \omega^{P'}_{n'}$, $\omega^{P'}_1 C_{\vec u} \omega^{P''}_{n''}$, $\omega^{P''}_1 C_{\vec u} \omega^{P'}_{n'}$ and $\omega^{P''}_1 C_{\vec u} \omega^{P''}_{n''}$. We have
  \begin{align*}
   \psi^{\circ\circ}(\omega^{P'}_1 C_{\vec u} \omega^{P'}_{n'}) &= \psi^{\circ\circ}(\omega^{P'}_1 (\omega^{P''}_{n''} \xi^P_1 \omega^{P'}_1 - \omega^{P'}_{n'} \xi^P_{n'} \omega^{P''}_1) \omega^{P'}_{n'}) \\           
    &= C_\tau \psi^{\circ\circ} (\omega_1^{P'} \omega_{n'}^{P'}) - \psi^{\circ\circ}(\omega_1^{P'} \omega^{P'}_{n'}) C_\tau = 0;
  \end{align*}
  \begin{align*}
   \psi^{\circ\circ}(\omega^{P'}_1 C_{\vec u} \omega^{P''}_{n''}) &= \psi^{\circ\circ}(\omega^{P'}_1 (\omega^{P''}_{n''} \xi^P_1 \omega^{P'}_1 - \omega^{P'}_{n'} \xi^P_{n'} \omega^{P''}_1) \omega^{P''}_{n''}) \\           
    &= C_\tau \omega^P_1 - \psi^{\circ\circ}(\omega_1^{P'} \omega^{P'}_{n'}) \xi_{n'}^P \psi^{\circ\circ}(\omega_1^{P''} \omega^{P''}_{n''}) 
    = \xi^P_2 R_{P,1} = 0
  \end{align*}
  and the two other cases are similar. 
  
  \rs{2}
  \item with $n'' = 1$. The generators of $e_\tau (C_{\vec u}) e_\tau$ are $\omega^{P'}_1 (\omega^{P''}_1)^{\ell_1} C_{\vec u} (\omega^{P''}_1)^{\ell_2} \omega^{P'}_{n'}$ for $\ell_1, \ell_2 \in \N$. As $\# \M_{P''} = \infty$, we have 
  \begin{align*}
   & \psi^{\circ\circ}(\omega^{P'}_1 (\omega^{P''}_1)^{\ell_1} C_{\vec u} (\omega^{P''}_1)^{\ell_2} \omega^{P'}_{n'}) \\ =\, & \psi^{\circ\circ}\left(\omega^{P'}_1 (\omega^{P''}_1)^{\ell_1} \left[\omega^{P''}_1 \omega^{P'}_{n'} \xi_1^P \omega^{P'}_1 - \omega^{P'}_{n'} \xi_1^P \omega^{P'}_1 \omega^{P''}_1\right] (\omega^{P''}_1)^{\ell_2} \omega^{P'}_{n'}\right) \\
    = \, & \delta_{\ell_1, 0} \delta_{\ell_2, 0} \delta_{\M_{P'}, \emptyset} \left(\omega^P_1 \xi^P_1 \xi^{P}_2 \cdots \xi^P_{n'-1} - \xi^{P}_2 \xi^P_3 \cdots  \xi^P_1 \omega^P_1\right) \\ &+ \delta_{\ell_1, 0} \delta_{\ell_2, 1} \omega^P_1 \xi^P_1 \omega^P_1 -\delta_{\ell_1, 1} \delta_{\ell_2, 0} \omega^P_1 \xi^P_1 \omega^P_1 \\
   = \, & \delta_{\ell_1, 0} \delta_{\ell_2, 0} \delta_{\M_{P'}, \emptyset} \left(C_\tau \xi^{P}_2 \cdots \xi^P_{n'-1} - \xi^{P}_2 \cdots \xi^P_{n'-1} C_\tau \right) \\ &+ (\delta_{\ell_1, 0} \delta_{\ell_2, 1} - \delta_{\ell_1, 1} \delta_{\ell_2, 0}) \xi^P_2 \omega^P_2 \omega^P_1 = 0.
  \end{align*}
 \end{enumerate}

 We finished to prove that $\psi^\circ: e_\tau \Gamma^\circ_\sigma e_\tau \rightarrow \Gamma_\tau$ is well defined. Let us prove that it induces a morphism $\psi: e_\tau \Gs e_\tau \rightarrow \Gamma_\tau$. We need to prove that $\psi^\circ (e_\tau I_\sigma e_\tau) = 0$. Using Theorem \ref{thmmin}, we know that $I_\sigma$ is generated by relations coming for minimal polygons. Any minimal polygon of $\sigma$ which does not contain $u$ as a side also exists in $\tau$ so we can focus on relations coming from $P'$ and $P''$. According to Lemma \ref{redideal}, for $i = 1, \dots, \min(n'-1,n)$, we have
  $$\psi^\circ(\alpha_{i+1} R_{P', i} \beta_i) \in \psi^\circ((R_{P,i}) + e_\tau (I_{P,u} + S^P) e_\tau) = \psi^\circ(e_\tau I_{P,u} e_\tau)$$
 and we have analogous observations for some relations coming from $P'$ and $P''$. We look at relations coming from $P'$ which cannot be simplified in that way:
 \begin{enumerate}[\rm ({R}a)]
  \item In this case, these relations are
   \begin{itemize}
    \item $R_{P',1} (\omega_{n'}^{P'})^\ell \omega_{n'-1}^{P'}$ for $\ell \geq 0$ if $n > 1$,
    \item $\omega_1^{P'} (\omega_{n'}^{P'})^\ell R_{P',n'-2}$ for $\ell \geq 0$ if $n > 1$,
    \item $\omega_1^{P'} (\omega_{3}^{P'})^\ell R_{P',1} (\omega_{3}^{P'})^{\ell'} \omega_{2}^{P'}$ for $\ell, \ell' \geq 0$ if $n = 1$,
    \item $R_{P',n'} (\omega_{n'}^{P'})^\ell \omega_{n'-1}^{P'}$ for $\ell \geq 0$,
    \item $\omega_1^{P'} (\omega_{n'}^{P'})^\ell R_{P',n'-1}$ for $\ell \geq 0$.
   \end{itemize}
  Suppose first that $\# \M_{P'} = 0$. If $n > 1$, we get
  \begin{align*}
   &\psi^\circ(R_{P',1} (\omega_{n'}^{P'})^\ell \omega_{n'-1}^{P'}) \\ =\,\,& \psi^\circ((\omega^{P'}_2 \omega^{P'}_1 - \xi^{P'}_3 \xi^{P'}_4 \cdots \xi^{P'}_{n'}) (\omega_{n'}^{P'})^\ell \omega_{n'-1}^{P'}) \\
    =\,\,& \omega^P_2 (\xi^P_2 \cdots \xi^P_1)^\ell \omega_1^P - \xi^P_3 \cdots \xi^P_n \xi^P_1 \psi^\circ\left(\omega^{P'}_1 \lambda_{P'_{n'}} (\omega^{P'}_{n'})^{m_{P'_{n'}} - 1 + \ell} \omega_{n'-1}^{P'}\right) \\
    =\,\,&  (\xi^P_3 \cdots \xi^P_2)^\ell \omega^P_2 \omega_1^P - \lambda_{P'_{n'}} \left(\xi^P_3 \cdots \xi^P_2\right)^{m_{P'_{n'}} - 1 + \ell}\xi^P_3 \cdots \xi^P_1 \omega_1^P \\
    =\,\,&  (\xi^P_3 \cdots \xi^P_2)^\ell R_{P, 1} = 0.
  \end{align*}
  In the same way, we prove that $\psi^\circ(\omega_1^{P'} (\omega_{n'}^{P'})^\ell R_{P',n'-2}) = 0$ and if $n = 1$, $\psi^\circ(\omega_1^{P'} (\omega_{3}^{P'})^\ell R_{P',1} (\omega_{3}^{P'})^{\ell'} \omega_{2}^{P'}) = 0$. Similarly, it is easy to compute that $\psi^\circ(R_{P',n'} (\omega_{n'}^{P'})^\ell \omega_{n'-1}^{P'}) = \psi^\circ (\omega_1^{P'} (\omega_{n'}^{P'})^\ell R_{P',n'-1}) = 0$.

  If $\# \M_{P'} \geq 2$, the computations are immediate. Finally, if $\# \M_{P'} = 1$, we have $(\alpha_{i+1} R_{P', i} \beta_i)_{i = 1\dots n'} + S^P = (\alpha_{i+1} \omega^{P'}_{i+1} \omega^{P'}_i \beta_i )_{i = 1\dots n'} + S^P$ thanks to Lemma \ref{RRp} and the result is easy to obtain.
  \item In this case, thanks to Lemma \ref{RRp}, $$e_{\tau} \left((R_{P', i})_{i = 1\dots n'} + S^P\right) e_\tau = e_\tau \left((\omega^{P'}_{i+1} \omega^{P'}_i)_{i = 1\dots n'} + S^P\right) e_\tau.$$ It is easy to check that 
   $$\psi^\circ\left(e_\tau \left(\omega^{P'}_{i+1} \omega^{P'}_i\right)_{i = 1\dots n'} e_\tau\right) \subset \left(\omega^{P}_{i+1} \omega^{P}_i\right)_{i = 1\dots n} + \left(\omega^{\tilde P}_{i+1} \omega^{\tilde P}_i\right)_{i = 1\dots \tilde n} $$
   where $\tilde P$ is the polygon with sides $\vct{P'_{n+2} P'_{n+3}}$, \dots, $\vct{P'_{n'-1} P'_{n'}}$ and $R_{P, i} = \omega_{i+1}^P \omega_i^P$ and $R_{\tilde P, i} = \omega_{i+1}^{\tilde P} \omega_i^{\tilde P}$ as $\# \M_P > 1$ and $\# \M_{\tilde P} > 1$ so the result is immediate in this case.
  \item if $n', n'' > 1$. In this case, we need to look at $R_{P', 1}\omega^{P'}_{n'}$, $\omega^{P'}_1 R_{P', n'-1}$ and $R_{P',n'}$. We get
  \begin{align*}
   \psi^\circ\left(R_{P', n'}\right) &= \psi^\circ\left(\omega^{P'}_1 \omega^{P'}_{n'} - \delta_{\M_{P'}, \emptyset}\xi^{P'}_2 \xi^{P'}_3 \cdots \xi^{P'}_{n'-1}\right) = 0,  \\
      \psi^\circ\left(R_{P', 1}\omega^{P'}_{n'}\right) &= \psi^\circ\left(\left(\omega^{P'}_2 \omega^{P'}_1 - \delta_{\M_{P'}, \emptyset} \xi^{P'}_3 \xi^{P'}_4 \cdots \xi^{P'}_{n'}\right)\omega^{P'}_{n'}\right) \\
     &= \psi^\circ\left(\omega^{P'}_{2} R_{P', n'}\right) = 0
   \end{align*}
   and the same for $\omega^{P'}_1 R_{P', n'-1}$. 
  \rs{2}
  \item if $n'' = 1$ and $n' \geq 3$. We have to check relations $R_{P', 1} (\omega^{P''}_1)^\ell \omega^{P'}_{n'}$, $R_{P', n'}$ and $\omega^{P'}_{1} (\omega^{P''}_1)^\ell R_{P', n'-1}$ for $\ell = 0$ or $\ell \geq 2$. It is clear, by definition of $\psi^{\circ}$, that $\psi^\circ(R_{P',n'}) = 0$. As we made the hypothesis that $\# \M_{P''} \in \{0, \infty\}$, we necessarily have $\# \M_{P''} = \infty$ and therefore $\# \M_P = \infty$ so 
  \begin{align*}
   \psi^\circ\left(R_{P', 1} \omega^{P'}_{n'}\right) &= \psi^\circ\left(\left[\omega^{P'}_2 \omega^{P'}_1 - \delta_{\M_{P'}, \emptyset} \xi^{P'}_3 \dots \xi^{P'}_{n'}\right] \omega^{P'}_{n'} \right) \\
    &= \delta_{\M_{P'}, 0} \left(\omega_2^{P} \xi^P_2 \dots \xi^P_{n'-1} - \xi_3^{P} \dots \xi^{P}_{n'-1} C_\sigma \right) \\
    &= \delta_{\M_{P'}, 0} \left(C_\sigma \xi^P_3 \dots \xi^P_{n'-1} - \xi_3^{P} \dots \xi^{P}_{n'-1} C_\sigma \right)  = 0.
  \end{align*}
  and if $\ell \geq 2$, $\psi^\circ(R_{P', 1} (\omega^{P''}_1)^\ell \omega^{P'}_{n'}) = 0$ by direct computation. The proof works analogously for $\omega^{P'}_{1} (\omega^{P''}_1)^\ell R_{P', n'-1}$.
  \rs{2}
  \item if $n'' = 1$ and $n' = 2$. We have to check $\omega^{P'}_{1} (\omega^{P''}_1)^{\ell_1} R_{P', 1} (\omega^{P''}_1)^{\ell_2} \omega^{P'}_{2}$ for $\ell_1, \ell_2 \geq 0$ and $R_{P', 2}$. It is similar as before.
 \end{enumerate}

 The only case which is not analogous for $P''$ is (Rc) when $n'' = 1$. In this case, $e_\tau ( R_{P'', 1} ) e_\tau$ is generated by relations $\omega^{P'}_1 (\omega^{P''}_1)^{\ell_1} R_{P'', 1} (\omega^{P''}_1)^{\ell_2} \omega^{P'}_{n'}$ for $\ell_1, \ell_2 \geq 0$. As in this case $\# \M_{P''} = \infty$, we get
 $$\psi^\circ \left(\omega^{P'}_1 (\omega^{P''}_1)^{\ell_1} R_{P'', 1} (\omega^{P''}_1)^{\ell_2} \omega^{P'}_{n'} \right) = \psi^\circ \left(\omega^{P'}_1 (\omega^{P''}_1)^{\ell_1 + 2 + \ell_2}  \omega^{P'}_{n'} \right) = 0$$
 so we finished to prove that $\psi^\circ(e_\tau I_\sigma e_\tau) = 0$ and $\psi: e_\tau \Gs e_\tau \rightarrow \Gamma_\tau$ is a well defined morphism of algebra.

 It is immediate that $\psi \circ \phi = \id_{\Gamma_\tau}$. We also have $\phi \circ \psi = \id_{e_\tau \Gs e_\tau}$. It is enough to check the generators of $e_\tau k Q_\sigma e_\tau$ one by one. It is easy in every case. Finally, $e_\tau \Gamma_\sigma e_\tau \cong \Gamma_\tau$.

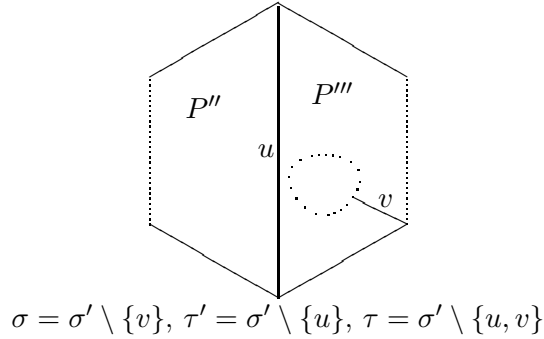
\begin{figure}
\begin{center}
\boxinminipage{$\renewcommand{\labelstyle}{\textstyle} \xymatrix@L=.05cm@!=0cm@R=.089cm@C=.089cm@M=0.00cm{ 
    & & & & & & & & & & & & & & & & & & & \ar@{-}[rrrrrrrrrrrrrrrrrrrddddddddddd] \ar@{-}[dddddddddddddddddddddddddddddddddddddddddddd]_u \\ \\ \\ \\ \\ \\ \\ \\ \\ \\ \\
    \ar@{-}[rrrrrrrrrrrrrrrrrrruuuuuuuuuuu] & & & & & & & & & & & & & & & & & & & & & & & & & & & & & & & & & & & & & & \ar@{..}[dddddddddddddddddddddd] \\ \\ \\ & & & & & & & & & & & & & & & & & & & & & & & & & & & P''' \\ \\ & & & & & & & & P'' \\ \\ \\ \\ \\ \\ \\ \\ \\ \\ \\ \\ \\ & & & & & & & & & & & & & & & & & & & & & & & & & & & & & & \ar@{..}@`{p+(4,6), p+(-16,8), p+(-4,-6)}[]  \\  \\ \\ \\
    \ar@{..}[uuuuuuuuuuuuuuuuuuuuuu]  & & & & & & & & & & & & & & & & & & & & & & & & & & & & & & & & & & & & & & \ar@{-}[dddddddddddlllllllllllllllllll] \ar@{-}[uuuullllllll]_v \\ \\ \\ \\ \\ \\ \\ \\ \\ \\ \\
    & & & & & & & & & & & & & & & & & & & \ar@{-}[llllllllllllllllllluuuuuuuuuuu] 
}$}

   $\sigma = \sigma' \setminus \{v\}$, $\tau' = \sigma' \setminus \{u\}$, $\tau = \sigma' \setminus \{u,v\}$
\end{center}
  \caption{Inductive argument: partial triangulation $\sigma'$  } \label{illind}   
  \end{figure}

 (b) We have to prove the isomorphism in case (Rc) when $0 < \# \M_{P'} < \infty$ or $0 < \# \M_{P''} < \infty$. We will do an induction on $(\# \M_{P'}, \# \M_{P''})$ with the product order. The cases $(0,0)$, $(0,\infty)$, $(\infty, 0)$ and $(\infty, \infty)$ have been proved in (a). Suppose that $0 < \# \M_{P'} < \infty$. As $P'$ is minimal, it is immediate that we can construct a partial triangulation $\sigma'$ such that $\sigma = \sigma' \setminus \{v\}$ where $s(\vec v)$ is a vertex of $P'$ and $t(\vec v) \in \M_{P'}$ for an orientation $\vec v$ of $v$. See Figure \ref{illind}. By minimality of $P'$, it is clear that the reduction situation in $\sigma'$ involving $P'$ and $v$ is of type (Ra) or (Rb). So we already know that $e_\sigma \Gamma_{\sigma'} e_\sigma \cong \Gs$. On the other hand, the minimal polygon $P'''$ containing $\vec u$ in $\sigma'$ has at least one puncture less than $P'$ (indeed $\M_{P'''} \subset \M_{P'} \setminus \{t(\vec v)\}$). Thus, by induction hypothesis, $e_{\tau'} \Gamma_{\sigma'} e_{\tau'} \cong \Gamma_{\tau'}$ where $\tau' = \sigma' \setminus \{u\}$. As before, the reduction situation involving $P$ and $v$ in $\tau'$ is of type (Ra) or (Rb) so we know that $e_\tau \Gamma_{\tau'} e_\tau \cong \Gamma_\tau$. Finally, we have
  $$e_\tau \Gamma_{\sigma} e_\tau \cong e_\tau e_\sigma \Gamma_{\sigma'} e_\sigma e_\tau = e_\tau e_{\tau'} \Gamma_{\sigma'} e_{\tau'} e_\tau \cong e_\tau \Gamma_{\tau'} e_\tau \cong \Gamma_\tau.$$
 A similar argument gives the induction step when $0 < \# \M_{P''} < \infty$. \qed

\subsection{Proof of Proposition \ref{boundord}} \label{proofboundord}
 Denote by $E_{i,j}$ the matrix with entry $1$ in cell $(i,j)$ and $0$ everywhere else. In each case, we give the images of the generators:
 
 (1) In this case, the isomorphism is generated by
   \begin{align*}
    \ic{\vct{P_1 P_2}, \vct{P_1 P_n}}  & \mapsto \lambda_{P_1}^{n-2} t^{-m} E_{1,n} \\
    \ic{\vct{P_i P_{i+1}}, \vct{P_i P_{i-1}}} & \mapsto t^m E_{i, i-1} & (1<i\leq n) \\
    \ic{\vct{P_1 P_n}, \vct{P_1 P_2}} & \mapsto \lambda_{P_1}^{2-n} t^{m+1} E_{n,1} \\
    \ic{\vct{P_i P_{i-1}}, \vct{P_i, P_{i+1}}} & \mapsto \lambda_{P_1} \lambda_{P_i}^{-1} E_{i-1,i} & (1 < i \leq n).    
   \end{align*}

 (2) In this case, the isomorphism is generated by
   \begin{align*}
    \ic{\vct{P_1 P_2}, \vct{P_1 P_n}}  & \mapsto \lambda_M (0, x^{m-1}) E_{1,n} \\
    \ic{\vct{P_i P_{i+1}}, \vct{P_i P_{i-1}}} & \mapsto \lambda_M (0, x^{m})  E_{i, i-1} & (1<i\leq n) \\
    \ic{\vct{P_1 P_n}, \vct{P_1 P_2}} & \mapsto \lambda_{P_1}^{-1} (x,x) E_{n,1} \\
    \ic{\vct{P_i P_{i-1}}, \vct{P_i, P_{i+1}}} & \mapsto \lambda_{P_i}^{-1} (1,1) E_{i-1,i} & (1 < i \leq n).
   \end{align*}

 (3) In this case, the isomorphism is generated by
   \begin{align*}
    \ic{\vct{P_1 P_2}, \vct{P_1 P_n}}  & \mapsto t^{-1} \varepsilon E_{1,n} \\
    \ic{\vct{P_i P_{i+1}}, \vct{P_i P_{i-1}}} & \mapsto \varepsilon  E_{i, i-1} & (1<i\leq n) \\
    \ic{\vct{P_1 P_n}, \vct{P_1 P_2}} & \mapsto \lambda_{P_1}^{-1} x E_{n,1} \\
    \ic{\vct{P_i P_{i-1}}, \vct{P_i, P_{i+1}}} & \mapsto \lambda_{P_i}^{-1} E_{i-1,i} & (1 < i \leq n).
   \end{align*}

 The proof that these definitions give isomorphisms is each time elementary. In the first case the central element used is $U := C_\sigma$. In the second and third case, it is 
 $U := \sum_{i = 1}^n (\xi^P_{i+1} \xi^P_{i+2} \cdots \xi^P_i)$. We have easily $k \llbracket x \rrbracket \cong k \llbracket U \rrbracket$. Then, the key point in each case is that it is easy to determine a $k \llbracket U \rrbracket$-basis of $e_u \Gamma_\sigma e_v$ using relations for any pair of sides $u$ and $v$. \qed

\subsection{Proofs of Theorem \ref{classiford2}} \label{proofclassiford2}
(1) This is a consequence of Corollary \ref{dimd} proven later (indeed, in this case, $e \Gs e = e \Delta_\sigma e$).

We will now prove (2). We need some preparation. Using the same strategy as for Proposition \ref{boundord}, we get the following lemma:
\begin{lemma} \label{bascases}
 \begin{enumerate}[\rm (1)]
  \item We consider the following triangulation $\sigma$ of a disc with one puncture and one marked point on the boundary: 
  $$
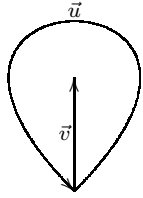
\xymatrix@L=.05cm@R=1.5cm@C=1.5cm@M=0.01cm{
   & & \\
   & \ar[u]^{\vec v} \ar@`{p+(15,15),p+(0,30),p+(-15,15)}[]_{\vec u} &   
  }$$
  with $m_{s(\vec v)} = 1$ and $m_{t(\vec v)} = m$ (and $m > 1$ by hypothesis). Then we get
  $$\Gamma_\sigma \cong \begin{bmatrix}
                         R_{m-1}  &  0 \times R \\
			 0 \times t^{m-1} R & 0 \times R
                        \end{bmatrix}$$
  using the notation of Proposition \ref{boundord} (2).

 \item We consider the following triangulation $\sigma$ of a disc with one puncture and two marked points on the boundary: 
  $$
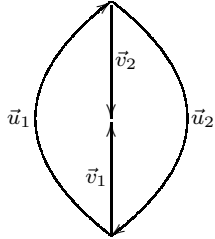
\xymatrix@L=.05cm@R=1.5cm@C=1.5cm@M=0.01cm{
   & \ar@/^1cm/[dd]^{\vec u_2} \ar[d]^{\vec v_2} & \\
   & & \\
   & \ar@/^1cm/[uu]^{\vec u_1} \ar[u]^{\vec v_1} & 
  }$$
  with $m_{s(\vec v_1)} = m_{s(\vec v_2)} = 1$ and $m_{t(\vec v_1)} = m$. Then we get
  $$\Gamma_\sigma \cong \begin{bmatrix}
                         R_m  &  R_{m-1} & 0 \times R & 0 \times R \\
			 tR_{m-1}  &  R_m & 0 \times tR & 0 \times R \\
			 0 \times t^m R  &  0 \times t^{m-1} R & 0 \times R & 0 \times R \\
		         0 \times t^m R  &  0 \times t^m R & 0 \times t R & 0 \times R
                        \end{bmatrix}.$$
 \end{enumerate}
\end{lemma}

\begin{proof}
 We just give the images of arrows of quivers:
 
 (1) We put 
  \vspace*{-1cm}\begin{multicols}{2} 
  \begin{align*} \ic{\vec u, \vec v} &\mapsto \lambda_{t(\vec v)} (0,1) E_{1,2} \\ \ic{-\vec v, -\vec v} &\mapsto (0, t) E_{2,2} \end{align*} 

  \begin{align*} \ic{\vec v, -\vec u} &\mapsto (0, t^{m-1}) E_{2,1}\\ \ic{-\vec u, \vec u} &\mapsto \lambda_{s(\vec v)}^{-1} (t,t) E_{1,1}.\end{align*}
  \end{multicols}

 (2) We put
  \vspace*{-1cm}\begin{multicols}{2}%
  \begin{align*}
   \ic{-\vec u_2, \vec v_1} &\mapsto (0,1) E_{1,3} \\
   \ic{-\vec u_1, \vec v_2} &\mapsto (0,1) E_{2,4} \\
   \ic{\vec v_1, \vec u_1} &\mapsto \lambda_M (0,t^{m-1}) E_{3,2} \\
   \ic{\vec v_2, \vec u_2} &\mapsto \lambda_M (0,t^{m}) E_{4,1}
  \end{align*}

  \begin{align*}
   \ic{-\vec v_1, -\vec v_2} &\mapsto (0,1) E_{3,4} \\
   \ic{-\vec v_2, -\vec v_1} &\mapsto (0,t) E_{4,3} \\
   \ic{\vec u_1, -\vec u_2} &\mapsto \lambda_{s(\vec u_1)}^{-1} (t,t) E_{2,1} \\
   \ic{\vec u_2, -\vec u_1} &\mapsto \lambda_{s(\vec u_2)}^{-1} (1,1) E_{1,2}. \qedhere
  \end{align*}
  \end{multicols}
\end{proof}

The following lemma is strongly inspired of the strategy of \cite{DeLu16, DeLu16-2}. We need here a small part of the results of these articles, but in bigger generality.

\begin{lemma} \label{inducord}
 We consider partial triangulations $\sigma$ and $\sigma'$ of two marked surfaces $(\Sigma, \M)$ and $(\Sigma', \M')$ such that there is an embedding $\Sigma' \subset \Sigma$ satisfying:
 \begin{itemize}
  \item $\M = \M' \sqcup \{M\}$ where $M$ in on a boundary component and $m_M = 1$;
  \item $\sigma = \sigma' \sqcup \{u, v\}$ where $u$ and $v$ are boundary edges of $\Sigma$ incident to $M$;
  \item $\Sigma \setminus \Sigma'$ is connected.
 \end{itemize}
 We take the following notation:
 $$
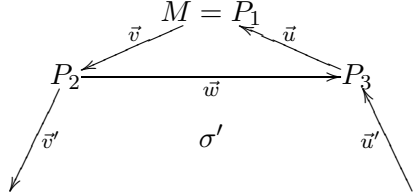
\xymatrix@L=.05cm@R=.5cm@C=.5cm@M=0.01cm{
  & & & M = P_1 \ar[lld]_{\vec v} & & & \\
  & P_2 \ar[rrrr]_{\vec w} \ar[ddl]^{\vec v'} & & & & P_3 \ar[ull]_{\vec u} & \\
  & & & \sigma' & & &   \\
  & & & & & &  \ar[uul]^{\vec u'}
 }$$
 where $\vec u$, $\vec v$, $\vec w$ form a triangle without puncture.

 Then $\Gamma_{\sigma'} = e_{\sigma'} \Gamma_\sigma e_{\sigma'}$ through the canonical inclusion of paths algebras $k Q_{\sigma'} \subset k Q_{\sigma}$. Moreover, this identification induces equalities of left $\Gamma_{\sigma'}$-modules $\Gamma_{\sigma'} \ic{-\vec u', \vec u} = e_{\sigma'} \Gamma_{\sigma} e_u$, $\Gamma_{\sigma'} \ic{\vec w, -\vec v} = e_{\sigma'} \Gamma_{\sigma} e_v$ and equalities of right $\Gamma_{\sigma'}$-modules $\ic{\vec u, -\vec w} \Gamma_{\sigma'} = e_u \Gamma_\sigma e_{\sigma'}$, $\ic{-\vec v, \vec v'} \Gamma_{\sigma'} = e_v \Gamma_\sigma e_{\sigma'}$. 
\end{lemma}

\begin{proof}
 Let us number the vertices of the triangle $P$ with sides $\vec u$, $\vec v$, $\vec w$ in such a way that $P_1 = M$. We consider the quiver $Q'_\sigma$ obtained from $Q_\sigma$ by removing $\ic{-\vec u, \vec v}$. As $R_{P,2} = \omega^P_3 \omega^P_2 - \lambda_M \ic{-\vec u, \vec v}$ and $C_{\vec u}, C_{\vec v} \in (R_{P,1}, R_{P,2}, R_{P,3})$, it is immediate that $\Gamma_\sigma$ is $k Q'_\sigma$ modulo all relations except $R_{P, 2}$, $C_{\vec u}$ and $C_{\vec v}$.

 As all relations defining $\Gamma_{\sigma'}$ are relations in $\Gamma_\sigma$, the inclusion $\iota: k Q_{\sigma'} \subset k Q'_\sigma$ induces a morphism of algebras $\phi: \Gamma_{\sigma'} \to e_{\sigma'} \Gamma_\sigma e_{\sigma'}$.  Thanks to the relations coming from $P$, any path of $Q'_\sigma$ can be rewritten as a multiple of a path without factor $\omega^P_2 \omega^P_1$ or $\omega^P_1 \omega^P_3$. Thus, it is clear that $\phi$ is surjective. 

 Using the same argument, we have:
 \begin{align*}
  e_{\sigma'} (R_{P,1}, R_{P,3}) e_{\sigma'} &= e_{\sigma'} (\omega^P_2 \omega^P_1 - \xi^P_3, \omega^P_1 \omega^P_3 - \xi^P_2 ) e_{\sigma'} \\ 
    &= e_{\sigma'}(\omega^P_2 \omega^P_1 \omega^P_3 - \xi^P_3 \omega^P_3, \omega^P_2 \omega^P_1 \omega^P_3 - \omega^P_2 \xi^P_2)e_{\sigma'} \\
    &= e_{\sigma'}(\omega^P_2 \xi^P_2 - \xi^P_3 \omega^P_3, \omega^P_2 \omega^P_1 \omega^P_3 - \omega^P_2 \xi^P_2)e_{\sigma'} \\ &= e_{\sigma'}(C_{\vec w}, \omega^P_2 \omega^P_1 \omega^P_3 - \omega^P_2 \xi^P_2)e_{\sigma'}
 \end{align*}
 so $\iota^{-1}(e_{\sigma'} (I^\circ_\sigma + I_\sigma) e_{\sigma'}) = I^\circ_{\sigma'} + I_{\sigma'}$ and therefore $\phi$ is injective.
 We proved the first part of the statement. 

 With the same strategy, there is a surjective morphism of left $\Gamma_{\sigma'}$-modules $\psi: \Gamma_{\sigma'} \ic{-\vec u', \vec u} \to e_{\sigma'} \Gamma_\sigma e_u$. And, modulo $(I_{\sigma'} + I^\circ_{\sigma'}) \ic{-\vec u', \vec u}$, we have
  \begin{align*} e_{\sigma'} (R_{P,1}, R_{P,3}) e_u &= e_{\sigma'}(\omega^P_2 \omega^P_1 - \xi^P_3)e_u + e_{\sigma'}(\omega^P_2 \omega^P_1 \omega^P_3 - \omega^P_2 \xi^P_2) \ic{-\vec u', \vec u} \\
   &= e_{\sigma'}(\omega^P_2 \omega^P_1 - \xi^P_3)e_u + e_{\sigma'}(\omega^P_2 \omega^P_1 \omega^P_3 - \xi^P_3 \omega^P_3) \ic{-\vec u', \vec u} \\
   &= e_{\sigma'}(\omega^P_2 \omega^P_1 - \xi^P_3)e_u
  \end{align*}
  so $e_{\sigma'} (R_{P,1}, R_{P,3}) e_u  \cap k Q_{\sigma'} \ic{-\vec u', \vec u} \subset (I_{\sigma'} + I^\circ_{\sigma'}) \ic{-\vec u', \vec u}$ and therefore $\psi$ is injective. The three other equalities are proved in the same way.
\end{proof}

\begin{definition} \label{bigcycle}
 Suppose that we are in cases (b) or (c) of Theorem \ref{classiford2} (2). Let $u_0$ be in $E$ which is not incident to a puncture (thus, by hypothesis, it is homotopic to a part of a boundary component). We consider a $n$-gon $P$ of $\sigma$ such that
 \begin{itemize}
  \item $\vct{P_n P_1}$ is an orientation of $u_0$;
  \item $\M_P$ contains all punctures of $(\Sigma, \M)$;
  \item $\M_P$ contains all holes of $\Sigma$ except the hole $u_0$ is incident to.
 \end{itemize}
 Then we say that $\xi^P_1 \xi^P_2 \cdots \xi^P_n$ is a \emph{big cycle} at $u_0$.
\end{definition}

\begin{lemma} \label{centreord}
 Under the assumptions of Definition \ref{bigcycle}, we get 
 \begin{enumerate}[\rm (1)]
  \item In cases (b) or (c) of Theorem \ref{classiford2} (2), all big cycles at $u_0$ are equal to the same element $S_{u_0}$ of $\Gamma_\sigma$.
  \item Suppose that we are in the situation of Theorem \ref{classiford2} (2). For $u \in E$, we denote:
   $$U_u := \left\{\begin{array}{ll}
                    e_u C_\sigma & \text{in case (a) or if $u$ is incident to a puncture;} \\
		    \lambda_M S_u^{m_M} & \text{in case (b) if $u$ is not incident to the puncture $M$;} \\
                    S_u & \text{in case (c).}
                   \end{array}\right.$$
   Then the element $U_\sigma := \sum_{u \in E} U_u$ is in the centre of $e \Gamma_\sigma e$.
 \end{enumerate}
\end{lemma}

\begin{proof}
 (1) Let us take two polygons $P_1$ and $P_2$ as in Definition \ref{bigcycle}. First of all, the hypotheses imply immediately that $P_1$ and $P_2$ cannot pass through a puncture (otherwise $P_1$ or $P_2$ would not contain this puncture). Thus, thanks to Theorem \ref{subtri}, we can suppose that $\sigma$ does not contain any arc incident to a puncture. 
 
 We will prove by induction on the the number of marked point on the boundary component incident to $u_0$ that the big cycles defined from $P_1$ and $P_2$ are equal in $\Gamma_\sigma$. If this number is $1$ or $2$, we necessarily have $P_1 = P_2$ so the result is immediate. If this number is at least $3$, it is then an easy combinatorial observation that there is a marked point $M$ such that:
 \begin{itemize}
  \item $M$ is on the boundary component $u_0$ is incident to;
  \item $M$ is not incident to $u_0$;
  \item $M$ is not incident to any (non-boundary) arc of $\sigma$.
 \end{itemize}
 Then, up to adding the arc $w$ if it is not in $\sigma$, the assumptions of Lemma \ref{inducord} are satisfied. For $i = 1,2$, we denote by $P'_i$ the polygon obtained from $P_i$ by replacing the sequence of sides $\vec u$, $\vec v$ by $-\vec w$ if it appears (it is possible using Proposition \ref{caracpoly} and the assumptions about $P_i$). Then, it is an immediate consequence of the relations that the the big cycle defined from $P_i$ is equal to the big cycle defined from $P'_i$ in $\Gamma_\sigma$. By induction hypothesis, the big cycles defined from $P'_1$ and $P'_2$ are equal in $\Gamma_{\sigma'}$ so they are equal in $\Gamma_\sigma$ thanks to Lemma \ref{inducord}.

 (2) Case (a) is immediate so we focus on Cases (b) and (c). Let $\ic{\vec u, \vec v}$ be an arrow of $Q_\sigma$ which links to arcs of $E$. Let us prove that $\ic{\vec u, \vec v} U_\sigma = U_\sigma \ic{\vec u, \vec v}$ \emph{i.e.} $\ic{\vec u, \vec v} U_v = U_u \ic{\vec u, \vec v}$. We consider several cases:
 \begin{enumerate}[\rm (a)]
  \item If $u$ and $v$ are both incident to punctures, it is an immediate consequence of the fact that $C_\sigma$ is central. 
  \item If none of $u$ and $v$ is incident to a puncture. If there is a polygon $P$ having $-\vec u$, $\vec v$ as two consecutive sides and satisfying hypotheses of Definition \ref{bigcycle}, then, with $P_1 = s(\vec u)$ we have
  \begin{align*} \ic{\vec u, \vec v} S_v &= \ic{\vec u, \vec v} \xi^P_2 \xi^P_3 \cdots \xi^P_n \xi^P_1 = \ic{\vec u, \vec v} \xi^P_2 \xi^P_3 \cdots \xi^P_n \lambda_{P_1} \ic{\vec u, \vec v} \\ &= \xi^P_1 \xi^P_2 \cdots \xi^P_n \ic{\vec u, \vec v} = S_u \ic{\vec u, \vec v}\end{align*}
  (where we used that $m_{P_1} = 1$). If there is a polygon $P$ having $-\vec v$, $\vec u$ as two consecutive sides and satisfying hypotheses of Definition \ref{bigcycle}, then, with $P_1 = s(\vec u)$ we have
  $$\ic{\vec u, \vec v} S_v = \ic{\vec u, \vec v} \xi^P_1 \xi^P_2 \cdots \xi^P_n = C_\sigma \xi^P_2 \cdots \xi^P_n  = \xi^P_2\cdots \xi^P_n \xi^P_1 \ic{\vec u, \vec v} = S_u \ic{\vec u, \vec v}.$$
  If none of these two cases are satisfied, then it is an easy consequence of the hypotheses that $-\vec v$ and $\vec u$ are consecutive sides of a polygon $P'$ with $\M_{P'} = \emptyset$ (because $u$ and $v$ are homotopic to parts of the boundary and $\ic{\vec u, \vec v}$ is an arrow). Then there is a polygon $P$ satisfying the hypothesis of Definition \ref{bigcycle} for $u$ (or for $v$) such that $v$ (or $u$) is a reduction arc for $P$ of type (Rc) as in Figure \ref{reds}. Moreover the polygon $P''$ defined by this reduction situation satisfy the hypotheses of Definition \ref{bigcycle} for $v$ (or for $u$). Using relations in $P'$, we get $\ic{\vec u, \vec v} S_v = S_u \ic{\vec u, \vec v}$ in any case.
  \item If $u$ is not incident to a puncture and $v$ is incident to a puncture. We are in Case (b) of Theorem \ref{classiford2} (2). Let $P$ be a polygon satisfying hypotheses of Definition \ref{bigcycle} for $u$. As $\ic{\vec u, \vec v}$ is an arrow, it is easy that $\vec u$ is an oriented side of $P$. Moreover $\vec v$ is a reduction arc for $P$ of type (Ra). Taking the notation of Figure \ref{reds}, we get
   $$ \ic{\vec u, \vec v} U_v = \lambda_{P'_{n'}} \omega^{P'}_1 \left(\omega^{P'}_{n'}\right)^{m_{P'_{n'}}} = \lambda_{P'_{n'}} (\xi^{P}_2 \cdots \xi^P_1)^{m_{P'_{n'}}} \ic{\vec u, \vec v} = U_u \ic{\vec u, \vec v}.$$
  \item If $u$ is incident to a puncture and $v$ is not, this is similar as (c). \qedhere
 \end{enumerate}
\end{proof}

To finish the proof of Theorem \ref{classiford2} (2), we suppose that we are in one of the three cases (a), (b) or (c). Consider the central element $U_\sigma$ defined in Lemma \ref{centreord}. We will prove that $e \Gamma_\sigma e$ is a $k\llbracket x \rrbracket$-lattice by mapping $x$ to $U_\sigma$. Let $u$ and $v$ be two edges of $\sigma$. Let us prove that $e_u \Gamma_\sigma e_v$ is free as a left $k\llbracket x \rrbracket$-module. First of all, thanks to Theorem \ref{subtri}, we can suppose that $\sigma$ contains only $u$ and $v$ in addition to boundary edges. We do an induction on the number of boundary edges.

If $u$ and $v$ are both boundary edges, or if we are in one of the cases of Lemma \ref{bascases}, the result is immediate by Proposition \ref{boundord} and Lemma \ref{bascases}: $x = t^m$ in cases (a) or (b) and $x = t$ in case (c). Otherwise, an easy combinatorial argument shows that the boundary contains a marked point $M$ satisfying $m_M = 1$ and neither $u$ nor $v$ is incident to $M$. Then Lemma \ref{inducord} permits to conclude the proof of Theorem \ref{classiford2} (2).

Finally, we prepare Theorem \ref{classiford2} (3).

\begin{lemma} \label{reford1}
 Let $E_0 \subset E$ correspond to a boundary component and $e_0$ be the corresponding idempotent. Let $P$ be the $n$-gon of $\sigma$ corresponding to this boundary component. If 
 \begin{itemize}
  \item[] $\# \M_P > 0$ and $m_{P_i} > 1$ for some $i$,
  \item[or] $m_{P_i}, m_{P_j} > 1$ for $i \neq j$, 
 \end{itemize}
 then $e_0 \Gamma_\sigma e_0$ is not finitely generated over its centre.
\end{lemma}

\begin{proof}
 To simplify, we suppose that $m_{P_1} > 1$. For $i = 1, \dots, n$, we denote $$\xi'_i := \ic{\vct{P_i P_{i-1}}, \vct{P_i P_{i+1}}}.$$ 

 We suppose first that we are in the first case or in the second case with $j \neq i \pm 1$. Let $\alpha := \xi'_1 \xi'_2 \cdots \xi'_n$. It is immediate that the $\alpha^n$'s form a linearly independent set of $e_0 \Gamma_\sigma e_0$ as no path appearing in any relation defining $\Gamma_\sigma$ is a factor of any $\alpha^n$. As a consequence, if $e_0 \Gamma_\sigma e_0$ was finitely generated over its centre, there would be an element $U = \alpha^\ell + U'$ in the centre of $e_0 \Gs e_0$ where $\ell > 0$ and $U'$ does not contain a multiple of $\alpha^\ell$ as a summand. Again, no path appearing in any relation is a factor of $\omega_1^P \alpha^\ell$. Therefore, any element of the complete path algebra equivalent to $\omega_1^P U$ modulo the relations defining $e_0 \Gamma_\sigma e_0$ has $\omega_1^P \alpha^\ell$ as a term and cannot be divisible to the right by $\omega_1^P$. It contradicts the fact that $U$ is in the centre.

 Let us now suppose that $m_{P_1}, m_{P_2} > 1$ and $\# \M_P = 0$ (so $n \geq 3$ by hypothesis). We consider the following elements of the path algebra:
  $$\beta_1 := \xi'_1 \xi'_2 \omega^P_2 \omega^P_1 \quad \text{and} \quad \beta_2 := \xi'_1 \xi'_2 \xi^P_3 \xi^P_4 \cdots \xi^P_n$$
 and we notice that $\beta_1 = \beta_2$ in $\Gamma_\sigma$ satisfy that any $\beta_1^\ell$ can only be rewritten up to the relations as a linear combination of 
 $\beta_{i_1} \beta_{i_2} \cdots \beta_{i_\ell}$ for $i_1, \dots, i_\ell \in \{1,2\}$. In particular they are linearly independent. Thus, as before, we should have $U =  \beta_1^\ell + U'$ in the centre where $\ell > 0$ and $U'$ does not have any summand that is a scalar multiple of $\beta_1^\ell$. If $n > 3$ or $m_{P_1} > 2$, we notice as before that $\xi'_1 \omega^P_1 \beta_1^\ell$ can only be rewritten as $\xi'_1 \omega^P_1 \beta_{i_1} \beta_{i_2} \cdots \beta_{i_\ell}$ for $i_1, \dots, i_\ell \in \{1,2\}$. In particular, it is never right divisible by $\xi'_1 \omega^P_1$ so it contradicts the fact that $U$ is in the centre. In the same way, if $n = 3$ and $m_{P_1} = 2$, $\xi_1' \omega_1^P \beta_1^\ell$ can only be rewritten as $\xi_1' \omega_1^P \beta_{i_1} \beta_{i_2} \cdots \beta_{i_\ell}$ or $\lambda_{P_1}^{-1} \omega_3^P \omega_2^P \xi_2' \omega_2^P \omega_1^P \beta_{i_2} \cdots \beta_{i_\ell}$ or $\lambda_{P_1}^{-1} \omega_3^P \omega_2^P \xi_2' \xi_3^P \beta_{i_2} \cdots \beta_{i_\ell}$ for $i_1, \dots, i_\ell \in \{1,2\}$ so it is not right divisible by $\xi_1' \omega_1^P$ and $U$ is not in the centre.
\end{proof}

\begin{lemma} \label{reford2}
 Let $E_0 \subset E$ correspond to a boundary component, $P$ be the $n$-gon of $\sigma$ corresponding to this boundary component. Suppose that $P$ has at least two punctures or a hole and $E$ contains at least a non-boundary arc which is not homotopic to a part of a boundary. Then $e \Gamma_\sigma e$ is not finitely generated over its centre.
\end{lemma}

\begin{proof}
 Using Lemma \ref{reford1}, we can suppose that $m_N = 1$ for any vertex $N$ of $P$. We suppose thanks to Theorem \ref{subtri} that $E$ contains only one non-boundary arc $\vec u$ such that $s(\vec u) = P_1$ (as $E$ is connected). We will prove that $e_u \Gamma_\sigma e_u$ is not finitely generated over its centre. Let us distinguish two cases:
 \begin{enumerate}[\rm (a)]
  \item If $t(\vec u)$ is not a vertex of $P$. We have $\# \M_P > 1$ so the elements
  $$\ic{\vec u, \vct{P_1 P_n}} \left(\xi^P_1 \xi^P_2 \xi^P_3 \cdots \xi^P_n \right)^\ell \xi^P_1 \ic{\vct{P_1 P_2}, \vec u}$$
 do not have any factor appearing in a relation. The end of the reasoning is the same as in Proof of Lemma \ref{reford1}.
  \item If $t(\vec u)$ is a vertex of $P$ and $\vec u$ is not homotopic to a part of $P$. The elements
  $$\ic{\vec u, \vct{P_1 P_n}} \left(\xi_1^P \xi_2^P \cdots \xi^{P}_{n}\right)^\ell \xi^{P}_1 \ic{\vct{P_1 P_n}, \vec u}$$
  permit to conclude as before. 
 \end{enumerate} 
 Notice that in both case, $e_u \Gs e_u$ is infinitely generated as an algebra.
\end{proof}

Lemmas \ref{reford1} and \ref{reford2} conclude the proof of Theorem \ref{classiford2} (3).

\subsection{Proof of Proposition \ref{altpres0}} \label{proofaltpres0}

We start by giving two technical lemmas:

\begin{lemma} \label{Csigsoc}
  In $\Delta_\sigma$, we have $C_\sigma J = 0$.
 \end{lemma}

 \begin{proof}
  We can suppose that $\Sigma$ has no boundary as the result for $\sigma$ follows from the one for the partial triangulation $\sigma'$ of $\Sigma'$. We can also suppose that $\sigma$ is a triangulation as the result will be induced to any partial triangulation $\tau \subset \sigma$ thanks to Corollary \ref{subtri2}.

  As $C_\sigma$ is in the center of $\Delta_\sigma$, it is enough to prove that for an arrow $\ic{\vec u, \vec v}$ of $Q_\sigma$ which is in $J$, we have $C_\sigma \ic{\vec u, \vec v} = 0$. Using relations defining $\Delta_\sigma$, we get that $C_\sigma \ic{\vec u, \vec v}$ is equal to
   $$\lambda_{t(\vec u)} \ic{-\vec u, -\vec u}^{m_{t(\vec u)}} \ic{\vec u, \vec v} = \lambda_{s(\vec u)} \ic{\vec u, \vec u}^{m_{s(\vec u)}} \ic{\vec u, \vec v} = \lambda_{t(\vec v)} \ic{\vec u, \vec v} \ic{-\vec v, -\vec v}^{m_{t(\vec v)}}.$$
   We will prove that $C_\sigma \ic{\vec u, \vec v} \in C_\sigma J^2$ which is enough by completeness with respect to $J$. Under the hypothesis that $\ic{\vec u, \vec v}$ is an arrow and $\sigma$ is a triangulation, the following cases are impossible:
   \begin{itemize}
    \item $\vec v = -\vec u$ ;
    \item $t(\vec u) = t(\vec v) \neq s(\vec u)$ and $\vec u \neq \vec v$;
    \item $t(\vec u) = t(\vec v) = s(\vec u)$
   \end{itemize}
   so we are in one of the following cases:
  \begin{itemize}
   \item If $s(\vec u)$, $t(\vec u)$ and $t(\vec v)$ are distinct then, using Lemma \ref{caracpoly}, there is a polygon with sides $\vec u$, $-\vec u$, $\vec v$, $-\vec v$ and Lemma \ref{RRp2} gives $\ic{-\vec u, -\vec u} \ic{\vec u, \vec v} = 0$.
   \item If $\vec u = \vec v$. As $\vec v$ follows $\vec u$ around $s(\vec u)$, we have $t(\vec u) \neq s(\vec u)$ and Lemma \ref{RRp} permits to conclude in the digon with sides $\vec u$, $-\vec u$.
   \item If $s(\vec u) = t(\vec u) \neq t(\vec v)$. Let $P$ the triangle enclosed by $\vec u$, $\vec v$, $-\vec v$. Thanks to Lemma \ref{RRp}, if $\# \M_P \geq 1$ 
  then $\ic{\vec u, \vec v} \ic{\vec v, -\vec u} \ic{\vec u, \vec v} = 0$, which permits to conclude. 
  Suppose that $\M_P = \emptyset$. Then $m_{t(\vec v)} > 1$ as $\ic{\vec u, \vec v} \in J$. So $C_\sigma \ic{\vec u, \vec v}  = \io{-\vec u, \vec v} \cdot \ic{\vec v, -\vec u} \ic{\vec u, \vec v} = \lambda_{t(\vec v)} \io{-\vec u, \vec v} \cdot \ic{-\vec v, -\vec v}^{m_{t(\vec v)} - 1}$ is left divisible by $\io{-\vec u, \vec u} \cdot \ic{\vec u, \vec v} \ic{-\vec v, -\vec v} = \io{-\vec u, \vec u} \cdot \io{-\vec u, \vec v} = \io{-\vec u, \vec u}^2 \cdot \ic{\vec u, \vec v}$. If $- \vec u$ encloses more than one puncture then $\io{-\vec u, \vec u}^2 = 0$. If it encloses one puncture $M$ then $\Sigma$ is a sphere with three punctures. Thus, $m_N > 1$ for all $N \in \M$. Thus 
  \begin{align*}
   & \io{-\vec u, \vec u}^2 \cdot \ic{\vec u, \vec v} \\=\,& \lambda_{s(\vec u)}^2 \ic{-\vec u, -\vec u}^{m_{s(\vec u)} - 1} \ic{-\vec u, \vec u}^2 \ic{\vec u, \vec u}^{m_{s(\vec u)} - 1} \ic{\vec u, \vec v} \\
    =\,& \lambda_{s(\vec u)}^2 \lambda_M \ic{-\vec u, -\vec u}^{m_{s(\vec u)} - 1} \ic{-\vec u, \vec u} \cdot \io{\vec u, -\vec u}^{m_M - 1} \ic{\vec u, \vec v} \ic{\vec v, \vec v}^{m_{s(\vec u)} - 1} 
  \end{align*}
  is right divisible by $C_\sigma \ic{\vec u, \vec v} \ic{\vec v, \vec v}^{m_{s(\vec u) - 1}} \in C_\sigma J^2$ and the result follows.
  \item If $s(\vec u) = t(\vec v) \neq t(\vec u)$, the reasoning is similar. \qedhere
  \end{itemize}
 \end{proof}

Let $I$ be the kernel of the canonical projection $k \uQ_\sigma \surj \Delta_\sigma$. We denote $I_0 := (C_{\vec v})_{\vec v} + C_\sigma J \subset I$.

 \begin{lemma} \label{altpresl}
  Let $\vec v$ be an oriented arc of $\sigma$, and $P$ be the minimal $n$-gon containing $\vec v$ with $\vct{P_1 P_2} = \vec v$. Then we have 
  $$(R_{P,1}) + S^P + I_0 + J I + I J = (\RR^\circ_{\vec v}) + S^P + I_0 + J I + I J.$$
 \end{lemma}

 \begin{proof}
  Let $\vec u := \vct{P_2 P_3}$ and $\vec w := \vct{P_n P_1}$. Thus, $-\vec v$ follows immediately $\vec u$ around $P_2$ and $-\vec w$ follows immediately $\vec v$ around $P_1$.

  Let us first suppose that $\sigma$ contains only $u$, $v$ and $w$. If $\# \M_P \geq 2$, the result is immediate. So we suppose that $\# \M_P \leq 1$. If $P$ has one puncture, we write $\M_P = \{M\}$. 
   Let us distinguish several cases:
  \begin{enumerate}[\rm (a)]
   \item If $P$ is a monogon (Case d). In this case $\vec u = \vec v = \vec w$ and $P$ has one puncture. We get easily that, modulo $I_0$,
   $$\ic{\vec u, -\vec v} \ic{\vec v, -\vec w} - R_{P,1} = \left\{\begin{array}{ll}
                                                                                  \lambda_M \ic{\vec v, -\vec v} & \text{if $m_M = 1$;} \\
										  \lambda_M e_v C_\sigma & \text{if $m_M = 2$;} \\
										  0 & \text{if $m_M > 2$.}
                                                                                 \end{array}\right.$$
   \item If $P$ is a digon (Case b). In this case $\vec u = \vec w$ and $P$ has one puncture. Again, we easily find that, modulo $I_0$,
   $$\ic{\vec u, -\vec v} \ic{\vec v, -\vec w} - R_{P,1} = \left\{\begin{array}{ll}
										  \lambda_M e_u C_\sigma & \text{if $m_M = 1$;} \\
										  0 & \text{if $m_M > 1$.}
                                                                                 \end{array}\right.$$
   \item If $P$ is a triangle. The result is immediate if $P$ has no puncture (Case a) so we consider the case where $\M_P = \{M\}$. We have
    $$R_{P,1} = \omega^P_2 \omega^P_1 - \lambda_M C_\sigma (\xi^P_3 \xi^P_1 \xi^P_2)^{m_M - 1} \xi^P_3$$
    and, modulo $I_0$, using Lemma \ref{xiJ}, $R_{P,1} = \omega^P_2 \omega^P_1$ except if $m_M = m_{P_3} = 1$ and $\vec w = -\vec u$. It induces the expected result (Case e or zero relation).
   \item If $n \geq 4$ and $P$ has a puncture or two non-consecutive sides coinciding. In this case, using Proposition \ref{caracpoly}, the assumptions of Lemma \ref{RRp} are satisfied and the result easily follows (zero relation).
   \item If $n = 4$, the sides of $P$ are $\vec w$, $\vec v$, $\vec u$, $-\vec u$ in this order and $\M_P = \emptyset$. The case $\vec w = -\vec v$ is an easy consequence of Proof of Lemma \ref{RRp2} so we suppose that $\vec w \neq -\vec v$ (Case c). We have
   $$\ic{\vec u, -\vec v} \ic{\vec v, -\vec w} - R_{P,1} = \lambda_{P_2} \lambda_{P_3} \ic{-\vec u, -\vec u}^{m_{P_3}-1} \ic{\vec u, -\vec v} \ic{-\vec v, \vec w} \ic{\vec w, \vec w}^{m_{P_2}-1}.$$
   Notice that, modulo $I$, we have
   $$\ic{-\vec u, -\vec u} \ic{\vec u, -\vec v} = \lambda_{P_2} \lambda_{P_1} \ic{\vec u, -\vec v} \ic{-\vec v, -\vec v}^{m_{P_2} - 1} \ic{-\vec v, \vec w} \ic{-\vec w, \vec v} \ic{\vec v, \vec v}^{m_{P_1} - 1} = 0$$
   using the fact that, by (b) in the digon with sides $-\vec v$, $-\vec w$, $\ic{-\vec v, \vec w} \ic{-\vec w, \vec v}$ is a multiple of $C_\sigma$ and $\ic{\vec u, -\vec v} \in J$. Thus, if $m_{P_3} > 1$, $\ic{\vec u, -\vec v} \ic{\vec v, -\vec w} - R_{P,1} \in IJ$ (as $\ic{-\vec v, \vec w} \in J$). The case $m_{P_3} = 1$ is direct.
   \item If $n = 4$, the sides of $P$ are $\vec w$, $\vec v$, $\vec u$, $-\vec w$ in this order and $\M_P = \emptyset$, the reasoning is the same as (e) (Case c for $\vec w$).
   \item If $n = 5$ and the sides of $P$ are $\vec w$, $\vec v$, $\vec u$, $-\vec u$, $-\vec w$ in this order. The case $\M_P \neq \emptyset$ is an easy consequence of Lemma \ref{RRp}. Suppose that $\M_P = \emptyset$ (Case f). Using Proof of Lemma \ref{findminpol}, we have that $t(\vec u)$ and $s(\vec w)$ are both only incident to one edge. Moreover, we have $s(\vec v) = t(\vec w) = s(\vec u) = t(\vec v)$. So we get
   $$\ic{\vec u, -\vec v} \ic{\vec v, -\vec w} - R_{P,1} = \io{-\vec u, -\vec u} \cdot \io{\vec u, -\vec w} \cdot \io{\vec w, \vec w}$$
   and, thanks to (c) in the triangle $\vec v$, $\vec u$, $-\vec u$, modulo $I$, $\ic{-\vec u, -\vec u} \ic{\vec u, -\vec v} = 0$. So, if $m_{P_3} > 1$, we get that $\ic{\vec u, -\vec v} \ic{\vec v, -\vec w} - R_{P,1} \in I J$ (as $\ic{\vec v, -\vec w} \in J$). In the same way, if $m_{P_5} > 1$, we have $\ic{\vec u, -\vec v} \ic{\vec v, -\vec w} - R_{P,1} \in J I$. Finally, if $m_{P_3} = m_{P_5} = 1$, we get
   $$\ic{\vec u, -\vec v} \ic{\vec v, -\vec w} - R_{P,1} = \lambda_{P_3} \lambda_{P_5} \io{\vec u, -\vec w}.$$
  \end{enumerate}
  It is immediate that we exhausted all the possibilities.

  In the general case, we can suppose that $\sigma$ contains exactly the sides of $P$. We do an induction on the number of arcs of $\sigma$. Let $x$ be an arc of $\sigma$ different than $u$, $v$ or $w$ and let $\tau := \sigma \setminus \{x\}$. It is easy to check case by case that $\RR^\circ_{\vec v}$ as the same value considered with respect to $\sigma$ or $\tau$. Let $P'$ be the minimal polygon of $\tau$ containing $\vec v$. It is immediate to check, taking notations of Lemma \ref{redideal}, that $e_u I_{P,x} e_w \subset JI + IJ$ (as relations $R_{P,i}$ generating $I_{P,x}$ start or end at $x$ and neither $x$ and $u$, neither $x$ and $w$ can form a self-folded triangle). Thus, according to Lemma \ref{redideal}, we get $R_{P, 1} - R_{P', 1} \in JI + IJ + S^P$. Therefore
  \begin{align*} (R_{P,1}) + S^P + I_0 + JI + IJ &= (R_{P',1}) + S^P + I_0 + JI + IJ \\ &= (\RR^\circ_{\vec v}) + S^P + I_0 + JI + IJ \end{align*}
  where the last equality comes from the induction hypothesis.
 \end{proof}

 We are now ready to prove Proposition \ref{altpres0}:

\begin{proof}[Proof of Proposition \ref{altpres0}]
  Summing the equalities of Lemma \ref{altpresl} for all $\vec v$, we get
  \begin{equation}(R_{P,\ell})_{P, \ell} + I_0 + JI + IJ = (\RR^\circ_{\vec v})_{\vec v} + (S^P)_{P} + I_0 + JI + IJ \label{ap1} \end{equation}
  where $P$ runs over all the minimal polygons, $\ell$ over all vertices of minimal polygons and $\vec v$ runs over all oriented arcs of $\sigma$. The left member is clearly included in $I$ by definition of $\Delta_\sigma$. It is immediate that $k \uQ_\sigma / I_0$ is a finitely generated $k$-module. Moreover, thanks to Theorem \ref{thmmin}, up to completion with respect to $J$, $I$ is generated by $(C_{\vec v})_{\vec v}$ and $(R_{P, \ell})_{P, \ell}$. Therefore, $I \subset I_0 + (R_{P, \ell})_{P, \ell}$. Finally, the left member of \eqref{ap1} is $I$. As moreover, $(S^P)_{P} \subset (\RR^\circ_{\vec v})_{\vec v}$, we rewrite \eqref{ap1} as follows:
  $$I = (\RR^\circ_{\vec v})_{\vec v} + C_\sigma J + (C_{\vec v})_{\vec v} + JI + IJ$$
  and we deduce by immediate induction that:
  $$I = (\RR^\circ_{\vec v}, C_{\vec v})_{\vec v} + C_\sigma J + \sum_{n' = 0}^n J^{n'} I J^{n-n'}$$
  for any $n > 0$. For $n$ big enough, we have $\sum_{n' = 0}^n J^{n'} I J^{n-n'} \subset C_\sigma J$ so we finally get
  $I = (\RR^\circ_{\vec v}, C_{\vec v})_{\vec v} + C_\sigma J$.
\end{proof}

\subsection{Proof of Theorem \ref{altpres2}} \label{proofaltpres2}

Let $\Pr$ be the set of minimal polygon of $\sigma$. If $\epsilon \in \{0,1\}^\Pr$, for any $\vec v \in \sigma$, we define $\RR^\epsilon_{\vec v} = \RR_{\vec v}$ if $\epsilon(P) = 1$ where $P$ is the minimal polygon with oriented side $\vec v$ and $\RR^\epsilon_{\vec v} = \RR^\circ_{\vec v}$ else. Finally, we put $$\Delta_\sigma^\epsilon := \frac{k \uQ_\sigma}{J C_\sigma + (C_{\vec v}, \RR^\epsilon_{\vec v})_{\vec v \in \sigma}}.$$

We will prove the following proposition:
\begin{proposition} \label{multipres}
 For any $\epsilon \in \{0,1\}^\Pr$, there is an isomorphism $\psi^\epsilon: \Delta^\epsilon_\sigma \to \Delta_\sigma$ satisfying $\psi^\epsilon(e_u) = e_u$ for any $u \in \sigma$ and $\psi^\epsilon(\ic{\vec u, \vec u}^{m_{s(\vec u)}}) = \ic{\vec u, \vec u}^{m_{s(\vec u)}}$ for any $\vec u \in \sigma$.
\end{proposition}
It implies Theorem \ref{altpres2} as $\Delta_\sigma^s = \Delta^{\mathbf{1}}_\sigma$.

\begin{proof}
For $\epsilon \in \{0,1\}^\Pr$, denote $|\epsilon| := \sum_{P \text{ minimal}} \epsilon(P)$. We prove by induction on $(-\# \sigma, |\epsilon|)$ the existence of $\psi^\epsilon$. As the case $|\epsilon| = 0$ is trivial, we suppose that $|\epsilon| > 0$. Let $P$ be a minimal polygon of $\sigma$ such that $\epsilon(P) = 1$ and let $\eta \in \{0,1\}^\Pr$ be defined by $\eta(P) = 0$ and $\eta(P') = \epsilon(P')$ for any $P' \neq P$. By induction hypothesis, the isomorphism $\psi^{\eta}$ exists. If $P$ does not correspond to Cases b with $m_M = 1$, c with $m_{P_3} = 1$, d with $m_M = 2$, e with $m_{P_3} = m_M = 1$ or f with $m_{P_3} = m_{P_5} = 1$ of Figure \ref{polnonz}, the result is immediate as $\Delta_\sigma^{\epsilon} = \Delta_\sigma^{\eta}$. 

If $P$ is as in Case b with $m_M = 1$, we consider the partial triangulation $\tau$ obtained from $\sigma$ by adding an arc linking $P_2$ and $M$. By induction hypothesis, there is an isomorphism $\psi_\tau^{\epsilon}: \Delta^\epsilon_\tau \to \Delta_\tau$ (where, for $\epsilon$, $P$ is replaced by the self-folded quadrilateral). Hence we have an isomorphism $e_\sigma \psi_\tau^{\epsilon} e_\sigma: e_\sigma\Delta^\epsilon_\tau e_\sigma \to e_\sigma \Delta_\tau e_\sigma = \Delta_\sigma$ by Corollary \ref{subtri2}. Moreover, it is easy to observe that $e_\sigma\Delta^\epsilon_\tau e_\sigma = \Delta^\epsilon_\sigma$ so the result follows in this case. Case e with $m_M = m_{P_3} = 1$ is solved in the same way by adding an arc linking $P_1$ to $M$. In Case f with $m_{P_3} = m_{P_5} = 1$ we can suppose by the same argument that there is an arc linking $P_3$ to $P_5$.

In each remaining case, we define $\phi^\circ: k \uQ_\sigma \to \Delta^{\eta}_\sigma$ and we prove that it induces an isomorphism $\phi: \Delta^\epsilon_\sigma \to \Delta^{\eta}_\sigma$ satisfying $\phi^\circ(e_u) = e_u$ for any $u \in \sigma$ and $\phi^\circ(\ic{\vec u, \vec u}^{m_{s(\vec u)}}) = \ic{\vec u, \vec u}^{m_{s(\vec u)}}$ for any $\vec u \in \sigma$. It permits to conclude by putting $\psi^\epsilon = \psi^{\eta} \circ \phi$.

If $P$ is as in Case d with $m_M = 2$ and $2$ invertible in $k$, for $q$ an arrow of $\uQ_\sigma$, we put
$$\phi^\circ(q) := \left\{\begin{array}{ll}
                              q - \lambda_M \io{-\vec u, \vec u}/2 & \text{if $q = \ic{\vec u, -\vec u}$;} \\
			      q & \text{else.}
                             \end{array}\right.$$
Indeed, $\phi^\circ(\ic{\vec u, -\vec u}^2) = 0$. For any $\vec x \in \sigma$ such that $s(\vec x) = P_1$ and $x \neq u$, we get $\io{-\vec u, \vec u} \cdot \ic{-\vec u, \vec x} = \ic{\vec x, \vec u} \cdot \io{-\vec u, \vec u} = 0$ (we use that in the case of a sphere with three punctures, $m_N \geq 2$ for any $N \in \M$). Hence we get $\phi^{\circ}(\io{\vec u, \vec x}) = \io{\vec u, \vec x}$, $\phi^{\circ}(\io{\vec x, -\vec u}) = \io{\vec x, -\vec u}$ and $\phi^\circ(\ic{\vec v, \vec v}^{m_{s(\vec v)}}) = \ic{\vec v, \vec v}^{m_{s(\vec v)}}$ for any $\vec v \in \sigma$. From these observations, $\phi^\circ$ induces a well defined morphism from $\Delta^\epsilon_\sigma$ to $\Delta^{\eta}_\sigma$. It is clearly invertible.

If $P$ is as in Case c with $m_{P_3} = 1$, denote $\tilde \lambda = \lambda_N$ and $\tilde \nu = \nu_\M$ if the polygon with sides $-\vec v$, $-\vec w$ contains a unique puncture $N$, $m_N = 1$, and the minimal polygon $P'$ containing $-\vec w$ is a triangle or $\eta(P') = 1$. In any other case, write $\tilde \lambda = 0$ and $\tilde \nu = 1$. It permits to have the equalities $\ic{-\vec v, \vec w} \ic{-\vec w, \vec v} = \tilde \lambda e_v C_\sigma$ and $\ic{-\vec w, \vec v} \ic{-\vec v, \vec w} = \tilde \lambda e_w C_\sigma$ in $\Delta_\sigma^\eta$. Then write:
$$\phi^\circ(q) := \left\{\begin{array}{ll}
                              \tilde \nu^{-1}(q - \lambda_{P_3} \io{-\vec v, \vec w}) & \text{if $q = \ic{\vec v, -\vec w}$;} \\
			      q & \text{else.}
                             \end{array}\right.$$
We easily get that $\phi^\circ(\RR_{\vec u}) = \phi^\circ(\RR_{\vec v}) = \phi^\circ(\RR_{\vec w}) = \phi^\circ(\RR_{-\vec u}) = 0$. Then, an easy computation gives
$\phi^\circ(\ic{\vec v, \vec v}) = \tilde \nu^{-1}(\ic{\vec v, \vec v} - \delta_{m_{P_2}, 1} \tilde \lambda \lambda_{P_2} \lambda_{P_3} e_v C_\sigma)$ so using $C_\sigma J = 0$, $\phi^\circ(\io{\vec v, -\vec w}) = \tilde \nu^{-1}(\io{\vec v, -\vec w} - \delta_{m_{P_1}, 1} \lambda_{P_1} \lambda_{P_3} \io{-\vec v, \vec w})$. Then, we get $\phi^\circ(\lambda_{P_1} \ic{\vec v, \vec v}^{m_{P_1}}) = \phi^\circ(\io{\vec v, -\vec w}) \ic{-\vec w, \vec v} = e_v C_\sigma$ and in the same way $\phi^\circ(\lambda_{P_1} \ic{-\vec w, -\vec w}^{m_{P_1}}) = e_w C_\sigma$. We now have to check relations $\RR^{\epsilon}_{\vec x}$ when $\vec x \neq \pm \vec u, \vec v, \vec w$. The only case where we do not have trivially $\phi^\circ (\RR^{\epsilon}_{\vec x}) = 0$ is when $f^\circ_{\vec x}$ has $\ic{\vec v, -\vec w}$ as a factor. In this case, an easy case by case analysis proves that there exist $f'$ multiple of a path which does not have $\ic{\vec v, -\vec w}$ as a factor and $\vec y \in \sigma$ with $s(\vec y) = P_1$ such that
\begin{itemize}
 \item[] $f^\circ_{\vec x} = \io{\vec y, -\vec w} f'$, and $t(\vec y) \neq P_1$ or $-\vec w$, $\vec y$, $-\vec y$, $\vec v$ are ordered around $P_1$,
 \item[or] $f^\circ_{\vec x} = f' \io{\vec v, \vec y}$, and $t(\vec y) \neq P_1$ or $-\vec w$, $-\vec y$, $\vec y$, $\vec v$ are ordered around $P_1$.
\end{itemize}
So it is enough to prove that $\phi^\circ(\io{\vec y, -\vec w}) = \io{\vec y, -\vec w}$ in the first case and $\phi^\circ(\io{\vec v, \vec y}) = \io{\vec v, \vec y}$ in the second case. By symmetry, we prove the first one. We use $\phi^\circ(\io{\vec y, -\vec w}) = \ic{\vec y, \vec v} \phi^\circ(\io{\vec v, -\vec w})$ and the easy observation that $\ic{\vec y, \vec v} \ic{-\vec v, \vec w} = \tilde \lambda \io{\vec y, -\vec w}$ so $\ic{\vec y, \vec v} \cdot \io{-\vec v, \vec w} = \delta_{m_{P_2}, 1} \tilde \lambda \lambda_{P_2} \io{\vec y, -\vec w}$ in  $\Delta^\eta_\sigma$. So $\phi^\circ$ induces a morphism $\phi: \Delta^\epsilon_\sigma \to \Delta^\eta_\sigma$, which is clearly invertible.

Finally, we consider Case f where $m_{P_3} = m_{P_5} = 1$ and there is an arc $\vec z$ from $P_3$ to $P_5$. In this case, we put
$$\phi^\circ(q) := \left\{\begin{array}{ll}
                              \tilde \nu_\M^{-1}(q - \lambda_{P_3} \lambda_{P_5} \io{-\vec v, -\vec w}) & \text{if $q = \ic{\vec v, -\vec w}$;} \\
			      q & \text{else.}
                             \end{array}\right.$$
It is immediate that $\phi^\circ(\RR_{\pm \vec u}) = \phi^\circ(\RR_{\vec z}) = \phi^\circ(\RR_{ \pm \vec w}) = \phi^\circ(\RR_{\vec v}) = 0$. We denote $\tilde \lambda := \lambda_N$ if $-\vec v$ encloses a special monogon with puncture $N$ and $\tilde \lambda = 0$ else. An easy computation gives $\phi^\circ(\ic{\vec v, \vec v}) = \nu_\M^{-1}(\ic{\vec v, \vec v} - \tilde \lambda \lambda_{P_3} \lambda_{P_5} e_v C_\sigma)$ so, using $C_\sigma J = 0$, $\phi^\circ(\io{\vec v, -\vec w}) = \nu_\M^{-1}(\io{\vec v, -\vec w} - \delta_{m_{P_1}, 1} \lambda_{P_1} \lambda_{P_3} \lambda_{P_5} \io{-\vec v, -\vec w})$. Using $\ic{-\vec v, \vec v}^2 - \tilde \lambda \ic{-\vec v, \vec v} \in (C_\sigma)$ and $C_\sigma J = 0$, we deduce $\phi^\circ(\lambda_{P_1} \ic{\vec x, \vec x}^{m_{P_1}}) = e_x C_\sigma$ for $\vec x = \pm \vec v, \vec u, -\vec w$. Moreover, $\phi^\circ(\io{\vec u, -\vec w}) = \io{\vec u, -\vec w}$ so $\phi^\circ(\RR^\epsilon_{-\vec z}) = 0$ and $\phi^\circ(\io{\vec v, -\vec v}) = \io{\vec v, -\vec v}$ so all other $\RR^\epsilon_{\vec x}$ are mapped to $0$.
\end{proof}

\subsection{Proof of Theorem \ref{basisDsig}} \label{proofbasisDsig}

We start by proving the theorem for $\Delta_\sigma$:

\begin{proof}[Proof of Theorem \ref{basisDsig} for $\Delta_\sigma$]
 We can check it for $\Delta_{\sigma'}$ as it is immediate that it induces the result for $\Delta_\sigma$. So we suppose that $\Sigma$ has no boundary. We will use the presentation of Proposition \ref{altpres0}.

 For (3), Lemma \ref{Csigsoc} gives that if $x \in (C_\sigma)$ then $J x = x J = 0$. So it is enough to prove (3)' If $J x = x J = 0$ then $x \in (C_\sigma)$.

 We will prove at the same time (1) and (3)'. It is enough to check the result for $e_u \Delta_\sigma e_v$ for any pair of edges $u$, $v$. Thanks to Corollary \ref{subtri2}, we can suppose that $\sigma$ contains only $u$ and $v$. If $u$ and $v$ are disconnected, it is immediate that $e_u \Delta_\sigma e_v = 0$. Let us suppose that $u$ and $v$ are connected. We call $P$ the minimal polygon containing $\vec u$ and $P'$ the minimal polygon containing $-\vec u$. 

 Suppose first that $u = v$. If $u$ has two distinct endpoints, we immediately get, using $\RR^\circ_{\vec u}$ and $\RR^\circ_{-\vec u}$:
  $$e_u \Delta_\sigma e_u = k[\omega_1^P, \omega_2^P]/(\omega_1^P \omega_2^P, \omega_2^P \omega_1^P, \lambda_{P_1} (\omega_1^P)^{m_{P_1}} - \lambda_{P_2} (\omega_2^P)^{m_{P_2}}) $$
 which clearly has a basis consisting of $e_u$, $e_u C_u$, $(\omega_1^P)^\ell$ for $1 \leq \ell \leq m_{P_1} - 1$ and $(\omega_2^P)^\ell$ for $1 \leq \ell \leq m_{P_2} - 1$ and it is immediate that if $a \in e_u \Delta_\sigma e_u$ satisfies $a \omega_1^P = a \omega_2^P = 0$ then $a \in (C_\sigma)$.

 Let us now suppose that $u$ is a loop. Denote $x := \ic{\vec u, -\vec u}$ and $y := \ic{-\vec u, \vec u}$.
 We use the implicit notation $\M_P = \{M\}$ if $\# \M_P = 1$ and $\M_{P'} = \{M'\}$ if $\# \M_{P'} = 1$. Moreover, if $\# \M_P > 1$ we denote $m_M = \infty$ and if $\# \M_{P'} > 1$ we denote $m_{M'} = \infty$. Denote $c :=  \lambda_{s(\vec v)} (xy)^{m_{s(\vec v)}}$ and $c' :=  \lambda_{s(\vec v)} (yx)^{m_{s(\vec v)}}$. We get easily: 
 \begin{itemize}
  \item if $m_M > 1$ and $m_{M'} > 1$ then 
   $$e_u \Delta_\sigma e_u = k\langle x, y \rangle / (c - c', x^2 - \delta_{m_M, 2} \lambda_M c, y^2 - \delta_{m_{M'}, 2} \lambda_M c, cx, cy);$$
  \item if $m_M = 1$ then $m_{M'} = \infty$ and
   $e_u \Delta_\sigma e_u = k\langle x, y \rangle / (c - c', x^2 - \lambda_M x, y^2)$;
  \item if $m_{M'} = 1$ then $m_{M} = \infty$ and
   $e_u \Delta_\sigma e_u = k\langle x, y \rangle / (c - c', x^2 , y^2 - \lambda_{M'} y)$.
 \end{itemize}
 The result follows in this case.

 Suppose now that $u \neq v$. Let us choose orientations of $u$ and $v$ such that $s(\vec u) = s(\vec v)$ and $\vec v$ follows immediately $\vec u$ around $s(\vec u)$ (in the case where $u$ or $v$ is a loop). We distinguish several cases 
 \begin{enumerate}[\rm (a)] 
  \item If $t(\vec v)$, $s(\vec u)$ and $t(\vec u)$ are distinct. We get $\ic{-\vec u, -\vec u} \ic{\vec u, \vec v} = \ic{\vec u, \vec v} \ic{-\vec v, -\vec v} = 0$ so, using the basis of $e_v \Delta_\sigma e_v$, any element of $e_u \Delta_\sigma e_v$ is a linear combination of elements of the form $\ic{\vec u, \vec v} \ic{\vec v, \vec v}^\ell$ for $0 \leq \ell < m_{s(\vec u)}$. Multiplying on the left by $\ic{\vec v, \vec u}$ maps these elements to $\ic{\vec v, \vec v}^{\ell+1}$ which are linearly independent so the statement is true in this case. As $\ic{\vec v, \vec u} \in J$, it also proves that no $x \in e_u \Delta_\sigma e_v \setminus \{0\}$ satisfies $J x = 0$.
  \item If $t(\vec v) = t(\vec u) \neq s(\vec u)$. In this case, $\ic{\vec u, \vec v} \ic{-\vec v, -\vec u}$ and $\ic{-\vec u, -\vec v} \ic{\vec v, \vec u}$ are multiple of $C_\sigma$. Therefore, any element of $e_u \Delta_\sigma e_v$ is a linear combination of $\ic{\vec u, \vec v} \ic{\vec v, \vec v}^\ell$ for $0 \leq \ell < m_{s(\vec u)}$ and $\ic{-\vec u, -\vec v} \ic{-\vec v, -\vec v}^{\ell}$ for $0 \leq \ell < m_{t(\vec u)}$. 
 We denote $\tilde \lambda_M = \lambda_M$ if $\M_P = \{M\}$ and $m_M = 1$ and $\tilde \lambda_M = 0$ else. In the same way, $\tilde \lambda_{M'} = \lambda_{M'}$ if $\M_{P'} = \{M'\}$ and $m_{M'} = 1$ and $\tilde \lambda_{M'} = 0$ else.

  Suppose that
  $$\sum_{\ell = 0}^{m_{s(\vec u)} - 1} \alpha_\ell \ic{\vec u, \vec v} \ic{\vec v, \vec v}^\ell + \sum_{\ell = 0}^{m_{t(\vec u)} - 1} \beta_\ell \ic{-\vec u, -\vec v} \ic{-\vec v, -\vec v}^\ell = 0.$$
    Multiplying on the left by $\ic{\vec v, \vec u}$ and using the structure of $e_v \Delta_\sigma e_v$, we get 
  $$\alpha_\ell = 0 \quad \text{for $\ell < m_{s(\vec u)} - 1$} \quad \text{and} \quad \alpha_{m_{s(\vec u)} - 1} + \beta_0 \tilde \lambda_M \lambda_{s(\vec u)} = 0$$
  and multiplying by $\ic{-\vec v, -\vec u}$, we get
  $$\beta_\ell = 0 \quad \text{for $\ell < m_{t(\vec u)} - 1$} \quad \text{and} \quad \alpha_0 \tilde \lambda_{M'} \lambda_{t(\vec u)} + \beta_{m_{t(\vec u)} - 1} = 0.$$
  If $m_{s(\vec u)} > 1$ or $m_{t(\vec u)} > 1$, we get that $\alpha_\ell = 0$ and $\beta_\ell = 0$ for any $\ell$ so these elements are linearly independent. If $m_{s(\vec u)} = m_{t(\vec u)} = 1$, we deduce $\nu_\M \alpha_0 = \nu_\M \beta_0 = 0$ so the conclusion follows as $\nu_\M$ is invertible.
  \item If $t(\vec u) = s(\vec u) \neq t(\vec v)$. Using $\RR^\circ_{-\vec v}$, $\ic{\vec u, \vec v} \ic{-\vec v, -\vec v} \in e_u \Delta_\sigma \ic{\vec u, \vec v}$ so every element of $e_u \Delta_\sigma e_v$ has the form $\omega \ic{\vec u, \vec v}$ for $\omega \in e_u \Delta_\sigma e_u$. We know that $\omega$ is a linear combination of $\epsilon \ic{\vec u, \vec u}^\ell \epsilon'$ for $\epsilon \in \{e_u, \ic{-\vec u, \vec u}\}$, $\epsilon' \in \{e_u, \ic{\vec u, -\vec u}\}$ and  $0 \leq \ell < m_{s(\vec u)}$. Using $\RR^\circ_{\vec u}$ and $\RR^\circ_{-\vec v}$, we get
  \begin{align*}\ic{\vec u, -\vec u} \ic{\vec u, \vec v}  &= \ic{\vec u, \vec v} \ic{\vec v, -\vec u} \ic{\vec u, \vec v} \\&= \ic{\vec u, \vec v} \left\{\begin{array}{ll}
                                                         \lambda_{t(\vec v)} \ic{-\vec v, -\vec v}^{m_{t(\vec v)}-1} & \text{if $\M_P = \emptyset$}; \\ \delta_{m_M, 1} \delta_{m_{t(\vec v)}, 1} \lambda_M \lambda_{t(\vec v)} e_v C_\sigma & \text{if $\M_P \neq \emptyset$}.
                                                        \end{array}\right. \\
		 &= \left\{\begin{array}{ll}
                                                         \lambda_{t(\vec v)} \left( \lambda_{s(\vec u)} \ic{-\vec u, \vec u} \ic{\vec u, \vec u}^{m_{s(\vec u)} - 1} \right)^{m_{t(\vec v)}-1} \ic{\vec u, \vec v} & \text{if $\M_P = \emptyset$}; \\ 0 & \text{if $\M_P \neq \emptyset$}.
                                                        \end{array}\right. \\
                 &= \left\{\begin{array}{ll}
                                                         \lambda_{t(\vec v)} \ic{\vec u, \vec v} & \text{if $\M_P = \emptyset$, $m_{t(\vec v)} = 1$}; \\ \lambda_{t(\vec u)} \lambda_{s(\vec u)} \ic{-\vec u, \vec u} \ic{\vec u, \vec u}^{m_{s(\vec u)} - 1}  \ic{\vec u, \vec v} & \text{if $\M_P = \emptyset$, $m_{t(\vec v)} = 2$}; \\ 0 & \text{else}.
                                                        \end{array}\right.
  \end{align*}
  where we used the relations computed before in $e_u \Delta_\sigma e_u$. Using this identity, we deduce that $\omega \ic{\vec u, \vec v}$ is a linear combination of $\epsilon \ic{\vec u, \vec u}^\ell \ic{\vec u, \vec v}$ for $\epsilon \in \{e_u, \ic{-\vec u, \vec u}\}$ and $0 \leq \ell < m_{s(\vec u)}$. Multiplying by $\ic{\vec v, -\vec u}$ on the right, these elements are linearly independent. Notice that, using this argument, we see that the only possibility, up to rescaling, to have $x \in e_u \Delta_\sigma e_v$ satisfying $x J = 0$ is to take $x = \ic{-\vec u, -\vec u}^{m_{s(\vec u)} -1} \ic{-\vec u, \vec u} \ic{\vec u, \vec v}$ if $\M_P = \emptyset$ and $m_{t(\vec v)} = 1$. In this case, according to the computation before, we get 
   $$\lambda_{t(\vec v)} x = \ic{-\vec u, -\vec u}^{m_{s(\vec u)}-1} \ic{-\vec u, \vec u} \ic{\vec u, -\vec u} \ic{\vec u, \vec v} = \lambda_{s(\vec u)}^{-1}C_\sigma \ic{\vec u, \vec v} \in (C_\sigma).$$
  \item If $t(\vec v) = s(\vec u) \neq t(\vec u)$. This is similar to the previous case.
  \item If $t(\vec v) = s(\vec u) = t(\vec u)$. Using the structure of $e_v \Delta_\sigma e_v$, every element of $e_u \Delta_\sigma e_v$ can be written as a linear combination of $\omega \ic{\vec u, \vec v} \eta$ where $\omega \in e_u \Delta_\sigma e_u$ and $\eta \in \{e_v, \ic{\vec v, -\vec v}\}$. Using the structure of $e_u \Delta_\sigma e_u$, $\omega$ is a linear combination of $\epsilon \ic{\vec u, \vec u}^\ell \epsilon'$ for $\epsilon \in \{e_u, \ic{-\vec u, \vec u}\}$, $\epsilon' \in \{e_u, \ic{\vec u, -\vec u}\}$ and $0 \leq \ell < m_{s(\vec u)}$. Using $\RR^\circ_{\vec u}$, we have
   $$
   \ic{-\vec v, -\vec u} \ic{\vec u, \vec v} 
     = \left\{ \begin{array}{ll}
                                      \lambda_{M} e_v C_\sigma & \text{if $\M_P = \{M\}$ and $m_M = 1$;} \\
					0 & \text{else,}
                                     \end{array}\right.
  $$
 and we deduce $\ic{\vec u, -\vec u} \ic{\vec u, \vec v} = \ic{\vec u, -\vec v} \ic{-\vec v, -\vec u} \ic{\vec u, \vec v} = 0$. As a consequence, any element of $e_u \Delta_\sigma e_v$ is a linear combination of $\epsilon \ic{\vec u, \vec u}^\ell \ic{\vec u, \vec v} \eta$ for $\epsilon \in \{e_u, \ic{-\vec u, \vec u}\}$, $\eta \in \{e_v, \ic{\vec v, -\vec v}\}$ and $0 \leq \ell < m_{s(\vec u)}$. It remains to prove that these elements are linearly independent. Suppose that $\sum_{\epsilon, \ell, \eta} \mu_{\epsilon, \ell, \eta} \epsilon \ic{\vec u, \vec u}^\ell \ic{\vec u, \vec v} \eta = 0$. Let us denote $x := \ic{\vec u, -\vec u}$, $y := \ic{-\vec u, \vec u}$ and $c := \lambda_{s(\vec u)}(xy)^{m_s(\vec u)} = \lambda_{s(\vec u)} (yx)^{m_s(\vec u)}$ in such a way that computations rules fit with the description of $e_u \Delta_\sigma e_u$ before.

 Multiplying by $\ic{\vec v, -\vec u}$ the previous equality on the right, we get
 \begin{align*}0 &= \sum_{\epsilon, \ell} \left( \mu_{\epsilon, \ell, e_v} \epsilon \ic{\vec u, \vec u}^\ell \ic{\vec u, -\vec u} + \mu_{\epsilon, \ell, \ic{\vec v, -\vec v}} \epsilon \ic{\vec u, \vec u}^\ell \ic{\vec u, -\vec v} \ic{\vec v, -\vec u}  \right) \\
   &= \sum_{\epsilon, \ell} \left( \mu_{\epsilon, \ell, e_v} \epsilon (xy)^\ell x + \mu_{\epsilon, \ell, \ic{\vec v, -\vec v}} \epsilon (xy)^\ell \alpha \right)
  \end{align*}
 where $\alpha := \ic{\vec u, -\vec v} \ic{\vec v, -\vec u}$ is equal to $\lambda_{N'} x$ if $\vec v$ encloses a special monogon with special puncture $N'$ and $\alpha = 0$ else. Using the structure of $e_u \Delta_\sigma e_u$, we get
 $$\mu_{\epsilon, \ell, e_v} = \left\{ \begin{array}{ll} -\lambda_{N'} \mu_{\epsilon, \ell, \ic{\vec v, -\vec v}} & \text{if $\vec v$ is special;} \\ 0 & \text{else.} \end{array} \right.$$

 Multiplying the equality by $\ic{-\vec v, -\vec u}$, an analogous reasoning gives
 \begin{itemize}
  \item $\mu_{\epsilon, \ell, \ic{\vec v, -\vec v}} = 0$ if $\ell < m_{s(\vec u)} - 1$ or $m_M > 1$ or $\epsilon = e_u$;
  \item $\mu_{y, m_{s(\vec u)} - 1, \ic{\vec v, -\vec v}} = -\lambda_M \lambda_{s(\vec u)}  (\mu_{e_u, 0, e_v} + \lambda_{N} \mu_{y, 0, e_v} )$ if $m_M = 1$ and $-\vec u$ encloses a special monogon with puncture $N$;
  \item $\mu_{y, m_{s(\vec u)} - 1, \ic{\vec v, -\vec v}} = -\lambda_M \lambda_{s(\vec u)}  \mu_{e_u, 0, e_v}$ if $m_M = 1$ and $-\vec u$ does not enclose a special monogon.
 \end{itemize}

 So we get $\mu_{\epsilon, \ell, \eta} = 0$ if $\epsilon = e_u$ or $m_M > 1$ or $\ell < m_{s(\vec u)} - 1$ or $-\vec u$ is not special or $\vec v$ is not special. If $m_M = 1$, $-\vec u$ is special and $\vec v$ is special, we have $\mu_{y, m_{s(\vec u)} - 1, e_v} = -\lambda_{N'} \mu_{y, m_{s(\vec u)} - 1, \ic{\vec v, \vec v}} = \lambda_\M \mu_{y, 0, e_v}$
 which permits to conclude in any case as $\nu_\M$ is invertible. Notice also that as $\ic{\vec v, -\vec u}$ and $\ic{-\vec v, -\vec u}$ are in $J$, we get that $x J = 0$ is impossible for a non-zero $x \in e_u \Delta_\sigma e_v$.
 \end{enumerate}

 (2) is an easy consequence of (1) and Proposition \ref{altpres0} which states that the ideal of relations of $\Delta_\sigma$ is generated by linear combinations of at most two paths.
\end{proof}

Then we deduce this generalized version of Theorem \ref{basisDsig} for $\Delta^s_\sigma$:
\begin{proposition}
 For any $\epsilon \in \{0,1\}^\Pr$, $\Br$ is mapped to a basis of $\Delta_\sigma^\epsilon$.
\end{proposition}

\begin{proof}
 We take the same notations as in Proof of Proposition \ref{multipres} and we follow the same inductive argument. By Theorem \ref{basisDsig} for $\Delta_\sigma$, $\Br$ is mapped to a basis of $\Delta_\sigma^{\mathbf{0}}$. Let $\epsilon \in \{0, 1\}^\Pr$ such that $|\epsilon| > 0$ and construct $\eta$ as in Proof of Proposition \ref{multipres}. Then, it is immediate looking at the definition of $\phi^\circ$ in each case that if $\Br$ is mapped to a basis of $\Delta^\eta_\sigma$ then $\Br$ is also mapped to a basis of $\Delta^\epsilon_\sigma$. 

 (2) is then an easy consequence of (1) as before.
\end{proof}

\subsection{Proof of Theorem \ref{symm}} \label{proofsymm}

Finally, we will prove Theorem \ref{symm}. We suppose now that $\sigma$ has no arc incident to the boundary. We start with two lemmas.

\begin{lemma} \label{tech2}
 If $\vec u, \vec v \in \sigma$ and $s(\vec u) = s(\vec v)$, we have:
 \begin{itemize}
  \item if $-\vec v$ encloses a special monogon with special puncture $M$ and $t(\vec u) \neq M$, $\ic{\vec u, \vec v} \ic{-\vec v, \vec u} =  \lambda_M \ic{\vec u, \vec u}$;
  \item if $\vec u$, $-\vec u$, $-\vec v$ form a self-folded triangle without puncture, then we have $\ic{\vec u, \vec v} \ic{-\vec v, \vec u} = \io{-\vec u, -\vec u}$;
  \item if $\vec v$, $-\vec v$ and $\vec u$ are (strictly) ordered around $s(\vec u)$, $-\vec v$ encloses two punctures $M$ and $N$, and $m_M = m_N = 1$, we have $\ic{\vec u, \vec v} \ic{-\vec v, \vec u} = \lambda_M \lambda_N e_u C_\sigma$;
  \item in any other case, $\ic{\vec u, \vec v} \ic{-\vec v, \vec u} = 0$.
 \end{itemize}
 We also have:
 \begin{itemize}
  \item if $\vec u$, $-\vec v$ form a digon with one puncture $M$ and $m_M = 1$, then we have $\ic{\vec u, \vec v} \ic{-\vec v, -\vec u} = \lambda_M e_u C_\sigma$;
  \item if $-\vec u$, $-\vec v$ form a digon with one puncture $M$ and $m_M = 1$ and $\vec u$ encloses a special monogon with special puncture $N$, then we have $\ic{\vec u, \vec v} \ic{-\vec v, -\vec u} = \lambda_M \lambda_N^2 e_u C_\sigma$;
  \item if $-\vec v$ encloses a special monogon with special puncture $M$ which is also enclosed by $-\vec u$ (not necessarily special), then we have $\ic{\vec u, \vec v} \ic{-\vec v, -\vec u} = \lambda_M \ic{\vec u, \vec u} \ic{\vec u, -\vec u}$; 
  \item if $-\vec v$ encloses a special monogon with special puncture $M$ which is also enclosed by $\vec u$ (not necessarily special), then $\ic{\vec u, \vec v} \ic{-\vec v, -\vec u} = \lambda_M \ic{\vec u, -\vec u}$; 
  \item if $-\vec v = \vec u$ encloses a unique puncture $M$ with $m_M = 2$ then we have $\ic{\vec u, \vec v} \ic{-\vec v, -\vec u} = \lambda_M e_u C_\sigma$;
  \item in any other case, $\ic{\vec u, \vec v} \ic{-\vec v, -\vec u} = 0$.
 \end{itemize}
\end{lemma}

\begin{proof}
 We can suppose that $\sigma$ contains only $u$ and $v$ thanks to Corollary \ref{subtri2}. We start by $\ic{\vec u, \vec v} \ic{-\vec v, \vec u}$. For this to be non-zero, we need $t(\vec v) = s(\vec u) = s(\vec v)$ so $\vec v$ is a loop. Then studying the different possible orders of $\vec v$, $-\vec v$, $\vec u$ (and maybe $-\vec u$) around $s(\vec v)$ gives easily the result using Proposition \ref{altpres0}. The reasoning is analogous for $\ic{\vec u, \vec v} \ic{-\vec v, -\vec u}$. The case $u = v$ is dealt in the same way.
\end{proof}

\begin{definition} \label{trace}
 We say that $\vec u \in \sigma$ is \emph{$2$-special} if it encloses a monogon containing two punctures $M_{\vec u}$ and $N_{\vec u}$ with $m_{M_{\vec u}} = m_{N_{\vec u}} = 1$. We define the subset $\Er$ of $\Delta_\sigma$ by 
 $$\Er := \{e_u C_\sigma \mid u \in \sigma\} \cup \{\lambda_{M_{\vec u}}^{-1} \lambda_{N_{\vec u}}^{-1} \ic{\vec u, -\vec u} \mid \text{$\vec u$ is $2$-special}\}.$$
 
 We define the linear map $\Er^*: \Delta_\sigma \to k$ such that $\Er^*(x) = 1$ if $x \in \Er$ and $\Er^*(x) = 0$ if $x$ is an element of the basis without multiple in $\Er$ (it is possible as $\Er$ consists of multiples of elements of $\Br$).
\end{definition}

\begin{lemma} \label{commC} 
 Suppose that $\sigma$ has no arc incident to a boundary.  Let $a$ and $b$ be two elements of $\Br$ such that $a \in e_u \Delta_\sigma e_v$ and $b \in e_v \Delta_\sigma e_u$ for some $u$ and $v$. Let $\mu \in k$. Then $\mu ab \in \Er$ if and only if $\mu ba \in \Er$.
\end{lemma}

\begin{proof}
 According to Corollary \ref{subtri2}, we can suppose that $u$ and $v$ are the only arcs of $\Delta_\sigma$. Suppose that $u \neq v$.

 According to Theorem \ref{basisDsig}, we can suppose that $s(\vec u) = s(\vec v)$ and $a = \ic{\vec u, \vec u}^\ell \ic{\vec u, \vec v}$. Suppose first that $b = \ic{\vec v, \vec v}^{\ell'} \ic{\vec v, \vec w}$ for $\vec w = \pm \vec u$. Then we have 
 $\mu a b = \mu \ic{\vec u, \vec u}^{\ell+\ell'} \ic{\vec u, \vec v} \ic{\vec v, \vec w}$ and, according to Theorem \ref{basisDsig}, $\mu a b = e_u C_\sigma$ can happen in three situations:
 \begin{enumerate}[\rm (a)]
  \item $\mu = \lambda_{s(\vec u)}$, $\ell + \ell' = m_{s(\vec u)} - 1$ and $\vec w = \vec u$;
  \item $\mu = \lambda_{s(\vec u)} \lambda_M^{-1}$, $\ell + \ell' = m_{s(\vec u)} - 1$, $\vec w = -\vec u$ and $\vec u$ encloses a special monogon with special puncture  $M \neq t(\vec v)$;
  \item $\mu = \lambda_{s(\vec u)} \lambda_M^{-1}$, $\ell + \ell' = m_{s(\vec u)}$, $\vec w = -\vec u$ and $\vec u$ encloses a special monogon with special puncture $M = t(\vec v)$.
 \end{enumerate}
 So we clearly have $\mu b a = e_v C_\sigma$.

 Moreover, $\mu a b = \lambda_{M_{\vec u'}}^{-1} \lambda_{N_{\vec u'}}^{-1} \ic{\vec u', -\vec u'}$ for a $2$-special oriented arc $\vec u'$ can happen only if $\vec u' = \vec u$, $t(\vec v) \in \{M_{\vec u}, N_{\vec u}\}$, $\vec w = -\vec u$, $\ell = \ell' = 0$ and $\mu = \lambda_{M_{\vec u}}^{-1} \lambda_{N_{\vec u}}^{-1}$. Then we have $\mu b a = e_v C_\sigma$.

 Suppose now that $b = \ic{-\vec v, -\vec v}^{\ell'} \ic{-\vec v, \vec w}$ with $\vec w = \pm \vec u$. We have $\mu a b = \mu \ic{\vec u, \vec u}^\ell \ic{\vec u, \vec v} \ic{-\vec v, \vec w} \ic{\vec w, \vec w}^{\ell'}$ so using Lemma \ref{tech2}, the only cases where $\mu a b =e_u C_\sigma$ are the following:
 \begin{enumerate}[\rm (a)]
  \item $-\vec v$ encloses a special monogon with special puncture $M \neq t(\vec u)$, $\vec w = \vec u$, $\ell + \ell' = m_{s(\vec u)} - 1$ and $\mu = \lambda_{s(\vec u)} \lambda_{M}^{-1}$; 
  \item $-\vec v$ encloses a special monogon with special puncture $M = t(\vec u)$, $\vec w = \vec u$, $\ell + \ell' = m_{s(\vec u)}$ and $\mu = \lambda_{s(\vec u)} \lambda_{M}^{-1}$;
  \item $\vec v$, $-\vec v$, $\vec u$ are strictly ordered around $s(\vec u)$, $-\vec v$ is $2$-special, $\vec w = \vec u$, $\ell = \ell' = 0$ and $\mu = \lambda_{M_{-\vec v}}^{-1} \lambda_{N_{-\vec v}}^{-1}$;
  \item $\vec u$, $-\vec v$ form a digon with one puncture $M$ such that $m_M = 1$, $\vec w = -\vec u$, $\ell = \ell' = 0$, and $\mu = \lambda_M^{-1}$;  
  \item $-\vec v$ is $2$-special, $\vec u$ is a special monogon with special puncture $N_{-\vec v}$, $\vec w = -\vec u$, $\ell = \ell' = 0$ and $\mu = \lambda_{M_{-\vec v}}^{-1} \lambda_{N_{-\vec v}}^{-2}$;
  \item $-\vec v$ and $\vec u$ are special monogons with special punctures $M$ and $N$, $\vec w = -\vec u$, $\ell + \ell'  = m_{s(\vec u)} - 1$ and $\mu = \lambda_{s(\vec u)} \lambda_M^{-1} \lambda_N^{-1}$.
 \end{enumerate}
 In cases (a), (b), (d) and (f), we have $\mu b a = e_v C_\sigma$. In cases (c) and (e), we have $\mu b a = \lambda_{M_{\vec v}}^{-1} \lambda_{N_{\vec v}}^{-1} \ic{-\vec v, \vec v}$ so $\mu b a \in \Er$.

 Finally, the only case where $\mu a b = \lambda_{M_{\vec u'}}^{-1} \lambda_{N_{\vec u'}}^{-1} \ic{\vec u', -\vec u'}$ for a $2$-special oriented arc $\vec u'$ happens, again thanks to Lemma \ref{tech2}, when $\vec u = \vec u'$, $-\vec v$ encloses $N_{\vec u}$, and $\mu = \lambda_{M_{\vec u}}^{-1} \lambda_{N_{\vec u}}^{-2}$. In this case, we have $\mu b a = e_v C_\sigma$.

 The case $u = v$ is analogous (and simpler).
\end{proof}

\begin{definition}
 For a $k$-algebra $A$, we say that a $k$-linear map $t: A \to k$ is a \emph{trace} if it satisfy $t(ab) = t(ba)$ for any $a, b \in A$. We say that it is a \emph{non-degenerate trace} if moreover the induced morphism of $A$-bimodule
  $$ A \to \Hom_k(A, k), \quad  a \mapsto t(a-) $$
  is an isomorphism.
\end{definition}

We recall the following classical observation:
\begin{lemma} \label{zerodivdet}
 Let $M$ be a $m \times \ell$ matrix with coefficient in $k$. If there exists $\mu \in k \setminus \{0\}$ such that $\mu \delta = 0$ for all maximal minors $\delta$ of $M$ then $M$ has a non-zero kernel.
\end{lemma}

\begin{proof}
  Up to adding rows of zeros, we can suppose that $m \geq \ell$. If $\ell = 1$, the result is obvious. Let us suppose that $\ell > 1$. If $\mu \delta' = 0$ for all $(\ell-1) \times (\ell-1)$-minors $\delta'$ of $M$ then the result is true by induction. So, without loss of generality, we can suppose that the upper left $(\ell-1) \times (\ell-1)$ minor $\delta'$ of $M$ satisfy $\mu \delta' \neq 0$. For $i = 1 \dots \ell$, let $\delta_i := (-1)^i \mu \delta'_i$ where $\delta'_i$ is the minor of $M$ with rows $1, 2, \dots, \ell-1$ and columns $1, 2, \dots, i-1,i+1, \dots, \ell$. Then $\delta \neq 0$ is in the kernel of $M$.
\end{proof}

We deduce the following equivalent characterizations of non-degenerate traces:

\begin{lemma} \label{caracnondeg}
 For a $k$-algebra $A$ which is free of finite rank over $k$ and a trace $t: A \to k$, the following are equivalent:
 \begin{enumerate}[(i)]
  \item $t$ is non-degenerate;
  \item for any $\mu \in k$ and $a \in A$,
  $(\forall x \in A, t(ax) \in \mu k) \Leftrightarrow a \in \mu A$.
 \end{enumerate}
\end{lemma}

\begin{proof} 
 Let $\rho: A \to \Hom_k(A, k)$ be defined by $\rho(a) = t(a-)$.

 $(i) \Rightarrow (ii)$ Suppose that $t$ is non-degenerate and let $\mu \in k$ and $a \in A$ such that $t(ax) \in \mu A$ for any $x \in A$. Then, as $A$ is free, there exists $f \in \Hom_k(A, k)$ such that $\rho(a) = \mu f$. By surjectivity of $\rho$, there exists $b \in A$ such that $f = \rho(b)$. By injectivity of $\rho$, we have $\mu b = a$.

 $(ii) \Rightarrow (i)$ First of all, taking $\mu = 0$ in (ii), we get that $\rho$ is injective. As $A$ and $\Hom_k(A, k)$ have the same rank over $k$ and $\rho$ is injective, $\mu := \det(\rho) \in k$ is not a zero divisor by Lemma \ref{zerodivdet}. Let $f \in \Hom_k(A, k)$ and $b := \rho' f$ where $\rho'$ is the adjugate of $\rho$. Thus $\rho(b) = \mu f$. Using $(ii)$, we then get $b = \mu a$ for some $a \in A$. Thus $\mu (f - \rho(a)) = 0$ so $f = \rho(a)$.
\end{proof}

\begin{lemma} \label{commC2} 
 The linear map $\Er^*$ is a non-degenerate trace. 
\end{lemma}

\begin{proof}
 As the product of any two elements of the basis is a multiple of an element of the basis thanks to Theorem \ref{basisDsig} and using Lemma \ref{commC}, we get immediately that $\Er^*(xy) = \Er^*(yx)$ for any $x, y \in \Delta_\sigma$ so $\Er^*$ is a trace. 

 Let us prove that $\Er^*$ is non-degenerate. We use the characterization of Lemma \ref{caracnondeg}. Let $a \in \Delta_\sigma$ and $\mu \in k$ satisfying $\Er^* (ax) \in \mu k$ for any $x \in \Delta_\sigma$. If $\mu$ is invertible, it is immediate that $a \in \mu \Delta_\sigma$. 

 Suppose that $\mu = 0$. If $a \neq 0$, then, as $J^n = 0$ for $n$ big enough, there exists $x, x' \in \Delta_\sigma$ such that $xax' \neq 0$ and $xax' J = Jxax' = 0$. Thanks to Theorem \ref{basisDsig} (2), $xax' \in (C_\sigma)$. Then, up to multiplying by an idempotent, $xax'$ is a non-zero multiple of $e_u C_\sigma$ for $u \in \sigma$. So $\Er^*(xax') \neq 0$. It is a contradiction.

 If $\mu$ is not invertible, notice that $\Delta'_{\sigma} := \Delta_\sigma / (\mu \Delta_\sigma)$ is the algebra of $\sigma$ defined over the ring $k' := k/(\mu k)$ and through this identification, $\Er^*$ is mapped to the corresponding trace over $\Delta'_\sigma$. Thus, applying the case $\mu = 0$ implies immediately that $a \in \mu \Delta_\sigma$.
\end{proof}

\begin{proof}[Proof of Theorem \ref{symm}]
 It is an immediate consequence of Lemma \ref{commC2}.
\end{proof}

Finally, we prove Proposition \ref{bimodres}:

\begin{proof}[Proof of Proposition \ref{bimodres}]
 First of all, it is classical that the last four terms form an exact sequence, where $\alpha$ is the multiplication in $\Delta_\sigma$, $\beta(a e_{\vec u} \tens e_{\vec u^+} b) = a \tens \ic{\vec u, \vec u^+} b - a \ic{\vec u, \vec u^+} \tens b$ and $\gamma$ is induced by the generating relations $\RR_{\vec v}$ (notice that the completion with respect to $J$ does not matter here). More precisely, if we define the morphism of $k Q_\sigma$-bimodules $\phi: k Q_\sigma \to \bigoplus_{\vec u \in \sigma} \Delta_\sigma e_{\vec u} \otimes e_{\vec u^+} \Delta_\sigma$ by $\phi(q_1 q_2 \cdots q_n) = \sum_{i= 1}^n q_1 \cdots q_{i-1} \tens q_{i+1} \cdots q_n$ for any path $q_1 \cdots q_n$, we let, for an oriented triangle $\vec w$, $\vec v$, $\vec u$ in $\sigma$, $\gamma(e_{-\vec u} \tens e_{\vec w}) = \phi(\RR_{\vec v})$. To conclude, it is enough to prove that $\gamma = \gamma^*$. Equivalently, we need to prove that
 $$(\Er^* \tens \Er^*) (b \phi(\RR_{\vec v}) a) = (\Er^* \tens \Er^*) (a \phi(\RR_{\vec v'}) b)$$ for two triangles $\vec w$, $\vec v$, $\vec u$ and $\vec w'$, $\vec v'$, $\vec u'$ in $\sigma$, $a \in e_{\vec w} \Br e_{\vec u'}$ and $b \in e_{\vec w'} \Br e_{\vec u}$. Recall that, as $\sigma$ is a triangulation, we have $\RR_{\vec v} = \ic{\vec u, -\vec v} \ic{\vec v, -\vec w} - \io{-\vec u, \vec w}$.
 \begin{itemize}
 \item Let us consider first the term 
 $$(\Er^* \tens \Er^*) (b \phi(\ic{\vec u, -\vec v} \ic{\vec v, -\vec w}) a) = (\Er^* \tens \Er^*) (b \ic{\vec u, -\vec v} \tens  a + b \tens \ic{\vec v, -\vec w} a ).$$
 The first part $(\Er^* \tens \Er^*) (b \ic{\vec u, -\vec v} \tens  a)$ is non-zero if and only if $a , b \ic{\vec u, -\vec v} \in k \Er$, which implies that $\vec w = \pm \vec u'$ and $\vec w' = \pm \vec v$. As there is at most one oriented triangle, that $v = w'$ and $w = u'$ are sides of, in the same order, we necessarily have $\vec u' = \vec w$, $\vec v' = \vec u$ and $\vec w' = \vec v$. So 
  $$(\Er^* \tens \Er^*) (b \ic{\vec u, -\vec v} \tens  a) = (\Er^* \tens \Er^*) (a \tens b \ic{\vec v', -\vec w'}) = (\Er^* \tens \Er^*) (a \tens \ic{\vec v', -\vec w'}b)$$
  (obviously this still holds when $(\Er^* \tens \Er^*) (b \ic{\vec u, -\vec v} \tens  a) = 0$ by symmetry of the reasoning). Moreover, in the same way, 
  $$(\Er^* \tens \Er^*) (b \tens \ic{\vec v, -\vec w} a) = (\Er^* \tens \Er^*) (a \ic{\vec u', -\vec v'} \tens  b)$$
  so $(\Er^* \tens \Er^*) (b \phi(\ic{\vec u, -\vec v} \ic{\vec v, -\vec w}) a) = (\Er^* \tens \Er^*) (a \phi(\ic{\vec u', -\vec v'} \ic{\vec v', -\vec w'}) b)$. 
 \item Let us now look at the second term $(\Er^* \tens \Er^*) (b \phi(\io{-\vec u, \vec w}) a)$. It is immediate that, if it is non-zero, then there is an oriented arc $\vec x$ such that $s(\vec x) = s(\vec w)$, $\vec x = \pm \vec w'$ and $\vec x^+ = \pm \vec u'$. As $\sigma$ is a triangulation, the only possibility is that $\vec x = \vec w'$ and $\vec x^+ = -\vec u'$. Continuing this argument, we then find, in this case
  \begin{align*} 
   (\Er^* \tens \Er^*) (b \phi(\io{-\vec u, \vec w}) a) &= \sum\nolimits_{\omega_1 \ic{\vec w', -\vec u'} \omega_2 = \io{-\vec u, \vec w}} \Er^*(b \omega_1) \Er^*(\omega_2 a) \\
       &= \sum\nolimits_{\omega_2 \ic{\vec w, -\vec u} \omega_1 = \io{-\vec u', \vec w'}} \Er^*(a \omega_2) \Er^*(\omega_1 b) \\
       &= (\Er^* \tens \Er^*) (a \phi(\io{-\vec u', \vec w'}) b).
  \end{align*}
 \end{itemize}
 Combining both identities, we finally get, as necessary, 
 \begin{align*} 
  & (\Er^* \tens \Er^*) (b \phi(\RR_{\vec v}) a) = (\Er^* \tens \Er^*) (a \phi(\RR_{\vec v'}) b). \qedhere
 \end{align*}
\end{proof}

\subsection{Proof of Proposition \ref{endo}} \label{proofendo}

Along this proof, we denote $\lambda^* = \mu_u(\lambda)$. Moreover, every algebra of partial triangulation is computed with respect to $\lambda$ if it is not specified.

We provide the following Lemma to simplify the situation:

\begin{lemma} \label{reducidemisom}
 Suppose that $\sigma \subset \sigma'$ for a partial triangulation $\sigma'$ satisfying that $\vec u^+$ and $(-\vec u)^+$ have the same value computed with respect to $\sigma$ or $\sigma'$. Then $\mu_u(\sigma) \subset \mu_u(\sigma')$ in an obvious way and, if $T$ (respectively $T'$) denotes the tilting complex computed for $\sigma$ (respectively $\sigma'$) as before, we have 
  $$\End_{\Kb(\proj \Delta_{\sigma})}(T) \cong e_{\sigma^*} \End_{\Kb(\proj \Delta_{\sigma'})}(T') e_{\sigma^*}$$
 where $e_{\sigma^*}$ is the idempotent of $\End_{\Kb(\proj \Delta_{\sigma'})}(T')$ corresponding to the set of arcs in $\sigma^*$ ($e_{u^*}$ corresponds to the summand $P'^*_u$ of $T'$).
\end{lemma}

\begin{proof}
 First of all, it is clear under these assumptions that $T \cong e_{\sigma^*} T' e_{\sigma}$ as $e_{\sigma^*} \End_{\Kb(\proj \Delta_{\sigma'})}(T')^{\op} e_{\sigma^*} \times \Delta_\sigma$-modules. Moreover, $- e_{\sigma}: \add(e_\sigma \Delta_{\sigma'}) \to \proj \Delta_\sigma$ is an equivalence of categories (as $\End_{\Delta_{\sigma'}}(e_\sigma \Delta_{\sigma'} ) = e_\sigma \Delta_{\sigma'} e_\sigma \cong \Delta_\sigma$ by Corollary \ref{subtri2}). Hence, $- e_\sigma : \Kb( \add(e_\sigma \Delta_{\sigma'})) \to \Kb(\proj \Delta_\sigma )$ is also an equivalence of category. Moreover, it is immediate that $e_{\sigma^*} T' \in \Kb(\add(e_\sigma \Delta_{\sigma'}))$, so
  \begin{align*}  & e_{\sigma^*} \End_{\Kb(\proj \Delta_{\sigma'})} ( T') e_{\sigma^*} = \End_{\Kb(\add (e_\sigma \Delta_{\sigma'}))} (e_{\sigma^*} T') \\ =\, & \End_{\Kb(\proj \Delta_\sigma)} (e_{\sigma^*} T' e_{\sigma}) = \End_{\Kb(\proj \Delta_\sigma)} (T). \qedhere\end{align*}
\end{proof}

In view of Lemma \ref{reducidemisom} and Corollary \ref{subtri2}, up to adding some arcs to $\sigma$, we can suppose that we are in one of the case of Figure \ref{figflip2}. We also fix notations of Figure \ref{figflip2}.
\begin{figure}
 \begin{center} \begin{tabular}{cccc}
                 \boxinminipage{$\renewcommand{\labelstyle}{\textstyle} \xymatrix@L=.05cm@!=0cm@R=.7cm@C=.7cm@M=0.00cm{ 
    \\ \\ \\ 
    \ar@{<-}@`{p+(13,11),p+(0,42),p+(-13,11)}[]_{\vec u^+ = \vec u^{*+}} \ar[uuu]|{\vec u = \vec u^*} \ar@{..}[d] \\ \quad
 }$}  & \boxinminipage{$\renewcommand{\labelstyle}{\textstyle} \xymatrix@L=.05cm@!=0cm@R=.7cm@C=.7cm@M=0.00cm{ \ar@{..}[d] \\
    \ar@{<-}@/_.8cm/[dddd]_(.2){\vec u^{+}} \ar@{-->}[dd]_{\vec u^*} \\ \\  \\ \\
    \ar[uu]_{\vec u} \ar@{<-}@/_.8cm/[uuuu]_(.2){\vec u^{*+}} \ar@{..}[d] \\ \quad
 }$} & 
		\boxinminipage{$\renewcommand{\labelstyle}{\textstyle} \xymatrix@L=.05cm@!=0cm@R=.7cm@C=.7cm@M=0.00cm{ \ar@{..}[d] \\
    \ar@{<-}@/_1.2cm/[dddd]_(.2){\vec u^{+}} \ar@{-->}@`{p+(-8,-11),p+(0,-28),p+(8,-11)}[]_(.85){\vec u^*} \\ \\ \vdots \\ \\
    \ar@{->}@`{p+(8,11),p+(0,28),p+(-8,11)}[]_(.85){\vec u}  \ar@{<-}@/_1.2cm/[uuuu]_(.2){\vec u^{*+}} \ar@{..}[d] \\ \quad
 }$} &  
                \boxinminipage{$$\renewcommand{\labelstyle}{\textstyle} \xymatrix@L=.05cm@!=0cm@R=.5cm@C=.5cm@M=0.01cm{
                   & & & & & & & \\
                   & \ar[rrrr]^{(-\vec u)^+} \ar@{..}[ul] & & & & \ar[dddd]^{(-\vec u^*)^+} \ar@{..}[ur] \\ \\ \\ \\ 
                   & \ar[uuuu]^{\vec u^{*+}} \ar@{..}[dl] \ar@{-->}[uuuurrrr]|(.25){\vec u^*} & & & & \ar[llll]^{\vec u^+} \ar@{..}[dr] \ar[uuuullll]|(.25){\vec u} \\ & & & & & & &
                 }$$} \\
	    (F1) &  (F1') & (F2) & (F3)
                \end{tabular}  
 \end{center}
 \caption{Flip of an oriented arc in a partial triangulation} \label{figflip2}
\end{figure}

From now on, we identify $\Delta_\tau$ and $e_\tau \Delta_\sigma e_\tau$ on the one hand and $\Delta^{\lambda^*}_\tau$ and $e_\tau \Delta^{\lambda^*}_{\sigma^*} e_\tau$ on the other hand using Corollary \ref{subtri2}. We also use notation of Section \ref{nonfroz}. More precisely, we use $\RR_{\vec v}$ and $f_{\vec v}$ for $\sigma$ and $\RR^*_{\vec v}$ and $f^*_{\vec v}$ for $\sigma^*$. We need the following lemmas:

\begin{lemma} \label{tech1}
 \begin{enumerate}[\rm (1)]
  \item In cases (F2) or (F3), in $\Delta_\tau$,
  \begin{itemize}
   \item $\ic{-\vec u, (-\vec u)^+} \ic{-(-\vec u)^+, (-\vec u^*)^{+}} = \ic{\vec u, \vec u^+} f^*_{-\vec u^*}$;
   \item $\io{\vec u^+, \vec u} \cdot \ic{-\vec u, (-\vec u)^+} = f^*_{-\vec u^*} \io{(-\vec u^*)^{+}, -(-\vec u)^+}$.
  \end{itemize}
  \item In case (F3), in $\Delta_\tau$,
  \begin{itemize}
   \item $\ic{\vec u, \vec u^+} \ic{-\vec u^+, \vec u^{*+}} = \ic{-\vec u, (-\vec u)^+} f^*_{\vec u^*}$;
   \item $\io{(-\vec u)^+, -\vec u} \cdot \ic{\vec u, \vec u^+} = f^*_{\vec u^*} \io{\vec u^{*+}, -\vec u^+}$.
  \end{itemize}
  \item In case (F2), $\ic{\vec u, \vec u^+} \ic{-\vec u^+, \vec u^{*+}} = 0$.
  \item In cases (F2) or (F3), the identity of $k Q_\tau$ induces an isomorphism $\psi: \Delta^{\lambda^*}_\tau \to \Delta_\tau$.
 \end{enumerate}
\end{lemma}

\begin{proof}
 (1), (2) and (3) are easy computations using Theorem \ref{altpres2} in Figure \ref{figflip2} (recall that in case (F2), $\vec u$ does not enclose a special monogon).

 (4) The only case where it is not immediate is when $\lambda^* \neq \lambda$. So we focus on case (F2) when $\vec u$ encloses a unique puncture $M$. In this case, we have $\lambda^*_M = - \lambda_M$. However, no relation of $\Delta^{\lambda^*}_\tau$ does involve $\lambda^*_M$ (see Theorem \ref{altpres2} and recall that $\vec u \notin \tau$ and $m_M > 1$ by hypothesis). It permits to conclude.
\end{proof}

\begin{lemma} \label{tech3}
 Suppose that we are in case (F1) or (F1'). Then there exists $\alpha \in e_{\vec u^+} \Delta_\tau e_{\vec u^{*+}}$ and an isomorphism $\psi: \Delta^{\lambda^*}_\tau \to \Delta_\tau$ 
such that
  \begin{enumerate}[\rm (1)]
   \item $\ic{\vec u, \vec u^+} \ic{-\vec u^+, \vec u^{*+}} = \ic{\vec u, \vec u^+} \alpha$;
   \item $\psi (\ic{-\vec u^+, \vec u^{*+}}) = \ic{-\vec u^+, \vec u^{*+}} - \alpha$;
   \item $\psi (C^*_\tau) = C_\tau$;
   \item $\alpha \io{\vec u^{*+}, -\vec u^+} = (1 - \nu_\M) e_{\vec u^+} C_\tau$ or $\alpha \io{\vec u^{*+}, -\vec u^+} = 0$;
   \item for any element $\omega$ of the basis $\Br$ of the form $\ic{\vec u^{*+}, x} \ic{\vec x, \vec x}^\ell$, we have
     \begin{itemize}
      \item[] $\alpha \omega = (1-\nu_\M) \ic{-\vec u^+, \vec u^{*+}} \omega$,
      \item[or ] $\alpha \omega$ is multiple of an element of $\Br$ strictly longer than $\ic{-\vec u^+, \vec u^{*+}} \omega$,
      \item[or ] $\alpha \omega = 0$.
     \end{itemize}
  \end{enumerate}
\end{lemma}

\begin{proof}
 We prove the result when $\Sigma$ has no boundary as it implies the general result via Definition \ref{defsp}. Suppose first that we are in case (F1) with $m_{t(\vec u)} > 2$, or in case (F1') with $m_{t(\vec u)} > 1$. Then we have $\ic{\vec u, \vec u^+} \ic{-\vec u^+, \vec u^{*+}} = 0$. Thus, we can take $\alpha = 0$. No relation of $\Delta_{\tau}$ (respectively $\Delta^{\lambda^*}_\tau$) depends on $\lambda_{t(\vec u)}$ (respectively $\lambda^*_{t(\vec u)}$). So, in case (F1), we can take  $\psi(q) = q$ for any $q \in Q_\tau$. Notice that in case (F1'), $\ic{-\vec u^{*+}, \vec u^+}$ appears only in $0$-relations and relations involving $\lambda^*_{s(\vec u)}$. So we can take $\psi(\ic{-\vec u^{*+}, \vec u^+}) = -\ic{-\vec u^{*+}, \vec u^+}$ and $\psi(q) = q$ for any other $q \in Q_\tau$ in case (F1').

 We will now focus on cases (F1) with $m_{t(\vec u)} = 2$ and (F1') with $m_{t(\vec u)} = 1$ (we have excluded (F1) with $m_{t(\vec u)} = 1$).
 
 \begin{itemize}[\rm (i)]
  \item[(F1)] with $m_{t(\vec u)} = 2$. We start by supposing that $\sigma$ contains $\vec v$ as follows:
   $$\renewcommand{\labelstyle}{\textstyle} 
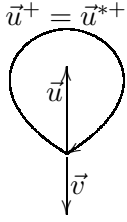
\xymatrix@L=.05cm@!=0cm@R=.4cm@C=.4cm@M=0.00cm{ 
    & & & \\ & & &  \\ \\ \\ 
    & & & \ar@{<-}@`{p+(13,8),p+(0,25),p+(-13,8)}[]_{\vec u^+ = \vec u^{*+}} \ar[uuu]^(.7){\vec u} \ar[dd]^{\vec v} \\
    \\ & & & 
   }$$
   and we will check that we can choose:
   \begin{itemize}[$\bullet$]
    \item $\alpha = \lambda_{t(\vec u)} \io{\vec u^+, -\vec u^{+}}$,
    \item $\psi (\ic{-\vec u^+, \vec u^{+}}) = \ic{-\vec u^+, \vec u^{+}} - \lambda_{t(\vec u)} \io{\vec u^+, -\vec u^{+}}$,
    \item $\psi(q) = q$ for any other arrow $q$ of $Q_\tau$.
   \end{itemize}
   (1) is immediate. We will prove that $\psi$ is an isomorphism, (2), (3) and (4) at the same time. Notice that $\lambda^*_{t(\vec u)} = -\lambda_{t(\vec u)}$ and $\lambda^*_M = \lambda_M$ for any other $M$.
   For each puncture $M$, we denote $\tilde \lambda_M = \delta_{m_M, 1} \lambda_M$. We denote $\tilde \lambda = \lambda_N$ in the case of a sphere with four puncture where the only puncture $N$ which is not on the diagram satisfies $m_N = 1$. In any other case, we denote $\tilde \lambda = 0$. Using Theorem \ref{altpres2}, we get successively:
   \begin{enumerate}[\rm (a)]
    \item $\io{\vec u^+, -\vec u^{+}}  \cdot \ic{\vec u^+, \vec v} = 0$. Indeed,
     $$\io{\vec u^+, -\vec u^{+}}  \cdot \ic{\vec u^+, \vec v} = \io{\vec u^+, \vec v} \cdot \ic{\vec v, -\vec u^+} \ic{\vec u^+, \vec v}.$$
     If $(\Sigma, \M)$ is not a sphere with three punctures, $$\ic{\vec v, -\vec u^+} \ic{\vec u^+, \vec v} = \tilde \lambda \tilde \lambda_{t(\vec v)} e_{\vec v} C_\tau \quad \text{so} \quad \io{\vec u^+, -\vec u^{+}}  \cdot \ic{\vec u^+, \vec v} = 0$$ as $\io{\vec u^+, \vec v} \in J$. If $(\Sigma, \M)$ is a sphere with three punctures, we get
     \begin{align*}
      & \io{\vec u^+, -\vec u^{+}}  \cdot \ic{\vec u^+, \vec v} = \io{\vec u^+, \vec v} \cdot \io{-\vec v, -\vec v} \\
             =\, & \lambda_{t(\vec v)} \io{\vec u^+, \vec u^+} \cdot \io{-\vec u^+, \vec v} \cdot \ic{-\vec v, -\vec v}^{m_{t(\vec v)} - 2} \\
             =\, & \lambda_{t(\vec v)} \io{\vec u^+, \vec u^+}^2 \cdot \ic{-\vec u^+, \vec v} \cdot \ic{-\vec v, -\vec v}^{m_{t(\vec v)} - 2} \\
             =\, & \lambda_{t(\vec v)} \lambda_{s(\vec u)}^2  \ic{\vec u^+, \vec u^+}^{2 m_{s(\vec u)} - 2} \cdot \ic{-\vec u^+, \vec v} \cdot \ic{-\vec v, -\vec v}^{m_{t(\vec v)} - 2} \in C_\tau J = 0
     \end{align*}
     where we used that in this case $m_{t(\vec v)} > 1$ and $2 m_{s(\vec u)} - 2 \geq m_{s(\vec u)}$.
    \item $\psi(\io{-\vec u^+, \vec v}) = \io{-\vec u^+, \vec v} $. Indeed, using (a), we get $\psi(\ic{-\vec u^+, \vec v}) = \ic{-\vec u^+, \vec v}$ and the result follows.
    \item $\psi(\io{\vec v, \vec u^+}) = \io{\vec v, \vec u^+}$. Analogous as (b).
    \item $\psi(\lambda_{s(\vec u^+)} \ic{\vec u^+, \vec u^+}^{s(\vec u^+)}) = e_{\vec u^+} C_\tau$. If follows from (c).
    \item $\psi(\lambda_{t(\vec u^+)} \ic{-\vec u^+, -\vec u^+}^{t(\vec u^+)}) = e_{\vec u^+} C_\tau$. Analogous as (d).
    \item $\psi(\lambda_{s(\vec v)} \ic{\vec v, \vec v}^{s(\vec v)}) = e_{\vec v} C_\tau$. Analogous as (d).
    \item $\psi(\ic{-\vec u^+, \vec u^+}^2) = \lambda^*_{t(\vec u)} e_{\vec u^+} C_\tau$. Indeed, by (a),
     \begin{align*}
      \psi(\ic{-\vec u^+, \vec u^+}^2) &= (\ic{-\vec u^+, \vec u^{+}} - \lambda_{t(\vec u)} \io{\vec u^+, -\vec u^{+}})^2 \\
		      &= \lambda_{t(\vec u)} e_{\vec u^+} C_\tau - 2 \lambda_{t(\vec u)} e_{\vec u^+} C_\tau + \lambda_{t(\vec u)}^2 \io{\vec u^+, -\vec u^{+}}^2 \\
		      &= -\lambda_{t(\vec u)} e_{\vec u^+} C_\tau = \lambda^*_{t(\vec u)} e_{\vec u^+} C_\tau.
     \end{align*}
   \end{enumerate}
    Then we get $\psi(C_{\vec u^+}) = \psi(C_{\vec v}) = 0$ and $\psi(C^*_\tau) = C_\tau$ by (d), (e) and (f). We get $\psi(\RR_{-\vec u^+}) = 0$ by (g). Finally, all other relations $\RR^*_{\vec x}$ defining $\Delta^{\lambda^*}_\tau$ involve $\io{-\vec u^+, \vec v}$, $\io{\vec v, \vec u^+}$ and arrows other than $\ic{-\vec u^+, \vec u^+}$ so we conclude that $\psi(\RR^*_{\vec x}) = \RR_{\vec x} = 0$ by (b) and (c). So $\psi$ is a morphism. It is invertible with inverse satisfying $\psi^{-1} (\ic{-\vec u^+, \vec u^{+}}) = \ic{-\vec u^+, \vec u^{+}} + \lambda_{t(\vec u)} \io{\vec u^+, -\vec u^{+}}$. The condition (4) comes from (a).

    We proved the existence of $\psi$ under the condition that $\vec v \in \sigma$. The general case is obtained by applying Corollary \ref{isoa} (see also Remark \ref{isoarem}). Indeed, put $\sigma_2 = \tau$, $\sigma^\circ = \{\vec u\}$. Let $\sigma_1$ be any partial triangulation which contains $\vec u$, $\vec u^+$ and $\vec v$, which does not intersect $\sigma^\circ$ and which is maximal for these properties. Then it is immediate that $\sigma_1$, $\sigma_2$ and $\sigma^\circ$ satisfy the hypotheses of Corollary \ref{isoa} (the isomorphism $\Delta^{\lambda^*}_{\sigma_1} \to \Delta_{\sigma_1}$ was just constructed). For (5), it is immediate via Corollary \ref{subtri2} that we can suppose that $\omega$ is left divisible by $\ic{\vec u^+, \vec v}$. Then, by (a), $\alpha \omega = 0$.

  \item[(F1')] with $m_{t(\vec u)} = 1$. We suppose first that $\sigma$ contains $\vec v$ as follows:
   $$\renewcommand{\labelstyle}{\textstyle} 
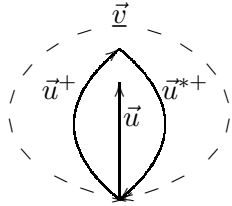
\xymatrix@L=.05cm@!=0cm@R=.4cm@C=.4cm@M=0.00cm{ 
    & & & \ar@{<-}@/_.6cm/[ddddd]_(.3){\vec u^{+}}  & & & \\ & & & \\ & & & \\ \\ \\ 
    & & & \ar[uuuu]_(.7){\vec u} \ar@{<-}@/_.6cm/[uuuuu]_(.7){\vec u^{*+}} \ar@{<--}@`{p+(25,6),p+(0,40),p+(-25,6)}[]_{\vec v}
   }$$
  and we will check that we can set:
   \begin{itemize}[$\bullet$]
    \item $\alpha = \lambda_{t(\vec u)} \io{\vec u^+, -\vec u^{*+}}$,
    \item $\psi (\ic{-\vec u^+, \vec u^{*+}}) = \ic{-\vec u^+, \vec u^{*+}} - \lambda_{t(\vec u)} \io{\vec u^+, -\vec u^{*+}}$,
    \item $\psi (\ic{\vec u^{*+}, -\vec u^+}) = \nu_\M^{-1} \ic{\vec u^{*+}, -\vec u^+}$,
    \item $\psi (\ic{-\vec u^{*+}, \vec u^+}) = -\ic{-\vec u^{*+}, \vec u^{+}} - \lambda_{t(\vec u)} \nu_\M^{-1} \io{\vec u^{*+}, -\vec u^{+}}$,
    \item $\psi (\ic{\vec u^{+}, \vec v}) = \ic{\vec u^{+}, \vec v} - \lambda_{t(\vec u)} \tilde \lambda_{t(\vec u^+)} \io{\vec u^+, -\vec v}$,
    \item $\psi(q) = q$ for any other arrow $q$ of $Q_\tau$.
   \end{itemize}
  Notice that $\lambda^*_{t(\vec u)} = -\lambda_{t(\vec u)}$, $\lambda^*_{s(\vec u)} = (-1)^{m_{s(\vec u)}} \lambda_{s(\vec u)}$ and $\lambda^*_M = \lambda_M$ for any other $M$. We use the same notation as in the latter case for $\tilde \lambda$ and $\tilde \lambda_M$. Identity (1) is immediate. We prove (2), (3) and (4). Using Theorem \ref{altpres2}, we get successively:
  \begin{enumerate}[\rm (a)]
   \item $\psi(\ic{-\vec u^+, -\vec u^+}) = \nu_\M^{-1} (\ic{-\vec u^+, -\vec u^+} - \lambda_{t(\vec u)} \tilde \lambda_{s(\vec u)} \tilde \lambda e_{\vec u^+} C_\tau)$. Indeed
   \begin{align*}
    \psi(\ic{-\vec u^+, -\vec u^+}) &= (\ic{-\vec u^+, \vec u^{*+}} - \lambda_{t(\vec u)} \io{\vec u^+, -\vec u^{*+}}) \nu_\M^{-1} \ic{\vec u^{*+}, -\vec u^+} \\
     &= \nu_\M^{-1} (\ic{-\vec u^+, -\vec u^+} - \lambda_{t(\vec u)} \tilde \lambda_{s(\vec u)} \tilde \lambda e_{\vec u^+} C_\tau)
   \end{align*}
   \item $\psi(\io{-\vec u^+, \vec u^{*+}}) = \io{-\vec u^+, \vec u^{*+}} - \lambda_{t(\vec u)} \tilde \lambda_{t(\vec u^+)} \io{\vec u^+, -\vec u^{*+}}$. Indeed, using (a), $C_\tau J = 0$ and $\nu_\M = 1$ when $m_{t(\vec u^+)} > 1$,
   \begin{align*}
    \psi(\io{-\vec u^+, \vec u^{*+}}) &= \io{-\vec u^+, -\vec u^{+}} (\ic{-\vec u^+, \vec u^{*+}} - \lambda_{t(\vec u)} \io{\vec u^+, -\vec u^{*+}}) \\
	  &= \io{-\vec u^+, \vec u^{*+}} - \lambda_{t(\vec u)} \tilde \lambda_{t(\vec u^+)} \io{\vec u^+, -\vec u^{*+}}
   \end{align*} 
    where we used at the end that $\ic{-\vec u^+, -\vec u^+} \ic{\vec u^+, -\vec u^{*+}} \in J C_\tau = 0$.
   \item $\psi(\ic{\vec u^+, \vec u^+}) = -\ic{\vec u^+, \vec u^+}$. Indeed
    \begin{align*}
     \ic{\vec v, \vec u^+} \ic{-\vec u^+, -\vec u^+} &\in J \ic{-\vec u^{*+}, \vec u^+} \ic{-\vec u^+, \vec u^{*+}} J \in C_\tau J = 0 \\
    \text{so} \quad 
     \ic{\vec v, -\vec u^{*+}} \cdot \io{\vec u^{*+}, -\vec u^+} &= \ic{\vec v, -\vec v} \ic{-\vec v, -\vec u^{*+}} \ic{\vec u^{*+}, -\vec u^+} \cdot \io{-\vec u^{+}, -\vec u^+}  \\
      &= \ic{\vec v, -\vec v} \cdot \io{\vec v, \vec u^{+}} \cdot \io{-\vec u^{+}, -\vec u^+} 
      \\ &= \tilde \lambda \io{\vec v, \vec v} \cdot \ic{\vec v, \vec u^+} \cdot \io{-\vec u^{+}, -\vec u^+} =  \tilde \lambda \tilde \lambda_{t(\vec u^+)} \io{\vec v, \vec u^+}
    \end{align*}
    \begin{align*}
     \text{so} \quad & -\psi(\ic{\vec u^+, \vec u^+}) \\ =\,&  \psi(\ic{\vec u^+, \vec v}) \ic{\vec v, -\vec u^{*+}} (\ic{-\vec u^{*+}, \vec u^{+}} + \lambda_{t(\vec u)} \nu_\M^{-1} \io{\vec u^{*+}, -\vec u^{+}}) \\ 
     =\,& (\ic{\vec u^{+}, \vec v} - \lambda_{t(\vec u)} \tilde \lambda_{t(\vec u^+)} \io{\vec u^+, -\vec v}) (\ic{\vec v, \vec u^+}  + \lambda_{t(\vec u)} \nu_\M^{-1} \tilde \lambda \tilde \lambda_{t(\vec u^+)} \io{\vec v, \vec u^+} ) \\
     =\,& (\ic{\vec u^+, \vec u^+} - \lambda_{t(\vec u)} \tilde \lambda_{t(\vec u^+)} \tilde \lambda e_{\vec u^+} C_\tau) (1 + \lambda_{t(\vec u)} \nu_\M^{-1} \tilde \lambda \tilde \lambda_{t(\vec u^+)} \io{\vec u^+, \vec u^+} ) \\
     =\,& \ic{\vec u^+, \vec u^+} + \lambda_{t(\vec u)} \tilde \lambda_{t(\vec u^+)} \tilde \lambda(\nu_\M^{-1}  - 1 - \nu_\M^{-1} \lambda_{t(\vec u)} \tilde \lambda_{t(\vec u^+)} \tilde \lambda \tilde \lambda_{s(\vec u^+)}   ) e_{\vec u^+} C_\tau \\ =\, & \ic{\vec u^+, \vec u^+}.
    \end{align*}
   \item $\psi(\io{-\vec v, \vec v}) = \io{-\vec v, \vec v}$. Indeed, $\ic{\vec u^{*+}, -\vec u^+} \ic{\vec u^+, \vec u^+} = 0$ so, by (c), using $\lambda^*_{s(\vec u)} =   (-1)^{m_{s(\vec u)}} \lambda_{s(\vec u)}$,
    \begin{align*}
     \psi(\io{-\vec u^{*+}, \vec u^+}) &= (\ic{-\vec u^{*+}, \vec u^{+}} + \lambda_{t(\vec u)} \nu_\M^{-1} \io{\vec u^{*+}, -\vec u^{+}}) \io{\vec u^+, \vec u^+} \\
            &= \io{-\vec u^{*+}, \vec u^{+}} + \lambda_{t(\vec u)} \tilde \lambda_{s(\vec u)} \nu_\M^{-1} \io{\vec u^{*+}, -\vec u^{+}}
    \end{align*}
    \begin{align*}
     \text{so} \;\; & \psi(\io{-\vec v, \vec v}) \\ =\, & \ic{-\vec v, -\vec u^{*+}} (\io{-\vec u^{*+}, \vec u^{+}} + \lambda_{t(\vec u)} \tilde \lambda_{s(\vec u)} \nu_\M^{-1} \io{\vec u^{*+}, -\vec u^{+}}) \psi(\ic{\vec u^+, \vec v}) \\
       =\, & (\io{-\vec v, \vec u^{+}} + \lambda_{t(\vec u)} \tilde \lambda_{s(\vec u)} \nu_\M^{-1}  \io{\vec v, \vec u^{+}}  \cdot \io{-\vec u^{+}, -\vec u^{+}}) \psi(\ic{\vec u^+, \vec v}) \\
       =\, & (\io{-\vec v, \vec u^{+}} + \lambda_{t(\vec u)} \tilde \lambda_{s(\vec u)} \nu_\M^{-1}  \tilde \lambda_{t(\vec u^+)}\io{\vec v, \vec u^{+}}) (\ic{\vec u^{+}, \vec v} - \lambda_{t(\vec u)} \tilde \lambda_{t(\vec u^+)} \io{\vec u^+, -\vec v}) \\
       =\, & \io{-\vec v, \vec v} + \lambda_{t(\vec u)} \tilde \lambda_{t(\vec u^+)} \tilde \lambda_{s(\vec u)} (\nu_\M^{-1} - 1 - \lambda_{t(\vec u)} \tilde \lambda_{t(\vec u^+)} \tilde \lambda_{s(\vec u)} \nu_\M^{-1} \ic{\vec v, -\vec v}) e_{\vec v} C_\tau \\
       =\, & \io{-\vec v, \vec v} + \lambda_{t(\vec u)} \tilde \lambda_{t(\vec u^+)} \tilde \lambda_{s(\vec u)} (\nu_\M^{-1} - 1 - \lambda_{t(\vec u)} \tilde \lambda_{t(\vec u^+)} \tilde \lambda_{s(\vec u)} \nu_\M^{-1} \tilde \lambda ) e_{\vec v} C_\tau = \io{-\vec v, \vec v}
    \end{align*}
    \item $\psi(\ic{-\vec u^+, -\vec u^+}^{m_{t(\vec u^+)}}) = \ic{-\vec u^+, -\vec u^+}^{m_{t(\vec u^+)}}$. This is an immediate consequence of (a) and $C_\tau J = 0$.
    \item $\psi(\ic{\vec u^{*+}, \vec u^{*+}}^{m_{s(\vec u^{*+})}}) = \ic{\vec u^{*+}, \vec u^{*+}}^{m_{s(\vec u^{*+})}}$. Analogous to (e).
    \item $\psi(\ic{\vec u^+, \vec u^+}^{m_{s(\vec u^+)}}) = (-1)^{m_{s(\vec u^+)}} \ic{\vec u^+, \vec u^+}^{m_{s(\vec u^+)}}$. Consequence of (c).
    \item $\psi(\ic{-\vec u^{*+}, -\vec u^{*+}}^{m_{t(\vec u^{*+})}}) = (-1)^{m_{s(\vec u^+)}} \ic{-\vec u^{*+}, -\vec u^{*+}}^{m_{t(\vec u^{*+})}}$. As (g).
    \item $\psi(\ic{-\vec u^+, \vec u^{*+}}\ic{-\vec u^{*+}, \vec u^+}) = \lambda_{t(\vec u)}^* e_{\vec u^+} C_\tau$. Indeed,
     \begin{align*}
      & -\psi(\ic{-\vec u^+, \vec u^{*+}}\ic{-\vec u^{*+}, \vec u^+}) \\ =\, & (\ic{-\vec u^+, \vec u^{*+}} - \lambda_{t(\vec u)} \io{\vec u^+, -\vec u^{*+}}) (\ic{-\vec u^{*+}, \vec u^{+}} + \lambda_{t(\vec u)} \nu_\M^{-1} \io{\vec u^{*+}, -\vec u^{+}}) \\
      =\,& (\lambda_{t(\vec u)} - \lambda_{t(\vec u)} + \lambda_{t(\vec u)} \nu_\M^{-1} - \lambda_{t(\vec u)}^2 \nu_\M^{-1} \tilde \lambda_{s(\vec u^+)} \tilde \lambda_{t(\vec u^+)}\tilde \lambda )e_{\vec u^+} C_\tau = \lambda_{t(\vec u)} e_{\vec u^+} C_\tau.
     \end{align*}
    \item $\psi(\ic{-\vec u^{*+}, \vec u^{+}}\ic{-\vec u^{+}, \vec u^{*+}}) = \lambda^*_{t(\vec u)} e_{\vec u^{*+}} C_\tau$. Analogous to (i).
    \item $\psi(\ic{\vec u^+, \vec v} \ic{-\vec v, -\vec u^{*+}}) = \psi(\io{-\vec u^+, \vec u^{*+}})$. Indeed,
     \begin{align*}
      \psi(\ic{\vec u^+, \vec v} \ic{-\vec v, -\vec u^{*+}}) &= (\ic{\vec u^+, \vec v} - \lambda_{t(\vec u)} \tilde \lambda_{t(\vec u^+)}  \io{\vec u^+, -\vec v})\ic{-\vec v, -\vec u^{*+}} \\ &= \io{-\vec u^+, \vec u^{*+}} - \lambda_{t(\vec u)} \tilde \lambda_{t(\vec u^+)} \io{\vec u^+, -\vec u^{*+}}
     \end{align*}
     and we conclude with (b).
    \item $\psi(\ic{\vec u^{*+}, -\vec u^{+}} \ic{\vec u^+, \vec v}) = \psi(\io{-\vec u^{*+}, -\vec v})$. Indeed,
     \begin{align*}
      \psi(\ic{\vec u^{*+}, -\vec u^{+}} \ic{\vec u^+, \vec v}) &= \nu_\M^{-1} \ic{\vec u^{*+}, -\vec u^+}(\ic{\vec u^+, \vec v} - \lambda_{t(\vec u)} \tilde \lambda_{t(\vec u^+)} \io{\vec u^+, -\vec v}) \\ &= \nu_\M^{-1} \io{-\vec u^{*+}, -\vec v} (1 - \lambda_{t(\vec u)} \tilde \lambda_{t(\vec u^+)} \io{\vec v, -\vec v})) \\ &= \nu_\M^{-1}  (1 - \lambda_{t(\vec u)} \tilde \lambda_{t(\vec u^+)} \tilde \lambda \tilde \lambda_{s(\vec u)}  ) \io{-\vec u^{*+}, -\vec v} \\ &= \io{-\vec u^{*+}, -\vec v}
     \end{align*}
     and by (c) and $\lambda^*_{s(\vec u)} =   (-1)^{m_{s(\vec u)}} \lambda_{s(\vec u)}$,
     \begin{align*}
      & \psi(\io{-\vec u^{*+}, -\vec v}) \\ =\,& (\ic{-\vec u^{*+}, \vec u^{+}} + \lambda_{t(\vec u)} \nu_\M^{-1} \io{\vec u^{*+}, -\vec u^{+}}) \io{\vec u^+, \vec u^+} \psi(\ic{\vec u^+, -\vec v}) \\ =\,& (\io{-\vec u^{*+}, \vec u^{+}} + \lambda_{t(\vec u)} \nu_\M^{-1} \tilde \lambda_{s(\vec u^+)} \io{\vec u^{*+}, -\vec u^{+}})  (\ic{\vec u^{+}, -\vec v} - \lambda_{t(\vec u)} \tilde \lambda_{t(\vec u^+)} \tilde \lambda \io{\vec u^+, -\vec v})
       \\ =\,& (1 + \lambda_{t(\vec u)} \nu_\M^{-1} \tilde \lambda_{s(\vec u^+)} \tilde \lambda \tilde \lambda_{s(\vec u^{*+})}) \io{-\vec u^{*+}, -\vec v} (1 - \lambda_{t(\vec u)} \tilde \lambda_{t(\vec u^+)} \tilde \lambda \io{-\vec v, -\vec v}) \\ =\,& \nu_\M^{-1}   (1 - \lambda_{t(\vec u)} \tilde \lambda_{t(\vec u^+)} \tilde \lambda \tilde \lambda_{s(\vec u)}) \io{-\vec u^{*+}, -\vec v} = \io{-\vec u^{*+}, -\vec v}.
     \end{align*}
    \item $\psi(\ic{-\vec v, -\vec u^{*+}} \ic{\vec u^{*+}, -\vec u^+}) = \psi(\io{\vec v, \vec u^+})$. Analogous to (l).   
  \end{enumerate}
  We conclude for (2), (3) and (4) as for (F1). We have to prove (5).
 \begin{itemize}[$\bullet$]
  \item If $t(\vec x) \neq s(\vec x), s(\vec u)$. We have 
  $$\alpha \ic{\vec u^{*+}, \vec x} =  \lambda_{t(\vec u)} \io{\vec u^+, \vec u^+} \cdot \ic{\vec u^+, -\vec u^{*+}} \ic{\vec u^{*+}, \vec x}.$$
  If $(\Sigma, \M)$ is not a sphere with four punctures, it vanishes. Otherwise, we deduce:
  \begin{align*}
   \alpha \ic{\vec u^{*+}, \vec x} \ic{\vec x, \vec x}^\ell &= \lambda_{t(\vec u)} \tilde \lambda_{t(\vec x)} \io{\vec u^+, \vec u^+} \cdot \io{-\vec u^+, \vec x} \cdot  \ic{\vec x, \vec x}^\ell \\ &= \lambda_{t(\vec u)} \tilde \lambda_{t(\vec x)} \tilde \lambda_{s(\vec u^{+})} \io{-\vec u^+, \vec x} \cdot \ic{\vec x, \vec x}^\ell 
  \end{align*}
  as $\ic{\vec u^+, \vec u^+} \ic{-\vec u^+, \vec x} = 0$. If $m_{t(\vec u^+)} > 1$, this is strictly right divisible by $\ic{-\vec u^+, \vec u^{*+}} \omega$ so it is  longer. If $m_{t(\vec u^+)} = 1$ we have necessarily $\ell = 0$ and 
  $$ \alpha \ic{\vec u^{*+}, \vec x} = \lambda_{t(\vec u)} \tilde \lambda_{t(\vec x)} \tilde \lambda_{s(\vec u^{+})} \tilde \lambda_{t(\vec u^+)} \ic{-\vec u^+, \vec x} = (1 - \nu_\M) \ic{-\vec u^+, \vec u^{*+}} \omega.$$
  \item If $t(\vec x) = s(\vec x)$ and $\vec x$ is winding counter-clockwisely. If the triangle with sides $\vec u^+$, $-\vec x$, $\vec u^{*+}$ contains at least one puncture, then $\alpha \ic{\vec u^{*+}, \vec x} = 0$. Otherwise, 
    \begin{align*}\alpha \ic{\vec u^{*+}, \vec x} \ic{\vec x, \vec x}^\ell &=  \lambda_{t(\vec u)} \io{\vec u^+, \vec u^+} \cdot \ic{\vec u^+, -\vec u^{*+}} \ic{\vec u^{*+}, \vec x}\ic{\vec x, \vec x}^\ell  \\ &= \lambda_{t(\vec u)} \io{\vec u^+, \vec u^+} \cdot \io{-\vec u^+, -\vec x} \cdot \ic{\vec x, \vec x}^\ell \\ &= \delta_{\ell, 0} \lambda_{t(\vec u)} \tilde \lambda_{s(\vec u^+)} \io{-\vec u^+, -\vec x} \end{align*}
   which is a strict multiple of $\ic{-\vec u^+, \vec u^{*+}} \omega$ so it is longer (or $0$).
  \item If $t(\vec x) = s(\vec x)$ and $\vec x$ is winding clockwisely. According to the previous case for $-\vec x$, the only possibility for $\alpha \omega$ to be non-zero is when the triangle with sides $\vec u^+$, $\vec x$, $\vec u^{*+}$ contains no puncture and $\ell = 0$. In this case,
   $$\alpha \ic{\vec u^{*+}, \vec x} = \lambda_{t(\vec u)} \tilde \lambda_{s(\vec u^+)} \io{-\vec u^+, \vec x} \cdot \ic{-\vec x, \vec x} = \lambda_{t(\vec u)} \tilde \lambda_{s(\vec u^+)} \tilde \lambda \io{-\vec u^+, \vec x}$$
  which is strictly longer than $\ic{-\vec u^+, \vec u^{*+}} \omega$ except if $m_{t(\vec u^+)} = 1$. In this case
   $$\alpha \ic{\vec u^{*+}, \vec x} = \lambda_{t(\vec u)} \tilde \lambda_{s(\vec u^+)} \tilde \lambda \tilde \lambda_{t(\vec u^+)} \ic{-\vec u^+, \vec x} = (1 - \nu_\M) \ic{-\vec u^+, \vec u^{*+}} \omega.$$
  \item If $t(\vec x) = s(\vec u)$. We have
  $$\alpha \ic{\vec u^{*+}, \vec x}\ic{\vec x, \vec x}^\ell =  \lambda_{t(\vec u)} \io{\vec u^+, \vec u^+} \cdot \ic{\vec u^+, -\vec u^{*+}} \ic{\vec u^{*+}, \vec x} \ic{\vec x, \vec x}^\ell$$
  which admit a factor $\ic{-\vec x, -\vec u^{*+}} \ic{\vec u^{*+}, \vec x} \in (C_\sigma)$ so the only way that this does not vanish is $\vec x = -\vec u^+$ and $\ell = 0$. In this case,
  $$\alpha \ic{\vec u^{*+}, \vec x} =  \lambda_{t(\vec u)} \tilde \lambda \io{\vec u^+, \vec u^+} e_{\vec u^+} C_\tau = \lambda_{t(\vec u)} \tilde \lambda \tilde \lambda_{s(\vec u^+)} e_{\vec u^+} C_\tau$$
  which is strictly longer than $\ic{-\vec u^+, \vec u^{*+}} \omega$ except if $m_{t(\vec u^+)} = 1$. In this case,
   \begin{align*}\alpha \ic{\vec u^{*+}, \vec x} &=  \lambda_{t(\vec u)} \tilde \lambda \tilde \lambda_{s(\vec u^+)} \tilde \lambda_{t(\vec u^+)} \ic{-\vec u^+, -\vec u^+} = (1 - \nu_\M)\ic{-\vec u^+, \vec u^{*+}} \omega. \qedhere \end{align*}
 \end{itemize}
 \end{itemize}
\end{proof}

We generalize the notation of Lemma \ref{tech3} by setting $\alpha = 0$ and taking $\psi$ as in Lemma \ref{tech1} in cases (F2) or (F3). Thus, we always have 
 \begin{itemize}
  \item $\psi (\ic{-\vec u^+, \vec u^{*+}}) = \ic{-\vec u^+, \vec u^{*+}} - \alpha$,
  \item $\psi (C^*_\tau) = C_\tau$,
  \item $\ic{-\vec u, (-\vec u)^+} \ic{-(-\vec u)^+, (-\vec u^*)^{+}} = \ic{\vec u, \vec u^+} f^*_{-\vec u^*}$,
  \item $\io{\vec u^+, \vec u} \cdot \ic{-\vec u, (-\vec u)^+} = f^*_{-\vec u^*} \io{(-\vec u^*)^{+}, -(-\vec u)^+}$,
 \end{itemize}
 and, in cases (F1), (F1') or (F3),
 \begin{itemize}
  \item $\ic{\vec u, \vec u^+} \ic{-\vec u^+, \vec u^{*+}} = \ic{\vec u, \vec u^+} \alpha + \ic{-\vec u, (-\vec u)^+} f^*_{\vec u^*}$,
  \item $\io{(-\vec u)^+, -\vec u} \cdot \ic{\vec u, \vec u^+} = f^*_{\vec u^*} \io{\vec u^{*+}, -\vec u^+}$
 \end{itemize}
 (in cases (F1) and (F1'), most of these equalities are trivial).

 Then, we get the following maps in $\Kb(\proj \Delta_\sigma)$:
 \begin{itemize}
  \item $\pi_+ := \left( \begin{bmatrix} \id_{e_{\vec u^+} \Delta_\sigma} & 0 \end{bmatrix}, 0\right): P_u^* \to e_{\vec u^+} \Delta_\sigma$;
  \item $\pi_- := \left( \begin{bmatrix} 0 & -\id_{e_{(-\vec u)^+} \Delta_\sigma} \end{bmatrix}, 0\right): P_u^* \to e_{(-\vec u)^+} \Delta_\sigma$;
  \item $\theta_+ := \left( \begin{bmatrix} \ic{-\vec u^+, \vec u^{*+}} - \alpha \\ -f^*_{\vec u^*} \end{bmatrix}, 0\right): e_{\vec u^{*+}} \Delta_\sigma \to P_u^*$ in cases (F1), (F1') or (F3);
  \item $\theta_- := \left( \begin{bmatrix} f^*_{-\vec u^*} \\ -\ic{-(-\vec u)^+, (-\vec u^*)^{+}} \end{bmatrix}, 0\right): e_{(-\vec u^*)^{+}} \Delta_\sigma \to P_u^*$ in cases (F2) or (F3);
  \item $\epsilon := \left( \io{\vec u^+, -\vec u^+}, \ic{-\vec u, -\vec u} \right): P_u^* \to P_u^*$ in case (F1);
  \item $\epsilon := \left( 0, \ic{-\vec u, -\vec u} \right): P_u^* \to P_u^*$ in case (F1');
  \item $\eta := -\left(\begin{sbmatrix} 0 & \id_{e_{\vec u^+} \Delta_\sigma} \\ 0 & 0 \end{sbmatrix} , \ic{\vec u, -\vec u} \right): P_u^* \to P_u^*$ in case (F2).
 \end{itemize}

\begin{lemma} \label{defmorphtilt}
 There exists a unique $\phi^\circ: k Q_{\sigma^*} \to \End_{\Kb(\proj \Delta_\sigma)} (T)$ satisfying $\phi^\circ(x) = \psi(x)$ for $x \in kQ_{\tau}$ and
 \begin{itemize}
  \item $\phi^\circ(e_{u^*}) = \id_{P_u^*}$;
  \item $\phi^\circ(\ic{-\vec u^+, \vec u^*}) = \pi_+$;
  \item $\phi^\circ(\ic{-(-\vec u)^+, -\vec u^*}) = \pi_-$ in case (F3);
  \item $\phi^\circ(\ic{\vec u^*, \vec u^{*+}}) = \theta_+$ in case (F1), (F1') or (F3);
  \item $\phi^\circ(\ic{-\vec u^*, (-\vec u^*)^{+}}) = \theta_-$ in cases (F2) or (F3);
  \item $\phi^\circ(\ic{-\vec u^*, -\vec u^*}) = \epsilon$ in case (F1) or (F1');
  \item $\phi^\circ(\ic{\vec u^*, -\vec u^*}) = \eta$ in case (F2).
 \end{itemize}
\end{lemma}

\begin{proof}
 Notice first that if $\vec u^{*+}$ or $(-\vec u^*)^{+}$ is a boundary component, some equality becomes trivial.

 Arrows of $Q_{\sigma^*}$ which are not in $e_\tau k Q_{\sigma^*} e_\tau$ are the one defined case by case. Thus, such a map, if it exists, is unique. For the well definition, it is enough to check each arrow of $e_\tau k Q_{\sigma^*} e_\tau$ which appear as a composition of arrows of $Q_{\sigma^*}$. We consider all cases:
 \begin{itemize}
  \item $\ic{-\vec u^+, \vec u^{*+}} = \ic{-\vec u^+, \vec u^*} \ic{\vec u^*, \vec u^{*+}}$ (in cases (F1), (F1') or (F3)): 
   $$\phi^\circ(\ic{-\vec u^+, \vec u^*}) \phi^\circ(\ic{\vec u^*, \vec u^{*+}}) = \pi_+ \theta_+ = \psi(\ic{-\vec u^+, \vec u^{*+}}) = \phi^\circ(\ic{-\vec u^+, \vec u^{*+}});$$
  \item $\ic{-(-\vec u)^+, (-\vec u^*)^{+}} = \ic{-(-\vec u)^+, -\vec u^*} \ic{-\vec u^*, (-\vec u^*)^{+}}$ (in case (F3)): this is similar;
  \item $\ic{-\vec u^+, (-\vec u^*)^{+}} = \ic{-\vec u^+, \vec u^*} \ic{\vec u^*, -\vec u^*} \ic{-\vec u^*, (-\vec u^*)^{+}}$ (in case (F2)):
   \begin{align*}\phi^\circ(\ic{-\vec u^+, \vec u^*}) \phi^\circ(\ic{\vec u^*, -\vec u^*}) \phi^\circ(\ic{-\vec u^*, (-\vec u^*)^{+}}) &= \pi_+ \eta \theta_- \\&= \phi^\circ(\ic{-\vec u^+, (-\vec u^*)^{+}}). \qedhere\end{align*}
 \end{itemize}
\end{proof}

\begin{lemma} \label{exseq}
 The following sequence is exact: 
  $$e_{\vec u^{*+}} \Delta_\sigma \oplus e_{(-\vec u^*)^{+}} \Delta_\sigma \xrightarrow{A} e_{\vec u^+} \Delta_\sigma \oplus e_{(-\vec u)^+} \Delta_\sigma \xrightarrow{\begin{bmatrix}\ic{\vec u, \vec u^+} & \ic{-\vec u, (-\vec u)^+} \end{bmatrix}} e_u \Delta_\sigma$$
  where 
  $$A = \left\{\begin{array}{ll}
                \begin{bmatrix} \ic{-\vec u^+, \vec u^{*+}} - \alpha & 0 \\ 0 & 0 \end{bmatrix} & \text{in case (F1) or (F1'),} \\
                \begin{bmatrix} \ic{-\vec u^+, \vec u^{*+}} & f^*_{-\vec u^*} \\ 0 & -\ic{-\vec u^+, \vec u^{*+}}\end{bmatrix} & \text{in case (F2),}\\
		\begin{bmatrix} \ic{-\vec u^+, \vec u^{*+}} & f^*_{-\vec u^*} \\ -f^*_{\vec u^*} & -\ic{-(-\vec u)^+, (-\vec u^*)^{+}}\end{bmatrix} & \text{in case (F3).}
               \end{array} \right.$$ 
\end{lemma}

\begin{proof}
 The composition vanishes by Lemmas \ref{tech1} and \ref{tech3}. Let us prove the exactness:
  \begin{enumerate}[\rm (a)]
   \item In case (F1) or (F1'). Let $x \in e_{\vec u^+} \Delta_\sigma$ such that $\ic{\vec u, \vec u^+} x = 0$ and write $x = \sum_{b \in \Br} \mu_b b$. According to Lemma \ref{tech3} (5), for any $\omega = \ic{\vec u^{*+}, \vec x} \ic{\vec x, \vec x}^\ell \in \Br$, we have $(\ic{-\vec u^+, \vec u^{*+}} - \alpha) \omega = \kappa \ic{-\vec u^+, \vec u^{*+}} \omega + \omega'$ where $\kappa \in \{1, \nu_\M\}$ is invertible and $\omega'$ is strictly longer than $\ic{-\vec u^+, \vec u^{*+}} \omega$ or $0$. Thus, as the length of non-zero paths is bounded, an immediate induction permits to make the assumption, up to subtracting an element of $\ima A$, that $x = \sum_{b \in \Br'} \mu_b b$ where $\Br'$ consists of the following elements of $\Br$: $e_{\vec u^+}$ and $\ic{\vec u^+, \vec x} \ic{\vec x, \vec x}^\ell$ where $s(\vec x) = s(\vec u^+)$, $0 \leq \ell < m_{s(\vec u^+)}$ and $\ell \neq m_{s(\vec u^+)} -1$ if $\vec x = \vec u^+$. It is immediate that the left multiplication by $\ic{\vec u, \vec u^+}$ is injective on $\Br'$ so $\ic{\vec u, \vec u^+} x = 0$ implies $x = 0$.
   \item In case (F2). Notice that $(-\vec u)^+ = \vec u^+$. Let $(x, y) \in e_{\vec u^+} \Delta_\sigma \oplus e_{\vec u^+} \Delta_\sigma$. Modulo $\ima A$, we can suppose that $y$ is a linear combination of elements of $\Br'$ where $\Br'$ denotes the same subset of $\Br$ as in (a). Then, reducing again modulo $\ima A$, we can suppose that $x$ is also a linear combination of elements of $\Br'$. Then, if $(x, y)$ is in the kernel of $\begin{bmatrix}\ic{\vec u, \vec u^+} & \ic{-\vec u, (-\vec u)^+} \end{bmatrix}$, we get that $(x, y)$ is a linear combination of $a := (\io{\vec u^+, -\vec u}, 0)$ and $b := (\io{\vec u^+, \vec u}, -\io{\vec u^+, -\vec u})$. On the other hand, $a = A (\ic{-\vec u^{*+}, \vec u}, 0)$ and $b = A (0, \ic{-\vec u^{*+}, \vec u})$.
   \item In case (F3). It is obtained by applying $S_{\vec u} \tens_{\Delta_\sigma} -$ to the standard projective bimodule resolution of $\Delta_\sigma$ (see also Proposition \ref{bimodres}).\qedhere
  \end{enumerate} 
\end{proof}

\begin{lemma} \label{phisurj}
 The map $\phi^\circ: k Q_{\sigma^*} \to \End_{\Kb(\proj \Delta_\sigma)} (T)$ is surjective.
\end{lemma}

\begin{proof}
 First of all, it is of course surjective onto $e_\tau \End_{\Kb(\proj \Delta_\sigma)} (T) e_\tau \cong \Delta_\tau$ as $\psi$ is surjective. If $f \in \Hom_{\Kb(\proj \Delta_\sigma)} (P_u^*, e_\tau \Delta_\sigma)$, it is immediate that $f$ factors through $\pi_{\pm}$ so $f$ is in the image of $\phi^\circ$ ($\pi_- = \pi_+ \eta$ in case (F2)). 
 
  If $f \in \Hom_{\Kb(\proj \Delta_\sigma)} (e_\tau \Delta_\sigma, P_u^*)$, using the exact sequence of Lemma \ref{exseq} and the fact that $e_\tau \Delta_\sigma$ is projective, we get that $f$ factors through $\theta_{\pm}$ (we replace $\theta_+$ by $\eta \theta_-$ in case (F2)). So $f$ is in the image of $\phi^\circ$ also in this case.

  Finally, take $(f_2, f_1) \in \End_{\Kb(\proj \Delta_\sigma)}(P^*_u)$. This endomorphism induces via cokernel a endomorphism $f_0$ of $X_\e$. It is easy to see that $f_0$ is a linear combination of the identity and morphisms induced by powers of $\epsilon$ (case (F1)) and $\eta$ (case (F2)). Thus, up to an element of the image of $\phi^\circ$, we can suppose that $f_0 = 0$. Then, up to homotopy, we can suppose that $f_1 = 0$. So $f_2$ factors through $A$ of Lemma \ref{exseq} and it permits to factor $(f_2, f_1)$ through $\pi_{\pm}$. Finally, $(f_2, f_1)$ is in the image of $\phi^\circ$.
\end{proof}

\begin{lemma} \label{relC}
 For any oriented edge $\vec v$ of $\sigma^*$, $\phi^\circ(C_{\vec v}) = 0$.
\end{lemma}

\begin{proof}
 This is immediate if $\vec v \neq \pm \vec u^*$ as in this case $C_{\vec v} \in e_\tau Q_{\sigma^*} e_\tau$ and $\phi^\circ(C_{\vec v}) = \psi(C_{\vec v}) = 0$. So we check $C_{\vec u^*} = -C_{-\vec u^*}$. In case (F3), we have
 \begin{align*}
  \phi^\circ(\lambda^*_{s(\vec u^*)} \ic{\vec u^*, \vec u^*}^{m_{s(\vec u^*)}}) &= \theta_+ \psi(\io{\vec u^{*+}, -\vec u^+}) \pi_+ \\
    &= \left(\begin{bmatrix} \ic{-\vec u^+, \vec u^{*+}} \\ -f^*_{\vec u^*} \end{bmatrix}, 0\right) \io{\vec u^{*+}, -\vec u^+} \left(\begin{bmatrix} \id_{e_{\vec u^+} \Delta_\sigma} & 0 \end{bmatrix}, 0\right) \\ &= \left(\begin{bmatrix} \lambda_{t(\vec u^+)} \ic{-\vec u^+, -\vec u^+}^{m_{t(\vec u^+)}} & 0\\ -f^*_{\vec u^*} \io{\vec u^{*+}, -\vec u^+} & 0\end{bmatrix}, 0\right)
    \\ &= \left(\begin{bmatrix} \lambda_{s(\vec u^+)} \ic{\vec u^+, \vec u^+}^{m_{s(\vec u^+)}} & 0\\ -f^*_{\vec u^*} \io{\vec u^{*+}, -\vec u^+} & 0\end{bmatrix}, 0\right)
 \end{align*}
 and, in the same way,
 \begin{align*}
  \phi^\circ(\lambda_{s(-\vec u^*)} \ic{-\vec u^*, -\vec u^*}^{m_{s(-\vec u^*)}}) = \left(\begin{bmatrix} 0 & -f^*_{-\vec u^*} \io{(-\vec u)^{*+}, -(-\vec u)^+}\\ 0 & \lambda_{s((-\vec u)^+)} \ic{(-\vec u)^+, (-\vec u)^+}^{m_{s((-\vec u)^+)}}\end{bmatrix}, 0\right)
 \end{align*}
 Using the following homotopy:
 $$\begin{bmatrix} \io{\vec u^+, \vec u} \\ -\io{(-\vec u)^+, -\vec u} \end{bmatrix} : e_u \Delta_\sigma \to e_{\vec u^+} \Delta_\sigma \oplus e_{(-\vec u)^+} \Delta_\sigma,$$
 we get that the following endomorphism of $P_{u}^*$ is homotopic to zero:
 $$\left(\begin{bmatrix} \lambda_{s(\vec u^+)} \ic{\vec u^+, \vec u^+}^{m_{s(\vec u^+)}} & \io{\vec u^+, \vec u} \cdot \ic{-\vec u, (-\vec u)^+} \\ -\io{(-\vec u)^+, -\vec u} \cdot \ic{\vec u, \vec u^+} & -\lambda_{s((-\vec u)^+)} \ic{(-\vec u)^+, (-\vec u)^+}^{m_{s((-\vec u)^+)}}\end{bmatrix}, 0\right)$$
 which is equal to $\phi^\circ (C_{\vec u^*})$, using the previous computation and Lemma \ref{tech1}.

 In case (F1) or (F1'), we have
 \begin{align*}
  \phi^\circ(\lambda_{s(\vec u^*)}^* \ic{\vec u^*, \vec u^*}^{m_{s(\vec u^*)}}) &= \theta_+ \psi(\io{\vec u^{*+}, -\vec u^+}) \pi_+ \\
    &= \left((\ic{-\vec u^+, \vec u^{*+}} - \alpha) \psi(\io{\vec u^{*+}, -\vec u^+}), 0\right) \\
    &= \left(\psi(\ic{-\vec u^+, \vec u^{*+}} ) \psi(\io{\vec u^{*+}, -\vec u^+}), 0\right) \\
    &= \left(\psi(e_{\vec u^+} C^*_\tau), 0\right) 
    = \left(\lambda_{s(\vec u)} \ic{\vec u^+, \vec u^+}^{m_{s(\vec u)}}, 0\right)
 \end{align*}
 (we use Lemma \ref{tech3} at the last step). 
 As we excluded the case where $u$ is the special arc of a special monogon, we have
 \begin{align*}\phi^\circ(\lambda^*_{t(\vec u^*)} \ic{-\vec u^*, -\vec u^*}^{m_{t(\vec u^*)}}) &=  \lambda^*_{t(\vec u^*)} \epsilon^{m_{t(\vec u^*)}} = \left(0, -\lambda_{t(\vec u)} \ic{-\vec u, -\vec u}^{m_{t(\vec u)}}\right) \\ &= \left(0, -\lambda_{s(\vec u)} \ic{\vec u, \vec u}^{m_{s(\vec u)}}\right).\end{align*}
 Using the homotopy $\io{\vec u^+, \vec u}: e_u \Delta_\sigma \to e_{\vec u^+} \Delta_\sigma$, we get that:
 $$\phi^\circ (C_{\vec u^*}) = \left( \lambda_{s(\vec u)} \ic{\vec u^+, \vec u^+}^{m_{s(\vec u)}}, \lambda_{s(\vec u)} \ic{\vec u, \vec u}^{m_{s(\vec u)}} \right)$$
 is homotopic to $0$.

 Finally, in case (F2), we have
 \begin{align*}\phi^\circ(\lambda_{s(\vec u^*)} \ic{\vec u^*, \vec u^*}^{m_{s(\vec u^*)}})  &= \eta \theta_- \io{\vec u^{*+}, -\vec u^+} \pi_+ \\  &= \left( \begin{bmatrix} \lambda_{s(\vec u^+)} \ic{\vec u^+, \vec u^+}^{m_{s(\vec u^+)}} & 0 \\ 0 & 0 \end{bmatrix}, 0 \right) \end{align*}
 and
 \begin{align*}
  \phi^\circ(\lambda_{s(\vec u^*)} \ic{-\vec u^*, -\vec u^*}^{m_{s(\vec u^*)}}) &= \theta_- \io{\vec u^{*+}, -\vec u^+} \pi_+ \eta \\
  &=\left( \begin{bmatrix} 0 & -f^*_{-\vec u^*} \io{\vec u^{*+}, -\vec u^+} \\ 0 & \lambda_{t(\vec u^+)} \ic{-\vec u^+, -\vec u^+}^{m_{t(\vec u^+)}} \end{bmatrix}, 0 \right) \\
  &=  \left( \begin{bmatrix} 0 & -\io{\vec u^+, \vec u} \cdot \ic{-\vec u, \vec u^+} \\ 0 & \lambda_{s(\vec u^+)} \ic{\vec u^+, \vec u^+}^{m_{s(\vec u^+)}} \end{bmatrix}, 0 \right)
 \end{align*}
 thanks to Lemma \ref{tech1}. Using the homotopy 
 $$\begin{bmatrix}
    \io{\vec u^+, \vec u} \\ -\io{(-\vec u)^+, -\vec u}
   \end{bmatrix}: e_u \Delta_\sigma \to e_{\vec u^+} \Delta_\sigma \oplus e_{(-\vec u)^+} \Delta_\sigma,$$
 we get the following null-homotopic endomorphism of $P_u^*$:
 $$\left(\begin{bmatrix}
          \lambda_{s(\vec u^+)} \ic{\vec u^+, \vec u^+}^{m_{s(\vec u^+)}} & \io{\vec u^+, \vec u} \cdot \ic{-\vec u, \vec u^+} \\
	  0 & -\lambda_{s(\vec u^+)} \ic{\vec u^+, \vec u^+}^{m_{s(\vec u^+)}}
         \end{bmatrix}, 0\right)$$
 which is $\phi^\circ (C_{\vec u^*})$. 
\end{proof}

\begin{lemma} \label{Rus}
 We have $\phi^\circ(\RR^*_{\vec u^*}) = \phi^\circ(\RR^*_{-\vec u^*}) = 0$.
\end{lemma}

\begin{proof}
 We start by $\RR^*_{-\vec u^*}$: 

 In cases (F2) or (F3), we have
 $$\phi^\circ(\ic{-\vec u^+, \vec u^*} \ic{-\vec u^*, (-\vec u^*)^{+}}) = \pi_+\theta_-  = f^*_{-\vec u^*} = \phi^\circ(f^*_{-\vec u^*})$$
 and $\phi^\circ(\RR^*_{-\vec u^*}) = 0$. 

 Suppose that we are in case (F1). We have
   $\phi^\circ(\ic{-\vec u^+, \vec u^*} \ic{-\vec u^*, -\vec u^*}) = \pi_+\epsilon  = \left(\io{\vec u^+, -\vec u^+}, 0\right) $
    and 
   $\phi^\circ(f^*_{-\vec u^*}) = \phi^\circ(\io{\vec u^+, \vec u^*}) = \io{\vec u^+, -\vec u^+} \pi_+$
   and therefore $\phi^\circ(\RR_{-\vec u^*}) = 0$ in this case. We used $\psi(\io{\vec u^+, -\vec u^+}) = \io{\vec u^+, -\vec u^+}$, see in particular Case (i) of Proof of Lemma \ref{tech3}. In case (F1'), we have $\phi^\circ(\ic{-\vec u^+, \vec u^*} \ic{-\vec u^*, -\vec u^*}) = 0 = f^*_{-\vec u^*}$ so $\phi^\circ(\RR^*_{-\vec u^*}) = 0$.

 Let us now consider $\RR^*_{\vec u^*}$: in case (F3), this is similar as before.

 Suppose that we are in case (F1). We have
   \begin{align*} \phi^\circ(\ic{-\vec u^*, -\vec u^*} \ic{\vec u^*, \vec u^+} ) &= \epsilon \theta_+ = \left(\psi(\io{\vec u^+, -\vec u^+} \cdot \ic{-\vec u^+, \vec u^{*+}}), 0\right) = \left(e_{\vec u^+} C_\sigma, 0\right) \\
     \text{and} \quad \phi^\circ(f^*_{-\vec u^*}) &= \phi^\circ(\io{\vec u^*, -\vec u^+}) = \theta_+ \io{\vec u^+, -\vec u^+} = \left(e_{\vec u^+} C_\sigma, 0\right)\end{align*} so $\phi^\circ(\RR_{\vec u^*}) = 0$. In case (F1'), $\phi^\circ(\ic{-\vec u^*, -\vec u^*} \ic{\vec u^*, \vec u^+}) = 0 = f^*_{\vec u^*}$ so the conclusion is immediate.

 In case (F2), we have
   $$\phi^\circ(\ic{\vec u^*, -\vec u^*}^2) = \eta^2 = \left(0, \ic{\vec u, -\vec u}^2 \right) = \left(0, f_{\vec u}\right)$$
   and the only case where $f_{\vec u} \neq 0$ is the same that the case where $f^*_{\vec u^*} \neq 0$: Case d of figure \ref{polnonz} with $m_M = 2$ (as we excluded $m_M = 1$ in this case). In this case,
   \begin{align*}\phi^\circ(f^*_{\vec u^*}) &= \phi^\circ(\lambda_M^* e_{u^*} C_{\sigma^*}) = -\lambda_M \eta \theta_- \io{\vec u^{*+}, -\vec u^+} \pi_+ \\ &= \left(\begin{bmatrix} -\lambda_M e_{\vec u^+} C_\sigma & 0 \\ 0 & 0 \end{bmatrix}, 0 \right) \end{align*}
   which is homotopic to $(0, f_{\vec u}) = (0, \lambda_M e_{u} C_{\sigma})$ via the homotopy
   $$ \begin{bmatrix} \lambda_M \io{\vec u^+, \vec u} - \tilde \lambda \io{\vec u^+, -\vec u} \\ 0 \end{bmatrix}: e_u \Delta_\sigma \to e_{\vec u^+} \Delta_\sigma \oplus e_{\vec u^+} \Delta_\sigma.$$
  where $\tilde \lambda = \lambda_\M$ if the digon with sides $\vec u^+$ and $\vec u^{*+}$ contains a unique puncture $N$ and $m_{s(\vec u)} = m_{s(\vec u^*)} = m_N = 1$, and $\tilde \lambda = 0$ in any other case.
\end{proof}

\begin{lemma} \label{mostrel}
 Let $x \in k Q_{\sigma^*} e_\tau$. If $x$ vanishes in $\Delta^{\lambda^*}_{\sigma^*}$ then $\phi^\circ(x) = 0$.
\end{lemma}

\begin{proof}
 First of all, if $x \in k Q_{\tau}$, it is immediate as $\phi^\circ(x) = \psi(x)$. 

 Let $x \in e_{u^*} k Q_{\sigma^*} e_\tau$ such that $x$ vanishes in $\Delta^{\lambda^*}_{\sigma^*}$. We have $\ic{-\vec u^+, \vec u^*} x = x^+ + y^+$ and $\ic{-(-\vec u)^+, -\vec u^*} x = x^- + y^-$ with $x^+, x^- \in k Q_\tau$ and $y^+, y^- \in (\RR^*_{\vec u^*}, \RR^*_{-\vec u^*})$. It is immediate that the following map is injective:
 \begin{align*} \bar \pi: \Hom_{\Kb(\proj \Delta_\sigma)} (\Delta_{\tau}, P_u^*) &\to \Hom_{\Kb(\proj \Delta_\sigma)} (\Delta_\tau, e_{\vec u^+} \Delta_\sigma \oplus e_{(-\vec u)^+} \Delta_\sigma)  \\
   f &\mapsto (\pi_+ f, \pi_- f)
 \end{align*}
 (notice that in case (F1), $\pi_- = 0$ and in case (F2), $\pi_- = \pi_+ \eta$). Moreover, 
  \begin{align*} \bar \pi(\phi^\circ(x)) &= (\phi^\circ(\ic{-\vec u^+, \vec u^*} x), \phi^\circ(\ic{-(-\vec u)^+, -\vec u^*} x)) \\ &= (\phi^\circ(x^+) + \phi^\circ(y^+), \phi^\circ(x^-) + \phi^\circ(y^-)) = (0,0)\end{align*}
 where we used Lemma \ref{Rus}. So $\phi^\circ(x) = 0$.
\end{proof}

\begin{lemma} \label{defphi}
 The morphism $\phi^\circ: k Q_{\sigma^*} \to \End_{\Kb(\proj \Delta_\sigma)} (T)$ induces a surjective morphism $\phi: \Delta^{\lambda^*}_{\sigma^*} \to \End_{\Kb(\proj \Delta_\sigma)} (T)$.
\end{lemma}

\begin{proof}
 The surjectivity comes from Lemma \ref{phisurj}. For the well-definition, we need to check that relations defining $\Delta^{\lambda^*}_{\sigma^*}$ are mapped to $0$. Using Lemmas \ref{relC}, \ref{Rus} and \ref{mostrel}, we need to check the $\RR^*_{\vec v}$'s which are in $\Delta^{\lambda^*}_{\sigma^*} e_{\vec u^*}$. In other terms, $\vec v = -\vec u^+$ always and $\vec v = -(-\vec u)^+$ in case (F3). By symmetry, it is enough to check $\RR^*_{-\vec u^+}$.

 We prove that $\phi^\circ(\RR^*_{-\vec u^+}) = 0$ case by case:
  \begin{enumerate}[\rm (a)]
   \item In case (F1), 
    $$\phi^\circ(\ic{\vec u^*, \vec u^+} \ic{-\vec u^+, \vec u^*}) = \theta_+ \pi_+ = (\ic{-\vec u^+, \vec u^{*+}} - \alpha, 0)$$
    which is homotopic, using $\ic{-\vec u^+, \vec u} : e_u \Delta_\sigma \to e_{\vec u^+} \Delta_\sigma$, to $(-\alpha, -f_{-\vec u^+})$. Moreover,
    \begin{align*}\phi^\circ(f^*_{-\vec u^+}) &= \phi^\circ(\io{-\vec u^*, -\vec u^*}) = -\lambda_{t(\vec u)} \epsilon^{m_{t(\vec u)}-1} \\&= -(\lambda_{t(\vec u)} \delta_{m_{t(\vec u)}, 2} \io{\vec u^+, -\vec u^+}, f_{-\vec u^+}) = (-\alpha, -f_{-\vec u^+}).\end{align*}
   \item In case (F1'),
     $$\phi^\circ(\ic{-\vec u^{*+}, \vec u^+} \ic{-\vec u^+, \vec u^*}) = \psi(\ic{-\vec u^{*+}, \vec u^+}) \pi_+ = \left(\psi(\ic{-\vec u^{*+}, \vec u^+}) , 0 \right).$$
     The homotopy $\ic{-\vec u^{*+}, \vec u} : e_u \Delta_\sigma \to e_{\vec u^{*+}} \Delta_\sigma$ gives that $(\ic{-\vec u^{*+}, \vec u^+}, 0) = 0$. So, if $m_{t(\vec u)} > 1$, $\phi^\circ(\ic{-\vec u^{*+}, \vec u^+} \ic{-\vec u^+, \vec u^*})$ is homotopic to $0$. In this case, we also have $\phi^\circ(f^*_{-\vec u^+}) = 0$.
     If $m_{t(\vec u)} = 1$, $\phi^\circ(\ic{-\vec u^{*+}, \vec u^+} \ic{-\vec u^+, \vec u^*})$ is homotopic to $(-\nu_\M^{-1} \lambda_{t(\vec u)} \io{\vec u^{*+}, -\vec u^+}, 0)$ (see Proof of Lemma \ref{tech3}). 
    Moreover, we have 
    \begin{align*}\phi^\circ(f^*_{-\vec u^+}) &= \phi^\circ(\lambda_{t(\vec u)}^* \io{\vec u^{*+}, \vec u^*}) = -\lambda_{t(\vec u)} \nu_\M^{-1} \io{\vec u^{*+}, -\vec u^+} \pi_+ \\ &= (-\nu_\M^{-1} \lambda_{t(\vec u)} \io{\vec u^{*+}, -\vec u^+}, 0).\end{align*}
   \item In cases (F2) or (F3),
     \begin{align*}\phi^\circ(\ic{-(-\vec u^*)^+, \vec u^+} \ic{-\vec u^+, \vec u^*}) &= \ic{-(-\vec u^*)^+, \vec u^+} \pi_+ \\ &= \left( \begin{bmatrix} \ic{-(-\vec u^*)^+, \vec u^+} & 0 \end{bmatrix}, 0 \right).\end{align*}
    The homotopy $\ic{-(-\vec u^*)^+, \vec u} : e_u \Delta_\sigma \to e_{(-\vec u^*)^+} \Delta_\sigma$ gives $$(\begin{bmatrix} \ic{-(-\vec u^*)^+, \vec u^+} & f_{-\vec u} \end{bmatrix}, 0) = 0, \quad \text{so,}$$
    \begin{align*} 
     & \phi^\circ(\ic{-(-\vec u^*)^+, \vec u^+} \ic{-\vec u^+, \vec u^*}) = \left( \begin{bmatrix} 0 & -f_{-\vec u} \end{bmatrix}, 0 \right) = f_{-\vec u} \pi_- \\ =\,& \phi^\circ(\io{(-\vec u^*)^+, -(-\vec u)^+} \cdot \ic{-(-\vec u^+), -\vec u^*}) = \phi^\circ(\io{(-\vec u^*)^+, -\vec u^*}) =\phi^\circ(f^*_{-\vec u^+}). \qedhere
    \end{align*}
  \end{enumerate}
\end{proof}

The next lemma concludes Proof of Proposition \ref{endo}:

\begin{lemma} \label{surjphi}
 The morphism $\phi: \Delta^{\lambda^*}_{\sigma^*} \surj \End_{\Kb(\proj \Delta_\sigma)}(T)$ is injective.
\end{lemma}

\begin{proof}
 Let $x \in \soc \ker \phi$. We have $x = \lambda e_v C_{\sigma^*}$ for some $v \in \sigma^*$ and $\lambda \in k$, thanks to Theorem \ref{basisDsig}. As $\phi$ coincide with $\psi$ which is injective on $\Delta^{\lambda^*}_\tau$, we get that $v = u^*$. Using Proof of Lemma \ref{relC}, we get
  $$\phi(e_{u^*} C_{\sigma^*}) = \left\{\begin{array}{ll}
				 \left(e_{\vec u^+} C_\sigma, 0\right) & \text{in case (F1) or (F1');} \\
                                 \left(\begin{bmatrix} e_{\vec u^+} C_\sigma & 0 \\ 0 & 0 \end{bmatrix}, 0\right) & \text{in case (F2);} \\
                                 \left(\begin{bmatrix} e_{\vec u^+} C_\sigma & 0 \\ -f^*_{\vec u^*} \io{\vec u^{*+}, -\vec u^+} & 0 \end{bmatrix}, 0\right) & \text{in case (F3).}
                                \end{array}\right.$$
 and in every case, using the trace $\Er^*$ of Definition \ref{trace} (see Lemma \ref{commC2}), we get that $\phi(\lambda e_{u^*} C_{\sigma^*}) = 0$ implies $\lambda = 0$ (recall that a trace on $\Delta_\sigma$ induces a trace on endomorphism rings of $\Kb(\proj \Delta_\sigma)$ by alternate sum of diagonal terms).
\end{proof}

\subsection{Proof of Lemma \ref{initcase}} \label{proofinitcase}
 We start by naming some elements of $\Delta_\sigma$. Denote \begin{align*} \ic{\vec u, -\vec u}' &= \ic{\vec u, -\vec u} - \lambda_M e_u \\ \ic{-\vec w, -\vec v}' &= \nu_\M^{-1} (\ic{-\vec w, -\vec v} - \lambda_M \io{\vec w, \vec v}).\end{align*} Then for any other arrow $\ic{\vec s, \vec t}$ of $Q_\sigma$, denote $\ic{\vec s, \vec t}' = \ic{\vec s, \vec t}$ and extend this notation as before for any pair of oriented side starting at the same point. Finally, denote $\io{\vec s, \vec s}' = \lambda_{s(\vec s)}'\ic{\vec s, \vec s}'^{m_{s(\vec s)}-1}$ for any $\vec s \in \sigma$ and $\io{\vec s, \vec t}' = \io{\vec s, \vec s}' \ic{\vec s, \vec t}'$ for any $\vec t \neq \vec s$ in $\sigma$ such that $s(\vec s) = s(\vec t)$. 

 Let us denote $\lambda_v = \lambda_P$ if $-\vec v$ encloses a special monogon with special puncture $P$ and $\lambda_v = 0$ in any other case. Denote also $\lambda_w = \lambda_Q$ if $\vec w$ encloses a special monogon with special puncture $Q$ and $\lambda_w = 0$ in any other case. Finally, denote $\tilde \lambda = \lambda_{s(\vec u)}$  if $m_{s(\vec u)} = 1$ and $\tilde \lambda = 0$ else. Notice that $\lambda_M \lambda_v \lambda_w \tilde \lambda = 1 - \nu_\M$. 

  Then we prove the following identities in $\Delta_\sigma$:
 \begin{enumerate}[\rm (a)]
  \item $\ic{\vec w, \vec v}' = \nu_\M^{-1}  (\ic{\vec w, \vec v} - \lambda_M \lambda_w \lambda_v \io{\vec w, \vec v})$. Indeed, using Proposition \ref{altpres0} and $C_\sigma J = 0$,
    \begin{align*}
      \ic{\vec w, \vec v}' &=  \ic{\vec w, -\vec w} \ic{-\vec w, -\vec v}' \ic{-\vec v, \vec v} \\ &=  \nu_\M^{-1} \ic{\vec w, -\vec w}  (\ic{-\vec w, -\vec v} - \lambda_M \io{\vec w, \vec v}) \ic{-\vec v, \vec v} \\
     &= \nu_\M^{-1}  (\ic{\vec w, \vec v} - \lambda_M \ic{\vec w, -\vec w} \cdot \io{\vec w, \vec v} \cdot \ic{-\vec v, \vec v}) \\
     &= \nu_\M^{-1}  (\ic{\vec w, \vec v} - \lambda_M \ic{\vec w, -\vec w}^2 \cdot \io{-\vec w, -\vec v} \cdot \ic{-\vec v, \vec v}^2) \\
     &= \nu_\M^{-1}  (\ic{\vec w, \vec v} - \lambda_M \lambda_w \lambda_v \ic{\vec w, -\vec w} \cdot \io{-\vec w, -\vec v} \cdot \ic{-\vec v, \vec v}) \\
     &= \nu_\M^{-1}  (\ic{\vec w, \vec v} - \lambda_M \lambda_w \lambda_v \io{\vec w, \vec v}).
    \end{align*}
  \item $\ic{\vec v, \vec v}' - \ic{\vec v, \vec v} \in (C_\sigma)$. We have
   \begin{align*}
    \ic{\vec v, -\vec u} \cdot \io{-\vec u, \vec v} &= e_v C_\sigma \in (C_\sigma) \\
     \text{and} \quad \ic{\vec v, \vec u} \ic{-\vec u, \vec v} &= \ic{\vec v, \vec u} \ic{-\vec u, \vec w} \ic{\vec w, \vec v} = \io{-\vec v, -\vec w} \cdot \ic{\vec w, \vec v} \\ &= \io{-\vec v, \vec w} \cdot \ic{\vec w, -\vec w}^2 \ic{-\vec w, \vec v} \\ &= \lambda_w \io{-\vec v, \vec w} \cdot \ic{\vec w, -\vec v} \ic{-\vec v, \vec v} = \lambda_w C_\sigma \ic{-\vec v, \vec v} \in (C_\sigma)
   \end{align*}
  So, using (a), we have, modulo $(C_\sigma)$,
   \begin{align*}
    \ic{\vec v, \vec v}' &= \ic{\vec v, \vec u} \ic{\vec u, -\vec u}' \ic{-\vec u, \vec v}' \\
	&=  \nu_\M^{-1} \ic{\vec v, \vec u} (\ic{\vec u, -\vec u} - \lambda_M e_u )   (\ic{-\vec u, \vec v} - \lambda_M \lambda_w \lambda_v \io{-\vec u, \vec v}) 
        = \nu_\M^{-1} \ic{\vec v, \vec v}
   \end{align*}
  and, if $\nu_\M \neq 1$, $\ic{\vec v, \vec v} \in (C_\sigma)$ so the result follows.
  \item $\io{\vec v, \vec v}' - \nu_\M^{-1} \io{\vec v, \vec v} \in (C_\sigma)$. This follows from (b) and $\lambda_{s(\vec v)}' = \nu_\M^{-1} \lambda_{s(\vec v)}$.
  \item $\io{\vec v, \vec v}^2 - \tilde \lambda \io{\vec v, \vec v} \in (C_\sigma)$. It is an easy computation:
   $$\io{\vec v, \vec v}^2 = \lambda_{s(\vec v)}^2 e_v \ic{\vec v, \vec v}^{2 m_{s(\vec v)} - 2} = \left\{\begin{array}{ll} \lambda_{s(\vec v)}^2 e_v = \lambda_{s(\vec v)} \io{\vec v, \vec v} & \text{if $m_{s(\vec v)} = 1$;} \\ \lambda_{s(\vec v)} e_v C_\sigma & \text{if $m_{s(\vec v)} = 2$;} \\ 0 & \text{else.}\end{array} \right.$$
  \item $\io{\vec w, \vec v}' = \nu_\M^{-1}  \io{\vec w, \vec v}$. Using (a), (c), (d) and $C_\sigma J = 0$,
    \begin{align*}
     \io{\vec w, \vec v}' &= \ic{\vec w, \vec v}' \cdot \io{\vec v, \vec v}' = \nu_\M^{-2} (\ic{\vec w, \vec v} - \lambda_M \lambda_w \lambda_v \io{\vec w, \vec v})\io{\vec v, \vec v} \\
      &= \nu_\M^{-2} \ic{\vec w, \vec v} (\io{\vec v, \vec v} - \lambda_M \lambda_w \lambda_v \io{\vec v, \vec v}^2) \\ &= \nu_\M^{-2}  (1-\lambda_M \lambda_w \lambda_v \tilde \lambda) \ic{\vec w, \vec v} \cdot \io{\vec v, \vec v} = \nu_\M^{-1}  \io{\vec w, \vec v}.
    \end{align*}
   \item $\ic{-\vec u, \vec w}' \ic{-\vec w, -\vec v}' = \io{\vec u, \vec v}'$. Indeed, using (e),
    \begin{align*}\ic{-\vec u, \vec w}' \ic{-\vec w, -\vec v}' &= \nu_\M^{-1} \ic{-\vec u, \vec w}  (\ic{-\vec w, -\vec v} - \lambda_M \io{\vec w, \vec v}) \\ &= \nu_\M^{-1} (\io{\vec u, \vec v} - \lambda_M \io{-\vec u, \vec v}) = \ic{\vec u, \vec w}' \cdot \io{\vec w, \vec v}' = \io{\vec u, \vec v}'.\end{align*}
   \item $\ic{-\vec w, -\vec v}' \ic{\vec v, \vec u}' = \io{\vec w, -\vec u}'$. This is similar as (f).
   \item $\io{-\vec v, -\vec w}' = \io{-\vec v, -\vec w} = \ic{\vec v, \vec u}' \ic{-\vec u, \vec w}'$. By (c) and (d),
    \begin{align*}
     \io{-\vec v, -\vec w}' &= \nu_\M^{-1} \ic{-\vec v, \vec v} \cdot \io{\vec v, \vec v} \cdot \ic{\vec v, -\vec w}' \\
	      &= \nu_\M^{-1} \io{-\vec v, \vec v} (\ic{\vec v, -\vec w} - \lambda_M \io{-\vec v, -\vec w}\cdot \ic{\vec w, -\vec w}) \\
	      &= \nu_\M^{-1}  (\io{-\vec v, -\vec w} - \lambda_M \io{-\vec v, -\vec v} \cdot \ic{-\vec v, \vec v}^2 \cdot \io{\vec v, \vec w} \cdot \ic{\vec w, -\vec w}^2) \\
	      &= \nu_\M^{-1}  (\io{-\vec v, -\vec w} - \lambda_M \lambda_v \lambda_w \io{-\vec v, -\vec v}^2 \cdot \ic{-\vec v, -\vec w} ) \\
              &= \nu_\M^{-1}  (\io{-\vec v, -\vec w} - \lambda_M \lambda_v \lambda_w \tilde \lambda \io{-\vec v, -\vec w} ) = \io{-\vec v, -\vec w}
    \end{align*}
    and the second equality is trivial.
   \item $\io{\vec v, \vec x}' = \io{\vec v, \vec x}$ if $s(\vec x) = s(\vec v)$ and $t(\vec x)$ is enclosed by $-\vec v$. By (c),
    \begin{align*}
     \io{\vec v, -\vec v}' &= \nu_\M^{-1} \io{\vec v, \vec v} \cdot \ic{\vec v, \vec u} \ic{\vec u, -\vec u}' \ic{-\vec u, \vec w} \ic{\vec w, -\vec v}' \\
		  &= \nu_\M^{-1} \io{\vec v, \vec v} \cdot (\ic{\vec v, \vec w} - \lambda_M \io{-\vec v, -\vec w}) \ic{\vec w, -\vec v}' \\
                  &= \nu_\M^{-1} \io{\vec v, \vec v} \cdot (\ic{\vec v, -\vec w} - \lambda_M \io{-\vec v, \vec w} \cdot \ic{\vec w, -\vec w}^2) \ic{-\vec w, -\vec v}' \\
		  &= \nu_\M^{-2} \io{\vec v, \vec v} \cdot (\ic{\vec v, -\vec w} - \lambda_M \lambda_w \io{-\vec v, -\vec w}) (\ic{-\vec w, -\vec v} - \lambda_M \io{\vec w, \vec v}) \\
		  &= \nu_\M^{-2} \io{\vec v, \vec v} \cdot (\ic{\vec v, \vec w} - \lambda_M \lambda_w \io{-\vec v, \vec w}) (\ic{\vec w, -\vec v} - \lambda_M \ic{\vec w, -\vec w}^2 \cdot \io{-\vec w, \vec v}) \\
                  &= \nu_\M^{-2} \io{\vec v, \vec v} \cdot (\ic{\vec v, \vec w} - \lambda_M \lambda_w \io{-\vec v, \vec w}) (\ic{\vec w, -\vec v} - \lambda_M \lambda_w  \io{\vec w, \vec v})
    \end{align*}
    and by (d), we have
    \begin{align*}\io{-\vec v, \vec w} \cdot \io{\vec w, \vec v} &= \io{-\vec v, -\vec v} \cdot \ic{-\vec v, \vec w} \cdot \io{\vec w, \vec w} \cdot \ic{\vec w, \vec v} = \io{-\vec v, -\vec v}^2 \cdot \ic{-\vec v, \vec w} \ic{\vec w, \vec v} \\&= \tilde \lambda \io{-\vec v, -\vec v} \cdot \ic{-\vec v, -\vec v} \ic{-\vec v, \vec v} = \tilde \lambda e_v C_\sigma \ic{-\vec v, \vec v} = \tilde \lambda \lambda_v e_v C_\sigma\end{align*}
    so 
    \begin{align*}
     \io{\vec v, -\vec v}' &= \nu_\M^{-2} \io{\vec v, \vec v} \cdot (\ic{\vec v, -\vec v} - \lambda_M \lambda_w(2 - \lambda_M \lambda_w \tilde \lambda \lambda_v  ) e_v C_\sigma) \\
             &= \nu_\M^{-2} (\io{\vec v, -\vec v} - \lambda_M \lambda_w \tilde \lambda (1 + \nu_\M) e_v C_\sigma)
    \end{align*}
    Moreover, we have $C_\sigma \ic{-\vec v, \vec x} = \io{-\vec v, \vec x} \cdot \ic{\vec x, \vec v} \ic{\vec v, \vec x} = \lambda_v \io{\vec v, \vec x}$ so
    \begin{align*}
     \io{\vec v, \vec x}' &= \nu_\M^{-2} (\io{\vec v, -\vec v} - \lambda_M \lambda_w \tilde \lambda (1 + \nu_\M) e_v C_\sigma) \ic{-\vec v, \vec x} \\
			&= \nu_\M^{-2} (1 - \lambda_M \lambda_w \tilde \lambda \lambda_v (1 + \nu_\M) ) \io{\vec v, \vec x} \\
                       &= \nu_\M^{-2} (1 - (1 - \nu_\M) (1 + \nu_\M) ) \io{\vec v, \vec x} = \io{\vec v, \vec x}.
    \end{align*}
   \item $\io{\vec x, -\vec v}' = \io{\vec x, -\vec v}$ if $s(\vec x) = s(\vec v)$ and $t(\vec x)$ is enclosed by $-\vec v$. Same as (i).
   \item $\io{-\vec w, \vec x}' = \io{-\vec w, \vec x}$ if $s(\vec x) = s(\vec w)$ and $t(\vec x)$ is enclosed by $\vec w$. Same as (i).
   \item $\io{\vec x, \vec w}' = \io{\vec x, \vec w}$ if $s(\vec x) = s(\vec w)$ and $t(\vec x)$ is enclosed by $\vec w$. Same as (i).
   \item $\lambda'_{s(\vec x)} \ic{\vec x, \vec x}'^{m_{s(\vec x)}} = e_x C_\sigma$ for any $\vec x \in \sigma$ such that $s(\vec x) = s(\vec u)$ and $\vec x \neq \pm \vec u$. If $t(\vec x)$ is enclosed by $-\vec v$ then using (i), we have
    $$\lambda'_{s(\vec x)} \ic{\vec x, \vec x}'^{m_{s(\vec x)}} = \ic{\vec x, \vec v} \cdot \io{\vec v, \vec x}' = \ic{\vec x, \vec v} \cdot \io{\vec v, \vec x} = e_x C_\sigma.$$
   This is analogous if $t(\vec x)$ is enclosed by $\vec w$. If $\vec x = \vec v$, take $\vec y$ such that $t(\vec y)$ is enclosed by $-\vec v$. Thanks to (i), we have
    $$\lambda'_{s(\vec v)} \ic{\vec v, \vec v}'^{m_{s(\vec v)}} = \io{\vec v, \vec y}' \cdot \ic{\vec y, \vec v} = \io{\vec v, \vec y} \cdot \ic{\vec x, \vec v} = e_v C_\sigma.$$
    This is analogous if $\vec x = -\vec v$, $\vec x = \vec w$ or $\vec x = -\vec w$.
   \item $\ic{\vec u, \vec u}'^{m_{s(\vec u)}} = \ic{-\vec u, -\vec u}'^{m_{s(\vec u)}}$. It follows from (e).
 \end{enumerate}
 
 From these identities, we deduce that the following map is a morphism of algebras:
 \begin{align*}
  \phi: \Delta_\sigma^\mu &\to \Delta_\sigma^\lambda \\
        e_x &\mapsto e_x & \text{for $x \in \sigma$}; \\
       \ic{\vec x, \vec y} &\mapsto \ic{\vec x, \vec y}' & \text{for $\ic{\vec x, \vec y} \in Q_{\sigma, 1}$}.
 \end{align*}
 Indeed, relations of the form $C_{\vec x}$ for $\Delta^\mu_\sigma$ are mapped to $0$ by $\phi$ because of (m) and (n) if $s(\vec x) = s(\vec v)$ or $t(\vec x) = s(\vec v)$ and trivially otherwise. The relation coming from the special monogon enclosed by $\vec u$ is mapped to $0$ easily. Relations coming from the triangle $-\vec u$, $\vec v$, $-\vec w$ are mapped to $0$ thanks to (f), (g) and (h). Relations $R_{P, n}$ coming from minimal polygons $P$ completely enclosed by $-\vec v$ or by $\vec w$ are mapped to $0$ thanks to (i), (j), (k) and (l) which permit to identify external paths winding around $s(\vec v)$.

 If we denote by $\psi: \Delta^\lambda_\sigma \to \Delta^\mu_\sigma$ the morphism obtained similarly, it is easy to prove that $\phi$ and $\psi$ are inverse of each other by using (e), $\mu_M = -\lambda_M$ and $\nu_\M' = \nu_\M^{-1}$ where $\nu_\M'$ is computed for $\mu$. \qed

\bibliographystyle{alphanum}
\bibliography{../biblio/biblio}

\end{document}